\documentclass[reqno]{amsart}
\usepackage{amsfonts}
\usepackage[all]{xy}
\usepackage{graphicx}
\usepackage{amssymb}
\usepackage{amsmath}
\usepackage{mathrsfs}
\usepackage{epsfig}
\usepackage{amscd}
\usepackage{graphicx, color}
\usepackage{caption}
\usepackage{tikz}

\usetikzlibrary{trees,snakes,mindmap,shapes,backgrounds,matrix,shadows}

\def\E{\ifmmode{\mathbb E}\else{$\mathbb E$}\fi}
\def\N{\ifmmode{\mathbb N}\else{$\mathbb N$}\fi}
\def\R{\ifmmode{\mathbb R}\else{$\mathbb R$}\fi}
\def\Q{\ifmmode{\mathbb Q}\else{$\mathbb Q$}\fi}
\def\C{\ifmmode{\mathbb C}\else{$\mathbb C$}\fi}
\def\H{\ifmmode{\mathbb H}\else{$\mathbb H$}\fi}
\def\Z{\ifmmode{\mathbb Z}\else{$\mathbb Z$}\fi}
\def\P{\ifmmode{\mathbb P}\else{$\mathbb P$}\fi}
\def\T{\ifmmode{\mathbb T}\else{$\mathbb T$}\fi}
\def\SS{\ifmmode{\mathbb S}\else{$\mathbb S$}\fi}
\def\DD{\ifmmode{\mathbb D}\else{$\mathbb D$}\fi}

\newcommand{\e}{\varepsilon}

\newcommand{\ben}{\begin{enumerate}}
\newcommand{\een}{\end{enumerate}}
\newcommand{\be}{\begin{equation}}
\newcommand{\ee}{\end{equation}}
\newcommand{\bea}{\begin{eqnarray}}
\newcommand{\eea}{\end{eqnarray}}
\newcommand{\beastar}{\begin{eqnarray*}}
\newcommand{\eeastar}{\end{eqnarray*}}
\newcommand{\bc}{\begin{center}}
\newcommand{\ec}{\end{center}}

\theoremstyle{theorem}
\newtheorem{thm}{Theorem}[section]
\newtheorem{cor}[thm]{Corollary}
\newtheorem{lem}[thm]{Lemma}
\newtheorem{prop}[thm]{Proposition}

\theoremstyle{definition}
\newtheorem{defn}[thm]{Definition}
\newtheorem{rem}[thm]{Remark}

\newtheorem{choice}[thm]{Choice}
\newtheorem{cond}[thm]{Condition}
\newtheorem{hypo}[thm]{Hypothesis}

\newtheorem*{thm*}{Theorem}
\numberwithin{equation}{section}
\def\R{{\mathbb R}}

\def\E{{\mathbb E}}
\def\Z{{\mathbb Z}}
\def\C{{\mathbb C}}
\def\R{{\mathbb R}}
\def\P{{\mathbb P}}

\def\N{{\mathbb N}}

\def\11{{\mathbb I}}

\def\C{\mathbb{C}}
\def\Z{\mathbb{Z}}

\def\T{\mathbb{T}}

\def\Q{\mathbb{Q}}
\def\E{\ifmmode{\mathbb E}\else{$\mathbb E$}\fi}
\def\N{\ifmmode{\mathbb N}\else{$\mathbb N$}\fi}
\def\R{\ifmmode{\mathbb R}\else{$\mathbb R$}\fi}
\def\Q{\ifmmode{\mathbb Q}\else{$\mathbb Q$}\fi}
\def\C{\ifmmode{\mathbb C}\else{$\mathbb C$}\fi}
\def\H{\ifmmode{\mathbb H}\else{$\mathbb H$}\fi}
\def\Z{\ifmmode{\mathbb Z}\else{$\mathbb Z$}\fi}
\def\P{\ifmmode{\mathbb P}\else{$\mathbb P$}\fi}
\def\SS{\ifmmode{\mathbb S}\else{$\mathbb S$}\fi}
\def\DD{\ifmmode{\mathbb D}\else{$\mathbb D$}\fi}
\def\R{{\mathbb R}}

\def\E{{\mathbb E}}
\def\Z{{\mathbb Z}}
\def\C{{\mathbb C}}
\def\R{{\mathbb R}}

\def\N{{\mathbb N}}

\def\MM{{\mathcal M}}

\def\e{\varepsilon}

\def\CB{{\mathcal B}}

\def\darr#1{\raise1.5ex\hbox{$\leftrightarrow$}
\mkern-16.5mu #1}

\def\roughly#1{\raise.3ex\hbox{$#1$\kern-.75em
\lower1ex\hbox{$\sim$}}}

\def\opname#1{\mathop{\kern0pt{\rm #1}}\nolimits}

\def\func{\operatorname}

\begin{document}
\def\mq{\mathfrak{q}}
\def\mp{\mathfrak{p}}
\def\mH{\mathfrak{H}}
\def\mh{\mathfrak{h}}
\def\ma{\mathfrak{a}}
\def\ms{\mathfrak{s}}
\def\mm{\mathfrak{m}}
\def\mn{\mathfrak{n}}
\def\mz{\mathfrak{z}}
\def\mw{\mathfrak{w}}
\def\Hoch{{\tt Hoch}}
\def\mt{\mathfrak{t}}
\def\ml{\mathfrak{l}}
\def\mT{\mathfrak{T}}
\def\mL{\mathfrak{L}}
\def\mg{\mathfrak{g}}
\def\md{\mathfrak{d}}
\def\mr{\mathfrak{r}}

\title[Partial collapsing and adiabatic gluing]{Partial collapsing degeneration of
Floer trajectories and adiabatic gluing}
\author{Yong-Geun Oh}
\author{Ke Zhu}
\date{April 13, 2022}
\address{Center for Geometry and Physics, Institute for Basic Science (IBS),
77 Cheongam-ro, Nam-gu, Pohang-si, Gyeongsangbuk-do, Korea 790-784 \&
POSTECH, Gyeongsangbuk do, Korea}
\email{yongoh1@postech.ac.kr}
\thanks{The first named author is supported by the IBS project \# IBS-R003-D1.}
\address{Department of Mathematics and Statistics, Minnesota State
University, Mankato, MN 56001 }
\email{ke.zhu@mnsu.edu}

\begin{abstract}
In the present paper, we study partial collapsing degeneration of
Hamiltonian-perturbed Floer trajectories for an adiabatic $\varepsilon$%
-family and its reversal adiabatic gluing, as the prototype of the partial
collapsing degeneration of $2$-dimensional (perturbed) $J$-holomorphic maps
to $1$-dimensional gradient segments. We consider the case when the Floer
equations are $S^1$-invariant on parts of their domains whose \emph{%
adiabatic limits} have positive lengths as $\varepsilon \to 0$, which we call
\emph{thimble-flow-thimble} configurations. The main gluing theorem we prove
also applies to the case with Lagrangian boundaries such as in the problem
of recovering holomorphic disks out of pearly configurations. In particular,
our gluing theorem gives rise to a new direct proof of the chain isomorphism property
between the Morse-Bott version of Lagrangian intersection Floer complex of $L$ by
Fukaya-Oh-Ohta-Ono and the \emph{pearly complex} of $L$
by Lalonde and Biran-Cornea (for monotone Lagrangian submanifolds).
It also provides another proof
of the present authors' earlier proof of the isomorphism property of the PSS
map without involving the target rescaling and the scale-dependent gluing.
\end{abstract}

\maketitle
\tableofcontents

\section{Introduction}

The `partial collapsing degeneration' of holomorphic curves or $S^1$-invariant
Hamiltonian Floer trajectories and its reversal process naturally arise in
many moduli problems in symplectic geometry. (See \cite{fukaya:homotopy},
\cite{oh:newton}, \cite{PSS}, \cite{mschwarz0}, \cite{mundet-tian}, \cite%
{biran-cor-2}, \cite{Ek}, \cite{SW}, \cite{AB} for example).
However its
reversal process, recovering smooth Floer trajectories out of a mixture of 2
dimensional configuration and 1 dimensional one is a much harder problem and
has not been much studied. The main purpose of the present article
is to study this partial collapsing degeneration and its reversal process, which we call the
\emph{adiabatic gluing}. We take the methodology we adopt in the present
article as the prototype of the gluing problem for the partial collapsing degeneration of
smooth $2$-dimensional cylindrical objects to a singular hetero-dimensional
$2$-dimensional pearly objects.

\begin{rem}
Mathematical content of the present work is essentially the same as that of our previous arXiv
posting arXiv:1103.3525. An immediate consequence of our gluing theorem was Theorem
\ref{thm:monotone-intro} below. (There was another proof given by Fukaya-Oh-Ohta-Ono via the
homological perturbation theory \cite{fooo:canonical}.) Since the methodology of our gluing construction
may have other applications related to partial collapsing degeneration of pseudoholomorphic
curves, we have decided to rewrite the paper and make its fine details more accessible by
considerably improving its presentation, organization and readability.
\end{rem}

In the present authors' earlier work \cite{oh-zhu}, we study the adiabatic
degeneration and its reversal process, scale-dependent gluing problem, of
maps $u:{\mathbb{R}}\times S^{1}\rightarrow M$ satisfying the following
1-parameter ($0<\varepsilon <\varepsilon _{0}$) family of Floer equations:
\begin{equation}
(du+P_{K_{\varepsilon }}(u))_{J_{\varepsilon }}^{(0,1)}=0\quad
\mbox{ or
equivalently }\,{\overline{\partial }}_{J_{\varepsilon
}}(u)+(P_{K_{\varepsilon }})_{J_{\varepsilon }}^{(0,1)}(u)=0,  \label{eq:KJe}
\end{equation}
when the degenerating Hamiltonian
\begin{equation*}
K_{\varepsilon }:{\mathbb{R}}\times S^{1}\times M\rightarrow {\mathbb{R}}
\end{equation*}
is $S^1$-invariant, i.e., $t$-independent on a part of its domain so that it
has the form
\begin{equation}
K_{\varepsilon }(\tau ,t,x)=%
\begin{cases}
\kappa _{\varepsilon }^{+}(\tau )\cdot H(t,x)\quad & \mbox{for }\,\tau \geq
R(\varepsilon ) \\
\kappa_{\varepsilon }^0(\tau )\cdot \varepsilon f(x)\quad & \mbox{for }%
\,|\tau |\leq R(\varepsilon ) \\
\kappa _{\varepsilon }^{-}(\tau )\cdot H(t,x)\quad & \mbox{for }\,\tau \leq
-R(\varepsilon )%
\end{cases}
\label{eq:KR}
\end{equation}%
where $\kappa _{\varepsilon }^{\pm }$ and $\kappa_{\varepsilon }^0$ are
suitable cut-off functions (See Section \ref{sec:floer} for the precise
definition).

Roughly speaking, the partial collapsing degeneration arises because $K_{\varepsilon }
$ restricts to Morse function $\varepsilon f$ on longer and longer cylinder $%
[-R(\varepsilon ),R(\varepsilon )]\times S^{1}$ in ${\mathbb{R}}\times S^{1}$%
. A basic assumption that we had put on \cite{oh-zhu} was that $%
R(\varepsilon )$ satisfies
\begin{equation}  \label{eq:length=0}
\lim_{\varepsilon \rightarrow 0}\varepsilon R(\varepsilon )=0.
\end{equation}

What we mean by \emph{partial-collapsing limit} in the present paper describes the
configuration arising in the limit when the left hand side is positive.

Motivated by these, we introduce the following

\begin{defn}[Adiabatic length]
Assume the limit in \eqref{eq:length=0} exists. We call the limit
\begin{equation*}
\lim_{\varepsilon \rightarrow 0}\varepsilon R(\varepsilon )
\end{equation*}
the \emph{adiabatic length} of the $\varepsilon$-family.
\end{defn}

The main purpose of the present paper is to prove an adiabatic gluing result
when the adiabatic length of the $\varepsilon$-family is positive, i.e.,
\begin{equation*}
\lim_{\varepsilon \rightarrow 0}\varepsilon R(\varepsilon )=l
\end{equation*}%
for $l >0$. Under this assumption, it was proved in \cite{oh:dmj} and \cite%
{mundet-tian} that as $\varepsilon \rightarrow 0$, a degenerating sequence
of Floer trajectories converges to a \textquotedblleft \emph{%
thimble-flow-thimble\textquotedblright } configuration denoted by
$$
(u_{-},\chi ,u_{+}).
$$
Here $u_{\pm }$ are the maps whose domains are the thimble-like domains
which we realize as subsets of $\R^3$ by
\bea
\Theta_+& =  & D_+ \cup C_+ \label{eq:Theta+}\\
\Theta_- & =: & D_- \cup C_-  \label{eq:Theta-}
\eea
where
\beastar
D_\pm & : =  &\{(x,y,z) \in \R^3 \mid x^2 + y^2 + z^2 = 1, \, \pm z \leq 0\} \\
C_\pm & : = & \{(x,y,z) \in \R^3 \mid x^2 + y^2 = 1, \, \pm z \geq 0 \}
\eeastar
respectively. (The domains $\Theta_\pm$ are only $C^1$  along  the seam
as  they are. With a slight abuse of notation, we denote  their smoothings by the same letters.)

We then require $u_\pm$ to  satisfy the equation on $\Theta_\pm$ respectively
which are similar to \eqref{eq:KJe} with $K_{\varepsilon }$ replaced by $K_{\pm }$ defined by
\begin{equation}
K_+ =
\begin{cases}
0 & \text{near $o_{+}$ on $D_+$} \\
H_{+}(t,x)\,dt & \text{on $C_+$}
\end{cases}
\label{eq:Kplus}
\end{equation}
and
\begin{equation}
K_- =
\begin{cases}
H_{-}(t,x)\,dt & \text{on $C^-$}  \\
0 & \text{near $o_{-}$ on $D_-$} .
\end{cases}
\label{eq:Kminus}
\end{equation}%

Here $H_{\pm }:S^{1}\times M\rightarrow {\mathbb{R}}$ are two Hamiltonian
functions independent of the variable $\tau $, and $\chi $ is a gradient
trajectory of $f$ that satisfies $\dot{\chi}+\text{grad}f(\chi )=$ on $[-l,l ]
$ such that
\begin{equation}
u_{-}(o_{-})=\chi (-l),\quad \chi (l)=u_{+}(o_{+}).
\label{eq:nodalmathching}
\end{equation}

\subsection{Statements of problem and the main result}

Let $\dot{\Sigma}_+$ be the Riemann sphere with one marked point $o_+$ and
one positive puncture $e_+$. We choose analytical charts at $o_+$ and at $e_+
$ on some neighborhoods $O_+$ and $E_+$ respectively, so that conformally $%
O_+\backslash o_+ \cong (-\infty,0]\times S^1$, and $E_+\backslash e_+\cong
[0,+\infty)\times S^1$. We use $t$ for the $S^1$ coordinate and $\tau$ for
the ${\mathbb{R}}$ coordinate. Then $\{-\infty\}\times S^1$ and $%
\{+\infty\}\times S^1$ correspond to $o_+$ and $e_+$ respectively.

Let $z_\pm:S^1 \to M$ be a nondegenerate periodic orbit of $H_\pm$ and
consider a finite energy solution $u_\pm: \dot \Sigma \to M$ of the Floer
equation \eqref{eq:Kplus}, \eqref{eq:Kminus} associated to $K_\pm$
respectively. By the finite energy condition and since $K_\pm \equiv 0$ near
$o_\pm$, $u_\pm$ extend smoothly across $o_\pm$ and can be regarded as a
smooth map defined on ${\mathbb{C}}$ that is holomorphic near the origin $0
\in {\mathbb{C}}$.

Now we consider the lifting $[z_\pm,w_\pm]$ of $z_\pm$ and introduce the
main moduli spaces of our interest in a precise term, where $w_\pm: D^2 \to M
$ are discs with $\partial w_\pm = z_\pm$.

\begin{defn}\label{defn:pie(z)} For given embedded loop $z: S^1 \to M$, we
denote by $\pi_2(z)$ the set
of relative homotopy classes of maps $w: D^2 \to M$ with
$w(\partial D^2) = z$.
\end{defn}

Let $J_{\pm }$ be a pair of domain-dependent almost complex structures on $M$
and let $u_{\pm }:{\mathbb{R}}\times S^{1}\rightarrow M$ be solutions of %
\eqref{eq:KJe} with $K_{\varepsilon }$ replaced by $K_{\pm }$ given in %
\eqref{eq:Kplus}, \eqref{eq:Kminus} in class $A_{\pm }\in \pi _{2}(z_{\pm })$
with a marked point $o_{\pm }\in S^{2}$ respectively, and $\chi
:[-l,l]\rightarrow M$ is a gradient segment of the Morse function $f$
connecting the two points $u_{+}(o_{+})$ and $u_{-}\left( o_{-}\right) $. We
assume $J_{\pm }$ satisfy $J_{\pm }\equiv J_{0}$ near the marked points $o_{\pm }
$ respectively and generic in that $u_{\pm }$ are Fredholm regular. We also
assume that the pair $(J_{0},f)$ is generic in that the configuration $%
(u_{-},\chi ,u_{+})$ satisfies the \textquotedblleft
thimble-flow-thimble\textquotedblright\ transversality defined in \cite%
{oh-zhu} Proposition 5.2. We will recall this definition below.

We note that $\Theta_\pm$ are conformally isomorphic to $S^2 \setminus \{e_\pm, o_\pm\}$
where $e_\pm$ correspond to $\pm \infty$ in $C_\pm$ respectively. Under this identification,
one may regard the standard coordinates $(\tau,t)$ on $C_\pm$ the cylindrical coordinates
on a punctured neighborhood of $e_\pm \in S^2$ respectively.

For each given pair
$$
(A_-,A_+) \in \pi_2(z_-) \times \pi_2(z_+),
$$
we define the moduli spaces
\begin{eqnarray*}
\mathcal{M}\left( K_{\pm },J_{\pm };z_{\pm };A_{\pm }\right) &=&\Big\{\left(
u_{\pm },o_{\pm }\right) |u_{\pm }:S^{2}\setminus \{e_{\pm },o_{\pm
}\}\rightarrow M \mbox{ solutions of \eqref{eq:KJe}} \\
&{}&\quad \mbox{ with $K_{\varepsilon }$ replaced by $K_{\pm }$ given in %
	\eqref{eq:Kplus}, \eqref{eq:Kminus}
respectively}, \\
&{}& \quad  [A_\pm \# u_\pm] = 0 \, \text{ \rm in }\,
\pi_2(M) \Big\}
\end{eqnarray*}%
and consider the evaluation maps
\begin{equation*}
ev_{\pm }:\mathcal{M}\left( K_{\pm },J_{\pm };z_{\pm };A_{\pm }\right)
\rightarrow M,\text{ \ }u_{\pm }\rightarrow ev_{\pm }\left( o_{\pm }\right) .
\end{equation*}

\begin{defn}
\label{tft-trans} We say that the configuration $(u_{-},\chi ,u_{+})$
satisfies the \textquotedblleft thimble-flow-thimble\textquotedblright\
transversality if the map
\begin{eqnarray*}
\phi _{f}^{2l}\circ ev_{-}\times ev_{+} &:&\mathcal{M}\left(
K_{-},J_{-};z_{-};A_{-}\right) \times \mathcal{M}\left(
K_{+},J_{+};z_{+};A_{+}\right) \rightarrow M\times M \\
\left( u_{-},u_{+}\right) &\rightarrow &\left( \phi _{f}^{2l}u_{-}\left(
o_{-}\right) ,u_{+}\left( o_{+}\right) \right)
\end{eqnarray*}%
is transversal to the diagonal $\triangle \subset M\times M$, where $\phi
_{f}^{2l}$ is the time-$2l$ flow of the Morse function $f$.
\end{defn}

\begin{figure}[hbt!]
% thimble-flow-thimble picture
\centering
\begin{tikzpicture}	
		% u_+ thimble		
		\draw (4,0) ellipse (0.5 and 1);
		\draw (4,1) arc
		[
		start angle=90,
		end angle=270,
		x radius=2,
		y radius =1
		] ;
		
		\node [right] at (2.5,0)  {$u_+$};
		\node [right] at (4.5,0)  {$z_+$};
		
		% u_- thimble		
		\draw (-4,0) ellipse (0.5 and 1);
		\draw (-4,-1) arc
		[
		start angle=270,
		end angle=450,
		x radius=2,
		y radius =1
		] ;
		
		\node [left] at (-2.5,0)  {$u_-$};
		\node [left] at (-4.5,0)  {$z_-$};
		
		% \chi flow
		\draw [->] (-2,0) to [out=30, in=210 ] (2,0);
		\node [above] at (0,0) {$\chi$}; 			
	\end{tikzpicture}
\caption{thimble-flow-thimble}
\end{figure}
The \textquotedblleft thimble-flow-thimble\textquotedblright\ transversality
can be achieved by a generic choice of the pair $(J_{0},f)$ (Corollary \ref%
{achieve-tft-trans}).

\begin{defn}\label{defn:pi2z+z-}
Let $(z_-,z_+)$ be a pair of nonintersecting embedded loops. We define
$$
\pi_2(z_-,z_+)
$$
to be the relative homotopy classes of maps $w: [0,1] \times S^1 \to M$
satisfying
$$
w(0,t) = z_-(t), \quad w(1,t) = z_+(t).
$$
\end{defn}

We note that for each given pair $(A_{-},A_{+})\in \pi _{2}(z_{-})\times \pi
_{2}(z_{+})$ we consider and \emph{fix} a class $B \in \pi _{2}(z_{-},z_{+})$
such that
\begin{equation}  \label{eq:A-AA+}
A_-\# B \# A_+ = 0 \quad \mbox{in }\, \pi_2(M).
\end{equation}
(Such $B \in \pi _{2}(z_{-},z_{+})$ is unique modulo the action of $\pi_1$).
We define

\begin{equation*}
{\mathcal{M}}^{\text{\rm tft}}: = {\mathcal{M}}^{\text{\rm tft}}(K_{-},J_{-};f;K_{+},J_{+};z_-,z_+;B)
\end{equation*}
 to be the fiber product
\begin{equation}
\mathcal{M}\left( K_{-},J_{-};z_{-};A_{-}\right) {}_{\phi _{f}^{2l}\circ
ev_{-}}\times _{ev_{+}}\mathcal{M}\left( K_{+},J_{+};z_{+};A_{+}\right) .
\label{eq:K-fK+}
\end{equation}
Here the superscript \textquotedblleft tft" stands for \textquotedblleft
thimble-flow-thimble". This definition makes sense thanks to the exponential
convergence of finite energy solution of \eqref{eq:KJe} as $\tau \to \pm \infty$ when $z_\pm$ are
nondegenerate periodic orbits.

We denote by
\begin{equation*}
{\mathcal{M}}^{\varepsilon }:={\mathcal{M}}(K_{\varepsilon },J_{\varepsilon
};z_{-},z_{+};B)
\end{equation*}%
the moduli space of solutions (resolved Floer trajectories) to \eqref{eq:KJe}%
,
\begin{figure}[hbt!]
% resolved Floer trajectory picture
\centering
\begin{tikzpicture}
% Near u_+ thimble part		
\draw (4,0) ellipse (0.5 and 1);
\draw (1.7,0.1) to [out=30, in=180] (4,1);
\draw (1.7,-0.1) to [out=30, in=180] (4,-1);
\node [right] at (4.5,0)  {$z_+$};

% Near u_- thimble part		
\draw (-4,0) ellipse (0.5 and 1);
\draw (-1.7,0.1) to [out=210, in=0] (-4,1);
\draw (-1.7,-0.1) to [out=210, in=0] (-4,-1);
\node [left] at (-4.5,0)  {$z_-$};

% middle cylinder near \chi flow
\draw (-1.7,0.1) to [out=30, in=210 ] (1.7,0.1);
\draw (-1.7,-0.1) to [out=30, in=210 ] (1.7,-0.1);
\node [above] at (0,0.3) {$u_{\varepsilon}$};		
		\end{tikzpicture}
\caption{resolved Floer trajectory}
\end{figure}
and consider
\begin{equation*}
{\mathcal{M}}_{(0,\varepsilon _{0}]}^{\text{\textrm{para}}%
}:=\bigcup_{0<\varepsilon \leq \varepsilon _{0}}{\mathcal{M}}^{\varepsilon }.
\end{equation*}%

\begin{rem} We mention the canonical identification
$$
 {\mathcal{M}}(K_{\varepsilon },J_{\varepsilon};[z_{-},w_-],[z_{+},w_+] )
 = {\mathcal{M}}(K_{\varepsilon },J_{\varepsilon};z_{-},z_{+};B)
$$
is independent of the pairs $(w_-,w_+)$ satisfying $[w_- \# B \# w_+] = 0$ in $\pi_2(M)$.
\end{rem}

In this paper, we prove the following main theorem:

\begin{thm}
${\mathcal{M}}_{(0,\varepsilon _{0})}^{\text{\rm para}}$ can be embedded into a
manifold with boundary $\overline{\mathcal{M}}_{[0,\varepsilon _{0})}^{\text{\rm para}}$
whose boundary is diffeomorphic to ${\mathcal{M}}^{\text{\rm tft}}$, i.e., there is a diffeomorphism
\begin{equation*}
Glue:[0,\varepsilon _{0})\times {\mathcal{M}}^{\text{\rm tft}}\rightarrow \overline{\mathcal{M}}_{[0,\varepsilon _{0})}^{\text{\rm para}}
\end{equation*}such that the following diagram commutes:
\begin{equation*}
\xymatrix{\MM^{\text{\rm para}}_{(0,\e_0)} \ar@{^{(}->}[r] \ar[d] &
\overline\MM^{\text{\rm para}}_{[0,\e_0)} \ar[r]^{Glue^{-1}}\ar[d] & [0,\e_0)
\times \MM^{\text{\rm tft}} \ar[dl]\\ (0,\e_0)
\ar@{^{(}->}[r] & [0,\e_0) &{} }.
\end{equation*}
\end{thm}

The statement of this theorem is the analog to the corresponding theorem,
Theorem 10.19 \cite{oh-zhu} for the case $l =0$. However the gluing analysis
and its off-shell framework are of nature very different from those of \cite%
{oh-zhu}: In \cite{oh-zhu}, we need to \emph{rescale the target} near the
nodal point of nodal Floer trajectories, while in the present paper we do
not need to rescale the target but we need to \emph{renormalize} the domain
variable $(\tau ,t)$ to $(\tau /\varepsilon ,t/\varepsilon )$ to obtain the
limiting gradient flow of $f$ arising from the thin part of the degeneration
of ${\mathcal{M}}^{\varepsilon }$ as $\varepsilon \rightarrow 0$ as done in
\cite{FO}. The length $l $ of the limiting gradient trajectory is determined
by the limit $\lim_{\varepsilon \rightarrow 0}\varepsilon R(\varepsilon )=l $%
, i.e, by how fast the length of thin part of degenerating Floer
trajectories increases relative to the speed as $\varepsilon \rightarrow 0$.
We do not need to rescale the target $M$ unlike the case in \cite{oh-zhu}.
This is because the presence of gradient line with \emph{non-zero} length
carries enough information to recover all the solutions in ${\mathcal{M}}%
^{\varepsilon }$ \emph{with immersed join points} nearby the elements from ${%
\mathcal{M}}(K_{-}|f|K_{+};B)$ in Gromov-Hausdorff topology, as
long as $\varepsilon $ is sufficiently small. On the other hand, as $%
l\rightarrow 0$ the Sobolev constant in present paper blows up, which
obstructs the implicit function theorem (see Remark \ref{Sobolev-blow-up}),
and makes blowing up target manifold $M$ appear to be necessary as in \cite%
{oh-zhu}. \emph{In both cases, we need to assume that the Floer trajectories
either at the nodal points or at the join point of the gradient trajectories
are immersed.} It appears that without this assumption gluing analysis is
much more delicate. The immersion condition can be achieved for a generic
Floer datum under the small index conditions (Proposition \ref%
{immersed-joint}) which naturally enter in the construction of various
operations in Floer theory.

Incidentally we would like to point out that the main theorem can be also
used to give another proof of isomorphism property of the PSS map which uses
a piecewise smooth cobordism different from that of \cite{oh-zhu}. Namely we
directly go from \textquotedblleft thimble-flow-thimble\textquotedblright\
configurations to resolved Floer trajectories by the above mentioned
adiabatic gluing theorem \emph{without passing through the nodal Floer
trajectories} unlike the one proposed in \cite{PSS} and realized in \cite%
{oh-zhu}. We also remark that the full gradient trajectory (i.e. $l=\infty $%
) case can be treated by the same method of this paper, which is in fact
easier, because we do not need to put the power weight $\rho \left( \tau
\right) $ on $\chi $ (see Remark \ref{weights}).

\subsection{Further ramifications and relation to other results}

Some remarks on the relationship with other relevant literature are now in
order. We would like to point out that our gluing analysis in the present
paper and in \cite{oh-zhu} can be applied to many other contexts where
similar thick-thin decomposition occurs (including $l=0$ case) during the
adiabatic degeneration arises. One of the very first instances where
importance of this kind of gluing analysis was mentioned is the papers \cite%
{fukaya:homotopy}, \cite{oh:newton}. An analogous analysis was also carried
out by Fukaya and the senior author in \cite{FO} when there is no
non-constant bubbles around.

\subsubsection{The case where $H= 0$: The Hamiltonian pearl complex}

We note that our gluing analysis covers the special case where $H = 0$. In
this case, the thimble itself from the moduli space $\mathcal{M}\left(
K_{\pm },J_{\pm };z_{\pm };A_{\pm }\right)$ entering in the
`thimble-flow-thimble' configuration $(u_-,\chi,u_+)$

\begin{figure}[hbt!]
% left and right thimbles picture
\centering
\begin{tikzpicture}
		% u_+ thimble		
		\draw (4,0) ellipse (0.5 and 1);
		\draw (4,1) arc
		[
		start angle=90,
		end angle=270,
		x radius=2,
		y radius =1
		] ;
		
		\node [right] at (2.5,0)  {$u_+$};
		\node [right] at (4.5,0)  {$z_+$};
		
		% u_- thimble		
		\draw (-4,0) ellipse (0.5 and 1);
		\draw (-4,-1) arc
		[
		start angle=270,
		end angle=450,
		x radius=2,
		y radius =1
		] ;
		
		\node [left] at (-2.5,0)  {$u_-$};
		\node [left] at (-4.5,0)  {$z_-$};
		
		% middle flow: dotted \chi flow
		\draw [dotted, ->] (-2,0) to [out=30, in=210 ] (2,0);
		\node [above] at (0,0) {$\chi$}; 	
	\end{tikzpicture}
\caption{left and right thimbles}
\end{figure}
\noindent is replaced by the configuration either of $\text{\textrm{flow-sphere}}$ or of $%
\text{\textrm{sphere-flow}}$ where the left or the right flow-trajectory is
the semi-infinite gradient trajectories either issued at (or ending at) a
critical point of the Morse function $f$ (See Subsection \ref%
{subsec:ham-pearl} for details).

\begin{figure}[hbt!]
% flow-sphere and sphere-flow picture
\centering
\begin{tikzpicture}
		
		% thimble-flow		
		\draw (3,0) ellipse (0.25 and 1);
		\draw (3,0) ellipse (1 and 1);		
		\node [above] at (3,0)  {$u_+$};
		
		\draw [->] (4,0) to [out=30, in=210] (5,0);
		\draw [fill] (5,0) circle [radius=0.02];
		\node [above] at (4.5,0) {$\chi_{\infty}$};
		\node [below] at (5,0) {$q$};

		% flow-thimble		
		
		\draw (-3,0) ellipse (0.25 and 1);
		\draw (-3,0) ellipse (1 and 1);		
		\node [above] at (-3,0)  {$u_-$};
		
		\draw [->] (-5,0) to [out=30, in=210] (-4,0);
		\draw [fill] (-5,0) circle [radius=0.02];
		\node [above] at (-4.5,0) {$\chi_{-\infty}$};
		\node [below] at (-5,0) {$p$};

		% middle flow: dotted \chi flow
		\draw [dotted, ->] (-2,0) to [out=30, in=210 ] (2,0);
		\node [above] at (0,0) {$\chi$};
		\end{tikzpicture}
\caption{flow-sphere and sphere-flow}
\end{figure}

The gluing analysis with non-constant bubbles around for the equation
\begin{equation*}
\frac{\partial u}{\partial \tau }+J_{0}\left( \frac{\partial u}{\partial t}%
-X_{\varepsilon f}(u)\right) =0
\end{equation*}%
as $\varepsilon >0$ small was pursued by Fukaya and the senior author at
that time but left unfinished since then, partly because the analysis turned
out to be much more nontrivial than they originally expected. The gluing
analysis given in the present paper can be similarly applied also to this
case \emph{under the assumption that the join points are immersed}. Section %
\ref{sec:variants} of the present paper briefly outlines the necessary
modification for this case in this setting.

\subsubsection{The case where $H= 0$: The Lagrangian pearl complex}

The Lagrangian counterpart of the pearl complex has been greatly studied by
Biran-Cornea who have found many interesting applications to the symplectic topology
of (monotone) Lagrangian submanifolds. (See \cite{biran-cor-1,biran-cor-2,biran-cor-4} to name a few.
Also see \cite{biran-cor-3,biran-cor-5} for a survey on the pearl complex and its applications.)
There has been also an interesting calculation by Cho-Hong-Lau \cite{cho-hong-lau} via the pearl complex
which enters in the homological mirror symmetry between the Fukaya category of toric
Fano manifolds and the matrix factorization.

One outcome of our analysis and the dimension counting argument
proves the following equivalence theorem which had been anticipated in \cite{oh:newton}.
We state the theorem only for the case where we do not need any
additional argument except some immediate translations of our gluing
analysis presented in this paper without using the virtual machinery. This
is related to the $A_{\infty }$-structure defined in \cite{FOOO}.

\begin{thm}
\label{thm:monotone-intro} Let $(M,\omega )$ be a monotone symplectic manifold and
$L\subset (M,\omega ) $ be a monotone Lagrangian submanifold. Fix a Darboux
neighborhood $U$ of $L$ and identify $U$ with a neighborhood of the zero
section in $T^{\ast }L$. Consider $k+1$ Hamiltonian deformations of $L$ by
autonomous Hamiltonian functions $F_{0},\ldots ,F_{k}:M\rightarrow {\mathbb{R%
}}$ such that
\begin{equation*}
F_{i}=\chi f_{i}\circ \pi
\end{equation*}%
where $f_{0},f_{1},\ldots f_{k}:L\rightarrow {\mathbb{R}}$ are generic Morse
functions, $\chi $ is a cut-off function such that $\chi =1$ on $U$ and
supported nearby U. Now consider
\begin{equation*}
L_{i,\varepsilon }=\func{Graph}(\varepsilon df_{i})\subset U\subset M,\quad
i=0,\ldots ,k.
\end{equation*}%
Assume transversality of $L_{i}$'s of the type given in \cite{FO}. Consider
the intersections $p_{i}\in L_{i}\cap L_{i+1}$ such that
\begin{equation*}
\mathop{\kern0pt{\rm dim}}\nolimits{\mathcal{M}}(L_{0},\ldots
,L_{k};p_{0},\ldots ,p_{k})=0,\,1.
\end{equation*}%
Then when $\varepsilon $ is sufficiently small, the moduli space ${\mathcal{M%
}}(L_{0,\varepsilon },\ldots ,L_{k,\varepsilon };p_{0},\ldots ,p_{k})$ is
diffeomorphic to the moduli space of pearl complex defined in \cite%
{oh:newton,oh:imrn}, and \cite{biran-cor-1,biran-cor-2}.
\end{thm}

The following is an immediate corollary of this theorem.

\begin{cor}
Under the same hypothesis on $(M,\omega )$ and $L$, the $A_\infty$-structure
proposed in \cite{fukaya:homotopy,oh:newton,oh:imrn}, which also coincides
with the quantum structure defined in \cite{biran-cor-1}, is isomorphic to
the $A_{\infty }$-structure defined in \cite{FOOO}.
\end{cor}
A complete proof of this result had been given in \cite{fooo:canonical} for general
compact \emph{weakly unobstructed Lagrangian submanifolds} via
the homological perturbation theory by reducing the filtered $A_\infty$-structure on the Bott-Morse
Floer chain complex $C^*(L) \otimes \Lambda_{0,\textrm{nov}}$ to one on
the Morse complex $CM^*(f) \otimes \Lambda_{0,\textrm{nov}}$.

The following special case of this corollary is worthwhile to mention
separately.

\begin{thm}[Theorem \protect\ref{thm:isomorphism}]
\label{thm:isomorphism-intro} Let $(M,\omega)$ and $L$ be as in Theorem \ref%
{thm:monotone-intro}. Then there exists some $\varepsilon_0 > 0$ such that the map
\begin{equation*}
\Theta_\varepsilon:(CF_*(L,L), \partial) \to (C(f,\rho,J), d)
\end{equation*}
induced by the diffeomorphism between the two relevant moduli spaces
defining the boundary maps for the two complexes induces a chain
isomorphism. In particular it the induces an isomorphism in homology
\begin{equation}\label{eq:QH=FH}
QH_*(L) \cong HF_*(L,L)
\end{equation}
for all $0 < \varepsilon \leq \varepsilon_0$.
\end{thm}

In \cite{biran-cor-1,biran-cor-2}, Biran and Cornea directly utilized \textquotedblleft
disk-flow-disk\textquotedblright\ configurations to define what so called
\emph{pearl complex} for embedded monotone Lagrangians (For \emph{immersed}
Lagrangians with certain positivity condition on the index of the
non-embedded points, \cite{AB} similarly defined Floer cohomology by pearly
trajectories). Theorem \ref{thm:isomorphism-intro} was anticipated in \cite%
{oh:newton} and stated as a part of theorem in \cite[Theorem 2.1]{biran-cor-3}.
 (See the first paragraph
\cite[Section 1.2.5]{biran-cor-2} for relevant remark.)

\bigskip In \cite{BO}, Bourgeois and Oancea studied the Floer equation of
autonomous Hamiltonian $H$ under the assumption that its 1-periodic orbits
are transversally nondegenerate on a symplectic manifold with contact type
boundary, in relation to the study of the linearized contact homology of a
fillable contact manifold and symplectic homology of its filling. Their
construction is based on Morse-Bott techniques for Floer trajectories where
a function of the type $\varepsilon f$ is used on the Morse-Bott set of
\emph{parameterized }periodic orbits of $H$ to break the $S^{1}$-symmetry of
the orbit space, whose loci consist of \emph{isolated}, \emph{unparameterized%
}, \emph{nonconstant} periodic orbits. The Hamiltonian pearl complex case
treated in our subsection \ref{subsec:ham-pearl} can be put in a similar
Morse-Bott setting for the Hamiltonian $H=0$, to break the $M$-symmetry of
the orbit space, whose loci consists of constant orbits (i.e. all points on $%
M$), by adding a small Morse function $\varepsilon f$ on $M$ itself.
Some gluing result involving the  configuration of disk-flow-disk also appeared in the Legendrian contact homology setting in \cite{EES}.

More recently,  Diogo and Lisi \cite{DL} studied the symplectic homology of complements of
smooth divisors under certain monotonicity assumptions. In relation to their study,
they use a Morse-Bott degeneration comparing split Floer cylinders with cascades to $J$%
-holomorphic curves. The configuration appearing in their study is similar to resolving \textquotedblleft
thimble-flow-thimble\textquotedblright\ configurations to Floer trajectories
in our case. But their method is very different from ours. (See \cite[Part 3]{DL} for details.)

\textbf{Acknowledgment:} The first named author would like to thank Kenji
Fukaya for having much discussion on this kind of gluing analysis after \cite%
{FO} was completed when they had attempted to push the gluing analysis further
to include the thick part without much success at that time.

The second named author would like to thank the math departments of Chinese
University of Hong Kong and University of Minnesota for the excellent research
environment, where major part of his research on this paper was carried out.
He also thanks Xiaowei Wang for constant encouragement.

\section{Set-up of the Floer equation}

\label{sec:invariant}

In this section, we formulate the set-up for the general Floer's perturbed
Cauchy-Riemann equation on compact Riemann surface with a finite number of
punctures. This requires a coordinate-free framework of the equation. Such a
description was given for example in \cite{Se} and \cite{MS}.

\subsection{Punctures with analytic coordinates}

\label{subsec:punctures}

We start with the description of positive and negative \emph{punctures}. Let
$\Sigma$ be a compact Riemann surface with a marked point $p \in \Sigma$.
Consider the corresponding punctured Riemann surface $\dot \Sigma$ with an
analytic coordinates $z: D\setminus \{p\} \to {\mathbb{C}}$ on a
neighborhood $D \setminus \{p\} \subset \dot \Sigma$. By composing $z$ with
a linear translation of ${\mathbb{C}}$, we may assume $z(p) = 0$.

We know that $D \setminus \{p\}$ is conformally isomorphic to both $%
[0,\infty) \times S^1$ and $(-\infty,0] \times S^1$.

\begin{enumerate}
\item We say that the pair $(p;(D,z))$ has a \emph{incoming cylindrical end}
(with analytic chart) if we have
\begin{equation*}
D = z^{-1}(D^2(1))
\end{equation*}
and are given by the biholomorphism
\begin{equation*}
(\tau, t) \in S^1 \times (-\infty,0] \mapsto e^{2\pi(\tau + it)} \in D^2(1)
\setminus \{0\} \mapsto z^{-1} \in D\setminus \{p\}.
\end{equation*}
We call the corresponding puncture $p \in \Sigma$ a \emph{positive puncture}.

\item We say that the pair $(p;(D,z))$ has a \emph{outgoing cylindrical end}
(with analytic chart) if we have
\begin{equation*}
D = z^{-1}(D^2(1))
\end{equation*}
and are given by the biholomorphism
\begin{equation*}
(\tau,t) \in S^1 \times [0,\infty) \mapsto e^{-2\pi(\tau + it)} \in D^2(1)
\setminus \{0\} \mapsto z^{-1} \in D \setminus \{p\}.
\end{equation*}
In this case, we call the corresponding puncture $(p;(D,z))$ a \emph{%
negative puncture} (with analytic chart).
\end{enumerate}

\subsection{Hamiltonian perturbations}

\label{eq:perturb}

Now we describe the Hamiltonian perturbations in a coordinate free fashion.

Let $\Sigma$ be a compact Riemann surface and $\dot \Sigma$ denote $\Sigma$
with a finite number of punctures and analytic coordinates. We denote by ${%
\mathcal{J}}_{0,\omega}$ the set of almost complex structures that are
cylindrical near the puncture with respect to the given analytic charts $z =
e^{\pm 2\pi(\tau + it)}$. Define ${\mathcal{J}}_\Sigma$ or ${\mathcal{J}}%
_{\dot\Sigma}$ to be the set of maps $J: \Sigma, \, \dot \Sigma \to {%
\mathcal{J}}_{0,\omega}$ respectively.

\begin{defn}
We call $K \in \Omega^1(\Sigma,C^\infty(M))$ \emph{cylindrical} at the
puncture $p \in \Sigma$ with analytic chart $(D,z)$, if it has the form
\begin{equation*}
K(\tau,t) = H(t)\, dt
\end{equation*}
in $D \setminus \{p\}$. We denote by ${\mathcal{K}}_{\dot\Sigma}$ the set of
such $K$'s.
\end{defn}

One important quantity associated to the one-form $K$ is a two-form, denoted
by $R_K$, and defined by
\begin{equation}  \label{eq:R_K}
R_K\left(\xi_1,\xi_2\right) = \xi_1[K(\xi_2)] - \xi_2[K(\xi_1)] -
\left\{K(\xi_2), K(\xi_1)\right\}
\end{equation}
for two vector fields $\xi_1, \, \xi_2$, where $\xi_1[(K(\xi_2)]$ denotes
directional derivative of the function $K(\xi_2)(z,x)$ with respect to the
vector field $\xi_1$ as a function on $\Sigma$, holding the variable $x \in
M $ fixed. It follows from the expression that $R_K$ is tensorial on $\Sigma$
which corresponds to the curvature of the Hamiltonian fibration equipped
with the Hamiltonian connection $K$.

The Hamiltonian-perturbed Cauchy-Riemann equation has the form
\begin{equation}
(du+P_{K}(u))_{J}^{(0,1)}=0\quad \mbox{ or equivalently }\,{\overline{%
\partial }}_{J}(u)+(P_{K})_{J}^{(0,1)}(u)=0  \label{eq:KJ}
\end{equation}%
on $\Sigma $ in general. In cylindrical coordinate $\left( \tau ,t\right) $
near the puncture $p$
\begin{equation*}
\left( P_{K}\right) ^{\left( 0,1\right) }\left( u\right) \left( \frac{%
\partial }{\partial \tau }\right) =-JX_{H}\left( u\right) .
\end{equation*}

For each given such a pair $(K,J)$, it defines a perturbed nonlinear
Cauchy-Riemann operator by
\begin{equation*}
{\overline \partial}_{(K,J)} u := {\overline \partial}_Ju + P_K(u)^{(0,1)}_J
= (du +P_K(u))^{(0,1)}_J.
\end{equation*}
Let $(\mathfrak{p, q)}$ be a given set of positive punctures $\mathfrak{p =
\{p_1, \cdots, p_k\}}$ and with negative punctures $\mathfrak{q = \{ q_1,
\cdots, q_l\}}$ on $\Sigma$. For each given Floer datum $(K,J)$ and a
collection $\vec z = \{z_*\}_{* \in \mathfrak{p \cup q}}$ of asymptotic
periodic orbits $z_*$ attached to the punctures $* = p_i$ or $* = q_j$, we
consider the perturbed Cauchy-Riemann equation
\begin{equation}  \label{eq:KJ-asymp}
\begin{cases}
{\overline \partial}_{(K,J)}(u) = 0 \\
u(\infty_*,t) = z_*(t).%
\end{cases}%
\end{equation}
Our main interest will lie in the case where $(|\mathfrak{p}|, |\mathfrak{q}%
|)$ is either $(1,0), \, (0,1)$ or $(1,1)$ in the present paper.

One more ingredient we need to give the definition of the
Hamiltonian-perturbed moduli space is the choice of an appropriate energy of
the map $u$. For this purpose, we fix a metric $h_\Sigma$ which is
compatible with the structure of the Riemann surface and which has the
cylindrical ends with respect to the given cylindrical coordinates near the
punctures, i.e., $h_\Sigma$ has the form
\begin{equation}  \label{eq:gSigma}
h_\Sigma = d\tau^2 + dt^2
\end{equation}
on $D_* \setminus \{*\}$. We denote by $dA_\Sigma$ the corresponding area
element on $\Sigma$.

Here is the relevant energy function

\begin{defn}[Energy]
For a given asymptotically cylindrical pair $(K,J)$, we define
\begin{equation*}
E_{(K,J)}(u) = \frac{1}{2}\int_\Sigma |du + P_K(u)|_J^2\, dA_\Sigma
\end{equation*}
where $|\cdot|_J$ is the norm of $\Lambda^{(0,1)}(u^*TM) \to \Sigma$ induced
by the metrics $h_\Sigma$ and $g_J: = \omega(\cdot, J \cdot)$.
\end{defn}

Note that this energy depends only on the conformal class of $h_\Sigma$,
i.e., depends only on the complex structure $j$ of $\Sigma$ and restricts to
the standard energy for the usual Floer trajectory moduli space given by
\begin{equation*}
E_{(H,J)} = \frac{1}{2}\int_{C_*} \left(\left|{\frac{\partial u}{\partial
\tau}}\right|_J^2 + \left|{\frac{\partial u}{\partial t}} -
X_H(u)\right|_J^2\right) \, dt\, d\tau
\end{equation*}
in the cylindrical coordinates $(\tau,t)$ on the cylinder $C_*$
corresponding to the puncture $*$. $E_{(K,J)}(u)$ can be bounded by a more
topological quantity depending only on the asymptotic orbits, or more
precisely their liftings to the \emph{universal covering space} of ${%
\mathcal{L}}_0(M)$, where the latter is the contractible loop space of $M$.
As usual, we denote such a lifting of a periodic orbit $z$ by $[z,w]$ where $%
w:D^2 \to M$ is a disk bounding the loop $z$.

We also consider the \emph{real blow-up} of $\dot \Sigma \subset \Sigma$ at
the punctures and denote it by $\overline\Sigma$ which is a compact Riemann
surface with boundary
\begin{equation*}
\partial \overline \Sigma = \coprod_{* \in \mathfrak{p} \cup \mathfrak{q}}
S^1_*
\end{equation*}
where $S^1_*$ is the exceptional circle over the point $*$. We note that
since there is given a preferred coordinates near the point $*$, each circle
$S^1_*$ has the canonical identification
\begin{equation*}
\theta_*: S^1_* \to {\mathbb{R}}/{\mathbb{Z}} = [0,1] \mod 1.
\end{equation*}

We note that for a given asymptotic orbits $\vec z$, one can define the
space of maps $u : \dot \Sigma \to M $ which can be extended to $\overline
\Sigma$ such that $u \circ \theta_* = z_*(t)$ for $* \in \mathfrak{p \cup q}$%
. Each such map defines a natural homotopy class $B$ relative to the
boundary. We denote the corresponding set of homotopy classes by $\pi_2(\vec
z)$. When we are given the additional data of bounding discs $w_*$ for each $%
z_*$, then we can form a natural homology (in fact a homotopy class),
denoted by $B \# \left(\coprod_{* \in \mathfrak{p} \cup \mathfrak{q}}
[w_*]\right) \in H_2(M)$, by `capping-off' the boundary components of $B$
using the discs $w_*$ respectively.

\begin{defn}
Let $\{[z_*,w_*]\}_{* \in \mathfrak{p \cup q}}$ be given. We say $B \in
\pi(\vec z)$ is \emph{admissible} if it satisfies
\begin{equation}  \label{eq:Bsharpws}
B \# \left(\coprod_{* \in \mathfrak{p} \cup \mathfrak{q}} [w_*]\right) = 0
\quad \mbox{in }\, H_2(M,{\mathbb{Z}})
\end{equation}
\end{defn}

where
\begin{equation*}
\# : \pi_2(\vec z) \times \prod_{* \in \mathfrak{p} \cup \mathfrak{q}}
\pi_2(z_*) \to H_2(M,{\mathbb{Z}})
\end{equation*}
is the natural gluing operation of the homotopy class from $\pi_2(\vec z)$
and those from $\pi_2(z_*)$ for $* \in \mathfrak{p} \cup \mathfrak{q}$. Now
we are ready to give the definition of the Floer moduli spaces.

\begin{defn}
Let $(K,J)$ be a Floer datum over $\Sigma$ with punctures $\mathfrak{p, \, q}
$, and let $\{[z_*,w_*]\}_{* \in \mathfrak{p\cup q}}$ be the given
asymptotic orbits. Let $B \in \pi_2(\vec z)$ be a homotopy class admissible
to $\{[z_*,w_*]\}_{* \in \mathfrak{p\cup q}}$. We define the moduli space
\begin{equation}  \label{eq:Mz*w*}
{\mathcal{M}}(K,J;\{[z_*,w_*]\}_*) = \{u: \dot \Sigma \to M \mid u \,
\mbox{ satisfies (\ref{eq:KJ-asymp}) and
$[u]\# (\coprod_{* \in \mp \cup \mq} [w_*]) =0$ }\}.
\end{equation}
\end{defn}

We note that the moduli space ${\mathcal{M}}(K,J;\{[z_{\ast },w_{\ast
}]\}_{\ast })$ is a finite union of the moduli spaces
\begin{equation*}
{\mathcal{M}}(K,J;\vec{z};B)\text{, with }B\#\left( \coprod_{\ast \in
\mathfrak{p}\cup \mathfrak{q}}[w_{\ast }]\right) =0.
\end{equation*}

\section{Adiabatic family of Floer moduli spaces}
\label{sec:floer}

In this section, we give the precise description of one-parameter family of
perturbed Cauchy-Riemann equations in a coordinate-free form parameterized
by $\varepsilon > 0$ such that as $\varepsilon \to 0$ the equation becomes
degenerate in a suitable sense which we will make precise.

We consider two genus-zero closed Riemann surfaces $\Sigma_{\pm}$ with two marked points
$$
 \{o_\pm, e_\pm \}.
$$
We denote by $\dot \Sigma_\pm$ the one-punctured surfaces given by
\be\label{eq:dotSigma+-}
\dot \Sigma_\pm = \Sigma_\pm \setminus \{e_\pm\}
\ee
respectively.  Then we consider the one-form $K_\pm$ such that
\begin{equation}  \label{eq:Kpm}
K_\pm(\tau,t,x) = \kappa^\pm(\tau) H_t(x)
\end{equation}
in terms of the given analytic coordinates on a punctured neighborhood
$U_\pm \setminus \{*\}\subset \dot \Sigma_\pm$ of the relevant puncture $*$ where
$U_\pm$ a disk-type neighborhood of $*$ in the associated compact surface $%
\Sigma$.

Now we consider a one-parameter family $(K_\varepsilon,J_\varepsilon)$ with $%
R = R(\varepsilon) \to \infty$,
\begin{equation}  \label{eq:length}
\varepsilon R(\varepsilon) \to l
\end{equation}
with $l \geq 0$ as $\varepsilon \to 0$. More precise description of $%
(K_\varepsilon, J_\varepsilon)$ is in order.

To define this family, we fix any two continuous functions $R(\varepsilon )$
and $S(\varepsilon )$ that satisfies
\begin{equation}
\varepsilon R(\varepsilon )\rightarrow l ,\,\varepsilon S(\varepsilon
)\rightarrow 0  \label{eq:RSe}
\end{equation}%
as $\varepsilon \rightarrow 0$.

We will make the following choice for the convenience of our exposition
later, which will appear frequently in our calculations.

\begin{choice}
\label{choice:RS-epsilon} We fix any $p>2$ and $0<\delta <1$. Then we take
\begin{equation}
R(\varepsilon )=\frac{l}{\varepsilon },\quad S(\varepsilon )=\frac{1}{2\pi }%
\ln \left(1+\frac{l}{\varepsilon }\right)  \label{eq:ReSe}
\end{equation}%
and set
\begin{equation*}
\tau \left( \varepsilon \right) =R\left( \varepsilon \right) +\frac{p-1}{%
\delta }S\left( \varepsilon \right).
\end{equation*}%
%\
\end{choice}

\begin{rem}
Here are some explanations on the above choices we have made.

\begin{itemize}
\item The choice $p>2$ is to get the Sobolev embedding
\begin{equation*}
W^{1,p}\hookrightarrow C^{0}
\end{equation*}
on the Riemann surface.

\item The choice $0<\delta <1$ is to get rid of the $0$-spectrum of $i\frac{%
\partial }{\partial t}$ on $S^{1}$. Any such choice of $p$ and $\delta $
will be suffice for our analysis.

\item The choice of $\tau \left( \varepsilon \right) $ is not canonical as
it depends on $p$ and $\delta $, but when $p\rightarrow 2 $ and $\delta
\rightarrow 1$, $\left\vert \tau \left( \varepsilon \right) -R\left(
\varepsilon \right) \right\vert $ is close to $S\left( \varepsilon \right) $.
\end{itemize}
\end{rem}

Then we decompose ${\mathbb{R}}$ into
\begin{equation*}
-\infty <-\tau \left( \varepsilon \right) -1<-\tau \left( \varepsilon
\right) <-R(\varepsilon )<R(\varepsilon )<\tau \left( \varepsilon \right)
<\tau \left( \varepsilon \right) +1<\infty .
\end{equation*}

We choose neighborhoods $U_{\pm }$ of $e_{\pm }$ and analytic charts $%
\varphi _{\pm }:U_{\pm }\rightarrow {\mathbb{C}}$ and the associated
coordinates $z=e^{2\pi (\tau +it)}$ so that
\begin{equation*}
\varphi _{+}(U_{+}\setminus \{e_{+}\})\cong \lbrack 0,\infty )\times
S^{1},\,\quad \varphi _{-}(U_{-}\setminus \{e_{-}\})\cong (-\infty ,0]\times
S^{1},
\end{equation*}%
we fix a function
\begin{equation}
\kappa ^{+}(\tau )=%
\begin{cases}
0 \quad & 0 \leq \tau \leq 1 \\
1\quad & {} \tau \geq 2%
\end{cases}
\label{eq:beta}
\end{equation}%
and let $\kappa ^{-}(\tau )=\kappa ^{+}(-\tau )$.

We define a glued surface with a cylindrical coordinates, denoted by
\begin{equation*}
\dot\Sigma: = \dot \Sigma_- \# \dot \Sigma_+
\end{equation*}
by identifying $o_- \in \dot \Sigma_-$ and $o_+ \in \dot \Sigma_+$. We denote by $o$
the corresponding nodal point of $\dot \Sigma$.
Then $\dot \Sigma$ carries two natural punctures $(e_-,e_+)$. We pick an
embedded path passing a point $o \in \dot \Sigma_+ \cap \dot \Sigma_-$ and connecting $%
e_-, e_+$. We fix the unique cylindrical coordinates $\dot\Sigma \cong {%
\mathbb{R}} \times S^1$ mapping the path to ${\mathbb{R}} \times \{0\}$ and $%
o$ to $(0,0)$.

\begin{choice}[Cut-off functions]
\label{choice:cut-off-functions} With this cylindrical coordinates, we
consider $\varepsilon >0$ and a family of cut-off functions defined by $%
\kappa _{\varepsilon }^{+}(\tau )=\kappa ^{+}(\tau -\tau (\varepsilon )+1)$
and $\kappa _{\varepsilon }^{-}(\tau )=\kappa _{\varepsilon }^{+}(-\tau )$.
It is easy to see
\begin{equation}
\kappa _{\varepsilon }^{+}(\tau )=%
\begin{cases}
1\,\quad \, & \mbox{for }\,\tau \geq \tau (\varepsilon )+1 \\
0\quad \, & \mbox{for }\,0\leq \tau \leq \tau (\varepsilon )%
\end{cases}%
,\quad \kappa _{\varepsilon }^{-}\left( \tau \right) =%
\begin{cases}
1\,\quad  & \mbox{for }\,\tau \leq -\tau (\varepsilon )-1 \\
0\quad  & \mbox{for }\,\tau (\varepsilon )\leq \tau \leq 0%
\end{cases}
\label{eq:betaR}
\end{equation}%
and $\kappa _{\varepsilon }^{0}(\tau )$ is a smooth cut-off function such
that
\begin{equation}
\kappa _{\varepsilon }^{0}(\tau )=%
\begin{cases}
1\quad  &
\mbox{for $\left\vert \tau \right\vert \leq R\left( \varepsilon
\right) $} \\
0\quad  &
\mbox{for $\left\vert \tau \right\vert \geq R\left(
\varepsilon \right) +1$},%
\end{cases}
\label{eq:beta0}
\end{equation}%
$\left\vert \kappa _{\varepsilon }^{0}(\tau )\right\vert \leq 1,\left\vert
(\kappa _{\varepsilon }^{0})^{\prime }(\tau )\right\vert \leq 2$.
\end{choice}

Using these cut-off functions, we now introduce the relevant $\varepsilon$%
-family of Floer data.

\begin{choice}[Floer data $(K_\protect\varepsilon, J_\protect\varepsilon)$]
\label{choice:floer-data} We define $(K_{\varepsilon },J_{\varepsilon })$ to
be the glued family
\begin{eqnarray}
K_{\varepsilon }(\tau ,t) &=&%
\begin{cases}
\kappa _{\varepsilon }^{+}(\tau )\cdot H_{t}\quad & \tau \geq R(\varepsilon )
\\
\kappa _{\varepsilon }^{0}(\tau )\cdot \varepsilon f\quad & |\tau |\leq
R(\varepsilon ) \\
\kappa _{\varepsilon }^{-}(\tau )\cdot H_{t}\quad & \tau \leq -R(\varepsilon
).%
\end{cases}
\label{eq:Kepsilon} \\
J_{\varepsilon }^{\pm }(\tau ,t,x) &=&%
\begin{cases}
J^{\kappa _{\varepsilon }^{+}(\tau )}(t,x)\quad & \tau \geq R(\varepsilon )
\\
J_{0}(x)\quad & |\tau |\leq R(\varepsilon ) \\
J^{\kappa _{\varepsilon }^{-}(\tau )}(t,x)\quad & \tau \leq -R(\varepsilon ).%
\end{cases}
\label{eq:Jepsilon}
\end{eqnarray}
\end{choice}

\begin{figure}[hbt!]
% Floer data picture
\begin{tikzpicture} [scale=0.9]		
		% draw rectangular regions
		\filldraw[thin,color=red!20, fill=red!5] (-2,1)--(2,1)--(2,-1)--(-2,-1);	
		\filldraw[thin,color=blue!20, fill=blue!5] (-6,1)--(-4,1)--(-4,-1)--(-6,-1);
		\filldraw[thin,color=blue!20, fill=blue!5] (6,1)--(4,1)--(4,-1)--(6,-1);

		% draw ellipses
		\draw (0,0) ellipse (0.5 and 1);
		\filldraw[ fill=red!5] (2,0) ellipse (0.5 and 1); 	
		\filldraw[ fill=red!5](-2,0) ellipse (0.5 and 1);
		
		\filldraw[ fill=blue!5] (4,0) ellipse (0.5 and 1);
		\filldraw[ fill=blue!5] (6,0) ellipse (0.5 and 1);
		
		\filldraw[ fill=blue!5] (-4,0) ellipse (0.5 and 1);
		\filldraw[ fill=blue!5] (-6,0) ellipse (0.5 and 1);
		
		% upper/lower boundary  axis
		\draw [thick] (-6,1) -- (6,1) coordinate ;
		\draw [thick] (-6,-1) -- (6,-1) coordinate ;
		
		% Hamiltonian function, middle part {\varepsilon}f
		\draw [blue] (-6,0)--(-4.5,0);   \draw [dotted] (-4.5,0)--(-4,0); % left Hamiltonian
		\draw [blue] [<-] (6,0)--(4.5,0);   \draw [dotted] (4.5,0)--(4,0); % right Hamiltonian
		\node [below right] at (6,0) {$\tau$};
		\draw [dashed] (-2,0) to (-4,0); %left piece transition
		\draw [dashed] (2,0) to (4,0);  %right piece transition
		
		% mark Hamiltonians
		\draw [red] (-2,0) -- (2,0) ;  % Morse function \varepsilon f
		\node [above, blue] at (-5,0) {$\kappa_{\varepsilon}^{-}H_t$}; \node [below, blue] at (-5,0) {$J_{\varepsilon}^-$};
		\node [above, blue] at (5,0) {$\kappa_{\varepsilon}^{+}H_t$}; \node [below, blue] at (5,0) {$J_{\varepsilon}^+$};
		\node [above left, red] at (0,0) {$\varepsilon f$};
		\node [below left, red] at (0,0) {$J_0$};
		
		% draw dots and marking along \tau axis
		\draw [fill] (2,0) circle [radius=0.03]; \draw [fill] (4,0) circle [radius=0.03];
		\draw [fill] (-2,0) circle [radius=0.03]; \draw [fill] (-4,0) circle [radius=0.03];
		\node [above] at (-4,0) {$-\tau(\varepsilon)$}; \node [above] at (-2,0) {$-l/{\varepsilon}$};
		\node [above] at (4,0) {$\tau(\varepsilon)$}; \node [above] at (2,0) {$l/{\varepsilon}$};	
	\end{tikzpicture}
\caption{Floer data}
\end{figure}

We will vary $R=R(\varepsilon )$ depending on $\varepsilon $ and study the
family of equations
\begin{equation}
(du+P_{K_{\varepsilon }}(u))_{J_{\varepsilon }}^{(0,1)}=0  \label{eq:duPK}
\end{equation}%
as $\varepsilon \rightarrow 0$.

\section{Adiabatic convergence}

\label{sec:adiabatic}

By definition of $K_{\varepsilon }$ and $J_{\varepsilon }$, as $\varepsilon
\rightarrow 0$, on the domain
\begin{equation*}
\lbrack -R(\varepsilon ),R(\varepsilon )]\times S^{1}
\end{equation*}%
we have $K_{\varepsilon }(\tau ,t)\equiv \varepsilon f$ and $J_{R}(\tau
,t)\equiv J_{0}$, and so \eqref{eq:KJe} becomes
\begin{equation*}
\frac{\partial u}{\partial \tau }+J_\varepsilon \left( \frac{\partial u}{%
\partial t}-\varepsilon X_{f}(u)\right) =0.
\end{equation*}%
Furthermore we require the Floer data $(K_\varepsilon,J_\varepsilon)$ to
satisfy
\begin{equation*}
K_{\varepsilon }(\tau ,t)\equiv H_{\pm }\,dt, \quad J_{\varepsilon }(\tau
,t)\equiv J_{\pm }
\end{equation*}
on
\begin{equation*}
{\mathbb{R}}\times S^{1}\setminus \lbrack -\tau (\varepsilon )-1,\tau
(\varepsilon )+1]\times S^{1}
\end{equation*}%
(\ref{eq:duPK}) is cylindrical at infinity, i.e., invariant under the
translation in $\tau $-direction at infinity.

Note that on any fixed compact set $B\subset {\mathbb{R}}\times S^{1}$, we
will have
\begin{equation*}
B\subset \lbrack -R(\varepsilon ),R(\varepsilon )]\times S^{1}
\end{equation*}%
for all sufficiently small $\varepsilon $. And as $\varepsilon \rightarrow 0$,
the equation (\ref{eq:duPK}) converges to ${\overline{\partial }}%
_{J_{0}}u=0$ on $B$ in that $J\rightarrow J_{0}$ and $K_{\varepsilon
}\rightarrow 0$ in $C^{\infty }$-topology. On the other hand, after
translating the region $(-\infty ,-R(\varepsilon ]$ to the right (resp. $%
[R(\varepsilon ),\infty )$ to the left) by $R(\varepsilon )
+ S(\varepsilon)$ in $\tau $-direction, (\ref{eq:duPK}) converges to
\begin{equation*}
\frac{\partial u}{\partial \tau }+J_{+}\left( \frac{\partial u}{\partial t}%
-X_{H_{+}}(u)\right) =0
\end{equation*}%
on $(-\infty ,0]\times S^{1}$ (resp. on $[0,\infty )\times S^{1}$) and ${%
\overline{\partial }}_{J_{0}}u=0$ on $[0,R(\varepsilon )]\times S^{1}$
(resp. on $[-R(\varepsilon ),0]\times S^{1}$).

Now we are ready to state the meaning of the \emph{adiabatic convergence}
for a sequence $u_{n}$ of solutions $(du_{n}+P_{K_{\varepsilon
_{n}}})_{J_{\varepsilon _{n}}}^{(0,1)}=0$ as $n\rightarrow \infty $.

After taking away bubbles, we can achieve the following derivative bound
which we will assume from now on.

\begin{hypo}
\label{hypo:C1bound} Let $\varepsilon_n \to 0$ and assume
\begin{equation}
\|du_{n}\|_{C^0} < C <\infty  \label{eq:|du|<C}
\end{equation}%
where we take the norm $\|du_{n}\|_{C^0}$ with respect to the metric $g$ on $%
M$.
\end{hypo}

\emph{From now on, we will always assume that we have both energy bound, and
this derivative bound so that bubble does not occur as $\varepsilon \to 0$.}

Put
\begin{equation*}
\Theta_\varepsilon := \left[-R(\varepsilon), R(\varepsilon)\right] \times
S^1.
\end{equation*}
For any compact subset $B \subset {\mathbb{R}} \times S^1$, we have
\begin{equation*}
B \subset \Theta_\varepsilon
\end{equation*}
eventually as $\varepsilon \to 0$.

\begin{defn}
Let $\{u_n\}$ be a sequence satisfying Hypothesis \ref{hypo:C1bound}. Let $B
\subset {\mathbb{R}} \times S^1$ be a precompact subdomain. We define the
\emph{local energy} of a sequence $\{u_n\}$ by
\begin{equation*}
E_{J,B}(u) := \limsup_{n \to \infty} \int_B |du_n|_J^2 dt \, d\tau.
\end{equation*}
\end{defn}

There are two cases to consider :

\begin{enumerate}
\item there exists a precompact subdomain $B \subset \R \times S^1$
and a subsequence $n_{i}$ and $c>0$ such that
$E_{J,B}(u_{n_{i}})>c>0$ for all sufficiently large $i$,
\item $\limsup_{n \to \infty} E_{J,\Theta_{\varepsilon_n}}(u_n) = 0$.
\end{enumerate}
For the case (1), standard argument produces a non-constant limit which
becomes a pseudoholomorphic sphere. To describe the situation (2) in a
precise manner, which is the case of our interest, we introduce the notion
of adiabatic convergence to a thimble-flow-thimble trajectory $(u_-,\chi,u_+)
$ in the next subsection. This will describe a neighborhood basis of $%
(u_-,\chi,u_+)$ at `infinity' of
\begin{equation*}
{\mathcal{F}} ^{1,p}(K_{\varepsilon_n},J_{\varepsilon_n};z_-,z_+;B)
\end{equation*}
We first recall the definition of Hausdorff distance of subsets $C, \, D
\subset X$
\begin{equation*}
d_{\text{\textrm{H}}}(C,D) = \sup_{x \in A}d(x,D) + \sup_{y \in D} d(C,y).
\end{equation*}

\subsection{Adiabatic deformations of domain}

\label{subsec:domain}

To provide a rigorous definition of the \emph{adiabatic convergence} we deal with,
we need to explain the precise setting and the way how we degenerate the domain
$$
\dot \Sigma = \dot \Sigma_- \# \dot \Sigma_+
$$
of the maps in our interest to the union of two spheres joined by
a line segment of length $2l $ parameterized by the interval $[-l ,l ]$.

For the aforementioned open subsets
$U_{\pm }\subset \Sigma _{\pm }$,  we identify $U_{\pm }\setminus \{o_{\pm }\}$ with the open
punctured unit disk $D^{2}\setminus \{0\}$ using the  analytic charts $\varphi
_{\pm }:U_{\pm }\rightarrow D^{2}\subset {\mathbb{C}}$ and the associated
coordinates $z$.
We first consider the annular
domain of $U_{\pm }\setminus \{o_{\pm }\}$:
\begin{eqnarray*}
\text{\textrm{Ann}}_{\alpha } &:= & \{z\in D^{2}\mid |\alpha |^{3/4}\leq
|z|\leq |\alpha |^{1/4}\} \\
& = & \{z\in D^{2}\mid |R_{\alpha }|^{-3/2}\leq |z|\leq |R_{\alpha
}|^{-1/2}\}.
\end{eqnarray*}
The choice of the power is dictated by that of Fukaya-Ono's deformation
given in \cite{fukaya-ono} in which $|\alpha |=R_{\alpha }^{-2}$ for $%
|\alpha |$ sufficiently small.

We note that the conformal modulus of $\text{\textrm{Ann}}_{\alpha }$ is $%
\Vert \alpha \Vert ^{-1/2}=R_{\alpha }$.

\begin{choice}[Conformal modulus $\protect\alpha(\protect\varepsilon)$]
Let
\begin{equation*}
S(\varepsilon )=\frac{1}{2\pi }\ln \left( 1+\frac{l }{\varepsilon }\right)
\end{equation*}
be the function given in Choice \ref{choice:RS-epsilon}. For each given $%
\varepsilon >0$, we choose $\alpha (\varepsilon )$ so that
\begin{equation*}
\frac{p-1}{\delta }S(\varepsilon )=\ln \Vert \alpha (\varepsilon )\Vert
^{-1/2}.
\end{equation*}
\end{choice}

Then we choose a biholomorphism
\begin{equation*}
\varphi _{\varepsilon }^{-}:\text{\textrm{Ann}}_{\alpha (\varepsilon
)}\rightarrow \lbrack -\tau (\varepsilon ),-R(\varepsilon )]\times S^{1}.
\end{equation*}%
Similarly we define
\begin{equation*}
\varphi _{\varepsilon }^{+}:\text{\textrm{Ann}}_{\alpha (\varepsilon
)}\rightarrow \lbrack R(\varepsilon ),\tau (\varepsilon )]\times S^{1}.
\end{equation*}

We denote
\begin{equation*}
C(\varepsilon )=[-\tau (\varepsilon ),\tau (\varepsilon )]\times S^{1}
\end{equation*}%
with the standard metric $g_{C(\varepsilon )}$. We also equip each of $\Sigma_\pm$ with
a K\"ahler metric, denoted by $g_\pm$, respectively.

\begin{choice}[Gluing domains]\label{choice:gluing-domains}
We take a family of glued surfaces
\begin{equation*}
\Sigma _{\varepsilon }=\left( \Sigma _{-}\setminus D^{2}(|\alpha
(\varepsilon )|^{3/4})\right) \cup _{\varphi _{\varepsilon }^{-}}\left(
[-\tau (\varepsilon ),\tau (\varepsilon )]\times S^{1}\right) {}_{(\varphi
_{\varepsilon }^{+})^{-1}}\cup \left( \Sigma _{+}\setminus D^{2}(|\alpha
(\varepsilon )|^{3/4})\right) .
\end{equation*}
Then through the identifications $\varphi _{\pm }:U_{\pm }\rightarrow D^{2}$%
, this family give rise to the following $\varepsilon $-parameterized family
of resolved cylinders $(\Sigma _{\varepsilon }^{\text{\textrm{adia}}%
},g_{\varepsilon }^{\text{\textrm{adia}}})$ equipped with the metric
provided by
\begin{equation*}
g_{\varepsilon }^{\text{\textrm{adia}}}=%
\begin{cases}
g_{+}\quad & \mbox{on }\,\Sigma _{+}\setminus U_{+}(|\alpha (\varepsilon )|)
\\
g_{C(\varepsilon )}\quad & \mbox{on }\,C(\varepsilon ) \\
g_{-}\quad & \mbox{on }\,\Sigma _{-}\setminus U_{-}(|\alpha (\varepsilon )|)%
\end{cases}%
\end{equation*}%
and suitably interpolated in between.
\end{choice}

Note that the conformal structure of $\Sigma _{\varepsilon }$ is
degenerating on the annuli region $C(\varepsilon ) $ when $\varepsilon
\rightarrow 0$, and any given compact subset $K_\pm \subset \Sigma _{\pm
}\setminus \{o_{\pm }\}$ is covered respectively by $\Sigma _{\pm }\setminus U_{\pm
}(\delta )$ for a sufficiently small $\delta >0$.

\subsection{Definition of adiabatic convergence}

Now we involve maps defined on the resolved domains $(\Sigma_\varepsilon^{%
\text{\textrm{adia}}},g_\varepsilon^{\text{\textrm{adia}}})$ and provide a
precise definition of adiabatic convergence or of adiabatic topology near
the thimble-flow-thimble moduli space
\begin{equation*}
{\mathcal{M}}^{\text{\textrm{tft}}}_{l}(z_-;f;z_+;B)
\end{equation*}
inside the off-shell spaces
\begin{equation*}
\overline {\mathcal{F}}^{\text{\textrm{para}}}(z_-;f;z_+;B) = \bigcup_{l
\geq l_0}{\mathcal{F}}^l(z_-;f;z_+,;B).
\end{equation*}

\begin{defn}[Adiabatic convergence]
\label{defn:adiabatic-convergence} \label{defn:adiabatic} Let $%
\{\varepsilon_j\}$ be sequence with $\varepsilon_j \to 0$ as $j \to \infty$.

We say a sequence $u_j$ of maps in ${\mathcal{F}}^{1,p}(K_{%
\varepsilon_j},J_{\varepsilon_j};z_-,z_+;B)$ \emph{$\{\varepsilon_j\}$%
-adiabatically converges} to a thimble-flow-thimble trajectory $%
(u_-,\chi,u_+)$ if the following hold:

\begin{enumerate}
\item $\lim_{j\rightarrow \infty }E_{J,\Theta _{\varepsilon _{j}}}(u_{j})=0.$

\item $\lim_{n\rightarrow \infty }d_{\text{\textrm{H}}}(u_{j}([-R(%
\varepsilon _{j}),R(\varepsilon _{j})]\times S^{1}),\chi ([-l,l])=0$, where $%
d_{\text{\textrm{H}}}$ is the Gromov-Hausdorff metric.

\item $u_{j}|_{\Sigma _{\pm }\setminus U_{\pm }(\zeta )}\rightarrow u_{\pm }$
in $C^{\infty }$ for any given $0<\zeta <1$, or equivalently,

$u_{j}(\cdot \pm \tau \left( \varepsilon _{j}\right) ,\cdot )\rightarrow
u_{\pm }$ in $C^{\infty }$ on any domain $\pm \lbrack \frac{1}{2\pi }\ln
\zeta ,\infty )\times S^{1}.$

\item $\lim_{\zeta \rightarrow 0}$ $\lim_{j\rightarrow \infty }\func{diam}%
\left( u_{j}|_{\varphi ^{\pm }\left( \text{\textrm{Ann}}_{\alpha
(\varepsilon _{j})}\right) \setminus \pm \left[ \frac{1}{2\pi }\ln \zeta
,\tau \left( \varepsilon \right) \right] \times S^{1}}\right) =0$, or
equivalently,

$\lim_{\zeta \rightarrow 0}$ $\lim_{j\rightarrow \infty }\func{diam}\left(
u_{j}\left( \pm \left[ R\left( \varepsilon _{j}\right) ,\tau \left(
\varepsilon _{j}\right) +\frac{1}{2\pi }\ln \zeta \right] \times
S^{1}\right) \right) =0.$
\end{enumerate}
\end{defn}

In fact, we can turn this adiabatic convergence into a topology by
describing a neighborhood basis of the topology at $\varepsilon =0$.

\begin{defn}[Neighborhood basis at infinity]
For any $\varepsilon ,\zeta >0$ we define
\begin{eqnarray}
&{}&d_{\text{\textrm{adia}}}^{\varepsilon ,\zeta }(u,(u_{-},\chi ,u_{+}))
\notag \\
:= &&\max \Big\{E_{J,\Theta _{\varepsilon }}(u),d_{\text{\textrm{H}}%
}(u([-R(\varepsilon ),R(\varepsilon )]\times S^{1}),\chi ([l,l]),  \notag \\
&&\func{diam}\left( u_{j}\left( \pm \left[ R\left( \varepsilon \right) ,\tau
\left( \varepsilon \right) +\frac{1}{2\pi }\ln \zeta \right] \times
S^{1}\right) \right) ,  \notag \\
&{}&d_{C_{\Sigma _{\pm }\setminus U_{\pm }(\zeta )}^{\infty }}(u(\cdot -\tau
\left( \varepsilon \right) ,u_{-}),d_{C_{\Sigma _{\pm }\setminus U_{\pm
}(\zeta )}^{\infty }}(u(\cdot +\tau \left( \varepsilon \right) ,\cdot
),u_{+})\Big\}  \label{eq:adiadist}
\end{eqnarray}%
Then the sequence $u_{j}$ \emph{$\{\varepsilon _{j}\}$-adiabatically
converges} to a thimble-flow-thimble trajectory $(u_{-},\chi ,u_{+})$ if and
only if
\begin{equation*}
\lim_{\zeta \rightarrow 0}\lim_{j\rightarrow \infty }{}d_{\text{\textrm{adia}%
}}^{\varepsilon _{j},\zeta }(u_{j},(u_{-},\chi ,u_{+}))=0.
\end{equation*}%
We define the open set $V_{\zeta ,\delta }^{\varepsilon }$ in ${\mathcal{F}}%
^{1,p}(K_{\varepsilon },J_{\varepsilon };z_{-},z_{+};B)$ by%
\begin{equation}
V_{\zeta ,\delta }^{\varepsilon }=\left\{ u\in {\mathcal{F}}%
^{1,p}(K_{\varepsilon },J_{\varepsilon };z_{-},z_{+};B\, )\Big|\, d_{\text{%
\textrm{adia}}}^{\varepsilon ,\zeta }(u,(u_{-},\chi ,u_{+}))<\delta \right\}
\text{.}  \label{eq:Ve}
\end{equation}
\end{defn}

For a given sequence $u_{j}$ satisfying (2), we consider the
reparameterization
\begin{equation*}
\overline{u}_{j}(\tau ,t)=u_{j}\left( \frac{\tau }{\varepsilon _{j}},\frac{t%
}{\varepsilon _{j}}\right)
\end{equation*}%
on the domain $[-\varepsilon _{j}R(\varepsilon _{j}),\varepsilon
_{j}R(\varepsilon _{j})]\times {\mathbb{R}}/2\pi \varepsilon _{j}{\mathbb{Z}}
$. A straightforward calculation shows that $\overline{u}_{j}$ satisfies
\begin{equation*}
\frac{\partial \overline{u}_{j}}{\partial \tau }+J_{0}\left( \frac{\partial
\overline{u}_{j}}{\partial t}-X_{f}(\overline{u}_{j})\right) =0
\end{equation*}%
or equivalently
\begin{equation*}
\frac{\partial \overline{u}_{j}}{\partial \tau }+J_{0}\frac{\partial
\overline{u}_{j}}{\partial t}+\func{grad}_{J_{0}}f(\overline{u}_{j})=0
\end{equation*}%
on $[-\varepsilon _{j}R(\varepsilon _{j}),\varepsilon _{j}R(\varepsilon
_{j})]\times {\mathbb{R}}/2\pi \varepsilon _{j}{\mathbb{Z}}$. For the
simplicity of notation, we will sometimes denote
\begin{equation*}
R_{j}=R(\varepsilon _{j}).
\end{equation*}%
The following result was proved in Part II of \cite{oh:dmj}. For symplectic
manifolds with Hamiltonian $S^{1}$ action, a similar result was also
obtained by Mundet i Riera and Tian for twisted holomorphic maps (See
Theorem 1.3 \cite{mundet-tian}), and by G. Xu \cite{Xu} when the twisted
holomorphic maps have Lagrangian boundary condition.

\begin{thm}[\protect\cite{oh:dmj}, \protect\cite{mundet-tian}]
\label{thm:centrallimit} Suppose
\begin{equation*}
l = \lim_{j \to \infty}\varepsilon_jR(\varepsilon_j), \quad \lim_{j \to
\infty}E_{J,\Theta_{R_j}}(u_j) = 0.
\end{equation*}
Then there exists a subsequence, again denoted by $u_j$, such that the
reparameterized maps $\overline u_j$ satisfy the following:

\begin{enumerate}
\item Consider the supremum
\begin{equation*}
\func{width} \overline u_j|_{[-l_j,l_j] \times S^1]}: = \sup_{\tau \in [-l_j,l_j]}
\func{diam} \func{Im} \overline u_j|_{\{\tau\} \times S^1}, \quad l_j : = \varepsilon_j R(\varepsilon_j)
\end{equation*}
Then $\func{width} \overline u_j|_{[-l_j,l_j] \times S^1]} \to 0$ and in
particular the center of mass of $\overline u_j:[-l_j,l_j] \times S^1$ defines a
smooth path
\begin{equation*}
\func{cm}(\overline u_j): [-l_j,l_j] \to M
\end{equation*}
and we can uniquely write
\begin{equation*}
\overline u_j(\tau,t) = \exp_{\func{cm}(\overline u_j)}(\tau) \xi_j(\tau,t)
\end{equation*}
so that $\int_{S^1} \xi_j(\tau,t)\,dt = 0$ for all $\tau \in [-l_j,l_j]$.
\item The path $\func{cm}(\overline u_j)$ converges to a gradient trajectory
$\chi:[-l,l] \to M$ satisfying $\dot \chi + \func{grad}_J f(\chi) = 0$ and $%
\xi_j \to 0$ in $C^\infty$-topology.
\end{enumerate}
\end{thm}

Under the assumption $\lim_{j\rightarrow \infty }E_{J,\Theta _{Rj}}(u_{j})=0$%
, after taking away bubbles, on $(-\infty ,K]\times S^{1}$ of any fixed $K$,
the translated sequences $u_{j}(\cdot -\tau \left( \varepsilon _{j}\right)
,\cdot ):{\mathbb{R}}\times S^{1}\rightarrow M$ of solutions $u_{j}$ of (\ref%
{eq:duPK}) as above converge to $u_{-}:{\mathbb{R}}\times S^{1}\rightarrow M$
that satisfies the equation
\begin{equation*}
\frac{\partial u_{-}}{\partial \tau }+J_{-}\left( \frac{\partial u_{-}}{%
\partial t}-X_{H_{-}}(u_{-})\right) =0
\end{equation*}%
in compact $C^{\infty }$-topology, where $H_{\pm }$ are the Hamiltonians
\begin{equation*}
H_{\pm }(\tau ,t,x)=\kappa ^{\pm }(\tau )H(t,x).
\end{equation*}%
Similar statement holds for $u_{j}(\cdot +R(\varepsilon _{j})+S\left(
\varepsilon _{j}\right) ,\cdot )$ at $+\infty $.

\section{Fredholm theory of Floer trajectories near gradient segments}

\label{sec:Qinmiddle}In this section, we study Floer trajectories near a
gradient segment $\chi $. Since $\chi $ itself is a Floer trajectory with $%
S^{1}$ symmetry, transversality is hard to achieve. We will set up
appropriate Banach manifold hosting $\chi $ such that $\chi $ is a
transversal Floer trajectory. During this section we let $J$ be a $t$%
-independent almost complex structure compatible with $\omega $.

\subsection{The Banach manifold set up}

For the gluing purpose, we treat $\chi $ as a $t$-independent Floer
trajectory $u_{\chi }:[-l,l]\times S^{1}\rightarrow M$ of the equation
\begin{equation}
\frac{\partial u}{\partial \tau }+J(u)\left( \frac{\partial u}{\partial t}%
-X_{f}(u)\right) =0.  \label{floer}
\end{equation}

\subsubsection{Center of mass}

We define
\begin{equation}
{\mathcal{B}}_{\chi }=\{u\in W^{1,p}([-l,l]\times S^{1},M)\mid \text{dist}%
\left( u\left( \tau ,t\right) ,\chi \left( \tau \right) \right) <d\text{ for
all }\tau ,t.\},  \label{B-mfd:FixEnds}
\end{equation}%
where $d>0$ is a small constant that for any loop $w:S^{1}\rightarrow \left(
M,g\right) $ whose image has diameter less than $d$, the center of mass of
the loop $w$ is well defined, and is denoted by cm$\left( w\right) $.
Therefore for any $u\in {\mathcal{B}}_{\chi }$, its \emph{center of mass
curve}
\begin{equation*}
\text{cm}\left( u\right) :\left[ -l,l\right] \rightarrow M
\end{equation*}
is well defined and is close to $\chi \left( \tau \right) $. Here
\begin{equation*}
\tau \mapsto \int_{S^1}u(\tau,t) \, dt
\end{equation*}
is the \emph{center of mass} of the loops $t \mapsto u(\tau ,t)$, and is
well defined for $u$ close enough to $\chi $: The center of mass of a loop $%
\gamma:S^1 \to M$ is defined to be the unique point $m_\gamma \in M$ such
that
\begin{equation*}
\int_0^1 \func{dist}^2(m,\gamma(t)) \, dt
\end{equation*}
is the minimum. (See \cite{katcher} for the detailed exposition on the
center of mass in general).

\begin{rem}
The center of mass of a curve in a Riemannian manifold is well defined
whenever the diameter of the curve is sufficiently small. In particular the
condition $\text{diam}(u(\tau ,t))<C\varepsilon $ enables us to define the
center of mass of the curve $u(\tau ,t)\;(t\in S^{1})$ when $\varepsilon $
is sufficiently small. Therefore we can also use a Darboux chart containing
the image of $u(\tau ,t)$ and may identify the curve $t\mapsto u(\tau ,t)$
as one in ${\mathbb{C}}^{n}$. With this understood, we will sometimes denote
the center of mass $\overline{u}(\tau )$ just as the average $%
\int_{S^{1}}u(\tau ,t)\,dt$.
\end{rem}

\subsubsection{Linearization operator and its Fourier expansion}

We choose a $J$-linear connection and take a trivialization
of $u_{\chi }^{\ast }TM$ using the associated parallel transport denoted by
\begin{equation*}
\Phi :u_{\chi }^{\ast }TM\rightarrow \left[ -l,l\right] \times S^{1}\times
\mathbb{C}^{n}
\end{equation*}
over $[-l,l]\times S^{1}$. In this trivialization, the linearization of the
above equation has the form
\begin{equation}
D\xi =\frac{\partial \xi }{\partial \tau }+J_{0}\frac{\partial \xi }{%
\partial t}+A\left( \tau \right) \xi  \label{L-floer}
\end{equation}%
for any vector field $\xi :[-l,l]\times S^{1}\rightarrow {\mathbb{C}}^{n}$
where $J_{0}$ is the standard complex structure on ${\mathbb{C}}^{n}$,
independent of $(\tau ,t)$, and
\begin{equation*}
A=:\xi \rightarrow \left( \nabla _{\chi ^{\prime }}\Phi +\nabla _{\Phi
}\nabla f(\chi (\tau ))\right) \xi
\end{equation*}%
is a 0-th order linear differential operator. If we change $f$ to $%
\varepsilon f$, then $\chi $ is changed to $\chi _{\varepsilon }:=\chi
\left( \varepsilon \cdot \right) $, $\Phi $ is changed to $\Phi
_{\varepsilon }:=\Phi \left( \varepsilon \cdot ,\cdot \right) $, and $A$ is
changed to $A_{\varepsilon }:=\varepsilon A$. So without loss of any
generality, we may assume that $A$ is as small as we want by considering the
Morse function $\varepsilon f$ for some small $\varepsilon$. We denote the
linearization operator of \eqref{floer} for the function $\varepsilon f$ by $%
D^\varepsilon_{\Phi}$ which has the form
\begin{equation*}
D^\varepsilon_{\Phi}\xi =\frac{\partial \xi }{\partial \tau }+J_{0}\frac{%
\partial \xi }{\partial t}+A_{\varepsilon }\left( \tau \right) \xi
\end{equation*}

We are going to construct a Banach manifold hosting Floer trajectories
nearby $u_{\chi }$ that matches $u_{\pm }$, on which good estimate of the
right inverse $Q$ of $D$ can be obtained. The idea is to use the Fourier series
(of variable $t\in S^{1}$) to decompose the variation vector field $\xi $
into the \textquotedblleft zero mode" part and the \textquotedblleft higher
mode" part. For the zero mode, it is tied with the Fredholm theory of the
linearized gradient operator $L$ of Morse function $f$ that%
\begin{equation*}
L\xi =\frac{\partial \xi }{\partial \tau }+A\left( \tau \right) \xi .
\end{equation*}%
For the higher mode part, it has no zero spectrum so uniform bound of the
right inverse can be obtained. This observation was used in \cite{fukaya-ono}
to compute Floer homology by Morse homology of $f$ using $S^{1}$-invariant
Kuranish structure. In our \textquotedblleft
thimble-flow-thimble\textquotedblright\ gluing, only the boundary value of
the \textquotedblleft zero mode" on gradient segments enters the matching
condition from the \textquotedblleft thimble-flow-thimble" transversality
(definition \ref{tft-trans}), while matching of the \textquotedblleft higher
modes" is not needed, thanks to the uniform bound of the right inverse on
higher mode space.

\begin{rem}
For this analysis via the Fourier decomposition on the region $[-l,l] \times
S^1$, it is important to assume $J = J_0$ to be $t$-independent after the
trivialization of $\chi ^{\ast }TM\rightarrow \lbrack -l,l]$ so that the
operator $L$ becomes $S^1$-invariant. %\red{On the other hand,
%if we use $t$-dependent $J$, the transversality of $\chi $ as a Floer
%trajectory is easier to achieve. This is the case for the transformed Floer
%equation $\left( \ref{LH-Floer-eq-e}\right) $ for Lagrangian pearl complex
%and will be discussed later.}
\end{rem}

\subsubsection{Weighted Sobolev norms: along $\protect\chi$}

Along the gradient flow $\chi \left( \tau \right) $, we decompose any
section $\xi \in \Gamma (W^{1,p}(u_{\chi }^{\ast }TM)$) into
\begin{equation*}
\xi = \xi_0 + \widetilde \xi
\end{equation*}
where
\begin{equation}
\xi _{0}(\tau ):=\int_{S^{1}}\xi (\tau ,t)dt  \label{VecField:flow}
\end{equation}%
is the time average of $\xi$, i.e., the zero mode, and
\begin{equation}
\tilde{\xi}(\tau ,t):=\xi (\tau ,t)-\xi _{0}(\tau ).
\label{VecField:reduced}
\end{equation}%
is the part of higher Fourier modes.

Then we define the Banach norm of $\xi \in T_{u_{\chi }}{\mathcal{B}}_{\chi
} $ to be
\begin{equation}
\Vert \xi \Vert _{\left( W_{\rho}^{1,p}\left[ -l,l\right] \times
S^{1}\right) }=\left\Vert \xi _{0}\right\Vert _{W^{1,p}([-l,l])}+\Vert
\tilde{\xi}\Vert _{W_{\rho }^{1,p}([-l,l]\times S^{1})},  \label{B-norm}
\end{equation}%
where
\begin{eqnarray*}
\left\Vert \xi _{0}\right\Vert _{W^{1,p}([-l,l])}^{p} &=&\int_{-l}^{l}\left(
|\xi _{0}|^{p}+|\nabla _{\tau }\xi _{0}|^{p}\right) d\tau , \\
\Vert \tilde{\xi}\Vert _{W_{\rho }^{1,p}([-l,l]\times S^{1})}^{p}
&=&\int_{S^{1}}\int_{-l}^{l}\left( |\tilde{\xi}|^{p}+|\nabla \tilde{\xi}%
|^{p}\right) \left( 1+\left\vert \tau \right\vert \right) ^{\delta }d\tau dt
\end{eqnarray*}%
We call $\rho \left( \tau \right) :=\left( 1+\left\vert \tau \right\vert
\right) ^{\delta }$ the weighting function for the above weighted $W_{\rho
}^{1,p}$\ norm. \ We remark that the norm $\Vert \tilde{\xi}\Vert _{W_{\rho
}^{1,p}([-l,l]\times S^{1})}$ is equivalent to the norm $\Vert \left(
1+\left\vert \tau \right\vert \right) ^{\frac{\delta }{p}}\tilde{\xi}\Vert
_{W^{1,p}([-l,l]\times S^{1})}$.

Similarly for any section $\eta \in \Gamma (W^{1,p}(u_{\chi }^{\ast
}TM)\otimes _{J}\Lambda ^{0,1}\left( [-l,l]\times S^{1}\right) $, we
decompose it into $\eta = \eta_0 + \widetilde \eta$ with
\begin{eqnarray*}
\eta _{0}(\tau ) &:&=\int_{S^{1}}\eta (\tau ,t)dt, \\
\tilde{\eta}(\tau ,t) &:&=\eta (\tau ,t)-\eta _{0}(\tau ).
\end{eqnarray*}%
and define
\begin{eqnarray*}
\left\Vert \eta _{0}\right\Vert _{L^{p}([-l,l])}^{p} &=&\int_{-l}^{l}|\eta
_{0}|^{p}d\tau , \\
\left\Vert \tilde{\eta}\right\Vert _{L_{\rho }^{p}([-l,l]\times S^{1})}^{p}
&=&\left\Vert \left( 1+\left\vert \tau \right\vert \right) ^{\frac{\delta }{p%
}}\tilde{\eta}\left( \tau ,t\right) \right\Vert _{L^{p}([-l,l]\times S^{1})}
\\
\left\Vert \eta \right\Vert _{L_{\rho}^{p}([-l,l]\times S^{1})}
&=&\left\Vert \eta _{0}\right\Vert _{L^{p}([-l,l])}+\left\Vert \tilde{\eta}%
\right\Vert _{L_{\rho }^{p}([-l,l]\times S^{1})}
\end{eqnarray*}

\subsubsection{Weighted Sobolov norms: along $u\in {\mathcal{B}}_{\protect%
\chi }$}

Along the general elements $u\in {\mathcal{B}}_{\chi }$ that are close
enough to $\chi $ in $C^0$-topology, (which are \textquotedblleft
thin\textquotedblright\ cylinders), we define the norm $\Vert \xi \Vert
_{W_{\rho}^{1,p}}$ for $\xi \in W^{1,p}\left( \Gamma \left( u^{\ast
}TM\right) \right) $ as the following: let
\begin{equation*}
\bar{\xi}\left( \tau ,t\right) =\func{Pal}_{\text{cm}\left( u\right)
}^{u}\xi \left( \tau ,t\right) ,
\end{equation*}%
where $\func{Pal}_{\text{cm}\left( u\right) }^{u}$ is the parallel transport
from $u\left( \tau ,t\right) $ to cm$\left( u\right) \left( \tau \right) $
along the shortest geodesic. Then along cm$\left( u\right) $, similar to $%
\chi \left( \tau \right) $ case, we decompose
\begin{equation*}
\bar{\xi}\left( \tau ,t\right) =\left( \bar{\xi}\right) _{0}\left( \tau
\right) +\widetilde{\bar{\xi}}\left( \tau ,t\right)
\end{equation*}%
and pull back to define
\begin{equation*}
\xi _{0}=\func{Pal}_{u}^{\text{cm}\left( u\right) }\left( \bar{\xi}\right)
_{0}\text{, \ }\tilde{\xi}=\func{Pal}_{u}^{\text{cm}\left( u\right) }%
\widetilde{\bar{\xi}}
\end{equation*}%
where $\func{Pal}_{u}^{\text{cm}\left( u\right) }$ is the parallel transport
from cm$\left( u\right) \left( \tau \right) $ to back $u\left( \tau
,t\right) $ along the shortest geodesic. Therefore we have well defined
decomposition%
\begin{equation*}
\xi \left( \tau ,t\right) =\xi _{0}\left( \tau \right) +\tilde{\xi}\left(
\tau ,t\right) .
\end{equation*}%
for $W^{1,p}\left( \Gamma \left( u^{\ast }TM\right) \right) $. Then we
define
\begin{eqnarray*}
\left\Vert \xi _{0}\right\Vert _{W^{1,p}\left( \left[ -l,l\right] \right) }
&=&\left\Vert \left( \bar{\xi}\right) _{0}\right\Vert _{W^{1,p}\left( \left[
-l,l\right] \right) } \\
\left\Vert \tilde{\xi}\right\Vert _{W_{\rho }^{1,p}\left( \left[ -l,l\right]
\times S^{1}\right) } &=&\left\Vert \widetilde{\bar{\xi}}\right\Vert
_{W_{\rho }^{1,p}\left( \left[ -l,l\right] \times S^{1}\right) } \\
\left\Vert \xi \right\Vert _{W_{\rho}^{1,p}\left( \left[ -l,l\right] \times
S^{1}\right) } &=&\left\Vert \xi _{0}\right\Vert _{W^{1,p}\left( \left[ -l,l%
\right] \right) }+\left\Vert \tilde{\xi}\right\Vert _{W_{\rho }^{1,p}\left( %
\left[ -l,l\right] \times S^{1}\right) }.
\end{eqnarray*}%
For $\eta \in L^{p}(u^{\ast }TM\otimes _{J}\Lambda ^{0,1}([-l,l]\times
S^{1}))$, the norm
\begin{equation*}
\left\Vert \eta \right\Vert _{L_{\rho}^{p}([-l,l]\times S^{1})}=\left\Vert
\eta _{0}\right\Vert _{L^{p}([-l,l])}+\left\Vert \tilde{\eta}\right\Vert
_{L_{\rho }^{p}([-l,l]\times S^{1})}
\end{equation*}%
is defined similarly by using parallel transport to cm$\left( u\right) $ and
decomposition $\eta =\eta _{0}+\tilde{\eta}$. Let
\begin{equation*}
L_{\rho}^{p}(u^{\ast }TM\otimes _{J}\Lambda ^{0,1}([-l,l]\times
S^{1}))=\left\{\eta \in L^{p}\left( [-l,l]\times S^{1}\right) \, \Big|\,
\left \Vert \eta \right\Vert _{L_{\rho}^{p}([-l,l]\times S^{1})}<\infty
\right\} .
\end{equation*}

\subsubsection{The nonlinear Floer operator as a Fredholm section}

We define the Banach bundle
\begin{equation*}
{\mathcal{L}}^{(0,1)}_{J,\chi} \to \mathcal{B}_\chi
\end{equation*}
where
\begin{equation*}
{\mathcal{L}}_{J,\chi }^{(0,1)} : =\bigcup\limits_{u\in {\mathcal{B}}_{\chi
}}L_{\rho}^{p}(u^{\ast }TM\otimes _{J}\Lambda ^{0,1}([-l,l]\times S^{1})).
\end{equation*}

%The boundary
%condition for $\xi_0$ and $\tilde{\xi}$ are
%\be\label{bdryCondition}%
%\xi_0(\pm l)=v_{\pm}, \qquad \tilde{\xi}(\pm l,t)=O(e^{-|l|}
%\|\tilde{\xi}\|_{1,p} )
%\ee%
%because the tangent vectors $v_{\pm}\in T_{p\pm}M$ has only zero
%Fourier mode.
%The second boundary condition is motivated from the asymptote of
%the outer curves and local models in [OZ1]. It should be regarded
%as some "finite version" (depending on $l$) of the decaying
%condition.
Then the Floer operator
\begin{equation*}
{\overline{\partial }}_{J_{0},f}:u\mapsto \left( {\partial }_{\tau }u+J_{0}{%
\partial }_{t}u+A\left( \tau \right) u\right) \otimes \left( d\tau -it\right)
\end{equation*}
gives a Fredholm section $\mathcal{F}$ of the Banach bundle ${\mathcal{L}}%
_{\chi, J}^{(0,1)}\rightarrow {\mathcal{B}}_{\chi }$ whose covariant
linearization
\begin{equation*}
D_u({\overline{\partial }}_{J_{0},f})
\end{equation*}
with respect to the connection $\nabla$ is given by
\begin{equation}  \label{eq:covariant-linearization}
D^\varepsilon_\Phi : \Omega^0_{1,p}(u^*TM) \to \Omega^{(0,1)}_{J;p}(u^*TM)
\end{equation}
in the trivialization $\Phi$ for each fixed $\varepsilon$. Here we write
\begin{eqnarray*}
\Omega^0_{1,p}(u^*TM)& : = & W^{1,p}([-l,l] \times S^1,M) \\
\Omega^{(0,1)}_{J;p}(u^*TM) & : = & L_{\rho}^{p}(u^{\ast }TM\otimes
_{J}\Lambda ^{0,1}([-l,l]\times S^{1})).
\end{eqnarray*}

\subsection{Decomposition into zero mode and the higher modes}

For any section $\xi \in W^{1,2}(u_{\chi }^{\ast }TM)$, we take the Fourier
expansion
\begin{equation*}
\xi (\tau ,t)=\sum_{-\infty }^{\infty }a_{k}(\tau )e^{2\pi ikt},
\end{equation*}%
where $a_{k}(\tau )$ are vectors in $T_{\chi (\tau )}M$. It is easy to see
\begin{equation*}
\xi _{0}(\tau )=a_{0}(\tau ),\qquad \tilde{\xi}\left( \tau \right)
=\sum_{k\neq 0}a_{k}(\tau )e^{2\pi ikt}.
\end{equation*}

\begin{defn}[$V_{0}$ and $\widetilde{V}$]
Let $V_{0}$ and $\widetilde{V}$ be the $L^{2}$-completions of the spans of
zero Fourier mode and of higher Fourier modes respectively, then we have the
$L^{2}$-decomposition
\begin{equation*}
W^{1,2}(u_{\chi }^{\ast }TM)=V_{0}\oplus \widetilde{V}
\end{equation*}%
where we still denote by $V_{0},\,\widetilde{V}$ the intersections
\begin{equation*}
V_{0}\cap W^{1,2}(u_{\chi }^{\ast }TM),\,\widetilde{V}\cap W^{1,2}(u_{\chi
}^{\ast }TM)
\end{equation*}%
respectively.
\end{defn}

We observe that $V_{0}$ and $\widetilde{V}$ are invariant subspaces of
operator
\begin{equation*}
D=\frac{\partial }{\partial \tau }+J_{0}\frac{\partial }{\partial t}+A(\tau
)
\end{equation*}
and so $D$ splits into
\begin{equation*}
D=D_{0}\oplus \widetilde{D}:V_{0}\oplus \widetilde{V}\rightarrow V_{0}\oplus
\widetilde{V}\text{,}
\end{equation*}%
where $\widetilde{D}=D|_{\widetilde{V}}:V_{0}\rightarrow V_{0}$, and $%
D_{0}=D|_{V_{0}}:\widetilde{V}\rightarrow \widetilde{V}$. Notice that $D_{0}=%
\frac{\partial }{\partial \tau }+A\left( \tau \right) $ is exactly the
linearized gradient operator $L$ of $f$.

For the construction of the right inverse $Q$ of $D$, we use Fourier
expansions of $\eta $. For any given $\eta $, we write
\begin{equation*}
\eta (\tau ,t)=\sum_{-\infty }^{\infty }b_{k}(\tau )e^{i2\pi kt}.
\end{equation*}%
Now the equation $D\xi =\eta $, i.e. $\partial _{\tau }\xi +J_{0}\partial
_{t}\xi +A(\tau )\xi =\eta $ splits into
\begin{equation}
a_{k}^{\prime }(\tau )+\left( A(\tau )-2\pi k\right) a_{k}\left( \tau
\right) =b_{k}(\tau )\qquad \text{for all }\;k\in {\mathbb{Z}}
\label{rightinverse}
\end{equation}%
Especially when $k=0$, it becomes
\begin{equation}  \label{eq:b0}
a_{0}^{\prime }\left( \tau \right) +A(\tau )a_{0}\left( \tau \right)
=b_{0}\left( \tau \right) .
\end{equation}

We note $\left( \ref{eq:b0}\right) ~$is exactly the linearized gradient flow
equation. We can always solve $\left( \ref{rightinverse}\right) $
\begin{equation}
a_{k}(\tau )=e^{-\int_{0}^{\tau }\left( A(\mu )-2\pi k\right) d\mu }\left[
\int_{0}^{\tau }b_{k}(s)e^{\int_{0}^{s}\left( A(\mu )-2\pi k\right) d\mu
}ds+C_{k}\right]  \label{sol}
\end{equation}%
by the variation of constants with $C_{k}$ arbitrary constant. Any choice of
$C_{k}$ will produce a right inverse of $D$.

However for the resulting right inverse to carry uniform bound independent
of $\varepsilon > 0$, we need to impose a good boundary condition on $a_k$,
which in turn requires us to make a good choice of the free constants $C_{k}$%
%.

\begin{cond}[Boundary condition]
\label{cond:boundary-conditions} We put the following boundary conditions

\begin{itemize}
\item For $k=0$, we put
\begin{equation}
a_{0}(\pm l)= \xi_{0}(\pm l)\in d(ev_{\pm })\left( T_{(u_{\pm },o_{\pm })}%
\mathcal{M}_{2}\left( K_{\pm },J_{\pm };A_{\pm }\right) \right) ,
\label{BdryCond:0Mode}
\end{equation}%
where $ev_{\pm }:{\mathcal{M}}_{2}\left( K_{\pm },J_{\pm };A_{\pm }\right)
\rightarrow M$ are the evaluation maps
\begin{equation*}
ev_{\pm }(u_{\pm },o_{\pm })=u_{\pm}(o_{\pm }).
\end{equation*}

\item For $k\neq 0$, we impose one-point boundary condition
\begin{equation}
a_{k}(l)=0\;\text{ if }k>0;\qquad a_{k}(-l)=0\;\text{ if }k<0.
\label{BdryCond:HigherMode}
\end{equation}
\end{itemize}
\end{cond}

The thimble-flow-thimble transversality condition, Proposition \ref%
{family-tft-trans}, enables us to solve this two point boundary problem %
\eqref{BdryCond:0Mode}.

The boundary condition can be always satisfied since the equation %
\eqref{eq:b0} is a first order linear ODE or one can choose $C_{k}$
arbitrarily. In fact the one-point boundary condition uniquely determines $%
a_{k}$ for $k\neq 0$:
\begin{equation}
a_{k}(\tau )=%
\begin{cases}
-\int_{\tau }^{l}b_{k}(s)e^{\int_{\tau }^{s}\left( A(\mu )-2\pi k\right)
d\mu }ds\; & \text{ if }k>0, \\
\int_{-l}^{\tau }b_{k}(s)e^{\int_{\tau }^{s}\left( A(\mu )-2\pi k\right)
d\mu }ds\; & \text{ if }k<0.%
\end{cases}
\label{sol:ak}
\end{equation}

This constructs the right inverse $Q$ of $D$ by applying the above Fourier
expansion to the value $Q(\eta) =: \xi$ of the operator $Q$. It follows from %
\eqref{eq:b0}  we see $b_{0}\left( \tau \right) =0$ if and only if $%
a_{0}\left( \tau \right) =0$, thus $Q$ also splits into
\begin{equation}  \label{eq:Q}
Q = Q_0 \oplus \widetilde Q: V_0 \oplus \widetilde V \to V_0 \oplus
\widetilde V
\end{equation}
where $Q_{0}=Q|_{V_{0}}$ and $\widetilde{Q}=Q|_{\widetilde{V}}$.

In fact, the above discussion shows that if the image of $Q$ becomes the
subspace
\begin{equation*}
W_{0}^{1,2}(u_{\chi }^{\ast }TM):=\{\xi \in W_{0}^{1,2}(u_{\chi }^{\ast
}TM)\mid a_{k}(l )=0\,\mbox{for }\,k>0,\quad a_{k}(-l )=0,,\mbox{for }\,k<0\}
\end{equation*}%
then restriction of $D$ to the subspace is an isomorphism with its inverse
given by $Q$.

\begin{rem}
The boundary condition \eqref{BdryCond:HigherMode} is geared more for the $%
L^{2}$ estimate of $Q$ rather than matching with the $J$-holomorphic spheres
$u_{\pm }$. We can't put two-point boundary condition for each $a_{k}(k\neq
0) $, since we have only one free constant $C_{k}$ in \eqref{sol}. There are
lots of choices of the right inverse $Q$, with various operator norm bounds,
but for the uniform $L^{2}$ estimate of $Q$ our choice seems to be the most
optimal one, which can be seen in the following estimate \eqref{KeyIdentify}.
\end{rem}

\subsection{$L^{2}$ estimate of the right inverse for the higher modes}

We estimate $\Vert \tilde{\xi}\Vert _{2}$. For $k\neq 0$,
\begin{eqnarray}
&&\left( 2\pi k\right) ^{2}\int_{-l}^{l}\left\vert a_{k}(\tau )\right\vert
^{2}d\tau  \notag \\
&\leq &\int_{-l}^{l}(\left\vert a_{k}^{\prime }\right\vert ^{2}+\left( 2\pi
k\right) ^{2}\left\vert a_{k}(\tau )\right\vert ^{2}d\tau  \notag \\
&=&\int_{-l}^{l}\left\vert a_{k}^{\prime }-2\pi ka_{k}\right\vert ^{2}d\tau
+2\pi k\int_{-l}^{l}2a_{k}\cdot a_{k}^{\prime }d\tau  \notag \\
&=&\int_{-l}^{l}\left\vert b_{k}\left( \tau \right) -A(\tau )a_{k}\left(
\tau \right) \right\vert ^{2}d\tau +2\pi k\int_{-l}^{l}\frac{d}{d\tau }%
\left\vert a_{k}(\tau )\right\vert ^{2}d\tau  \notag \\
&=&\int_{-l}^{l}\left\vert b_{k}\left( \tau \right) -A(\tau )a_{k}\left(
\tau \right) \right\vert ^{2}d\tau +2\pi k(\left\vert a_{k}(l)\right\vert
^{2}-\left\vert a_{k}(-l)\right\vert ^{2})  \notag \\
&\leq &2\left( \int_{-l}^{l}\left\vert b_{k}\right\vert ^{2}d\tau +\delta
^{2}\int_{-l}^{l}\left\vert a_{k}\right\vert ^{2}d\tau \right) +2\pi
k(\left\vert a_{k}(l)\right\vert ^{2}-\left\vert a_{k}(-l)\right\vert ^{2})
\label{KeyIdentify}
\end{eqnarray}%
provided $\left\vert A\left( \tau \right) \right\vert _{\infty }<\delta $.
By the boundary condition \eqref{BdryCond:HigherMode} of $a_{k}$, the second
summand of the last inequality is never positive, and noting $\delta
<1-1/p<1 $, so we get
\begin{equation*}
\int_{-l}^{l}\left\vert a_{k}\left( \tau \right) \right\vert ^{2}d\tau \leq
\frac{2}{\left( 2\pi k\right) ^{2}-2\delta ^{2}}\int_{-l}^{l}\left\vert
b_{k}\left( \tau \right) \right\vert ^{2}d\tau \leq \int_{-l}^{l}\left\vert
b_{k}\left( \tau \right) \right\vert ^{2}d\tau \text{ when }k\neq 0,
\end{equation*}%
Summing over $k\neq 0$ we get
\begin{equation}
\Vert \tilde{\xi}\Vert _{L^{2}\left[ -l,l\right] }^{2}\leq \Vert \tilde{\eta}%
\Vert _{L^{2}\left[ -l,l\right] }^{2}.  \label{Estimate:L2}
\end{equation}

From \eqref{BdryCond:HigherMode} and \eqref{KeyIdentify}, we also get
\begin{equation*}
0\leq (\left( 2\pi k\right) ^{2}-2\delta ^{2})\int_{-l}^{l}\left\vert
a_{k}\right\vert ^{2}\leq 2\int_{-l}^{l}\left\vert b_{k}\right\vert
^{2}+2\pi k(0-\left\vert a_{k}(-l)\right\vert ^{2})
\end{equation*}%
for $k>0$ and similar inequality for $k<0$. Hence
\begin{equation*}
|a_{k}(-l)|\leq \sqrt{\frac{1}{k\pi }}\Vert b_{k}\Vert _{L^{2}([-l,l])}\;%
\text{ if }k>0;\qquad |a_{k}(l)|\leq \sqrt{\frac{-1}{k\pi }}\Vert b_{k}\Vert
_{L^{2}([-l,l])}\;\text{ if }k<0.
\end{equation*}%
Squaring and summing over $k\neq 0$ we also have
\begin{equation}
\sum_{k>0}ka_{k}^{2}(-l)\leq \frac{1}{\pi }\Vert \tilde{\eta}\Vert
_{2}^{2},\qquad \sum_{k<0}|k|a_{k}^{2}(l)\leq \frac{1}{\pi }\Vert \tilde{\eta%
}\Vert _{2}^{2}.  \label{eq:sumkak}
\end{equation}

\begin{rem}
It would be nice if we can get $C^{0}$ estimate $\tilde{\xi}(\pm l,t)$ in
terms of $\Vert \tilde{\eta}\Vert _{2}$, but it seems that we can at most
get the $W^{\frac{1}{2},2}$ norm estimate of $\tilde{\xi}(\pm l,t)$ by the
above summation inequalities. However, we will later derive the $C^{0}$
estimate of $\tilde{\xi}$ by $W^{1,p}$ estimate and Sobolev embedding.
\end{rem}

\begin{rem}
\bigskip If we extend $\eta $ to be $0$ outside $\left[ -l,l\right] $, then
we can think $\eta $ is on the full gradient trajectory $\chi ,$
corresponding to the case where $l=\infty $. Then we can still use the above
method to construct $\widetilde{\xi }$ with slightly different boundary
condition%
\begin{equation*}
a_{k}(\infty )=0\;\text{ if }k>0;\qquad a_{k}(-\infty )=0\;\text{ if }k<0.
\end{equation*}%
(Here $a_{k}(\pm \infty )$ makes sense, since in their defining integrals $%
b_{k}\left( \tau \right) $ is compactly supported). For such $\widetilde{\xi
}$ we gain stronger inequality
\begin{equation*}
\Vert \tilde{\xi}\Vert _{L^{2}\left( -\infty ,\infty \right) }^{2}\leq \Vert
\tilde{\eta}\Vert _{L^{2}\left[ -l,l\right] }^{2}.
\end{equation*}%
We choose this $\tilde{\xi}$ in the remaining part of our paper, because the
above stronger inequality is needed to construct the approximate right
inverse.
\end{rem}

\subsection{$L^{p}$ estimate of the right inverse}

We start with estimating the $W^{1,p}$-norm $\|\xi _{0}\|_{W^{1,p}}$ of the
zero modes. Notice that $\xi _{0}$ satisfies the linearized gradient flow
equation%
\begin{equation*}
\frac{\partial }{\partial \tau }\xi _{0}\left( \tau \right) +A(\tau )\xi
_{0}\left( \tau \right) =\eta _{0}\left( \tau \right) .
\end{equation*}%
Note the gradient flow $\chi $ can always be extended to a full gradient
flow connecting two critical points, whose Fredholm theory and
transversality has been well established. Therefore by extending $\eta $ to
be $0$ outside $\left[ -l,l\right] $ and using the right inverse for $%
L_{\chi }$ on the full gradient trajectory $\chi $, $\xi _{0}$ is always
solvable and%
\begin{equation*}
\left\Vert \xi _{0}\right\Vert _{W^{1,p}\left( -\infty ,\infty \right) }\leq
C_{p}\left\Vert \eta _{0}\right\Vert _{L^{p}[-l,l]},
\end{equation*}%
especially
\begin{equation}
\left\Vert \xi _{0}\right\Vert _{W^{1,p}[-l,l]}\leq C_{p}\left\Vert \eta
_{0}\right\Vert _{L^{p}[-l,l]}  \label{eq:xi0leta0}
\end{equation}%
for some constant $C_{p}\left( p\geq 2\right) .$

Now assume $p\geq 2$ and examine the $L^p$-norm estimate of the higher
modes. Without loss of generality we assume $l$ is an integer greater than $1
$ (The general case $l\geq l_{0}>0$ is similar). Let $B_{m}=[m-1/2,m+1/2]%
\times S^{1}$ and $B_{m}^{+}=[m-1,m+1]\times S^{1}$. Given $\eta $ on $[-l,l]
$, we extend it on $[-l-1,l+1]$ by letting $\eta =0$ outside $[-l,l]$. Using
the method in the above section we construct $Q(\eta )$, then restrict it on
$[-l,l]$.

By the elliptic regularity we can improve $L^{2}$ estimate to $L^{p}$
estimate for the higher modes
\begin{equation*}
\widetilde \xi = \xi - \xi_0,
\end{equation*}
if we enlarge the region a bit. In our case the inequality is of the form
\begin{equation}
\Vert \tilde{\xi}\Vert _{W^{1,p}(B_{m})}\leq C(\Vert D\tilde{\xi}\Vert
_{L^{p}(B_{m}^{+})}+\Vert \tilde{\xi}\Vert _{L^{2}(B_{m}^{+})})  \label{L2Lp}
\end{equation}%
where the constant $C$ is independent on $m$. Summing $m$ from $-l$ to $l$,
we get
\begin{eqnarray}
\Vert \tilde{\xi}\Vert _{W^{1,p}([-l,l]\times S^{1})} &\leq &C(\Vert D\tilde{%
\xi}\Vert _{L^{p}([-l-1,l+1]\times S^{1})}+\Vert \tilde{\xi}\Vert
_{L^{2}([-l-1,l+1]\times S^{1}})  \notag  \label{L2LpS} \\
&\leq &C(\Vert \tilde{\eta}\Vert _{L^{p}([-l,l]\times S^{1})}+\Vert \tilde{%
\eta}\Vert _{L^{p}([-l,l]\times S^{1})})  \notag \\
&=&2C\Vert \tilde{\eta}\Vert _{L^{p}([-l,l]\times S^{1})}
\end{eqnarray}%
where the second inequality holds because of $\tilde{\eta}=0$ on $\pm
\lbrack l,l+1]$ and the $L^{2}$ estimate \eqref{Estimate:L2}. Here $C$ is
independent of $l$.

For $\xi_{0}$, we have \eqref{eq:xi0leta0}. Altogether we get the right
inverse
\begin{equation*}
Q:L^{p}\left( [-l,l]\times S^{1}\right) \to W^{1,p}\left( [-l,l]\times
S^{1}\right) ; \, \eta \rightarrow \xi
\end{equation*}
with uniform bound
\begin{equation*}
\left\Vert Q\eta \right\Vert _{W^{1,p}\left( [-l,l]\times S^{1}\right) }\leq
C\left\Vert \eta \right\Vert _{L^{p}\left( [-l,l]\times S^{1}\right)}
\end{equation*}
with $C$ independent of the size $l$.

\begin{rem}
\label{Donaldson-Method}Since the construction of the higher modes of $\xi $
is geared for the uniform right inverse bound, not for the matching
condition, we can obtain it by an easier method: We first extend $\eta $
outside $\left[ -l,l\right] $ trivially by letting $\eta =0$ there, then we
can apply the method in Chapter 3 of \cite{Don} to construct the right
inverse of higher mode part of $\frac{\partial }{\partial \tau }+J_{0}\frac{%
\partial }{\partial t}:C^{\infty }\left( \mathbb{R\times }S^{1},\mathbb{C}%
^{n}\right) \rightarrow C^{\infty }\left( \mathbb{R\times }S^{1},\mathbb{C}%
^{n}\right) $, since $J_{0}\frac{\partial }{\partial t}:C^{\infty }\left(
S^{1},\mathbb{C}^{n}\right) \rightarrow C^{\infty }\left( S^{1},\mathbb{C}%
^{n}\right) $ is $\tau $-independent operator and has nonzero spectrum when
restricted on higher mode subspace . Then we simply restrict the obtained $%
\xi =Q\eta $ on $\left[ -l,l\right] $. Then the $L^{2}$ and $L^{p}$ estimate
for $\left\Vert Q\right\Vert $ has already been obtained in \cite{Don}. For
the operator $D^\varepsilon_{\Phi}=\frac{\partial }{\partial \tau }+J_{0}%
\frac{\partial }{\partial t}+A_{\varepsilon }(\tau )$, it is a $\varepsilon $%
-small perturbation of $\frac{\partial }{\partial \tau }+J_{0}\frac{\partial
}{\partial t}$, hence its right inverse bound is inherited from $Q$. In the
next section the right inverse bound can be obtained in the same way, since
for small $\delta $, putting the $\left( 1+\left\vert \tau \right\vert
\right) ^{\delta }$ weight function amounts to taking a $\delta $%
-perturbation of $\frac{\partial }{\partial \tau }+J_{0}\frac{\partial }{%
\partial t}$.
\end{rem}

\subsection{$L_{\protect\rho }^{p}$ estimate of the right inverse\label%
{Lpr_estimate}}

On a fixed domain $[-l,l]\times S^{1}$, the $W^{1,p}$ norm and the $W_{\rho
}^{1,p}$ norm are equivalent, they defines the same function space. But when
$l\rightarrow \infty $, the weighted norm gives better control of the
\textquotedblleft Morse-Bott" variation. This is a soft technique to get
around the point estimate of $\tilde{\xi}(\pm l/\varepsilon ,t)$ that we are
lacking.

Choose $0<\delta <1$. By conjugating with the multiplication of $\rho \left(
\tau \right) ^{\frac{1}{p}}$ with a weighting function given by
\begin{equation}  \label{eq:rho}
\rho \left( \tau \right) =\left(1+\left\vert \tau \right\vert
\right)^{\delta },
\end{equation}
the operator $D:W_{\rho }^{1,p}\rightarrow L_{\rho }^{p}$ is equivalent to $%
D_{\rho }:W^{1,p}\rightarrow L^{p}$, with%
\begin{equation}  \label{eq:D-rho}
D_{\rho }=D-\frac{\left( \rho \left( \tau \right) ^{\frac{1}{p}}\right)
^{\prime }}{\left( \rho \left( \tau \right) \right) ^{\frac{1}{p}}}=D-\frac{%
\delta /p}{1+\left\vert \tau \right\vert }.
\end{equation}%
Then we have the following diagram%
\begin{equation*}
\xymatrix{ W^{1,p} \ar[r]^{D_{\rho }} & L^{p} \\ W_{\rho }^{1,p} \ar[r]^D
\ar[u]^{\rho \left( \tau \right) ^{\frac{1}{p}}} & L_{\rho }^{p}
\ar[u]_{\rho \left( \tau \right) ^{\frac{1}{p}}}}
\end{equation*}%
As for the restriction of the operator
\begin{equation*}
D=\frac{\partial }{\partial \tau }+J_{0}\frac{\partial }{\partial t}+A(\tau
)
\end{equation*}
to the higher modes $\widetilde{V}$ is invertible, $D-\frac{\delta /p}{%
1+\left\vert \tau \right\vert }$ is also invertible since the restriction $%
J_{0}\partial _{t}+A(\tau )$ on
\begin{equation*}
\bigoplus_{k\neq 0}E_{k}\subset L^{2}(S^{1},{\mathbb{R}}^{n})
\end{equation*}%
has its spectrum outside $(-1,1)\subset {\mathbb{R}}$ and we have choose
that $0<\frac{\delta }{p}<1$. Similarly to \eqref{L2LpS}, we obtain
\begin{equation}
\Vert \tilde{\xi}\Vert _{W_{\rho }^{1,p}([-l,l]\times S^{1})}\leq 2C\Vert
\tilde{\eta}\Vert _{L_{\rho }^{p}([-l,l]\times S^{1})}  \label{L2LpSw}
\end{equation}%
where $C$ is independent on $l$.

Now we achieve the following pointwise decay estimate.

\begin{prop}
\label{prop:pointwise-decay} Let $\xi = \xi_0 + \widetilde \xi$ be as above.

\begin{enumerate}
\item For the higher modes, we have
\begin{equation}  \label{eq:higher-modes-Lp}
|\tilde{\xi}\left( \tau ,t\right) |_{C^{0}}\leq \frac{1}{\left\vert \tau
\right\vert ^{\delta }}\Vert \tilde{\xi}\Vert _{_{W_{\rho
}^{1,p}([-l,l]\times S^{1})}}\leq \frac{2C}{\left\vert \tau \right\vert
^{\delta }}\Vert \tilde{\eta}\Vert _{L_{\rho }^{p}([-l,l]\times S^{1})}
\end{equation}%
through by Sobolev embedding $W^{1,p}\hookrightarrow C^{0}$

\item For the zero-mode, we have
\begin{eqnarray}  \label{eq:zero-mode-C0}
|\xi _{0}\left( \tau \right) -\xi _{0}(\pm l)|_{C^{0}} &\leq &C\left\vert
\tau \pm l\right\vert ^{\gamma }\Vert \xi _{0}\Vert _{W^{1,p}([-l,l])}
\notag \\
&\leq &2C\left\vert \tau \pm l\right\vert ^{\delta }\Vert \eta _{0}\Vert
_{L^{p}([-l,l])}  \label{Stablization}
\end{eqnarray}%
when $\left\vert \tau \pm l\right\vert <1$ and $0<\delta <1-1/p$.
\end{enumerate}
\end{prop}

\begin{proof}
The gain of this $W_{\rho }^{1,p}$-estimate of $\widetilde{\xi }$ gives rise
to \eqref{eq:higher-modes-Lp}. The inequality \eqref{eq:zero-mode-C0}
follows from the Sobolev embedding $W^{1,p}\left( \left[ -l,l\right] \right)
\hookrightarrow C^{\gamma }\left( \left[ -l,l\right] \right) $ $\left( l\geq
l_{0}>0\right) $ with $\gamma =1-\frac{1}{p}$.
\end{proof}

\begin{rem}
\label{weights} The power weight $\rho \left( \tau \right) =\left(
1+\left\vert \tau \right\vert \right) ^{\delta }$ takes care of the decay of
the high modes. We choose this weight because the \emph{gradient} \emph{%
segment} $\chi $ converges to its noncritical endpoints $p_{\pm }$ in the
linear order, not in the exponential order. If the $\chi $ is a \emph{full
gradient trajectory} connecting two nondegenerate Morse critical points, we
do not need any weight because the higher mode $\tilde{\xi}$ with finite $%
W^{1,p\text{ }}$norm automatically vanishes at $\tau =\pm \infty $. If one
of the critical end points of $\chi $ is Morse-Bott, we need to put the
exponential weight to capture the correct convergence rate. The power weight
and exponential weight can be unified by one function, say $\left(
\int_{0}^{\left\vert \tau \right\vert }\left\vert \frac{1}{\nabla f\left(
\chi \right) }\right\vert d\tau \right) ^{\delta }+1$.
\end{rem}

\subsection{$\protect\varepsilon $-reparameterization of the gradient segment
and adiabatic weight}

When $\varepsilon \rightarrow 0$, we reparameterize the gradient segment $%
\chi \left( \tau \right) $ to $\chi _{\varepsilon }\left( \tau \right)
:=\chi \left( \varepsilon \tau \right) $ for gluing. We define the map $%
\chi_\varepsilon$ by
\begin{equation}  \label{eq:chi-epsilon}
\chi_\varepsilon(\tau): = \chi(\varepsilon \tau).
\end{equation}
For given $(\varepsilon,\chi)$, we associate a Banach manifold using the
\emph{$\varepsilon$-weighted norms}
\begin{equation*}
\Vert \xi \Vert _{W_{\rho _{\varepsilon }}^{1,p}([-l/\varepsilon
,l/\varepsilon ]\times S^{1}),}=\left\Vert \xi _{0}\left( \tau \right)
\right\Vert _{W_{\varepsilon }^{1,p}([-l/\varepsilon ,l/\varepsilon
])}+\Vert \tilde{\xi}\left( \tau \right) \Vert _{W_{\rho _{\varepsilon
}}^{1,p}([-l/\varepsilon ,l/\varepsilon ]\times S^{1}),}
\end{equation*}%
and
\begin{equation*}
\Vert \eta \Vert _{L_{\rho _{\varepsilon }}^{p}([-l/\varepsilon
,l/\varepsilon ]\times S^{1})}=\left\Vert \eta _{0}\left( \tau \right)
\right\Vert _{L_{\varepsilon }^{p}([-l/\varepsilon ,l/\varepsilon ])}+\Vert
\tilde{\eta}\left( \tau \right) \Vert _{L_{\rho _{\varepsilon
}}^{p}([-l/\varepsilon ,l/\varepsilon ]\times S^{1})}
\end{equation*}%
and consider their family over $\varepsilon \in (0,\infty)$.

\subsubsection{One-dimensional problem of gradient flows}

We consider the space maps
\begin{eqnarray*}
{\mathcal{B}} &: = & W^{1,p}\left( \left[ -l,l\right] ,M\right) \\
{\mathcal{B}}^\varepsilon &: = & W_{\varepsilon }^{1,p}\left( \left[
-l/\varepsilon ,l/\varepsilon \right] ,M\right)
\end{eqnarray*}
and the Banach bundles over them respectively:
\begin{eqnarray*}
\bigcup\limits_{\chi \in {\mathcal{B}}} L^{p}\left( \left[ -l,l\right] ,\chi
^{\ast }TM\right) & \rightarrow & {\mathcal{B}} \\
\bigcup\limits_{\chi_\varepsilon \in {\mathcal{B}}^\varepsilon}
L_{\varepsilon }^{p}\left( \left[ -l/\varepsilon ,l/\varepsilon \right]
,\chi _{\varepsilon }^{\ast }M\right) & \rightarrow & {\mathcal{B}}%
^\varepsilon.
\end{eqnarray*}

\begin{defn}[$\protect\varepsilon $-weighted norm]
For $\xi _{0}$ and $\eta _{0}$, the so called geometric $W_{\varepsilon
}^{1,p}$ and $L_{\varepsilon }^{p}$ are
\begin{equation}
\left\Vert \xi _{0}\right\Vert _{W_{\varepsilon }^{1,p}[-l/\varepsilon
,l/\varepsilon ]}^{p}=\int_{-l/\varepsilon }^{l/\varepsilon }\left(
\varepsilon \left\vert \xi _{0}\right\vert ^{p}+\varepsilon ^{1-p}\left\vert
\nabla \xi _{0}\right\vert ^{p}\right) d\tau   \label{geometric-e-weight}
\end{equation}%
and
\begin{equation*}
\left\Vert \eta _{0}\right\Vert _{L_{\varepsilon }^{p}[-l/\varepsilon
,l/\varepsilon ]}^{p}=\int_{-l/\varepsilon }^{l/\varepsilon }\varepsilon
^{1-p}\left\vert \eta _{0}\right\vert ^{p}d\tau .
\end{equation*}
\end{defn}

The geometric $\varepsilon $-weight is useful in the adiabatic limit
problems, for example, was used in \cite{FO}. Some explanation on the
relationship between this norm and the ordinary (unweighted) Sobolev norm is
now in order.

One loses uniform right inverse bound of $\frac{\partial }{\partial \tau }%
+\varepsilon \nabla $grad$f$ on $\chi _{\varepsilon }$ if just use usual $%
W^{1,p}$ norm, since the spectrum of $\varepsilon \nabla $grad$f$ goes to $0$
as $\varepsilon \rightarrow 0$. The norm $\left\Vert \cdot \right\Vert
_{W_{\varepsilon }^{1,p}[-l/\varepsilon ,l/\varepsilon ]}$ is just the usual
Sobolev norm after the reparameterization
\begin{equation*}
R_{\varepsilon }:\chi _{\varepsilon }\rightarrow \chi ,\chi _{\varepsilon
}\left( \tau ^{\prime }\right) \rightarrow \chi \left( \tau \right) ,\text{
where }\tau ^{\prime }=\varepsilon \tau ,\tau ^{\prime }\in \left[
-l/\varepsilon ,l/\varepsilon \right]
\end{equation*}%
which defines an isometry
\begin{equation*}
R_\varepsilon: {\mathcal{B}}_\chi^\varepsilon \to {\mathcal{B}}_\chi
\end{equation*}
where ${\mathcal{B}}_\chi^\varepsilon$ is the set of $\varepsilon$%
-reparamterzations of elements of ${\mathcal{B}}_\chi$. Indeed, letting
\begin{equation*}
\widehat{\xi _{0}}\left( \tau \right) =\left( R_{\varepsilon }\right) ^{\ast
}\xi _{0}=\xi _{0}\left( \varepsilon \tau \right) ,
\end{equation*}%
then one can verify
\begin{equation*}
\left\Vert \xi _{0}\right\Vert _{W_{\varepsilon }^{1,p}[-l/\varepsilon
,l/\varepsilon ]}=\left\Vert \widehat{\xi _{0}}\right\Vert _{W^{1,p}[-l,l]}.
\end{equation*}%
In other words, the map
\begin{equation*}
R_{\varepsilon }:W_{\varepsilon }^{1,p}\left( \left[ -l/\varepsilon
,l/\varepsilon \right] ,M\right) \rightarrow W^{1,p}\left( \left[ -l,l\right]
,M\right)
\end{equation*}%
is an isometry between two Banach manifolds with respect to these norms.
Driven by the change of $L^{p}$-norm
\begin{eqnarray*}
&&\int_{-l}^{l}\left\vert \frac{\partial \chi }{\partial \tau }\left( \tau
\right) +\text{grad}f\left( \chi (\tau )\right) \right\vert ^{p}d\tau \\
&=&\int_{-l/\varepsilon }^{l/\varepsilon }\left\vert \frac{1}{\varepsilon }%
\frac{\partial \chi _{\varepsilon }}{\partial \tau ^{\prime }}(\tau ^{\prime
})+\text{grad}f\left( \chi _{\varepsilon }(\tau ^{\prime })\right)
\right\vert ^{p}\varepsilon d\tau ^{\prime } \\
&=&\int_{-l/\varepsilon }^{l/\varepsilon }\left\vert \frac{\partial \chi
_{\varepsilon }}{\partial \tau ^{\prime }}\left( \tau ^{\prime }\right)
+\varepsilon \text{grad}f\left( \chi _{\varepsilon }(\tau ^{\prime })\right)
\right\vert ^{p}\varepsilon ^{1-p}d\tau ^{\prime },
\end{eqnarray*}%
under the reparameterization, we define
\begin{equation*}
\left\Vert \eta _{0}\right\Vert _{L_{\varepsilon }^{p}[-l/\varepsilon
,l/\varepsilon ]}^{p}=\int_{-l/\varepsilon }^{l/\varepsilon }\varepsilon
^{1-p}\left\vert \eta _{0}\right\vert ^{p}d\tau.
\end{equation*}%
Then the two sections
\begin{eqnarray*}
\widehat{\eta _{0}} &=&\frac{\partial }{\partial \tau }+\text{grad}f:{%
\mathcal{B}} \rightarrow \bigcup\limits_{\chi \in {\mathcal{B}}}L^{p}\left( %
\left[ -l,l\right] ,\chi ^{\ast }TM\right) \\
\eta _{0} &=& \frac{\partial }{\partial \tau ^{\prime }}+\varepsilon \text{%
grad}f:{\mathcal{B}}^\varepsilon \rightarrow \bigcup\limits_{\chi
_{\varepsilon }\in {\mathcal{B}}^\varepsilon }L_{\varepsilon }^{p}\left( %
\left[ -l/\varepsilon ,l/\varepsilon \right] ,\chi _{\varepsilon }^{\ast
}M\right)
\end{eqnarray*}%
in the two Banach bundles are isometrically conjugate to each other. The
relation between the two sections is
\begin{equation}
\eta _{0}\left( \chi _{\varepsilon }\left( \tau ^{\prime }\right) \right)
=\varepsilon \widehat{\eta _{0}}\left( \chi \left( \tau \right) \right) .
\label{e-eta-relation}
\end{equation}%
It is easy to check that under these norms, the right inverse $%
Q_{0}^{\varepsilon }$ of the linearized gradient operator
\begin{equation*}
D_{0}^{\varepsilon }=\frac{\partial }{\partial \tau }+\varepsilon \nabla
\text{grad}f:W_{\varepsilon }^{1,p}\rightarrow L_{\varepsilon }^{p}
\end{equation*}%
has uniform operator bounds over $\varepsilon >0$, since we have
\begin{equation}
\left\Vert Q_{0}^{\varepsilon }\right\Vert =\left\Vert Q_{0}\right\Vert
\label{eq:Qe0Q}
\end{equation}%
for all $\varepsilon $. The Sobolev constant
\begin{equation*}
c_{p}^{\varepsilon }:=\sup_{\xi _{0}\neq 0}\frac{\left\vert \xi
_{0}\right\vert }{\left\Vert \xi _{0}\right\Vert _{W_{\varepsilon
}^{1,p}\left( \left[ -l/\varepsilon ,l/\varepsilon \right] \right) }}%
=\sup_{\xi _{0}\neq 0}\frac{\left\vert \widehat{\xi _{0}}\right\vert }{%
\left\Vert \widehat{\xi _{0}}\right\Vert _{W^{1,p}\left( \left[ -l,l\right]
\right) }}
\end{equation*}%
also has uniform bound for all $\varepsilon $, since it is equal to the $%
W^{1,p}$ Sobolev constant of $\left\Vert \widehat{\xi _{0}}\right\Vert $on $%
\left[ -l,l\right] $, and we have assumed $l\geq l_{0}>0$. \

\subsubsection{Comparison with the two dimensional adiabatic family}

The power weight $\rho $ \eqref{eq:rho} is transformed to
\begin{equation}
\rho _{\varepsilon }\left( \tau \right) =\varepsilon ^{1-p+\delta }\rho
\left( \tau \right) =\varepsilon ^{1-p+\delta }\left( 1+\left\vert \tau
\right\vert \right) ^{\delta },  \label{power-e-weight}
\end{equation}%
here and is used to define the weighted Sobolev norms
\begin{eqnarray*}
\Vert \tilde{\xi}\left( \tau \right) \Vert _{W_{\rho _{\varepsilon
}}^{1,p}([-l/\varepsilon ,l/\varepsilon ]\times S^{1})} &=&\Vert \tilde{\xi}%
\left( \tau \right) \left( \rho _{\varepsilon }\left( \tau \right) \right) ^{%
\frac{1}{p}}\Vert _{W^{1,p}([-l/\varepsilon ,l/\varepsilon ]\times S^{1}),}%
\text{ \ } \\
\Vert \tilde{\eta}\left( \tau \right) \Vert _{L_{\rho _{\varepsilon
}}^{p}([-l/\varepsilon ,l/\varepsilon ]\times S^{1})} &=&\Vert \tilde{\eta}%
\left( \tau \right) \left( \rho _{\varepsilon }\left( \tau \right) \right) ^{%
\frac{1}{p}}\Vert _{L^{p}([-l/\varepsilon ,l/\varepsilon ]\times S^{1})}.
\end{eqnarray*}%
The $\varepsilon ^{1-p+\delta }$\ factor in $\rho _{\varepsilon }\left( \tau
\right) $ is to make the norms $W_{\rho _{\varepsilon }}^{1,p}$ and $%
W_{\varepsilon }^{1,p}$ comparable, which is important later in our right
inverse estimate via the weight comparison. By conjugation with the
multiplication operator by the weight function $\rho _{\varepsilon }\left(
\tau \right) $,
\begin{equation*}
D_{\varepsilon }=\frac{\partial }{\partial \tau }+J_{0}\frac{\partial }{%
\partial t}+\varepsilon \nabla f(\chi _{\varepsilon
}):W^{1,p}([-l/\varepsilon ,l/\varepsilon ]\times S^{1})\rightarrow
L^{p}([-l/\varepsilon ,l/\varepsilon ]\times S^{1})
\end{equation*}%
is equivalent to $D_{\rho _{\varepsilon }}:W_{\rho _{\varepsilon
}}^{1,p}\left( [-l/\varepsilon ,l/\varepsilon ]\times S^{1}\right)
\rightarrow L_{\rho _{\varepsilon }}^{p}\left( [-l/\varepsilon
,l/\varepsilon ]\times S^{1}\right) $, with
\begin{equation*}
D_{\rho _{\varepsilon }}=D_{\varepsilon }-\frac{\left( \rho _{\varepsilon
}\left( \tau \right) ^{\frac{1}{p}}\right) ^{\prime }}{\rho _{\varepsilon
}\left( \tau \right) ^{\frac{1}{p}}}=D_{\varepsilon }-\frac{\delta /p}{%
1+\left\vert \tau \right\vert },
\end{equation*}%
by similar diagram of $D_{\rho }$ in Section \ref{Lpr_estimate}. Since
\begin{equation*}
\frac{\Vert \tilde{\xi}\left( \tau \right) \Vert _{W_{\rho _{\varepsilon
}}^{1,p}([-l/\varepsilon ,l/\varepsilon ]\times S^{1})}}{\Vert \tilde{\eta}%
\left( \tau \right) \Vert _{L_{\rho _{\varepsilon }}^{p}([-l/\varepsilon
,l/\varepsilon ]\times S^{1})}}=\frac{\Vert \tilde{\xi}\left( \tau \right)
\Vert _{W_{\rho }^{1,p}([-l/\varepsilon ,l/\varepsilon ]\times S^{1})}}{%
\Vert \tilde{\eta}\left( \tau \right) \Vert _{L_{\rho }^{p}([-l/\varepsilon
,l/\varepsilon ]\times S^{1})}}
\end{equation*}%
the right inverse bound obtained in the $W_{\rho }^{1,p}$ setting passes to
the $W_{\rho _{\varepsilon }}^{1,p}$ setting.

\bigskip

%For the higher modes, since the spectrum is uniformly away from $0,$ we have
%uniform right inverse bound already, so we need not to introduce $%
%\varepsilon $-weighted Sobolev norm.

We remark the following useful inequality for $\xi _{0}$ on $\chi \left(
\varepsilon \tau \right) $, by the Sobolev embedding $W^{1,p}\left( \left[
0,l\right] \right) \hookrightarrow C^{\gamma }\left( \left[ 0,l\right]
\right) $ with $\gamma =1-\frac{1}{p}$:
\begin{eqnarray}
\left\vert \xi _{0}\left( \tau \right) -\xi _{0}\left( \pm l/\varepsilon
\right) \right\vert &=&\left\vert \widehat{\xi _{0}}\left( \varepsilon \tau
\right) -\widehat{\xi _{0}}\left( \pm l\right) \right\vert  \notag \\
&\leq &C\left( l\right) \left\vert \varepsilon \tau -\pm l\right\vert
^{\gamma }\left\Vert \widehat{\xi _{0}}\right\Vert _{W^{1,p}[0,l]}  \notag \\
&=&C\left( l\right) \left\vert \varepsilon \tau \pm l\right\vert ^{\gamma
}\left\Vert \xi _{0}\right\Vert _{W_{\varepsilon }^{1,p}[-l/\varepsilon
,l/\varepsilon ]}  \label{Schauder-Convergence}
\end{eqnarray}%
Especially for $\tau =\pm \left( l/\varepsilon -T\left( \varepsilon \right)
\right) ,$ where $T\left( \varepsilon \right) =\frac{1}{3}\frac{p-1}{\delta }%
S\left( \varepsilon \right) =\frac{1}{3}\frac{p-1}{\delta }\frac{1}{2\pi }%
\ln \left( 1+\frac{l}{\varepsilon }\right) $, we have%
\begin{eqnarray*}
\left\vert \xi _{0}\left( \tau \right) -\xi _{0}\left( \pm l/\varepsilon
\right) \right\vert &\leq &C\left( l\right) \left( \varepsilon T\left(
\varepsilon \right) \right) ^{\gamma }\left\Vert \xi _{0}\right\Vert
_{W_{\varepsilon }^{1,p}[-l/\varepsilon ,l/\varepsilon ]} \\
&\leq &C\left( l\right) \left\vert \varepsilon \ln \varepsilon \right\vert
^{\gamma }\left\Vert \xi _{0}\right\Vert _{W_{\varepsilon
}^{1,p}[-l/\varepsilon ,l/\varepsilon ]} \\
&\leq &C\left( l\right) \varepsilon ^{\widetilde{\gamma }}\left\Vert \xi
_{0}\right\Vert _{W_{\varepsilon }^{1,p}[-l/\varepsilon ,l/\varepsilon ]}
\end{eqnarray*}%
for any $0<\widetilde{\gamma }<1-\frac{1}{p}$, if $\varepsilon $ is
sufficiently small.

\section{Moduli space of \textquotedblleft
thimble-flow-thimble\textquotedblright\ configurations}

\label{subsec:thimble-flow-thimble}

This subsection is the first stage of the deformation of the parameterized
moduli space entering in the construction of the chain homotopy map between $%
\Psi \circ \Phi $ and the identity on $HF(H,J)$ in \cite{PSS}. This is a key
step towards in this kind of adiabatic gluing of dimensional-varying moduli
spaces in general. The material in the first half of this section is largely
taken from Section 5.1 \cite{oh-zhu} with some modifications.

A ``thimble-flow-thimble" configuration consists of two perturbed $J$%
-holomorphic discs joined by a gradient flow line between their marked
points. In this section we will define the moduli space of such
configurations.

For notation brevity, we just denote
\begin{equation*}
{\mathcal{M}}^{l}(K^{\pm },J^{\pm };[z_{-},w_{-}];f;[z_{+},w_{+}])={\mathcal{%
M}}^{l}([z_{-},w_{-}];f;[z_{+},w_{+}]),
\end{equation*}%
omitting the Floer datum $(K^{\pm },J^{\pm })$, as long as it does not cause
confusion.

Given the two moduli spaces
$$
{\mathcal{M}}([z_{\pm},w_{\pm}]) : = {\mathcal{M}}(K^{\pm },J^{\pm };[z_{\pm},w_{\pm}])
$$
and the Morse function $f$, we
define the moduli space
\begin{equation*}
{\mathcal{M}}^{l}([z_{-},w_{-}];f;[z_{+},w_{+}])
\end{equation*}%
by the set consisting of \textquotedblleft \textquotedblleft
thimble-flow-thimble" configurations $(u_{-},\chi ,u_{+})$ of \emph{flow time%
} $2l$ such that
\index{${\mathcal{M}}^{l}([z_{-},w_{-}];f;[z_{+},w_{+}])$}
\begin{equation}
{\mathcal{M}}^{l}([z_{-},w_{-}];f;[z_{+},w_{+}])=\left\{ (u_{-},\chi
,u_{+})\left\vert
\begin{array}{c}
u_{\pm }\in {\mathcal{M}}(K^{\pm },J^{\pm };%
[z_{\pm},w_{\pm}]),\, \\
\text{ }\chi :[-l,l]\rightarrow M\, \textrm{satisfying}
\\
\dot{\chi}-\nabla f(\chi )=0, \\
u_{-}(o_{-})=\chi (-l),u_{+}(o_{+})=\chi (l).%
\end{array}%
\right. \right\}
\end{equation}%
Then the moduli space of \textquotedblleft thimble-flow-thimble"
configurations is defined to be
\index{${\mathcal{M}}^{\text{\textrm{para}}}([z_{-},w_{-}];f;[z_{+},w_{+}])$}
\begin{equation}
{\mathcal{M}}^{\text{\textrm{para}}}([z_{-},w_{-}];f;[z_{+},w_{+}]):=%
\bigcup_{l\geq l_{0}}{\mathcal{M}}^{l}([z_{-},w_{-}];f;[z_{+},w_{+}]).
\end{equation}

\subsection{Off-shell description of thimble-flow-thimble" moduli spaces}

We first provide the Banach manifold hosting ${\mathcal{M}}%
^{l}([z_{-},w_{-}];f;[z_{+},w_{+}])$.
As before, we consider the homotopy class $B \in \pi_2(z_-,z_+)$ determined by
$$
[w_- \# B \# w_+] = 0 \quad \text{\rm in }\, \pi_2(M)
$$
and define
\begin{multline}
{\mathcal{B}}_{l}^{\text{\textrm{tft}}}(z_{-},z_{+};B):=\{(u_{-},\chi
,u_{+})\mid u_{\pm }\in W^{1,p}(\dot{\Sigma},M;z_{\pm }),  \label{Bmfd} \\
\chi \in W^{1,p}([-l,l],M),u_{-}(o_{-})=\chi (-l),\;u_{+}(o_{+})=\chi (l) \\
[u_ -\# \chi \# u_+] = B \}
\end{multline}%
for $p>2$. For any $u\in {\mathcal{B}}_{l}^{\text{\textrm{tft}}%
}(z_{-},z_{+};B)$, the tangent space $T_{u}{\mathcal{B}}_{l}^{\text{\textrm{%
tft}}}(z_{-},z_{+};B)$ is given by
\begin{multline}  \label{eq:TuB}
W^{1,p}_u(z_-,z_+;\text{\textrm tft}) \\
: = \Big\{(\xi _{-},a,\xi _{+})\mid \xi _{\pm }\in W^{1,p}(u_{\pm }^{\ast
}TM),\xi _{-}(o_{-})=a(-l),\,\xi _{+}(o_{+})=a(l)\Big\}.
\end{multline}
More specifically, the latter space consists of $\xi=(\xi_-,a,\xi_+)$, where
$\xi_{\pm}\in W^{1,p}(u_{\pm}^*TM)$, $a\in W^{1,p}(\chi^*TM)$, with the
matching condition
\begin{equation}  \label{match}
\xi_-(o_-)=a(-l),\quad \xi_+(o_+)=a(l).
\end{equation}
The union forms the tangent bundle
\begin{equation*}
T_{u}{\mathcal{B}}_{l}^{\text{\textrm{tft}}}(z_{-},z_{+};B) \to {\mathcal{B}}%
_{l}^{\text{\textrm{tft}}}(z_{-},z_{+};B)
\end{equation*}
which is a Banach vector bundle.

Next for each $u=(u_{-},\chi ,u_{+})\in {\mathcal{B}}_{l}^{\text{\textrm{tft}%
}}(z_{-},z_{+};B)$, we define
\begin{equation*}
L_{u}^{p}(z_{-},z_{+})=L^{p}(\Lambda ^{(0,1)}u^{\ast }TM)
\end{equation*}%
and form the Banach bundle
\begin{equation*}
{\mathcal{L}}_{l}^{\text{\textrm{tft}}}(z_{-},z_{+};B)=\bigcup_{u\in {%
\mathcal{B}}_{l}^{\text{\textrm{tft}}}(z_{-},z_{+};B)}L_{u}^{p}(z_{-},z_{+})
\to {\mathcal{B}}_{l}^{\text{\textrm{tft}}}(z_{-},z_{+};B).
\end{equation*}
We refer to \cite{Fl} for a more detailed description of the asymptotic
behavior of the elements in ${\mathcal{B}}_{l}^{\text{\textrm{tft}}%
}(z_{-},z_{+})$ in the context of Floer moduli spaces.

We denote the set of pairs $(\ell,u)$ and form the unions
\index{${\mathcal{B}}^{\text{\textrm{tft}}}(z_-,z_+;B)$}
\begin{eqnarray*}
{\mathcal{B}}^{\text{\textrm{tft}}}(z_-,z_+;B) & = & \bigcup_{l\ge l_0} {%
\mathcal{B}}^{\text{\textrm{tft}}}_{l}(z_-,z_+;B) \\
{\mathcal{L}}^{\text{\textrm{tft}}}(z_-,z_+;B) & = & \bigcup_{l\ge l_0} {%
\mathcal{L}}^{\text{\textrm{tft}}}_{l}(z_-,z_+;B)
\end{eqnarray*}

\begin{rem}
\begin{enumerate}
\item For given $(l,u) \in {\mathcal{B}}^{\text{\textrm{tft}}}_{l}(z_-,z_+;B)
$, we have
\begin{equation*}
T_{(l,u)}{\mathcal{B}}^{\text{\textrm{tft}}}_{l}(z_-,z_+;B)\simeq T_{u}{%
\mathcal{B}}_{l}^{\text{\textrm{tft}}}(z_{-},z_{+};B)\times T_{l}\mathbb{R},
\end{equation*}
namely its element consists of the quadruples
\begin{equation*}
\xi =(\xi _{-},a,\xi _{+},\mu)
\end{equation*}
with
\begin{equation*}
\xi _{\pm }\in W^{1,p}(u_{\pm }^{\ast }TM), \, a\in W^{1,p}(\chi ^{\ast
}TM), \, \mu \in T_{l}{\mathbb{R}}\cong {\mathbb{R}}
\end{equation*}
satisfying the matching condition
\begin{equation}
\xi _{-}(o_{-})=a(-l)-\frac{\mu }{l}\dot{\chi}(-l),\quad \xi
_{+}(o_{+})=a(l)+\frac{\mu }{l}\dot{\chi}(l) \label{p-match}
\end{equation}%
in off-shell level, with $(u_-,\chi, u_+)$ not necessarily a solution.
Here the $\mu $ comes from the variation of the length $l$ of the domain of
gradient flows.

\item For $\mu \in T_{l}\mathbb{R}$, the induced germ of paths in ${\mathcal{%
B}}^{\text{\textrm{tft}}}(z_{-},z_{+};B)$ is given by
\begin{equation*}
s\mapsto (l+s\mu ,u_{s})
\end{equation*}%
where $u_{s}=(u_{-},\chi _{s},u_{+})\in {\mathcal{B}}_{\left( l-s\mu \right)
}^{\text{\textrm{tft}}}(z_{-},z_{+};B)$, where $\chi _{s}(\tau ^{\prime })$
is the reparameterization of $\chi (\tau )$,%
\begin{equation*}
\chi _{s}(\tau ^{\prime }):=\chi \left( \frac{l\tau ^{\prime }}{l-s\mu }%
\right) ,\text{ \ \ \ \ \ }\tau ^{\prime }\in \left[ -\left( l-s\mu \right)
,l-s\mu \right] ,
\end{equation*}%
for $s$ nearby $0$. There is a canonical way of associating points on $\chi $
to points on $\chi _{s}$ given by
\begin{equation}
\chi \left( \tau \right) \longleftrightarrow \chi _{s}\left( \tau ^{\prime
}\right) ,\text{ \ \ \ \ \ where }\tau =\frac{l\tau ^{\prime }}{l-s\mu }.
\label{reparametrize-chi}
\end{equation}%
all of which correspond to the same point in $M$ contained in the image of $%
\chi $.
\end{enumerate}
\end{rem}

Now we fix $l>0$ and consider a natural section
\begin{equation*}
e:{\mathcal{B}}_{l}^{\text{\textrm{tft}}}(z_{-},z_{+};B)\rightarrow {%
\mathcal{L}}_{l}^{\text{\textrm{tft}}}(z_{-},z_{+})
\end{equation*}%
such that $e(u)\in L_{u}^{p}(z_{-},z_{+})$ is given by
\begin{equation}
e(u)=({\overline{\partial }}_{(K_{-},J_{-})}u_{-},\dot{\chi}-\nabla f(\chi ),%
{\overline{\partial }}_{(K_{+},J_{+})}u_{+})  \label{e}
\end{equation}%
where the $u_{\pm }$ and $\chi $ satisfy the matching condition in %
\eqref{Bmfd}. The linearization of $e$ at $u\in e^{-1}(0)={\mathcal{M}}%
^{l}([z_{-},w_{-}];f;[z_{+},w_{+}])$ induces a linear operator
\begin{equation}
E_{u}:=D_{u}e:T_{u}{\mathcal{B}}_{l}^{\text{\textrm{tft}}}(z_{-},z_{+};B)%
\rightarrow L_{u}^{p}(z_{-},z_{+})  \label{Eu}
\end{equation}%
where we have
\begin{equation*}
T_{u}{\mathcal{B}}_{l}^{\text{\textrm{tft}}%
}(z_{-},z_{+};B)=W_{u}^{1,p}(z_{-},z_{+})
\end{equation*}%
given in \eqref{eq:TuB}, and the value $D_{u}e(\xi )=:\eta $ has the
expression
\begin{equation*}
\eta =(\eta _{-},b,\eta _{+})=\left( D_{u_{-}}{\overline{\partial }}%
_{(K_{-},J_{-})}(\xi _{-}),\,\frac{Da}{d\tau }-\nabla _{a}\func{grad}%
(f),D_{u_{+}}{\overline{\partial }}_{(K_{+},J_{+})}(\xi _{+})\,\right)
\end{equation*}%
for $\xi =(\xi _{-},a,\xi _{+})$. We remark that if we regard $u\in {%
\mathcal{B}}^{\text{\textrm{tft}}}(z_{-},z_{+})$, then for $(\xi _{-},a,\xi
_{+},\mu )\in $ $T_{u}{\mathcal{B}}_{l}^{\text{\textrm{tft}}}(z_{-},z_{+};B),
$ $D_{u}e(\xi )=:\eta $ has the expression $\eta =(\eta _{-},b,\eta _{+})$
where
\begin{equation}
b=\frac{D}{d\tau }a-\nabla _{a}\func{grad}(f)+\frac{\mu }{l}\dot{\chi}\left(
\tau \right) .  \label{eq:b=}
\end{equation}%
Here we have used the formula
\begin{eqnarray*}
D_{u}e(0,0,0,\mu ) &=&\left. \frac{d}{ds}\right\vert _{s=0}\left[ \frac{%
d\chi _{s}}{d\tau ^{\prime }}\left( \tau ^{\prime }\right) -\nabla f\left(
\chi _{s}\left( \tau ^{\prime }\right) \right) \right]  \\
&=&\left. \frac{d}{ds}\right\vert _{s=0}\left[ \frac{d\chi }{d\tau ^{\prime }%
}\left( \tau \right) -\nabla f\left( \chi \left( \tau \right) \right) \right]
\\
&=&\left. \frac{d}{ds}\right\vert _{s=0}\left[ \frac{l}{l-s\mu }\frac{d\chi
}{d\tau }\left( \tau \right) \right] =\frac{\mu }{l}\dot{\chi}\left( \tau
\right) ,
\end{eqnarray*}%
noting $\chi _{s}\left( \tau ^{\prime }\right) =\chi \left( \tau \right) $
from $\left( \ref{reparametrize-chi}\right) $. For the simplicity of
notation, we denote
\begin{eqnarray*}
W_{u}^{1,p}(z_{-},z_{+};{\text{\textrm{tft}}}):= &&T_{u}{\mathcal{B}}_{l}^{%
\text{\textrm{tft}}}(z_{-},z_{+};B) \\
&\subset &W^{1,p}(u_{-}^{\ast }TM)\times W^{1,p}(\chi ^{\ast }TM)\times
W^{1,p}(u_{+}^{\ast }TM)
\end{eqnarray*}
when $[u] = B \in \pi_2(z_-,z_+)$. (Recall Definition \ref{defn:pi2z+z-} for the definition of $\pi_2(z_-,z_+)$.)

\subsection{Index calculation}

Now we show $E_u$ is Fredholm and compute its index.

\begin{prop}
\label{prop:dfdindex} Let $l> 0$ be fixed. The operator $E_u$ is a Fredholm operator and we have
\begin{equation}
\func{Index}E_u=\mu _{H_{-}}([z_{-},w_{-}])-\mu
_{H_{+}}([z_{+},w_{+}]).
\end{equation}for any
\begin{equation*}
u=(u_{-},\chi ,u_{+})\in {\mathcal{M}}^{l}([z_{-},w_{-}];f;[z_{+},w_{+}]).
\end{equation*}
\end{prop}

\begin{proof}
We compute the kernel and the cokernel of
\begin{equation*}
E_u:W_{u}^{1,p}(z_{-},z_{+};{\text{\textrm{tft}}})\rightarrow
L_{u}^{p}(z_{-},z_{+}).
\end{equation*}%
By the matching condition \eqref{p-match} it is clear that
\begin{multline}
\func{ker}E_u =\Big\{(\xi _{-},a, \xi _{+})\mid \xi _{\pm }\in \func{ker}
D_{u_{\pm }}{\overline{\partial }}_{(K^{\pm },J^{\pm })}, \frac{Da}{\partial
\tau }-\nabla _{a}\func{grad}f(\chi )=0, \\
\qquad \xi _{-}(o_{-})=a(-l),\,\xi _{+}(o_{+})=a(l)\Big\},
\end{multline}%
It is easy to see
\begin{equation*}
\xi _{+}(o_{+})=d\phi _{f}^{2l}\cdot \xi _{-}(o_{-})
\end{equation*}%
for any $(\xi _{-},a, \xi _{+})\in \ker E_u$ noticing that $a$ is determined
by its initial value $a(-l)$ and by the equation
\begin{equation}
\frac{Da}{\partial \tau }-\nabla _{a}\func{grad}f(\chi )=0.  \label{eq:a}
\end{equation}%
Therefore the $(\xi _{-},\xi _{+},a)\in \ker E_u$ have one-to-one
correspondence to the points of $\left[ d\left( \phi _{f}^{2l}ev_{-}\times
ev_{+}\right) \right] ^{-1}\left( \Delta _{u_{+}\left( o_{+}\right) }^{\text{%
\textrm{tn}}}\right) ,$ where
\begin{equation*}
d\left( \phi _{f}^{2l}ev_{-}\times ev_{+}\right) :\ker D_{u_{-}}{\overline{%
\partial }}_{(K^{-},J^{-})}\times \ker D_{u_{+}}{\overline{\partial }}%
_{(K_{+},J_{+})}\rightarrow T_{u_{+}\left( o_{+}\right) }M\times
T_{u_{+}\left( o_{+}\right) }M
\end{equation*}%
and
\begin{equation*}
\Delta _{u_{+}\left( o_{+}\right) }^{\text{\textrm{tn}}}=\left\{ \left(
v,v\right) |v\in T_{u_{+}\left( o_{+}\right) }M\right\} \subset
T_{u_{+}\left( o_{+}\right) }M\times T_{u_{+}\left( o_{+}\right) }M,
\end{equation*}
which has codimension $2n$. We consider the subspaces
\begin{equation*}
V_{\pm }:=ev_{o_{\pm }}\left( \ker D_{u_{\pm }}{\overline{\partial }}%
_{(K^{\pm },J^{\pm })}\right) \subset T_{u_{\pm }\left( o_{\pm }\right) }M.
\end{equation*}%
We have
\begin{eqnarray}
\mathop{\kern0pt{\rm dim}}\nolimits\ker E_u &=&\mathop{\kern0pt{\rm dim}}%
\nolimits\ker D_{u_{-}}{\overline{\partial }}_{(K^{-},J^{-})}+%
\mathop{\kern0pt{\rm dim}}\nolimits\ker D_{u_{+}}{\overline{\partial }}%
_{(K_{+},J_{+})}-2n  \notag \\
&{}&\quad +4n-\mathop{\kern0pt{\rm dim}}\nolimits\left( \left( d\phi
_{f}^{2l}\cdot V_{-}\right) \times V_{+}+\Delta _{u_{+}\left( o_{+}\right)
}^{\text{\textrm{tn}}}\right)  \notag \\
&=&\mathop{\kern0pt{\rm dim}}\nolimits\ker D_{u_{-}}{\overline{\partial }}%
_{(K^{-},J^{-})}+\mathop{\kern0pt{\rm dim}}\nolimits\ker D_{u_{+}}{\overline{%
\partial }}_{(K_{+},J_{+})}  \notag \\
&&-\mathop{\kern0pt{\rm dim}}\nolimits\left( d\phi _{f}^{2l}\cdot
V_{-}+V_{+}\right)  \label{kerdim:tft}
\end{eqnarray}%
where in the last identity we have used the linear algebra fact that
\begin{equation*}
\mathop{\kern0pt{\rm dim}}\nolimits\left( A\times B+\Delta \right) =%
\mathop{\kern0pt{\rm dim}}\nolimits\Delta +\mathop{\kern0pt{\rm dim}}%
\nolimits\left( A+B\right)
\end{equation*}%
for linear subspaces $A,B$ in $V$ and diagonal $\Delta $ in $V\times V$.
Next we compute the cokernel of $E_u$. Let $E_u^{\ast }$ be the adjoint
operator of $E_u$, such that
\begin{equation*}
E_u^{\ast }:L_{u}^{p}(z_{-},z_{+})^{\ast }\rightarrow W^{1,p}_u(z_{-},z_{+};{%
\text{\textrm{tft}}})^{\ast }
\end{equation*}%
for $u = (u_,\chi,u_+)$.
Using the nondegenerate $L^{2}$ pairing
\begin{equation*}
L^{p}(\Lambda ^{(0,1)}u^{\ast }TM)\times L^{q}(\Lambda ^{(1,0)}u^{\ast
}TM)\rightarrow {\mathbb{R}}
\end{equation*}%
we identify $L_{u}^{p}(z_{-},z_{+})^{\ast }$ with $L^{q}(\Lambda
^{(1,0)}u^{\ast }TM)$.

On the other hand, we can identify $W^{1,p}_u(z_{-},z_{+};{\text{\textrm{tft}%
}})^{\ast }$ with
\begin{eqnarray*}
&{}&W^{-1,q}_u(z_{-},z_{+};{\text{\textrm{tft}}}) \\
:= &&\left\{ (\xi _{-},a,\xi _{+})\in W^{1,p}(z_{-},z_{+};\text{\textrm tft})\mid \xi
_{-}(o_{-})=a(-l),\,\xi _{+}(o_{+})=a(l)\right\} ^{\perp }
\end{eqnarray*}%
in the direct product
\begin{equation}  \label{eq:W-1q}
W^{-1,q}_u(z_{-},z_{+})=W^{-1,q}(u_{-}^{\ast }TM)\times W^{-1,q}(\chi ^{\ast
}TM)\times W^{-1,q}(u_{+}^{\ast }TM)
\end{equation}
where $(\cdot )^{\perp }$ denotes the $L^{2}$-orthogonal complement. Here we
have $1<q<2$ since $2<p<\infty $.

We denote by
\begin{equation*}
E_u^\dagger: L^q(\Lambda^{(1,0)}u^*TM) \to W^{-1,q}_u(z_-,z_+;{\text{\textrm{%
tft}}})
\end{equation*}
the corresponding $L^2$-adjoint with respect to these identifications.

Now we derive the formula for $E_u^{\dagger }$. Recall by definition, we
have
\begin{equation*}
\langle E_u\xi ,\eta \rangle =\langle \xi ,E_u^{\dagger }\eta \rangle .
\end{equation*}%
Then for any given $\eta :=(\eta _{-},b,\eta _{+})\in \func{ker}E_u^\dagger
\subset L^{q}(z_{-},z_{+})$ it satisfies
\begin{eqnarray*}
0 &=&\int_{-l}^{l}\left\langle \frac{Da}{\partial \tau }-\nabla \func{grad}%
f(\chi )a,b\right\rangle \\
&{}&\quad +\int_{\dot{\Sigma}_{-}}\left\langle D_{u_{-}}{\overline{\partial }%
}_{(K^{-},J^{-})}\xi _{-},\eta _{-}\right\rangle +\int_{\dot{\Sigma}%
_{+}}\left\langle D_{u_{+}}{\overline{\partial }}_{(K_{+},J_{+})}\xi
_{+},\eta _{+}\right\rangle
\end{eqnarray*}%
for all $(\xi _{+},a,\xi _{-})\in W_{u}^{1,p}(z_{-},z_{+};{\text{\textrm{tft}%
}})$, i.e., for all the triples satisfying the matching condition
\begin{equation}
\xi _{-}(o_{-})=a(-l),\quad \xi _{+}(o_{+})=a(l).  \label{eq:matching}
\end{equation}%
Integrating by parts, we have
\begin{eqnarray}
0 &=&\langle a(l),b(l)\rangle -\langle a(-l),b(-l)\rangle
+\int_{-l}^{l}\left\langle -\frac{Db}{\partial \tau }-\nabla \func{grad}%
f(\chi )b,a\right\rangle  \notag  \label{eq:0=aebe} \\
&{}&\quad -\int_{\dot{\Sigma}_{-}}\langle (D_{u_{-}}{\overline{\partial }}%
_{(K^{-},J^{-})})^{\dagger }\eta _{-},\xi _{-}\rangle -\int_{\dot{\Sigma}%
_{+}}\langle (D_{u_{+}}{\overline{\partial }}_{(K^{+},J^{+})})^{\dagger
}\eta _{+},\xi \rangle .
\end{eqnarray}%
Here $D_{u_{\pm }}\partial _{(K^{\pm },J^{\pm })})$ is the formal adjoint of
$D_{u_{\pm }}{\overline{\partial }}_{(K^{\pm },J^{\pm })})$ which has its
symbol of that of the Dolbeault operator $\partial $ (near $z=0$ in ${%
\mathbb{C}}$) and so elliptic. \emph{Here we note that we are using a metric
on $\Sigma _{\pm }\cong {\mathbb{C}}$ that is standard near the origin }$%
o_{\pm }$ \emph{and cylindrical near the end.}

Substituting \eqref{eq:matching} into this we can rewrite \eqref{eq:0=aebe}
into
\begin{eqnarray*}
0 &=&\int_{-l}^{l}\left\langle -\frac{Db}{\partial \tau }-\nabla \func{grad}%
f(\chi )b,a\right\rangle +\langle \xi _{+}(o_{+}),b(\varepsilon )\rangle
-\langle \xi _{-}(o_{-}),b(0)\rangle \\
&{}&\quad -\int_{\dot{\Sigma}_{-}}\langle (D_{u_{-}}{\overline{\partial }}%
_{(K^{-},J^{-})})^{\dagger }\eta _{-},\xi _{-}\rangle -\int_{\dot{\Sigma}%
_{+}}\langle (D_{u_{+}}{\overline{\partial }}_{(K^{+},J^{+})})^{\dagger
}\eta _{+},\xi \rangle .
\end{eqnarray*}%
Note $a$ can be varied arbitrarily on the interior $(-l,l)$ and can be
matched to any given $\xi _{\pm }(o_{\pm })$ at $-l,\,l$. Therefore
considering the variation $\xi _{-}=\xi _{+}=0$, we derive that $b$ must
satisfy
\begin{equation*}
\left\langle -\frac{Db}{\partial \tau }-\nabla \func{grad}f(\chi
)b,a\right\rangle =0
\end{equation*}%
for all $a$ with $a(-l)=0=a(l)$. Therefore $b$ satisfies
\begin{equation}
-\frac{Db}{\partial \tau }-\nabla \func{grad}f(\chi )b=0  \label{eq:b}
\end{equation}%
on $[-l,l]$ first in the distribution sense and then in the classical sense
by the bootstrap regularity of the ODE \eqref{eq:b} and so it is smooth. \
Let
\begin{equation*}
P^{\dag }:T_{u_{-}\left( o_{-}\right) }M\rightarrow T_{u_{+}\left(
o_{+}\right) }M,\text{ }b\left( -l\right) \rightarrow b\left( l\right)
\end{equation*}%
be the linear map for solutions $b$ of ODE $\left( \ref{eq:b}\right) $, then
\
\begin{equation*}
b\left( l\right) =P^{\dag }b\left( -l\right) .
\end{equation*}

Substituting \eqref{eq:b} into the above we obtain
\begin{eqnarray*}
0 &=&-\langle \xi _{-}(o_{-}),b(-l)\rangle +\int_{\dot{\Sigma}_{-}}\langle
(D_{u_{-}}{\overline{\partial }}_{(K^{-},J^{-})})^{\dagger }\eta _{-},\xi
_{-}\rangle \\
&{}&\quad +\langle \xi _{+}(o_{+}),b(l))\rangle +\int_{\dot{\Sigma}%
_{+}}\langle (D_{u_{+}}{\overline{\partial }}_{(K^{+},J^{+})})^{\dagger
}\eta _{+},\xi _{+}\rangle .
\end{eqnarray*}%
Now we can vary $\xi _{\pm }$ independently and hence we have
\begin{eqnarray*}
0 &=&-\langle \xi _{-}(o_{-}),b(-l)\rangle -\int_{\dot{\Sigma}_{-}}\langle
(D_{u_{-}}{\overline{\partial }}_{(K^{-},J^{-})})^{\dagger }\eta _{-},\xi
_{-}\rangle \\
0 &=&\langle \xi _{+}(o_{+}),b(l)\rangle +\int_{\dot{\Sigma}_{+}}\langle
(D_{u_{+}}{\overline{\partial }}_{(K^{+},J^{+})})^{\dagger }\eta _{+},\xi
_{+}\rangle
\end{eqnarray*}%
and hence
\begin{eqnarray}
(D_{u_{-}}{\overline{\partial }}_{(K^{-},J^{-})})^{\dagger }\eta
_{-}-b(-l)\delta _{o_{-}} &=&0  \notag \\
(D_{u_{+}}{\overline{\partial }}_{(K^{+},J^{+})})^{\dagger }\eta
_{+}+b(l)\delta _{o_{+}} &=&0.  \label{cdbar}
\end{eqnarray}%
Here $\delta _{o}$ denotes the Dirac-delta measure supported at the point $%
\{o\}\subset \Sigma $. Due to the choice of our metric on the domain ${%
\mathbb{C}}$ of $u_{\pm }$, $\eta _{\pm }$ must have the singularity of the
type $\frac{1}{\bar{z}}$ which is the fundamental solution to $\partial \eta
=\vec{b}\delta _{o}$ which lies in $L^{q}$ for any $1<q<2$. Therefore one
can solve the distributional equation
\begin{equation*}
(D_{u_{\pm }}{\overline{\partial }}_{(K^{\pm },J^{\pm })})^{\dagger }\eta =%
\vec{b}\cdot \delta _{o_{\pm }}
\end{equation*}%
provided that $\vec{b}$ satisfies the Fredholm alternative:
\begin{equation*}
\langle \vec{b},\xi _{\pm }\left( o_{\pm }\right) \rangle =\int_{\dot{\Sigma}%
_{\pm }}\langle \vec{b}\cdot \delta _{o_{\pm }},\xi _{\pm }\rangle =0\text{ }
\end{equation*}%
for all $\xi _{\pm }\in \ker \left( (D_{u_{\pm }}{\overline{\partial }}%
_{(K^{\pm },J^{\pm })})^{\dagger }\right) ^{\dagger }=\ker D_{u_{\pm }}{%
\overline{\partial }}_{(K^{\pm },J^{\pm })}$. Namely $\vec{b}\in \left(
V_{\pm }\right) ^{\perp }$, where
\begin{equation*}
V_{\pm }:=ev_{o_{\pm }}\left( \ker D_{u_{\pm }}{\overline{\partial }}%
_{(K^{\pm },J^{\pm })}\right) .
\end{equation*}%
We fix such a solution denoted by $\eta _{\vec{b}}\in L^{q}$.

Then \eqref{cdbar} can be written as
\begin{equation*}
(D_{u_{-}}{\overline{\partial }}_{(K^{-},J^{-})})^{\dagger }(\eta _{-}-\eta
_{b(-l)})=0,\quad (D_{u_{+}}{\overline{\partial }}_{(K^{+},J^{+})})^{\dagger
}(\eta _{+}+\eta _{b(l)})=0
\end{equation*}%
i.e.,
\begin{eqnarray*}
\eta _{-}+\eta _{b(-l)} &\in &\func{ker}(D_{u_{-}}{\overline{\partial }}%
_{(K_{-},J_{-})})^{\dagger }, \\
\eta _{+}-\eta _{b(l)} &\in &\func{ker}(D_{u_{+}}{\overline{\partial }}%
_{(K_{+},J_{+})})^{\dagger }.
\end{eqnarray*}%
Therefore we have the exact sequence
\begin{eqnarray*}
0& \rightarrow & \func{Graph}P^{\dag }\cap \left( V_{-}^{\perp }\times
V_{+}^{\perp }\right) \overset{i}{\rightarrow }\func{ker}E^{\dagger }(u) \\
&\overset{j}{\rightarrow }& \func{ker}(D_{u_{+}}{\overline{\partial }}%
_{(K_{+},J_{+})})^{\dagger }\oplus \func{ker}(D_{u_{-}}{\overline{\partial }}%
_{(K_{-},J_{-})})^{\dagger }\rightarrow 0:
\end{eqnarray*}%
Here the first homomorphism is the map
\begin{equation*}
i(b_{1},b_{2})=(\eta _{b_{1}},b_{b_{1}},-\eta _{b_{2}})
\end{equation*}%
where $b_{b_{1}}$ is a solution of \eqref{eq:b} satisfying $%
b_{b_{1}}(-l)=b_{1}$. Note that we have $b_{b_{1}}(l)=b_{2}$ if and only if $%
(b_{1},b_{2})\in \func{Graph}P^{\dag }$. And the second map $j$ is given by
\begin{equation*}
j(\eta _{-},b,\eta _{+})=\left( \eta _{-}+\eta _{b(-l)},\eta _{+}-\eta
_{b(l)}\right) .
\end{equation*}%
and so $\func{ker}E^{\dagger }(u)$ has its dimension given by
\begin{eqnarray}
&&\mathop{\kern0pt{\rm dim}}\nolimits\func{ker}\left( D_{u_{+}}{\overline{%
\partial }}_{(K_{+},J_{+})}\right) ^{\dagger }+\mathop{\kern0pt{\rm dim}}%
\nolimits\func{ker}\left( D_{u_{-}}{\overline{\partial }}_{(K_{-},J_{-})}%
\right) ^{\dagger }+\mathop{\kern0pt{\rm dim}}\nolimits\func{Graph}%
P^{\dagger }\cap \left( V_{+}^{\perp }\times V_{-}^{\perp }\right)  \notag \\
&=&\mathop{\kern0pt{\rm dim}}\nolimits\func{ker}\left( D_{u_{+}}({\overline{%
\partial }}_{(K_{+},J_{+})}\right) ^{\dagger }+\mathop{\kern0pt{\rm dim}}%
\nolimits\func{ker}\left( D_{u_{-}}{\overline{\partial }}_{(K_{-},J_{-})}%
\right) ^{\dagger }+\mathop{\kern0pt{\rm dim}}\nolimits\left( P^{\dagger
}\cdot V_{-}^{\perp }\cap V_{+}^{\perp }\right)  \label{dim:tft_coker}
\end{eqnarray}%
Equivalently $E_u$ has a closed range and its $\func{coker}E_u$ has
dimension the same as this. Combining this dimension counting of $\func{coker%
}E_u$ with that of $\func{ker}E_u$ in $\left( \ref{kerdim:tft}\right) $, we
conclude that $E_u$ is Fredholm and has index given by
\begin{eqnarray*}
\func{Index}E_u &=&\left( \Big(\mathop{\kern0pt{\rm dim}}\nolimits\ker
D_{u_{+}}{\overline{\partial }}_{(K^{+},J^{+})}+\mathop{\kern0pt{\rm dim}}%
\nolimits\ker D_{u_{-}}{\overline{\partial }}_{(K^{-},J^{-})}-%
\mathop{\kern0pt{\rm dim}}\nolimits(P\cdot V_{-}+V_{+})\right) \\
&{}&-\left( \Big(\mathop{\kern0pt{\rm dim}}\nolimits\ker (D_{u_{+}}{%
\overline{\partial }}_{(K^{+},J^{+})})^{\dagger }{}+\mathop{\kern0pt{\rm
dim}}\nolimits\ker (D_{u_{-}}{\overline{\partial }}_{(K^{-},J^{-})})^{%
\dagger }+\mathop{\kern0pt{\rm dim}}\nolimits\left( P^{\dagger }\cdot
V_{-}^{\perp }\cap V_{+}^{\perp }\right) )\right) \\
&=&\func{Index}D_{u_{+}}{\overline{\partial }}_{(K^{+},J^{+})}+\func{Index}%
D_{u_{-}}{\overline{\partial }}_{(K^{-},J^{-})}-2n \\
&=&(n+\mu _{H_{-}}([z_{-},w_{-}]))+(n-\mu _{H_{+}}([z_{+},w_{+}])) \\
&=&\mu _{H_{-}}([z_{-},w_{-}])-\mu _{H_{+}}([z_{+},w_{+}]).
\end{eqnarray*}%
Here we have used
\begin{eqnarray*}
&&\mathop{\kern0pt{\rm dim}}\nolimits(P\cdot V_{-}+V_{+})+%
\mathop{\kern0pt{\rm dim}}\nolimits\left( P^{\dagger }\cdot V_{-}^{\perp
}\cap V_{+}^{\perp }\right) \\
&=&\mathop{\kern0pt{\rm dim}}\nolimits(P\cdot V_{-}+V_{+})+%
\mathop{\kern0pt{\rm dim}}\nolimits\left( \left( P\cdot V_{-}\right) ^{\perp
}\cap V_{+}^{\perp }\right) \\
&=&\mathop{\kern0pt{\rm dim}}\nolimits(P\cdot V_{-}+V_{+})+%
\mathop{\kern0pt{\rm dim}}\nolimits(P\cdot V_{-}+V_{+})^{\perp }= 2n,
\end{eqnarray*}%
for the second identity and
\begin{eqnarray*}
\func{Index}D_{u_{-}}{\overline{\partial }}_{(K_{-},J_{-})} &=&(n+\mu
_{H_{-}}([z_{-},w_{-}])) \\
\func{Index}D_{u_{+}}{\overline{\partial }}_{(K_{+},J_{+})} &=&(n-\mu
_{H_{+}}([z_{+},w_{+}]))
\end{eqnarray*}%
for the third identity. \

To justify that $P^{\dagger }\cdot V_{-}^{\perp }=\left( P\cdot V_{-}\right)
^{\perp }$: For convenience we let
\begin{eqnarray*}
P &:&T_{u_{-}\left( o_{-}\right) }M\rightarrow T_{u_{+}\left( o_{+}\right)
}M,\text{ }a\left( -l\right) \rightarrow a\left( l\right) \\
P^{\dag } &:&T_{u_{-}\left( o_{-}\right) }M\rightarrow T_{u_{+}\left(
o_{+}\right) }M,\text{ }b\left( -l\right) \rightarrow b\left( l\right)
\end{eqnarray*}%
be the linear maps for solutions $a$ and $b$ of ODE $\left( \ref{eq:a}%
\right) $ and $\left( \ref{eq:b}\right) $ respectively. Then for solutions $%
a,b$ we have \
\begin{equation*}
\frac{d}{d\tau }\left\langle a\left( \tau \right) ,b\left( \tau \right)
\right\rangle =\left\langle \nabla \func{grad}f\left( \chi \right) a\left(
\tau \right) ,b\left( \tau \right) \right\rangle +\left\langle a\left( \tau
\right) ,-\nabla \func{grad}f\left( \chi \right) b\left( \tau \right)
\right\rangle =0,
\end{equation*}%
This in particular implies
\begin{equation*}
\left\langle Pa,P^{\dag }b\right\rangle =\left\langle a\left( l\right)
,b\left( l\right) \right\rangle =\left\langle a\left( -l\right) ,b\left(
-l\right) \right\rangle =\left\langle a,b\right\rangle .
\end{equation*}%
and hence%
\begin{equation*}
P^{\dag }V_{-}^{\bot }=\left( PV_{-}\right) ^{\bot }.
\end{equation*}
\end{proof}

\begin{cor}
\label{achieve-tft-trans}Suppose $u_{\pm }\in {\mathcal{M}}([z_{\pm },w_{\pm
}];A_{\pm })$ are Fredholm regular, then $u\in {\mathcal{M}}%
^{l}([z_{-},w_{-}];f;[z_{+},w_{+}])$ is Fredholm regular (in the sense that $%
E_u$ is surjective) if and only if the configuration $u=\left( u_{-},\chi
,u_{+}\right) $ satisfies the \textquotedblleft
thimble-flow-thimble\textquotedblright\ transversality in definition \ref%
{tft-trans}.
\end{cor}

\begin{proof}
In $\left( \ref{dim:tft_coker}\right) $ of the above proposition we have
obtained
\begin{eqnarray*}
\mathop{\kern0pt{\rm dim}}\nolimits \func{ker}E_u^{\dagger }&=&%
\mathop{\kern0pt{\rm dim}}\nolimits\func{ker}\left( D_{u_{+}}({\overline{%
\partial }}_{(K_{+},J_{+})}\right) ^{\dagger }+\mathop{\kern0pt{\rm dim}}%
\nolimits\func{ker}\left( D_{u_{-}}{\overline{\partial }}_{(K_{-},J_{-})}%
\right) ^{\dagger } \\
&{}& \quad +\mathop{\kern0pt{\rm dim}}\nolimits\left( P^{\dagger }\cdot
V_{-}^{\perp }\cap V_{+}^{\perp }\right) \\
&=&\mathop{\kern0pt{\rm dim}}\nolimits\left( P^{\dagger }\cdot V_{-}^{\perp
}\cap V_{+}^{\perp }\right) = \mathop{\kern0pt{\rm dim}}\nolimits\left(
PV_{-}+V_{+}\right) ^{\perp }.
\end{eqnarray*}
Hence $E_u$ is surjective if and only if $\left( PV_{-}+V_{+}\right) ^{\perp
}=\left\{ 0\right\} ,$ i.e.
\begin{equation*}
PV_{-}+V_{+}=T_{u_{+}\left( o_{+}\right) }M.
\end{equation*}
But this is equivalent to%
\begin{equation*}
d\phi _{f}^{2l}ev_{-}\times ev_{+}\pitchfork \Delta _{u_{+}\left(
o_{+}\right) }^{\text{\textrm{tn}}},
\end{equation*}%
as derived in $\left( \ref{kerdim:tft}\right) $.
\end{proof}

\subsection{Transversality of thimble-flow-thimble moduli spaces under the
perturbation of $f\label{sec:tft-transversality}$}

We now show $E_u$ is surjective for generic $f$, for any $u=\left(
u_{-},\chi ,u_{+}\right) \in {\mathcal{M}}^{l}([z_{-},w_{-}];f;[z_{+},w_{+}])
$ of \emph{any} $l>0$. We need the variation of $l$ to get this
surjectivity. More precisely we have the following

\begin{prop}
\label{family-tft-trans}Suppose that $u_{\pm }\in {\mathcal{M}}(K_{\pm
},J_{\pm };[z_{\pm },w_{+}])$ is Fredholm regular, i.e., its
linearization is surjective. Then for each given $l_{0}>0$, there exists a
dense subset of $f\in C^{\infty }\left( M\right) $ such that any element $u=\left( u_{-},\chi ,u_{+},l\right) $ in
\begin{equation*}
{\mathcal{M}}^{\text{\rm para}}([z_{-},w_{-}];f;[z_{+},w_{+}])=\bigcup_{l\geq
l_{0}}{\mathcal{M}}^{l}([z_{-},w_{-}];f;[z_{+},w_{+}])
\end{equation*}is Fredholm regular, in the sense that $E_u$ is surjective. ${\mathcal{M}}^{\text{\rm para}}([z_{-},w_{-}];f;[z_{+},w_{+}])$ is a smooth manifold with
dimension equal to the index of $E_u$:
\begin{eqnarray*}
&{}& \func{dim}{\mathcal{M}}^{\text{\rm para}}([z_{-},w_{-}];f;[z_{+},w_{+}])
\\
& = &\mu_{H_{-}}([z_{-},w_{-}])-\mu
_{H_{+}}([z_{+},w_{+}])+1.
\end{eqnarray*}
\end{prop}

\begin{proof}
Consider the section $e$ of the following Banach bundle
\begin{equation*}
e :{\mathcal{B}}^{\text{\textrm{tft}}}_{l}(z_-,z_+;B)\times C^{\infty
}\left( M\right) \rightarrow {\mathcal{L}}^{\text{\textrm{tft}}%
}(z_{-},z_{+})
\end{equation*}
given by
\begin{equation*}
e\left( u,f\right) =\left({\overline{\partial }}_{(K_{-},J_{-})}u_{-},\dot{%
\chi}-\nabla f(\chi ),{\overline{\partial }}_{(K_{+},J_{+})}u_{+}\right).
\end{equation*}
where
$$
u=\left( u_{-},\chi ,u_{+},l\right) \in { {\mathcal{B}}^{\text{\textrm{tft}}%
}_{l}(z_-,z_+;B) \subset \mathcal{B}}_{l}^{\text{%
\textrm{tft}}}(z_{-},z_{+}).
$$
Denote the linearization of $e$ at $\left( u,f\right) $
by
\begin{equation*}
E_{(u,f)}:=D_{(u,f)}e:T_{u}{\mathcal{B}}^{\text{\textrm{tft}}%
}_{l}(z_-,z_+;B)\times T_{f}C^{\infty }\left( M\right) \rightarrow
L_{u}^{p}(z_{-},z_{+}).
\end{equation*}%
For $\xi _{\pm }\in W^{1,p}(u_{\pm }^{\ast }TM)$, $a\in W^{1,p}(\chi ^{\ast
}TM)$, $\mu \in T_{l}\mathbb{R}$ and $h\in T_{f}C^{\infty }\left( M\right) $%
,
\begin{equation*}
E_{(u,f)}:\left( \xi _{-},a,\xi _{+},\mu ,h\right) \rightarrow \eta :=(\eta
_{-},b,\eta _{+})
\end{equation*}%
where
\begin{eqnarray*}
\eta _{-} &=&D_{u_{-}}{\overline{\partial }}_{(K_{-},J_{-})}\xi _{-}, \\
b &=&\frac{D}{d\tau }a-\nabla _{a}\func{grad}(f)+\frac{\mu }{l}\dot{\chi}%
\left( \tau \right) -\nabla h(\chi ), \\
\eta _{+} &=&D_{u_{+}}{\overline{\partial }}_{(K_{+},J_{+})}\xi _{+}\,.
\end{eqnarray*}

Next we show the cokernel of $E_{(u,f)}$ vanishes. Let $E_{(u,f)}^{\ast }$
be the adjoint operator of $E_{(u,f)}$, such that
\begin{equation*}
E_{(u,f)}^{\ast }:L_{u}^{p}(z_{-},z_{+})^{\ast }\rightarrow \left(
W^{1,p}_u(z_{-},z_{+};{\text{\textrm{tft}}})\times T_{f}C^{\infty }\left(
M\right) \right) ^{\ast }.
\end{equation*}%
(Recall the definition of $W^{1,p}_u(z_{-},z_{+};{\text{\textrm{tft}}})$
from \eqref{eq:TuB}.)
Using the nondegenerate $L^{2}$ pairing
\begin{equation*}
L^{p}(\Lambda ^{(0,1)}u^{\ast }TM)\times L^{q}(\Lambda ^{(1,0)}u^{\ast
}TM)\rightarrow {\mathbb{R}}
\end{equation*}%
we identify $L_{u}^{p}(z_{-},z_{+})^{\ast }$ with $L^{q}(\Lambda
^{(1,0)}u^{\ast }TM)$.

On the other hand, we can identify $\left( W^{1,p}_u(z_{-},z_{+};{\text{%
\textrm{tft}}})\times T_{f}C^{\infty }\left( M\right) \right) ^{\ast }$ with
$$
W^{-1,q}_u(z_{-},z_{+};{\text{\rm tft}})\times \left( T_{f}C^{\infty
}\left( M\right) \right) ^{\ast }
$$
where $W^{-1,q}_u(z_{-},z_{+};{\text{\rm tft}})$
is defined to be
\begin{equation*}
\left\{ (\xi _{-},a,\xi _{+},\mu )\in W^{1,p}_u(z_{-},z_{+})\mid \xi
_{-}(o_{-})=a(-l)-\mu \dot{\chi}(-l),\,\xi _{+}(o_{+})=a(l)+\mu \dot{\chi}(l)\right\} ^{\perp }
\end{equation*}in the direct product
$W^{-1,q}_u(z_{-},z_{+})$ given in \eqref{eq:W-1q}
where $(\cdot )^{\perp }$ denotes the $L^{2}$-orthogonal complement. Here we
have $1<q<2$ since $2<p<\infty $.

We denote by
\begin{equation*}
E_{(u,f)}^{\dagger }:L^{q}(\Lambda ^{(1,0)}u^{\ast }TM)\rightarrow
W^{-1,q}_u(z_{-},z_{+};{\text{\textrm{tft}}})\times \left( T_{f}C^{\infty
}\left( M\right) \right) ^{\ast }
\end{equation*}%
the corresponding $L^{2}$-adjoint with respect to these identifications.
Recall by definition, we have
\begin{equation*}
\langle E_{(u,f)}\xi ,\eta \rangle =\langle \xi ,E_{(u,f)}^{\dagger }\eta
\rangle .
\end{equation*}%
For any given $\eta :=(\eta _{-},b,\eta _{+})\in \func{ker}E^{\dagger
}(u,f)\subset L^{q}(\Lambda ^{(1,0)}u^{\ast }TM),$ it satisfies
\begin{equation*}
\langle E_{(u,f)}\xi ,\eta \rangle =0
\end{equation*}
for all $(\xi _{+},a,\xi _{-},\mu )\in T_{u}{\mathcal{B}}^{\text{\textrm{tft}%
}}_{l}(z_-,z_+;B)$, especially for $(\xi _{+},a,\xi _{-},0)\in T_{u}{%
\mathcal{B}}_{l}^{\text{\textrm{tft}}}(z_{-},z_{+}),$ namely
\begin{eqnarray}
0 &=&\int_{-l}^{l}\left\langle \frac{Da}{\partial \tau }-\nabla \func{grad}%
f(\chi )a,b\right\rangle -\int_{-l}^{l}\left\langle \nabla h(\chi
),b\right\rangle  \notag \\
&&{}+\int_{\dot{\Sigma}_{-}}\left\langle D_{u_{-}}{\overline{\partial }}%
_{(K^{-},J^{-})}\xi _{-},\eta _{-}\right\rangle +\int_{\dot{\Sigma}%
_{+}}\left\langle D_{u_{+}}{\overline{\partial }}_{(K_{+},J_{+})}\xi
_{+},\eta _{+}\right\rangle  \label{tft-coker-h}
\end{eqnarray}%
for all the triples satisfying the matching condition
\begin{equation}
\xi _{-}(o_{-})=a(-l),\quad \xi _{+}(o_{+})=a(l).  \label{match2}
\end{equation}%
Letting $a=\xi _{-}=\xi _{+}=0$, then $\left( \ref{tft-coker-h}\right) $
becomes
\begin{equation*}
\int_{-l}^{l}\left\langle \nabla h(\chi ),b\right\rangle =0
\end{equation*}%
for all $h\in C^{\infty }\left( M\right) $, so $b=0$. Now $\left( \ref%
{tft-coker-h}\right) $ becomes%
\begin{equation*}
\int_{\dot{\Sigma}_{-}}\left\langle D_{u_{-}}{\overline{\partial }}%
_{(K^{-},J^{-})}\xi _{-},\eta _{-}\right\rangle +\int_{\dot{\Sigma}%
_{+}}\left\langle D_{u_{+}}{\overline{\partial }}_{(K_{+},J_{+})}\xi
_{+},\eta _{+}\right\rangle =0.
\end{equation*}%
Notice that $\xi _{-}$ and $\xi _{+}$ can vary independently since $a\left(
-l\right) $ and $a\left( l\right) $ can be any vector without restriction
and hence the matching condition $\left( \ref{match2}\right) $ does not put
any restriction on $\xi _{\pm }\left( o_{\pm }\right) $. Therefore we have%
\begin{eqnarray*}
\int_{\dot{\Sigma}_{-}}\left\langle D_{u_{-}}{\overline{\partial }}%
_{(K^{-},J^{-})}\xi _{-},\eta _{-}\right\rangle &=&0, \\
\int_{\dot{\Sigma}_{+}}\left\langle D_{u_{+}}{\overline{\partial }}%
_{(K_{+},J_{+})}\xi _{+},\eta _{+}\right\rangle &=&0
\end{eqnarray*}%
for \emph{all} $\xi _{\pm }$. In other words, $\eta _{\pm }$ lie in coker$%
D_{u_{\pm }}{\overline{\partial }}_{(K_{\pm },J_{\pm })}$. By the
hypothesis, we conclude $\eta _{\pm }=0$. Therefore the cokernel of $%
E_{(u,f)}$ vanishes in $E_{(u,f)}:W^{1,p}\rightarrow L^{p}$ setting.

Now we raise the regularity to $W^{k,p}$ setting for the use of Sard-Smale
theorem. For any $k\geq 1$ and $\zeta \in W^{k-1,p}\subset L^{p}$, from
vanishing of the above cokernel we can always solve $E_{(u,f)}\xi =\zeta $
with $\xi \in W^{1,p}$. By ellipticity of $E_{(u,f)}$ we have $\xi \in
W^{k,p}$. Therefore, the cokernel of $E_{(u,f)}$ vanishes if we treat $%
E_{(u,f)}$ as a map from $W^{k,p}$ to $W^{k-1,p}$ for any $k$. We fix a
large enough $k$ (which depends on $n,c_{1}(A_{\pm }),\mu _{H_{\pm
}}([z_{\pm },w_{\pm }]$) to define the $W^{k,p}$ Banach norm on ${\mathcal{B}%
}^{\text{\textrm{tft}}}_{l}(z_-,z_+;B)$. Therefore, the universal moduli
space
\begin{equation*}
{\mathcal{M}}_{\text{\textrm{univ}}}^{\text{\textrm{para}}%
}([z_{-},w_{-}];f;[z_{+},w_{+}])=e^{-1}\left( 0\right)
\end{equation*}%
is a $W^{k,p}$ Banach manifold in ${\mathcal{B}}^{\text{\textrm{tft}}%
}_{l}(z_-,z_+;B)\times C^{\infty }\left( M\right) $. For the natural
projection
\begin{equation*}
\pi :{\mathcal{M}}_{\text{\textrm{univ}}}^{\text{\textrm{para}}%
}([z_{-},w_{-}];f;[z_{+},w_{+}])\rightarrow C^{\infty }\left( M\right) ,
\end{equation*}%
since our $k$ is large enough, we can apply Sard-Smale theorem and conclude
its regular values $f$ form a set of second category in $C^{\infty }\left(
M\right) $. For any regular value $f$,
\begin{equation*}
{\mathcal{M}}^{\text{\textrm{para}}}([z_{-},w_{-}];f;[z_{+},w_{+}])=\pi
^{-1}\left( f\right) \cap {\mathcal{M}}_{\text{\textrm{univ}}}^{\text{%
\textrm{para}}}([z_{-},w_{-}];f;[z_{+},w_{+}])
\end{equation*}%
is a finite dimensional smooth submanifold with
\begin{eqnarray*}
&&\mathop{\kern0pt{\rm dim}}\nolimits{\mathcal{M}}^{\text{\textrm{para}}%
}([z_{-},w_{-}];f;[z_{+},w_{+}]) \\
&=&\text{index}(\pi )=\text{index}(E\left( u\right) ) \\
&=&\text{ }\mu _{H_{-}}([z_{-},w_{-}])-\mu _{H_{+}}([z_{+},w_{+}])+1.
\end{eqnarray*}%
The index$E\left( u\right) $ here differs from previous Proposition by $1$
because the linearization $E\left( u\right) $ here is in $T_{u}{\mathcal{B}}%
^{\text{\textrm{tft}}}(z_{-},z_{+})$ instead of $T_{u}{\mathcal{B}}_{l}^{%
\text{\textrm{tft}}}(z_{-},z_{+})$. Since $f$ is a regular value of $\pi $, $%
E\left( u\right) :T_{u}{\mathcal{B}}^{\text{\textrm{tft}}}(z_{-},z_{+})%
\rightarrow \mathcal{L}^{\text{\textrm{tft}}}_u(z_{-},z_{+})$ is surjective
and $u$ is Fredholm regular.
\end{proof}

\subsection{Generic immersiveness of nodal points}

Next we establish the condition to ensure the joint points $u_{\pm }\left(
o_{\pm }\right) $ are immersed. This condition satisfies for a generic
choice of almost complex structures $J$ and will be needed in the proof of
the surjectivity of our gluing: The Hausdorff convergence imposed in
Definition \ref{defn:adiabatic} (2) is not strong enough to detect multiple
covers of the thin cylinder although it captures all simple thin cylinders.
Immersion condition at the joint points then make the thin part of the
adiabatic limit with immersed joint points be automatically simple.

\begin{prop}
\label{immersed-joint}For generic $J_{0}\in \mathcal{J}_{\omega }$ and $f\in
C^{\infty }\left( M\right) $, any element $u=\left( u_{-},\chi
,u_{+},l\right) \in {\mathcal{M}}^{\text{\textrm{para}}%
}([z_{-},w_{-}];f;[z_{+},w_{+}])$ whose $u_{\pm }$ are somewhere injective
must be Fredholm regular, and have both $u_{\pm }\left( o_{\pm }\right) $
immersed, provided that
\begin{equation*}
\mu _{H_{-}}([z_{-},w_{-}])-\mu _{H_{+}}([z_{+},w_{+}]))<2n-1,
\end{equation*}%
or equivalently
\begin{equation*}
\mathop{\kern0pt{\rm dim}}\nolimits{\mathcal{M}}^{\text{\textrm{para}}%
}([z_{-},w_{-}];f;[z_{+},w_{+}])\leq 2n-1.
\end{equation*}%
Especially, for any ${\mathcal{M}}^{\text{\textrm{para}}%
}([z_{-},w_{-}];f;[z_{+},w_{+}];A_{\pm })$ of virtual dimension $0,1$, both $%
u_{\pm }\left( o_{\pm }\right) $ are immersed for generic $\left(
J_{0},f\right) $.
\end{prop}

\begin{proof}
The proof is a small variant of Theorem 1.2 of \cite{oh-zhu2}. Consider the
section $\Upsilon $ of the following Banach bundle%
\begin{equation*}
\Upsilon :{\mathcal{B}}^{\text{\textrm{tft}}}_{l}(z_-,z_+;B)\times C^{\infty
}\left( M\right) \times \mathcal{J}_{\omega }\rightarrow {\mathcal{L}}^{%
\text{\textrm{tft}}}(z_{-},z_{+};B)\times H_{-}^{1,0}\times H_{+}^{1,0}
\end{equation*}
given by the projection map
\begin{equation*}
\Upsilon\left( u,f,J_{0}\right) = \left( {\overline{\partial }}%
_{(K_{-},J_{-})}u_{-},\dot{\chi}-\nabla f(\chi ),{\overline{\partial }}%
_{(K_{+},J_{+})}u_{+},{\partial }_{J_{0}}u_{-}\left( o_{-}\right) ,{\partial
}_{J_{0}}u_{+}\left( o_{+}\right) \right)
\end{equation*}
where
\begin{equation*}
H_{\pm }^{0,1}\left( u_{\pm }\right) = \func{Hom}_{i,J_{0}}\left( T_{o_{\pm
}}\Sigma _{\pm },T_{u_{\pm }\left( o_{\pm }\right) }M\right)
\end{equation*}%
is a rank $2n$ vector bundle over ${\mathcal{B}}^{\text{\textrm{tft}}%
}_{l}(z_-,z_+;B)\times C^{\infty }\left( M\right) \times \mathcal{J}_{\omega
}$ whose fiber consists of $\left( j,J_{0}\right) $-linear maps from $%
T_{o_{\pm }}\Sigma _{\pm }$ to $T_{u_{\pm }\left( o_{\pm }\right) }M$. Here
the $j$ for $\Sigma _{\pm }\simeq \left( D^{2},o_{\pm }\right) $ is fixed.
Any cokernel element
\begin{equation*}
\left( \eta _{-},b,\eta _{+},\alpha _{-},\alpha _{+}\right) \in {\mathcal{L}}%
^{\text{\textrm{tft}}}(z_{-},z_{+})\times H_{-}^{1,0}\times H_{+}^{1,0}
\end{equation*}%
of $D_{\left( u,f,J_{0}\right) }\Upsilon $ must satisfy%
\begin{eqnarray}
0 &=&\int_{-l}^{l}\left\langle \frac{Da}{\partial \tau }-\nabla \func{grad}%
f(\chi )a,b\right\rangle -\int_{-l}^{l}\left\langle \nabla h(\chi
),b\right\rangle  \notag \\
&&{}+\int_{\dot{\Sigma}_{-}}\left\langle D_{u_{-}}{\overline{\partial }}%
_{(K^{-},J^{-})}\xi _{-}+\frac{1}{2}du_{-}\circ B\circ j,\eta
_{-}\right\rangle  \notag \\
&&+\int_{\dot{\Sigma}_{+}}\left\langle D_{u_{+}}{\overline{\partial }}%
_{(K_{+},J_{+})}\xi _{+}\frac{1}{2}du_{+}\circ B\circ j,\eta
_{+}\right\rangle  \notag \\
&&+\left\langle D_{u_{-}}{\partial }_{(K^{-},J^{-})}\xi _{-},\alpha
_{-}\delta _{o_{-}}\right\rangle _{\dot{\Sigma}_{-}}+\left\langle D_{u_{+}}{%
\partial }_{(K_{+},J_{+})}\xi _{+},\alpha _{+}\delta _{o_{+}}\right\rangle _{%
\dot{\Sigma}_{+}}  \label{coker-jt}
\end{eqnarray}%
for all $(\xi _{+},a,\xi _{-},\mu )\in T_{u}{\mathcal{B}}^{\text{\textrm{tft}%
}}_{l}(z_-,z_+;B)$, $h\in T_{f}C^{\infty }\left( M\right) $ and $B\in
T_{J_{0}}\mathcal{J}_{\omega }$, where $\delta _{o_{\pm }}$ are delta
functions at $o_{\pm }$ on $\dot{\Sigma}_{\pm }$.

Letting $a=\xi _{-}=\xi _{+}=B=0$, then $\left( \ref{coker-jt}\right) $
becomes
\begin{equation*}
\int_{-l}^{l}\left\langle \nabla h(\chi ),b\right\rangle =0
\end{equation*}%
for all $h\in C^{\infty }\left( M\right) $, so $b=0$. Now $\left( \ref%
{coker-jt}\right) $ becomes%
\begin{eqnarray*}
0 &=&\int_{\dot{\Sigma}_{-}}\left\langle D_{u_{-}}{\overline{\partial }}%
_{(K^{-},J^{-})}\xi _{-}+\frac{1}{2}du_{-}\circ B\circ j,\eta
_{-}\right\rangle +\left\langle D_{u_{-}}{\partial }_{(K^{-},J^{-})}\xi
_{-},\alpha _{-}\delta _{o_{-}}\right\rangle _{\dot{\Sigma}_{-}} \\
&&+\int_{\dot{\Sigma}_{+}}\left\langle D_{u_{+}}{\overline{\partial }}%
_{(K_{+},J_{+})}\xi _{+}+\frac{1}{2}du_{+}\circ B\circ j,\eta
_{+}\right\rangle +\left\langle D_{u_{+}}{\partial }_{(K_{+},J_{+})}\xi
_{+},\alpha _{+}\delta _{o_{+}}\right\rangle _{\dot{\Sigma}_{+}}.
\end{eqnarray*}%
Notice that $\xi _{-}$ and $\xi _{+}$ can vary independently since $a\left(
-l\right) $ and $a\left( l\right) $ can be any vector without restriction
and hence the matching condition $\left( \ref{match2}\right) $ does not put
any restriction on $\xi _{\pm }\left( o_{\pm }\right) $. Therefore we have%
\begin{eqnarray*}
\int_{\dot{\Sigma}_{-}}\left\langle D_{u_{-}}{\overline{\partial }}%
_{(K^{-},J^{-})}\xi _{-}+\frac{1}{2}du_{-}\circ B\circ j,\eta
_{-}\right\rangle +\left\langle D_{u_{-}}{\partial }_{(K^{-},J^{-})}\xi
_{-},\alpha _{-}\delta _{o_{-}}\right\rangle _{\dot{\Sigma}_{-}} &=&0, \\
\int_{\dot{\Sigma}_{+}}\left\langle D_{u_{+}}{\overline{\partial }}%
_{(K_{+},J_{+})}\xi _{+}+\frac{1}{2}du_{+}\circ B\circ j,\eta
_{+}\right\rangle +\left\langle D_{u_{+}}{\partial }_{(K_{+},J_{+})}\xi
_{+},\alpha _{+}\delta _{o_{+}}\right\rangle _{\dot{\Sigma}_{+}} &=&0,
\end{eqnarray*}%
which are identical to $\left( 2.12\right) $ in \cite{oh-zhu2}. Then the
remaining steps are the same as in \cite{oh-zhu2} to show $\eta _{\pm }=0$
and $\alpha _{\pm }=0$, hence $\Upsilon \,$\ is transversal to the section
zero sections
\begin{eqnarray*}
&&o_{{\mathcal{L}}^{\text{\textrm{tft}}}(z_{-},z_{+})}\times
o_{H_{-}^{1,0}}\times H_{+}^{1,0}\text{ and} \\
&&o_{{\mathcal{L}}^{\text{\textrm{tft}}}(z_{-},z_{+})}\times
H_{-}^{1,0}\times o_{H_{-}^{1,0}},
\end{eqnarray*}
which both have codimension $2n$ in $o_{{\mathcal{L}}^{\text{\textrm{tft}}%
}(z_{-},z_{+})}\times H_{-}^{1,0}\times H_{+}^{1,0}$. Hence for generic $%
\left( J_{0},f\right) $, by Sard Theorem as in \cite{oh-zhu2},
\begin{equation*}
\Upsilon \,^{-1}\left( o_{{\mathcal{L}}^{\text{\textrm{tft}}%
}(z_{-},z_{+})}\times o_{H_{-}^{1,0}}\times H_{+}^{1,0}\cup o_{{\mathcal{L}}%
^{\text{\textrm{tft}}}(z_{-},z_{+})}\times H_{-}^{1,0}\times
o_{H_{-}^{1,0}}\right)
\end{equation*}%
has dimension $2n$ less than
\begin{equation*}
\Upsilon \,^{-1}\left( o_{{\mathcal{L}}^{\text{\textrm{tft}}%
}(z_{-},z_{+})}\times H_{-}^{1,0}\times H_{+}^{1,0}\right) ={\mathcal{M}}^{%
\text{\textrm{para}}}([z_{-},w_{-}];f;[z_{+},w_{+}]),
\end{equation*}
so is of dimension
\begin{equation*}
\mu _{H_{-}}([z_{-},w_{-}])-\mu _{H_{+}}([z_{+},w_{+}])+1-2n,
\end{equation*}%
which is negative if the index condition of the proposition holds. But this
means the set of the $\left( u_{-},\chi ,u_{+}\right) $ who fails the
immersion condition at $o_{-}$ or $o_{+}$ is empty. The proposition is
proved.
\end{proof}

\section{$J$-holomorphic curves from cylindrical domain}

The (perturbed) $J$-holomorphic curves from punctured Riemann surface $\dot{%
\Sigma}$ have been extensively studied in the literature. Since we will use
cylindrical coordinates of punctured Riemannian surface in our gluing
analysis , we briefly review the Banach norms in this setting.
For our $%
u_{\pm }$, the domain  can be thought as
\begin{equation*}
\dot \Sigma _{\pm } \simeq ( -\infty ,\infty) \times S^{1}.
\end{equation*}%
Let $O_{\pm
\text{ }}\simeq -(\pm \infty ,0] \times S^{1}$ be the
neighborhood near the puncture $o_{\pm }$ equipped with cylindrical coordinates
as before.

Without loss of generality let's assume the image of $O_{\pm }$
under $u_{\pm }\ $has diameter less than the injective radius of $\left(
M,g\right) $ (otherwise we can shift the ${\mathbb{R}}$ component to
reparameterize $u_{\pm }$)$.$ One can use the following norm to define the
Banach manifold hosting such Floer trajectories. Let%
\begin{equation*}
\mathcal{B}_{+}\left( z_{+}\right) =\left\{ u_{+}: \dot\Sigma _{+}\rightarrow
M\in W_{loc}^{1,p},\left\vert
\begin{array}{c}
u_{+}\left( -\infty ,t\right) =p_{+}\text{ for some }p_{+}\in M\text{, } \\
u_{+}\left( \infty ,t\right) =z_{+}\left( t\right) , \\
u_{+}\left( \tau ,t\right) =\exp _{p_{+}}\xi \left( \tau ,t\right) \text{
near }p_{+}\text{, } \\
u_{+}\left( \tau ,t\right) =\exp _{z_{+}\left( t\right) }\xi \left( \tau
,t\right) \text{ near }z_{+}\text{,} \\
\text{and }e^{\frac{2\pi \delta \left\vert \tau \right\vert }{p}}\xi \in
W^{1,p}\left( {\mathbb{R}}\times S^{1}\right) .%
\end{array}%
\right. \right\} .
\end{equation*}%
(Note that $p_{+}$ is allowed to vary in $M$ in the above definition.)
We define $\mathcal{B}_{-}$ similarly with $\infty $ and $-\infty $ switched
around. We recall the decomposition
$$
\mathcal{B}_\pm(z_\pm) = \bigcup_{A_\pm \in \pi_2(z_\pm)} \mathcal{B}_\pm(z_\pm;A_\pm).
$$
For any section $\xi _{\pm }\in W_{loc}^{1,p}\left( u_{\pm }^{\ast
}TM\right) $, we define its Banach norm to be%
\begin{equation*}
\left\Vert \xi _{\pm }\right\Vert _{W_{\alpha }^{1,p}\left( \Sigma _{\pm
}\right) }:=\left\Vert \widetilde{\xi _{\pm }}\right\Vert _{W_{\delta
}^{1,p}\left( \Sigma _{\pm }\right) }+\left\vert \xi _{\pm }\left( o_{\pm
}\right) \right\vert ,
\end{equation*}%
where $\widetilde{\xi _{\pm }}$ is defined by
\begin{equation*}
\widetilde{\xi _{\pm }}=\xi _{\pm }-\varphi _{\pm }\left( \tau ,t\right)
\func{Pal}_{\pm }\left( \tau ,t\right) \left( \xi _{\pm }\left( o_{\pm
}\right) \right)
\end{equation*}%
and
\begin{equation*}
\left\Vert \widetilde{\xi _{\pm }}\right\Vert _{W_{\delta }^{1,p}\left(
\Sigma _{\pm }\right) }:=\left\Vert e^{\frac{2\pi \delta \left\vert \tau
\right\vert }{p}}\widetilde{\xi _{\pm }}\right\Vert _{W^{1,p}\left( {\mathbb{%
R}}\times S^{1}\right) }.
\end{equation*}%
Here $\varphi _{\pm }\left( \tau ,t\right) $ is a smooth cut-off function
such that
\begin{equation*}
\varphi _{\pm }\left( \tau ,t\right) =%
\begin{cases}
1\quad  & \mbox{ on }\,\pm \lbrack -\infty ,-1)\times S^{1} \\
0\quad  & \mbox{ outside the neighborhood of }O_{\pm }%
\end{cases}%
\end{equation*}%
and $\left\vert d\varphi _{\pm }\right\vert \leq 2$, and $\func{Pal}_{\pm
}\left( \tau ,t\right) $ is the parallel transport along the shortest
geodesics from $u_{\pm }\left( o_{\pm }\right) $ to $u_{\pm }\left( \tau
,t\right) $. This defines the Banach manifold structure on $\mathcal{B}_{\pm
}\left( z_{\pm }\right) $ with its tangent space at $u_{\pm }$ given by
\begin{equation*}
T_{u_{\pm }}\mathcal{B}_{\pm }\left( z_{\pm }\right) \simeq W_{\delta
}^{1,p}\left( \Gamma \left( u_{\pm }^{\ast }TM\right) \right) \oplus
T_{p_{\pm }}M.
\end{equation*}

Next we define the weighted norm
\begin{equation*}
\left\Vert \eta \right\Vert _{L_{\delta }^{p}\left( \Sigma _{\pm }\right)
}:=\left\Vert e^{\frac{2\pi \delta \left\vert \tau \right\vert }{p}}\eta
\right\Vert _{L^{p}\left( {\mathbb{R}}\times S^{1}\right) }
\end{equation*}%
for sections $\eta _{\pm }\in \Gamma \left( u_{\pm }^{\ast }TM\right)
\otimes \Lambda ^{0,1}\left( \Sigma _{\pm }\right) $, and denote by
\begin{equation*}
L_{\delta }^{p}\left( \Gamma \left( u_{\pm }^{\ast }TM\right) \otimes
\Lambda ^{0,1}\left( \Sigma _{\pm }\right) \right)
\end{equation*}%
the set of such $\eta $s with $\left\Vert \eta \right\Vert _{L_{\delta
}^{p}\left( \Sigma _{\pm }\right) }<\infty $ .
Then for each given $A_\pm \in \pi_2(z_\pm)$,
we form the Banach bundle
\begin{equation*}
\mathcal{L}_{\pm }\left( z_{\pm };A_\pm \right) :=\bigcup\limits_{u_{\pm }\in
\mathcal{B}_{\pm }\left( z_{\pm };A_\pm \right) }L_{\delta }^{p}\left( \Gamma
\left( u_{\pm }^{\ast }TM\right) \otimes \Lambda ^{0,1}\left( \Sigma _{\pm
}\right) \right)
\end{equation*}%
over $\mathcal{B}_{\pm }\left( z_{\pm };A_\pm \right) $. It is well-known that
\begin{equation*}
\overline{\partial }_{\left( K_{\pm },J_{\pm }\right) }:\mathcal{B}_{\pm
}\left( z_{\pm };A_\pm \right) \rightarrow \mathcal{L}_{\pm }\left( z_{\pm };A_\pm \right)
\end{equation*}%
is a smooth Fredholm section, and for generic $\left( K_{\pm },J_{\pm
}\right) $ the Floer trajectory $u_{\pm }\in \left( \overline{\partial }%
_{\left( K_{\pm },J_{\pm }\right) }\right) ^{-1}\left( 0\right) $ are
Fredholm regular, and hence
\begin{equation*}
D_{u_{\pm }}\overline{\partial }_{\left( K_{\pm },J_{\pm }\right)
}:T_{u_{\pm }}\mathcal{B}_{\pm }\left( z_{\pm };A_\pm \right) \rightarrow L_{\delta
}^{p}\left( \Gamma \left( u_{\pm }^{\ast }TM\right) \otimes \Lambda
^{0,1}\left( \Sigma _{\pm }\right) \right)
\end{equation*}%
is surjective and has bounded right inverse, denoted by $Q_{\pm }$. Here we
recall
\begin{equation*}
T_{u_{\pm }}\mathcal{B}_{\pm }\left( z_{\pm };A_\pm\right) \simeq W_{\delta
}^{1,p}\left( \Gamma \left( u_{\pm }^{\ast }TM\right) \right) \oplus
T_{p_{\pm }}M.
\end{equation*}%
For the simplicity of notations in later calculations, we denote by
\begin{equation}
v_{\pm }:=\varphi _{\pm }\left( \tau ,t\right) \func{Pal}_{\pm }\left( \tau
,t\right) \left( \xi _{\pm }\left( o_{\pm }\right) \right)   \label{Vpm}
\end{equation}%
the cut-off \emph{constant vector field} extending the vector $\xi _{\pm
}\left( o_{\pm }\right) $.

\section{Approximate solutions}

In this section, we construct an approximate solution $u_{\textrm{app}}^{\varepsilon
} $ of the Floer equation. For the notational simplicity, we denote
\begin{equation}  \label{eq:tau(e)}
\tau(\varepsilon) = \frac{l}{\varepsilon} + \frac{p-1}{\delta}S(\varepsilon)
\end{equation}
which will appear very often in the discussion henceforth. We also denote
the translation map
\begin{equation}  \label{eq:Itaue}
I_{\tau_0}: {\mathbb{R}} \times S^1 \to {\mathbb{R}} \times S^1, \quad
I_{\tau_0}(\tau)(\tau,t) = (\tau-\tau_0,t)
\end{equation}
for $\tau_0 \in {\mathbb{R}}$. To simplify the notations, we also denote
\begin{equation}  \label{eq:E(x,v)}
E(x,y) = (\exp_x)^{-1}(y)
\end{equation}
whenever $d(x,y)< \iota_M$ where $\iota_M$ is the injectivity radius of $%
(M,g)$.

We recall the decomposition of ${\mathbb{R}}$ into
\begin{equation*}
-\infty <-\tau (\varepsilon )-1<-\tau (\varepsilon )<-R(\varepsilon
)<R(\varepsilon )<\tau (\varepsilon )<\tau (\varepsilon )+1<\infty
\end{equation*}%
in the beginning of Section 3, where
\begin{equation*}
R(\varepsilon )=\frac{l}{\varepsilon },\quad \tau (\varepsilon
)=R(\varepsilon )+\frac{p-1}{\delta }S(\varepsilon ),~\text{ }S\left(
\varepsilon \right) =\frac{1}{2\pi }\ln \left( 1+\frac{l}{\varepsilon }%
\right) .
\end{equation*}%
And we denoted $K_{\pm }(\tau ,t,x)=\kappa ^{\pm }(\tau )H_{\pm }(t,x)$.
Note that this latter representation of $K_{\pm }$ depend on the choice of
analytic coordinates $(\tau ,t)$ compatible to the parameter $t$
parameterizing $H_{\pm }$ near the punctures respectively. The coordinates
are unique modulo translations by $\tau $.

Now let $u_\pm$ be solutions of the equations $(du +
P_{K_\pm})_{J_\pm}^{(0,1)}=0$ respectively and fix the coordinate
representations of $u_\pm = u_\pm(\tau,t)$ so that they are compatible with
the choice of analytic coordinates given at the punctures. This can be
always done by adjusting the choice of analytic coordinates near the two
punctures $(e_+,o_+)$ and $(e_-,o_-)$ respectively.

With this preparation, we define our approximate solution by
\begin{equation}
u_{\textrm{app}}^{\varepsilon }\left(\tau ,t\right) =%
\begin{cases}
u_{-}(\tau +\tau (\varepsilon ),t) & -\infty <\tau \leq -l/\varepsilon -1 \\
\exp _{\chi (\varepsilon \tau )}\left[ \left( 1-\kappa _{\varepsilon
}^{0}\left( \tau \right) \right) E(\chi (\varepsilon \tau ),
u_{-}^{\varepsilon }\left( \tau ,t\right) ) \right] & -l/\varepsilon -1\leq
\tau \leq -l/\varepsilon \\
\chi \left( \varepsilon \tau \right) & -l/\varepsilon \leq \tau \leq
l/\varepsilon , \\
\exp _{\chi (\varepsilon \tau )}\left[ \left( 1-\kappa _{\varepsilon
}^{0}\left( \tau \right) \right) E(\chi (\varepsilon \tau ),
u_{+}^{\varepsilon }\left( \tau ,t\right) ) \right] & l/\varepsilon \leq
\tau \leq l/\varepsilon +1 \\
u_{+}(\tau -\tau (\varepsilon ),t) & l/\varepsilon +1\leq \tau <\infty%
\end{cases}
\label{eq:uappe}
\end{equation}%
where
\begin{eqnarray*}
u_{-}^{\varepsilon }\left( \tau ,t\right) &=&u_{-}(\tau +\tau (\varepsilon
),t), \\
u_{+}^{\varepsilon }\left( \tau ,t\right) &=&u_{+}(\tau -\tau (\varepsilon
),t), \\
S\left( \varepsilon \right) &=&\frac{1}{2\pi }\ln \left( 1+l/\varepsilon
\right) ,~
\end{eqnarray*}%
and $\kappa _{\varepsilon }^{0}\left( \tau \right) $ is a smooth cut-off
functions defined in $\left( \ref{eq:beta0}\right) $, $\kappa _{\varepsilon
}^{0}\left( \tau \right) =1$ when $\left\vert \tau \right\vert \leq
l/\varepsilon $ and $\kappa _{\varepsilon }^{0}\left( \tau \right) =0$ when $%
\left\vert \tau \right\vert \geq l/\varepsilon +1$.

The assignment of the approximate solution to each $(u_{-},\chi ,u_{+})$ and
$0<\varepsilon <\varepsilon _{0}$ defines a smooth map
\begin{eqnarray}
&{}&\text{preG}:{\mathcal{M}}_{1}(K_{-},z_{-};A_{-}){}_{ev_{+}}\times
_{ev_{-l}}{\mathcal{M}}(f;[-l,l ]){}_{ev_{l }}\times _{ev_{-}}{\mathcal{M}}%
_{1}(K_{+},z_{+};A_{+})\times (0,\varepsilon _{0}]  \notag
\label{eq:preglue} \\
&{}&\hskip1.5in\rightarrow \bigcup_{0<\varepsilon \leq \varepsilon _{0}}{%
\mathcal{M}}^{\varepsilon }(K_{\varepsilon };z_{-},z_{+};B).
\end{eqnarray}

Later in carrying out the $\overline{\partial }_{\left( K_{\varepsilon
},J_{\varepsilon };\varepsilon f\right) }$-error estimate, we make some
simplification of the expression of $u_{\textrm{app}}^{\varepsilon }$. Notice that
the interpolation in $u_{\textrm{app}}^{\varepsilon }$ takes place \emph{in a ball of
radius }$C\varepsilon $ around $p_{\pm }$, since
\begin{eqnarray}
\sup_{l/\varepsilon \leq \left\vert \tau \right\vert \leq l/\varepsilon +1}%
\text{dist}\left( \chi _{\varepsilon }\left( \tau \right) ,p_{\pm }\right)
&=&\sup_{l/\varepsilon \leq \left\vert \tau \right\vert \leq l/\varepsilon
+1}\text{dist}\left( \chi _{\varepsilon }\left( \tau \right) ,\chi
_{\varepsilon }\left( \pm l/\varepsilon \right) \right)   \notag \\
&\leq &\varepsilon \sup \left\vert \nabla f\right\vert \leq C\varepsilon
\label{dist-xp}
\end{eqnarray}%
and
\begin{eqnarray}
\sup_{l/\varepsilon \leq \left\vert \tau \right\vert \leq l/\varepsilon +1}%
\text{dist}\left( u_{\pm }^{\varepsilon }\left( \tau \right) ,p_{\pm
}\right)  &=&\sup_{l/\varepsilon \leq \left\vert \tau \right\vert \leq
l/\varepsilon +1}\text{dist}\left( u_{\pm }^{\varepsilon }\left( \tau
\right) ,u_{\pm }^{\varepsilon }\left( \pm \infty \right) \right)   \notag \\
&\leq &Ce^{-2\pi \frac{p-1}{\delta }S\left( \varepsilon \right) }=C\left(
1+l/\varepsilon \right) ^{-\frac{p-1}{\delta }}  \label{dist-up} \\
&\leq &\widetilde{C}\left( l\right) \varepsilon ,
\end{eqnarray}%
where the first inequality is because we have
\begin{equation*}
\left\vert u_{\pm }^{\varepsilon }\left( \tau ,t\right) -u_{\pm
}^{\varepsilon }\left( \pm \infty ,t\right) \right\vert \leq Ce^{-2\pi
\left\vert \tau -\pm (\tau (\varepsilon ))\right\vert }
\end{equation*}%
by the $J$-holomorphic property of \ $u_{\pm }$ and the second inequality is
because we have chosen $0<\delta <p-1$ and set $\widetilde{C}\left( l\right)
=Cl^{-\frac{p-1}{\delta }}$.

At the interpolation $\left\vert \tau \right\vert \in \lbrack l/\varepsilon
,l/\varepsilon +1]$, both $u_{\pm }^{\varepsilon }\left( \tau ,t\right) $
and $\chi \left( \varepsilon \tau \right) $ are in the ball of radius $%
C\varepsilon $ around $p_{\pm }$. The expression
\begin{equation*}
\exp _{\chi \left( \varepsilon \tau \right) }\left[ \left( 1-\kappa
_{\varepsilon }^{0}\left( \tau \right) \right) E(\chi (\varepsilon \tau ),
u_{\pm}^{\varepsilon }\left( \tau ,t\right)) \right] =:
v_\varepsilon^\pm(\tau,t)
\end{equation*}%
in \eqref{eq:uappe} is $C^{1}$ close to $\chi (\varepsilon \tau )+\left(
1-\kappa _{\varepsilon }^{0}\left( \tau \right) \right) \left( u_{\pm
}^{\varepsilon }\left( \tau ,t\right) -\chi \left( \varepsilon \tau \right)
\right) $, i.e.
\begin{equation*}
\kappa _{\varepsilon }^{0}\left( \tau \right) \chi (\varepsilon \tau
)+\left( 1-\kappa _{\varepsilon }^{0}\left( \tau \right) \right) u_{\pm
}^{\varepsilon }\left( \tau ,t\right) =: \widetilde v_\varepsilon^\pm(\tau,t)
\end{equation*}%
in coordinates, where the $+$ is from the vector space structure from $%
T_{p_{\pm }}M$, and the $C^{1}$ difference is of order $C\varepsilon $. To
see this, we identify the geodesic ball of radius $C\varepsilon $ around $%
p_{\pm }$ to the ball in $T_{p_{\pm }}M$, and equip it with the Euclidean
metric $g_{p_{\pm }}$. If we deform the metric involved in the exponential
map in the expression
\begin{equation}  \label{eq:exp-expression}
\exp _{\chi \left( \varepsilon \tau \right) }\left[ \left( 1-\kappa
_{\varepsilon }^{0}\left( \tau \right) \right) E(\chi (\varepsilon \tau ),
u_{\pm }^{\varepsilon }\left( \tau ,t\right)) \right]
\end{equation}
to the Euclidean metric $g_{p_{\pm }}$, the result becomes $\kappa
_{\varepsilon }^{0}\left( \tau \right) \chi (\varepsilon \tau )+\left(
1-\kappa _{\varepsilon }^{0}\left( \tau \right) \right) u_{\pm
}^{\varepsilon }\left( \tau ,t\right) $. \ Since the geodesic equation is a
second order differential equation whose coefficients are polynomials on
metric $g$ and its first order derivatives, by the differentiability of
solutions of ODE on finite interval $\left[ 0,1\right] $ with respect to its
initial condition and parameter, we see the $C^{1}$ norm of the map %
\eqref{eq:exp-expression} depends on the $C^{1}$ norm of the section $g\in
\Gamma \left( Sym_{+}\left( M\right) \right) $ with bounded Lipshitz
constant, where $Sym_{+}\left( M\right) $ is the space of positive definite
symmetric tensors on $M$. Since the $C^{1}$ difference of the metrics $%
g_{p_{\pm }}$ and $g$ is of order $C\varepsilon $ inside such ball, we have
\begin{eqnarray*}
\func{dist}\left( v_\varepsilon^\pm(\tau,t),\widetilde
v_\varepsilon^\pm(\tau,t) \right) &\leq & C\varepsilon , \\
\Vert \nabla v_\varepsilon^\pm(\tau,t) -\func{Pal}\nabla \widetilde
v_\varepsilon^\pm(\tau,t)\Vert &\leq & C\varepsilon ,
\end{eqnarray*}%
where $\func{Pal}$ is the parallel transport along short geodesic from $%
\widetilde v_\varepsilon^\pm(\tau,t)$ to $v_\varepsilon^\pm(\tau,t)$.

After the above simplification, we use the more explicit approximate solution%
\begin{equation}
u_{\textrm{app}}^{\varepsilon }\left( \tau ,t\right) =%
\begin{cases}
u_{-}(\tau +\tau (\varepsilon ),t) & -\infty <\tau \leq -l/\varepsilon \\
\kappa _{\varepsilon }^{0}\left( \tau \right) \chi (\varepsilon \tau
)+\left( 1-\kappa _{\varepsilon }^{0}\left( \tau \right) \right)
u_{-}^{\varepsilon }\left( \tau ,t\right) & -l/\varepsilon -1\leq \tau \leq
-l/\varepsilon \\
\chi (\varepsilon \tau ) & -l/\varepsilon \leq \tau \leq l/\varepsilon \\
\kappa _{\varepsilon }^{0}\left( \tau \right) \chi (\varepsilon \tau
)+\left( 1-\kappa _{\varepsilon }^{0}\left( \tau \right) \right)
u_{+}^{\varepsilon }\left( \tau ,t\right) & l/\varepsilon \leq \tau \leq
l/\varepsilon +1 \\
u_{+}(\tau -\tau (\varepsilon ),t) & l/\varepsilon +1\leq \tau <\infty%
\end{cases}
\label{eq:uappe-easy}
\end{equation}%
where the vector sum $+$ is from the linear space structure of $T_{p_{\pm
}}M $.

\begin{rem}
Apparently there enters no local model inserted at the joint points $p_{\pm
} $ to smooth out the join points in the construction of the above
approximate solution. Implicitly there is, though. The local model at $%
p_{\pm } $ is $u_{\pm }^{lmd}:{\mathbb{R}}\times S^{1}\rightarrow \mathbb{C}%
^{n}\simeq \left( T_{p_{\pm }}M,J_{p_{\pm }}\right) $,
\begin{equation*}
u_{\pm }^{lmd}\left( \tau ,t\right) =A_{\pm }z+a_{\pm }\tau ,
\end{equation*}%
where $z=e^{2\pi \left( \tau +it\right) }$, $A_{\pm }=u_{\pm }^{\prime
}\left( o_{\pm }\right) $, and $a_{\pm }=\nabla f\left( p_{\pm }\right) $.
Then one can see the local model is the linearized version of the above
interpolation of $u_{\pm }$ and $\chi $ in $u_{\textrm{app}}^{\varepsilon }$. Because
we can identify the portions of the approximate solution to $u_{\pm }$ and $%
\chi $ respectively and borrow Fredholm theories there, we do not need to
develop a Fredholm theory of the local model. However, when the gradient
flow length is $0$, the Fredholm theory of local model is needed for gluing,
because during compactification to nodal Floer trajectories the information
of $\nabla f$ is lost (See \cite{oh-zhu}).
\end{rem}

\section{Off-shell formulation of resolved Floer trajectories for $\protect%
\varepsilon> 0$}

\label{sec:off-shell}

Let $z_\pm$ be a pair of nondegenerate closed orbits of Hamiltonians $H_\pm$ respectively,
and let a relative homotopy class $B \in \pi_2(z_-,z_+)$ be given.

In this section, we define the Banach manifold to host resolved Floer
trajectories near the \textquotedblleft
thimble-flow-thimble\textquotedblright\ Floer trajectories $(u_{-},(\chi
,l),u_{+})$,

\beastar
\CB_{\text{\rm res}}(z_-.z_+;B)& = & \bigcup_{\varepsilon > 0}
\CB_{\text{\rm res}}^\varepsilon (z_-,z_+;B) \\
\CB_{\text{\rm res}}^\varepsilon (z_-.z_+;B)& = &  \bigcup_{l > 0}\mathcal{B}_{\text{\rm res}}^{\varepsilon }(z_{-},z_{+};B;l)
\eeastar
for which the descriptions of the Banach manifold $\mathcal{B}%
_{\text{\rm res}}^{\varepsilon }(z_{-},z_{+};B)$ and the space $\mathcal{B}%
_{\text{\rm res}}^{\varepsilon }(z_{-},z_{+};B;l)$ are in order for any $\varepsilon
\in (0,\varepsilon _{0})$ and $l \in (0,\infty )$, for a small constant $%
\varepsilon _{0}>0$ to be determined later.

For this purpose, we first introduce the weighting function $\beta _{\delta
,\varepsilon }$.

\begin{choice}[Weight function $\protect\beta _{\protect\delta ,\protect%
\varepsilon }$]
We define the function $\beta _{\delta ,\varepsilon }$ to be
\begin{equation}
\beta _{\delta ,\varepsilon }\left( \tau \right) =%
\begin{cases}
1 & \tau <-\tau (\varepsilon ) \\
e^{2\pi \delta (\tau +\tau (\varepsilon ))} & -\tau (\varepsilon )\leq \tau
\leq -l/\varepsilon  \\
\kappa _{\varepsilon }^{0}\left( \tau \right) \varepsilon ^{1-p+\delta
}\left( 1+\left\vert \tau \right\vert \right) ^{\delta }+\left( 1-\kappa
_{\varepsilon }^{0}\left( \tau \right) \right) e^{2\pi \delta (\tau +\tau
(\varepsilon ))} & -l/\varepsilon -1\leq \tau \leq -l/\varepsilon  \\
\varepsilon ^{1-p+\delta }\left( 1+\left\vert \tau \right\vert \right)
^{\delta } & -l/\varepsilon \leq \tau \leq l/\varepsilon  \\
\kappa _{\varepsilon }^{0}\left( \tau \right) \varepsilon ^{1-p+\delta
}\left( 1+\left\vert \tau \right\vert \right) ^{\delta }+\left( 1-\kappa
_{\varepsilon }^{0}\left( \tau \right) \right) e^{2\pi \delta (-\tau +\tau
(\varepsilon ))} & l/\varepsilon \leq \tau \leq l/\varepsilon +1 \\
e^{2\pi \delta (-\tau +\tau (\varepsilon ))} & l/\varepsilon +1\leq \tau
\leq \tau (\varepsilon ) \\
1 & \tau >\tau (\varepsilon )%
\end{cases}
\label{beta-weight}
\end{equation}%
where $\kappa _{\varepsilon }^{0}\left( \tau \right) $ is the smooth cut-off
function defined in \ $\left( \ref{eq:beta0}\right) $ such that $\kappa
_{\varepsilon }^{0}\left( \tau \right) =1$ for $\left\vert \tau \right\vert
\leq l/\varepsilon $ and $\kappa _{\varepsilon }^{0}\left( \tau \right) =0$
for $\left\vert \tau \right\vert \geq l/\varepsilon +1$.
\end{choice}

Note that
\begin{equation}  \label{eq:beta>1}
\beta _{\delta ,\varepsilon }\left( \tau \right) \geq 1
\end{equation}
everywhere, which will be important to have the uniform Sobolev constant we
will discuss later in Section \ref{sec:quadratic}.

%\vspace{0.3cm}
\begin{figure}[hbt!]
% weighting functions, pre-glued solution, and Floer data picture
\begin{tikzpicture} [scale=0.8]
	% \tau axis
	\draw [thick][->] (-6,0) -- (6,0) coordinate ; \node [below right] at (6,0) {$\tau$};
	\draw [fill] (2,0) circle [radius=0.1]; \draw [fill] (4,0) circle [radius=0.1];
	\draw [fill] (-2,0) circle [radius=0.1]; \draw [fill] (-4,0) circle [radius=0.1];
	
	\draw [fill=brown] (3,0) circle [radius=0.06]; \draw [fill=brown] (1,0) circle [radius=0.06];
	\draw [fill=brown] (-3,0) circle [radius=0.06]; \draw [fill=brown] (-1,0) circle [radius=0.06];
	
	\node [above] at (-4,0) {$-\tau(\varepsilon)$}; \node [above] at (-2,0) {$-l/{\varepsilon}$};
	\node [above right] at (0,0) {$0$};
	\node [above] at (4,0) {$\tau(\varepsilon)$}; \node [above] at (2,0) {$l/{\varepsilon}$};

	\node [below] at (-3,0) {$-\frac{l}{\varepsilon}-T(\varepsilon)$};
	\node [below] at (-1,0) {$-\frac{l}{\varepsilon}+T(\varepsilon)$};
	\node [below] at (3,0) {$\frac{l}{\varepsilon}+T(\varepsilon)$};
	\node [below] at (1,0) {$\frac{l}{\varepsilon}-T(\varepsilon)$};
	
	% y axis, with all markings
	\draw [thick][->] (0,-6) -- (0,6) coordinate ;
	
	\node [right] at (0,3.3) {$\beta_{\delta,\varepsilon}(\tau)$};
	\node [below right, red] at (0,1.5) {$\varepsilon^{1-p}$};
	\node [above right, red] at (0,-2) {$\chi_{\varepsilon}$};
	\node [above right, red] at (0,-4) {$\varepsilon f$};
	\node [below right, red] at (0,-4) {$J_0$};
	
	% 4 vertical black lines
	\draw [dashed] (2,-6) -- (2,6) ;
	\draw [dashed] (-2,-6) -- (-2,6);
	\draw [dashed] (4,-6) -- (4,6) ;
	\draw [dashed] (-4,-6) -- (-4,6);
	
	\draw [brown] [dotted] (1,-6) -- (1,6) ;   % 4 vertical brown doted lines
	\draw [brown] [dotted] (-1,-6) -- (-1,6);
	\draw [brown] [dotted] (3,-6) -- (3,6) ;
	\draw [brown] [dotted] (-3,-6) -- (-3,6);
	
	% geometric weight, red horizontal lines, with two paddings
	\draw [red] (-2,1.5) -- (2,1.5) ;
	\draw [dashed] [red] (2,1.5) -- (3,1.5) ;
	\draw [dashed] [red] (-2,1.5) -- (-3,1.5) ;
	
	% polynomial weight, green parabola, with two padings
	\draw [green] (-2,2.5) to [out=-30, in=210 ] (2,2.5);
	\draw [dashed] [green] (2,2.5) to [out=30, in=240 ] (3,3.5);
	\draw [dashed] [green] (-2,2.5) to [out=150, in=-60 ] (-3,3.5);
	
	\node [green] at (0,2.8) {$\rho_{\varepsilon}(\tau)={\varepsilon}^{1-p+\delta}(1+|\tau|)^{\delta}$};
	
	% exponential weight, blue curve on right, with constant padding
	\draw [dashed] [blue] (1,5) to [out=-80, in=150 ] (2,2.5);
	\draw [blue] (2,2.5) to [out=-30, in=180 ] (4,2);
	\draw [blue] (4,2)--(5,2);
	
	\draw [dashed] [blue] (-1,5) to [out=-100, in=30 ] (-2,2.5); % exponential weight, blue curve on left, with constant padding
	\draw [blue] (-2,2.5) to [out=210, in=0 ] (-4,2);
	\draw [blue] (-4,2)--(-5,2);
	
	\node [above left, blue] at (-4,2) {$e^{2\pi\delta(\tau-\tau(\varepsilon))}$};
	\node [above right, blue] at (4,2) {$e^{2\pi\delta(-\tau+\tau(\varepsilon))}$};
	
	\node at (-6,3) {weighting functions}; % weighting functions

	% approximate solution, DFD, and gradient flow \chi_{\varepsilon}
	\draw [blue] (-5,-2)--(-2.5,-2);   \draw [dotted] [blue] (-2.5,-2)--(-2,-2);
	\draw [red] (-2,-2) -- (2,-2) ;
	\draw [blue] (5,-2)--(2.5,-2);   \draw [dotted] [blue] (2.5,-2)--(2,-2);
	
	\node [above left, blue ] at (-4,-2) {$u^{\varepsilon}_{-}$};
	\node [above right, blue] at (4,-2) {$u^{\varepsilon}_{+}$};
	
	\node at (-6,-1) {preglued solution}; %pregluing
	
	% Hamiltonian function, middle part {\varepsilon}f
	\draw [blue] (-6,-4)--(-4.5,-4);   \draw [dotted] (-4.5,-4)--(-4,-4); % left Hamiltonian
	\draw [blue] (6,-4)--(4.5,-4);   \draw [dotted] (4.5,-4)--(4,-4); % right Hamiltonian
	
	\draw [dashed] (-2,-4) to (-4,-4); %left piece transition
	\draw [dashed] (2,-4) to (4,-4);  %right piece transition
	
	\draw [red] (-2,-4) -- (2,-4) ;  % Morse function \varepsilon f
	
	\node [above, blue] at (-5,-4) {$\kappa_{\varepsilon}^{-}H_t$}; \node [below, blue] at (-5,-4) {$J_{\varepsilon}^-$};
	\node [above, blue] at (5,-4) {$\kappa_{\varepsilon}^{+}H_t$}; \node [below, blue] at (5,-4) {$J_{\varepsilon}^+$};
	\node [above] at (-3,-4) {$0$};   \node [below] at (-3,-4) {$J_0$};
	\node [above] at (3,-4) {$0$};   \node [below] at (3,-4) {$J_0$};
	
	\node at (-6,-3) {Floer data}; % Floer datum
	
\end{tikzpicture}
\caption{weighting functions, pre-glued solution, and Floer data}
\end{figure}
\vspace{0.3cm}

In the following figure we put the graphs of various weighting functions
together, where the higher constant weight $\varepsilon ^{1-p}$ in the
adiabatic weight $\Vert \cdot \Vert _{W_{\varepsilon }^{1,p}}$ is in red
horizontal line, power order weight $\rho _{\varepsilon }(\tau )$ is in
green, and the exponential weight is in blue. The weight $\beta _{\delta
,\varepsilon }(\tau )$ is the glue of the power weight and the exponential
weight with smoothing at the corners, but to avoid too many graphs in the
picture we did not draw the smoothing. The two intervals $[l/\varepsilon
-T(\varepsilon ),l/\varepsilon +T(\varepsilon )]$ and $[-l-T(\varepsilon
),-l/\varepsilon +T(\varepsilon )]$ cut by four brown vertical lines are the
places where weighting function comparison occurs in right inverse
estimates, where $T\left( \varepsilon \right) =\frac{1}{3}\frac{p-1}{\delta }%
S\left( \varepsilon \right) $. For convenience of readers, we also include
the schematic picture of the preglued solution $u_{\textrm{app}}^{\varepsilon }$ and
Floer datum $(K_{\varepsilon },J_{\varepsilon })$ of the perturbed
Cauchy-Riemann equation. Note that in interval $\pm \lbrack l/\varepsilon
,\tau (\varepsilon )]$, the Hamiltonian $K_{\varepsilon }=0$ and $%
J_{\varepsilon }=J_{0}$ because of the cut-off function $\kappa
_{\varepsilon }^{\pm }(\tau )$.

This Banach manifold can be regarded as the gluing of the Banach manifolds $%
\mathcal{B}_{\pm }$ and $\mathcal{B}_{\chi _{\varepsilon }}$. More precisely
${\mathcal{B}}_{\text{\rm res}}^{\varepsilon }(z_{-},z_{+};B;l/\varepsilon )$ $\left(
l\geq l_{0}>0\right) $ consists of maps $u:\Sigma _{\varepsilon }\rightarrow
M$ satisfying:

\begin{enumerate}
\item $\Sigma _{\varepsilon }$ is diffeomorphic ${\mathbb{R}}\times S^{1}$
but decomposed into $3$ cylinders according to the Floer data, where $S_{\pm
}={\mathbb{R}}\times S^{1}$ are for the two the two ends, and $[-\tau
(\varepsilon ),\tau (\varepsilon )]\times S^{1}$ is for the part near
gradient segments: We take the family $\Sigma _{\varepsilon }$ of glued surfaces
constructed by the gluing process laid out in Choice \ref{choice:gluing-domains},
which we just denote by
\begin{equation}
\Sigma _{\varepsilon }=S_{-}\cup ([-\tau (\varepsilon ),\tau (\varepsilon
)]\times S^{1})\cup S_{+}.  \label{eq:Sigmae}
\end{equation}

\item $u\in W_{loc}^{1,p}(\Sigma _{\varepsilon },M)$

\item $\lim_{\tau \rightarrow +\infty }u(\tau ,t)=z_{+}(t)$ and $\lim_{\tau
\rightarrow -\infty }u(\tau ,t)=z_{-}(t)$ for all $t\in S^{1}$.

\item For $u\in \mathcal{B}_{\varepsilon }$, and any variation vector field $%
\xi \in \Gamma \left( W^{1,p}\left( u^{\ast }TM\right) \right) $, we define
the Banach norm to be
\begin{equation}
\left\Vert \xi \right\Vert _{\varepsilon }=\left\Vert \widetilde{\xi }%
\right\Vert _{W_{\beta _{\delta ,\varepsilon }}^{1,p}\left( \Sigma
_{\varepsilon }\right) }+\left\Vert \xi _{0}\right\Vert _{W_{\varepsilon
}^{1,p}\left( -l/\varepsilon ,l/\varepsilon \right) }+\left\vert \xi
_{0}\left( \pm l/\varepsilon \right) \right\vert ,  \label{eq:xienorm}
\end{equation}%
where $\xi $ is decomposed into the \textquotedblleft zero
mode\textquotedblright\ $\xi _{0}$ and \textquotedblleft higher
mode\textquotedblright\ $\widetilde{\xi }$ for $\left\vert \tau \right\vert
\leq l/\varepsilon $,  \
\begin{eqnarray}
\xi _{0}\left( \tau \right)  &=&\left\{
\begin{tabular}{ll}
$\int_{S^{1}}\xi \left( \tau ,t\right) dt,$ & $\left\vert \tau \right\vert
\leq l/\varepsilon $ \\
$\kappa _{\varepsilon }^{0}\left( \tau \right) \int_{S^{1}}\xi \left( \pm
l/\varepsilon ,t\right) dt,$ & $\left\vert \tau \right\vert \geq
l/\varepsilon $%
\end{tabular}%
\right. ,  \notag \\
\widetilde{\xi }\left( \tau ,t\right)  &=&\xi \left( \tau ,t\right) -\xi
_{0}\left( \tau \right) ,  \label{eq:xienorms} \\
\left\Vert \widetilde{\xi }\right\Vert _{W_{\beta _{\delta ,\varepsilon
}}^{1,p}\left( \Sigma _{\varepsilon }\right) }^{p} &=&\int \int_{\Sigma
_{\varepsilon }}\left( |\widetilde{\xi }|^{p}+|\nabla \widetilde{\xi }%
|^{p}\right) \beta _{\delta ,\varepsilon }\left( \tau \right) dvol_{\Sigma
_{\varepsilon }}.  \notag
\end{eqnarray}
\end{enumerate}

Therefore, we have an $\varepsilon $-family of Banach manifolds ${\mathcal{B}%
}_{\text{\rm res}}^{\varepsilon }(z_{-},z_{+};l/\varepsilon )$, and an $\varepsilon $%
-family of equations ${\overline{\partial }}_{(J_{\varepsilon
},K_{\varepsilon })}u^{\varepsilon }=0$ defined on each Banach bundle
\begin{equation*}
\pi :{\mathcal{L}}_{\text{\rm res}}^{\varepsilon }(z_{-},z_{+};l/\varepsilon
)\rightarrow {\mathcal{B}}_{\text{\rm res}}^{\varepsilon }(z_{-},z_{+};l/\varepsilon ),
\end{equation*}%
where
\begin{equation*}
{\mathcal{L}}_{\text{\rm res}}^{\varepsilon }(z_{-},z_{+};l/\varepsilon )=\bigcup_{u\in
{\mathcal{B}}_{\text{\rm res}}^{\varepsilon }(z_{-},z_{+};l/\varepsilon )}L_{\beta
_{\delta ,\varepsilon }}^{p}(\Lambda ^{0,1}(u^{\ast }TM)).
\end{equation*}%
Here each fiber $L_{\beta _{\delta ,\varepsilon }}^{p}(\Lambda
^{0,1}(u^{\ast }TM))$ consists of sections $\eta \in L^{p}\left( \Lambda
^{0,1}(u^{\ast }TM)\right) $ with $\Vert \eta \Vert _{\varepsilon }<\infty $
and the norm $\Vert \eta \Vert _{\varepsilon }$ is given by
\begin{equation}
\left\Vert \eta \right\Vert _{\varepsilon }=\left\Vert \widetilde{\eta }%
\right\Vert _{L_{\beta _{\delta ,\varepsilon }}^{p}\left( \Sigma
_{\varepsilon }\right) }+\left\Vert \kappa _{\varepsilon }^{0}\left( \tau
\right) \eta _{0}\left( \tau \right) \right\Vert _{L_{\varepsilon }^{p}\left[
-l/\varepsilon ,l/\varepsilon \right] },  \label{eq:etanorm}
\end{equation}%
where $\overline{\eta },\eta _{0}$ and $\widetilde{\eta }$ are defined
similarly as those for $\xi ,$ namely
\begin{eqnarray}
\eta _{0}\left( \tau \right)  &=&\left\{
\begin{tabular}{ll}
$\int_{S^{1}}\eta \left( \tau ,t\right) dt,$ & $\left\vert \tau \right\vert
<l/\varepsilon $ \\
$0$ & $\left\vert \tau \right\vert \geq l/\varepsilon $%
\end{tabular}%
\right. ,  \notag \\
\widetilde{\eta }\left( \tau ,t\right)  &=&\eta \left( \tau ,t\right) -\eta
_{0}\left( \tau \right) ,  \label{eq:etanorms} \\
\left\Vert \widetilde{\eta }\right\Vert _{L_{\beta _{\delta ,\varepsilon
}}^{p}\left( \Sigma _{\varepsilon }\right) }^{p} &=&\int \int_{\Sigma
_{\varepsilon }}\left\vert \widetilde{\eta }\right\vert ^{p}\beta _{\delta
,\varepsilon }\left( \tau \right) dvol_{\Sigma _{\varepsilon }}.  \notag
\end{eqnarray}

We fix a constant $l_0> 0$ once and for all and define
\begin{eqnarray}
{\mathcal{B}}_{\text{\rm res}}^{\varepsilon }(z_{-},z_{+};B) &=&\bigcup_{l\geq l_{0}}{%
\mathcal{B}}_{\text{\rm res}}^{\varepsilon }(z_{-},z_{+};B;l/\varepsilon )
\label{eq:CBe} \\
{\mathcal{L}}_{\text{\rm res}}^{\varepsilon }(z_{-},z_{+};B) &=&\bigcup_{l\geq l_{0}}{%
\mathcal{L}}_{\text{\rm res}}^{\varepsilon }(z_{-},z_{+};B;l/\varepsilon ).
\label{eq:CLe}
\end{eqnarray}%
For $(u,l/\varepsilon )\in {\mathcal{B}}_{\text{\rm res}}^{\varepsilon }(z_{-},z_{+};B)$%
, its tangent space consists of elements $(\xi ,\mu )$ where $\xi \in T_{u}{%
\mathcal{B}}_{\text{\rm res}}^{\varepsilon }(z_{-},z_{+};B;l/\varepsilon )$ and $\mu
\in T_{l/\varepsilon }{\mathbb{R}}_{+}\cong {\mathbb{R}}$ with the norm
\begin{equation*}
\Vert (\xi ,\mu )\Vert _{\varepsilon }=\Vert \xi \Vert _{\varepsilon }+|\mu
|.
\end{equation*}%
%
%
%
%
%
%
%
%
%
%
%
%
%
%
%
%
%
%
%
%
%
%
%
%
%
%
%
%Notice in the above Banach manifold $\mathcal{B}_{\varepsilon ,l}\left(
%z_{-},z_{+}\right) $ there is a parameter $l$. \ We need to consider $u$ in
%the $l$ family of Banach manifolds $\mathcal{B}_{\varepsilon ,l}\left(
%z_{-},z_{+}\right) $. We have%
%\begin{equation*}
%T_{u}\mathcal{B}_{\varepsilon }\left( z_{-},z_{+}\right) \simeq T_{u}%
%\mathcal{B}_{\varepsilon ,l}\left( z_{-},z_{+}\right) \oplus T_{l}\mathbb{R}.
%\end{equation*}%
%For $\xi \in T_{u}\mathcal{B}_{\varepsilon ,l}\left( z_{-},z_{+}\right) $
%and $\mu \in T_{l}\mathbb{R}$, we define the Banach norm in $\mathcal{B}%
%_{\varepsilon }\left( z_{-},z_{+}\right) $ to be
%\begin{equation*}
%\left\Vert \left( \xi ,\mu \right) \right\Vert _{\varepsilon }=\left\Vert
%\xi \right\Vert _{\varepsilon }+\left\vert \mu \right\vert .
%\end{equation*}%
Geometrically $\mu $ corresponds to the variation of conformal structure of
the neck cylinder $\left[ -l/\varepsilon ,l/\varepsilon \right] \times S^{1}$
by varying the length of the cylinder but keeping the radius of $S^{1}$
fixed. For $\mu \in T_{l/\varepsilon }\mathbb{R}$, the induced path in ${%
\mathcal{B}}_{\text{\rm res}}^{\varepsilon }\left( z_{-},z_{+};B\right) $, starting
from $u\in {\mathcal{B}}_{\text{\rm res}}^{\varepsilon }\left(
z_{-},z_{+};B;l/\varepsilon \right) $, is $u_{s}\in {\mathcal{B}}%
_{\text{\rm res}}^{\varepsilon }\left( z_{-},z_{+};B;l/\varepsilon -s\mu \right) $,
where $u_{s}(\tau ^{\prime },t)$ is the reparameterization of $u(\tau ,t)$
on the neck part,%
\begin{equation}
u_{s}(\tau ^{\prime },t):=u\left( \frac{\tau ^{\prime }}{l-s\varepsilon \mu }%
,t\right) ,\text{ \ \ \ \ \ }\left( \tau ^{\prime },t\right) \in \left[
-\left( l/\varepsilon -s\mu \right) ,l/\varepsilon -s\mu \right] \times
S^{1},  \label{reparameterize-tau}
\end{equation}%
for $s$ nearby $0$. There is a canonical way to associate points on $u$ to
points on $u_{s}:$%
\begin{equation*}
u\left( \tau ,t\right) \longleftrightarrow u_{s}\left( \tau ^{\prime
},t\right) ,\text{ \ \ \ \ \ where }\tau ^{\prime }=\frac{(l-s\varepsilon
\mu )\tau }{l}.
\end{equation*}%
Using this identification we can realize the variation of conformal
structure as a vector field on $\left[ -l/\varepsilon ,l/\varepsilon \right]
\times S^{1}$.

In summary, we have defined the Banach norm for $\xi $ in $T_{u}{\mathcal{B}}%
_{\text{\rm res}}^{\varepsilon }(z_{-},z_{+};l/\varepsilon )$ by
\begin{equation}
\left\Vert \xi \right\Vert _{\varepsilon }=\left\Vert \widetilde{\xi }%
\right\Vert _{W_{\beta _{\delta ,\varepsilon }}^{1,p}\left( \mathbb{R}\times
S^{1}\right) }+\left\Vert \xi _{0}\right\Vert _{W_{\varepsilon }^{1,p}\left( %
\left[ -l/\varepsilon ,l/\varepsilon \right] \right) }+\left\vert \xi
_{0}\left( \pm l/\varepsilon \right) \right\vert ,  \label{eq:normxie}
\end{equation}%
and the Banach norm for $\eta \in \Gamma \left( L^{p}\left( u^{\ast
}TM\otimes \Lambda ^{0,1}\left( \Sigma _{\varepsilon }\right) \right)
\right) $ by
\begin{equation}
\left\Vert \eta \right\Vert _{\varepsilon }=\left\Vert \widetilde{\eta }%
\right\Vert _{L_{\beta _{\delta ,\varepsilon }}^{p}\left( \mathbb{R}\times
S^{1}\right) }+\left\Vert \kappa _{\varepsilon }^{0}\left( \tau \right) \eta
_{0}\left( \tau \right) \right\Vert _{L_{\varepsilon }^{p}\left[
-l/\varepsilon ,l/\varepsilon \right] }.  \label{eq:normeta}
\end{equation}

\begin{rem}
In the norm $\left\Vert \xi \right\Vert _{\varepsilon }$ we could have
dropped the term $\left\vert \xi _{0}\left( \pm l/\varepsilon \right)
\right\vert $ but still get an equivalent norm, because from Sobolev
embedding we have
\begin{equation*}
\left\vert \xi _{0}\left( \pm l/\varepsilon \right) \right\vert \leq
C\left\Vert \xi _{0}\right\Vert _{W_{\varepsilon }^{1,p}\left(
[-l/\varepsilon ,l/\varepsilon ]\right) }
\end{equation*}%
where $C$ is uniform for all $\varepsilon $ and $l\geq l_{0}>0$. However we
still keep the $\left\vert \xi _{0}\left( \pm l/\varepsilon \right)
\right\vert $ term because $\left\Vert \widetilde{\xi }\right\Vert
_{W_{\beta _{\delta ,\varepsilon }}^{1,p}\left( \left\{ \left\vert \tau
\right\vert >l/\varepsilon \right\} \times S^{1}\right) }+\left\vert \xi
_{0}\left( \pm l/\varepsilon \right) \right\vert $ mimics the Banach norm of
$\xi _{\pm }$ and $\left\Vert \widetilde{\xi }\right\Vert _{W_{\beta
_{\delta ,\varepsilon }}^{1,p}\left( \left\{ \left\vert \tau \right\vert
\leq l/\varepsilon \right\} \times S^{1}\right) }+\left\Vert \xi _{0}\left(
\tau \right) \right\Vert _{W_{\varepsilon }^{1,p}\left( \left[
-l/\varepsilon ,l/\varepsilon \right] \right) }$ mimics the Banach norm of $%
\xi _{\chi _{\varepsilon }}$ and in this sense the norm $\left\Vert \xi
\right\Vert _{\varepsilon }$ is roughly the sum of them.
\end{rem}

Finally we consider the Fredholm section
\begin{equation*}
{\mathcal{B}}_{\text{\rm res}}^{\varepsilon }(z_{-},z_{+};B)\rightarrow {\mathcal{L}}%
_{\text{\rm res}}^{\varepsilon }(z_{-},z_{+};B)
\end{equation*}%
given by
\begin{equation*}
(u,l/\varepsilon )\mapsto {\overline{\partial }}_{K_{\varepsilon
},J_{\varepsilon };l/\varepsilon }(u):{\mathcal{B}}_{\text{\rm res}}^{\varepsilon
}(z_{-},z_{+};B)\rightarrow {\mathcal{L}}_{\text{\rm res},u}^{\varepsilon
}(z_{-},z_{+};B;l/\varepsilon )
\end{equation*}%
defined fiberwise over $l/\varepsilon \in {\mathbb{R}}$. We denote the
linearization of this map by $D^\varepsilon_{\Phi;\text{\textrm{para}}}$
which has expression
\begin{equation}
D^\varepsilon_{\Phi;\text{\textrm{para}}}(\xi ,\mu
)=D^\varepsilon_{\Phi}(\xi )+\frac{\mu }{l}\nabla \varepsilon f(\chi
_{\varepsilon }).  \label{eq:tildeDe}
\end{equation}

\section{$\overline{\partial }_{\left( K_{\protect\varepsilon },J_{\protect%
\varepsilon }\right) }$-error estimate\label{section:d-bar-error}}

In this section, we estimate the norm $\Vert \overline{\partial }%
_{(K_{\varepsilon },J_{\varepsilon })}u_{\textrm{app}}^{\varepsilon }\Vert
_{\varepsilon }$. Let $\eta =\overline{\partial }_{(K_{\varepsilon
},J_{\varepsilon })}u_{\textrm{app}}^{\varepsilon }$. We do this estimation
separately in several regions separately: \medskip

(1). When $\left\vert \tau \right\vert \leq l/\varepsilon $, $%
u_{\textrm{app}}^{\varepsilon }\left( \tau ,t\right) =\chi \left( \varepsilon \tau
\right) $, so%
\begin{equation*}
\overline{\partial }_{\left( K_{\varepsilon },J_{\varepsilon }\right)
}u_{\textrm{app}}^{\varepsilon }\left( \tau ,t\right) =\overline{\partial }_{\left(
J_{0},\varepsilon f\right) }\chi \left( \varepsilon \tau \right) =\frac{%
\partial }{\partial \tau }\chi \left( \varepsilon \tau \right) -\varepsilon
\nabla f\left( \chi \left( \varepsilon \tau \right) \right) =0;
\end{equation*}

(2). When $l/\varepsilon \leq \left\vert \tau \right\vert \leq l/\varepsilon
+1$, say $\tau >0$ case we have $u_{\textrm{app}}^{\varepsilon }\left( \tau ,t\right)
=\kappa _{\varepsilon }^{0}\left( \tau \right) \chi (\varepsilon \tau
)+\left( 1-\kappa _{\varepsilon }^{0}\left( \tau \right) \right) u_{\pm
}^{\varepsilon }\left( \tau ,t\right) $, so%
\begin{eqnarray}
\left\vert \eta \left( \tau ,t\right) \right\vert &=&\left\vert \overline{%
\partial }_{\left( J_{0},\varepsilon f\right) }u_{\textrm{app}}^{\varepsilon
}\right\vert  \notag \\
&=&\left\vert \left( \kappa _{\varepsilon }^{0}\left( \tau \right) \right)
^{\prime }\left( \chi \left( \varepsilon \tau \right) -u_{\pm }^{\varepsilon
}\right) +\left( 1-\kappa _{\varepsilon }^{0}\left( \tau \right) \right)
\overline{\partial }_{\left( J_{0},\varepsilon f\right) }u_{\pm
}^{\varepsilon }\right\vert  \notag \\
&\leq &C\left( \left\vert \chi \left( \varepsilon \tau \right) -u_{\pm
}^{\varepsilon }\left( \tau ,t\right) \right\vert +\left\vert \varepsilon
\nabla f\left( u_{\pm }^{\varepsilon }\left( \tau ,t\right) \right)
\right\vert \right)  \notag \\
&\leq &\widetilde{C}\left( l\right) \varepsilon +C\varepsilon :=C\left(
l\right) \varepsilon ,  \label{d-bar-error}
\end{eqnarray}%
where in the first inequality we have used $\overline{\partial }_{\left(
J_{0},\varepsilon f\right) }u_{\pm }^{\varepsilon }=\varepsilon \nabla
f\left( u_{\pm }^{\varepsilon }\right) $ and in the second inequality used $%
\left( \ref{dist-xp}\right) $ and $\left( \ref{dist-up}\right) $, and $%
\widetilde{C}\left( l\right) =Cl^{-\frac{p-1}{\delta }}$. The weight $\beta
_{\delta ,\varepsilon }$ in this interval satisfies%
\begin{equation*}
\beta _{\delta ,\varepsilon }\left( \tau \right) =e^{2\pi \delta \left(
-\tau +\tau \left( \varepsilon \right) \right) }\leq e^{2\pi \delta \left(
-l/\varepsilon +\tau \left( \varepsilon \right) \right) }=e^{\left(
p-1\right) \ln \left( 1+l/\varepsilon \right) }\leq D\left( l\right)
\varepsilon ^{1-p},
\end{equation*}%
where the constant $D\left( l\right) \approx l^{p-1}$. Since $\left\vert
\tau \right\vert \geq l/\varepsilon $, we do not need to distinguish the $0$%
-mode and higher mode of $\eta $. So we have
\begin{equation*}
\left\Vert \eta \left( \tau \right) \right\Vert _{L_{\beta _{\delta
,\varepsilon }}^{p}\left( \left[ l/\varepsilon 1,l/\varepsilon +1\right]
\times S^{1}\right) }\leq \left( \int_{l/\varepsilon }^{l/\varepsilon
+1}\left( C\left( l\right) \varepsilon \right) ^{p}\cdot D\left( l\right)
\varepsilon ^{1-p}d\tau \right) ^{\frac{1}{p}},
\end{equation*}%
A straightforward estimate shows the following

\begin{lem}
If we write
\begin{equation*}
\left( \int_{l/\varepsilon }^{l/\varepsilon +1}\left( C\left( l\right)
\varepsilon \right) ^{p}\cdot D\left( l\right) \varepsilon ^{1-p}d\tau
\right) ^{\frac{1}{p}} =:E(l) \varepsilon^{\frac1p}
\end{equation*}
then we have
\begin{equation*}
E\left( l\right) \approx \left( C+l^{-\frac{p-1}{\delta }}\right) l^{\frac{%
p-1}{p}}.
\end{equation*}
\end{lem}

Similar result holds for $-l/\varepsilon -1\leq \tau \leq -l/\varepsilon $
case.

(3). When $\left\vert \tau \right\vert >l/\varepsilon $, say $\tau
>l/\varepsilon $, we recall
\begin{equation*}
u_{\textrm{app}}^{\varepsilon }\left( \tau ,t\right) =u_{+}\circ I_{\tau (\varepsilon
)}(\tau ,t)=u_{+}(\tau -\tau (\varepsilon ),t)
\end{equation*}%
and so satisfies
\begin{equation*}
\overline{\partial }_{\left( K_{\varepsilon },J_{\varepsilon }\right)
}u_{\textrm{app}}^{\varepsilon }\left( \tau ,t\right) =\overline{\partial }_{\left(
K_{\varepsilon }^{+},J_{\varepsilon }^{+}\right) }u_{+}\left( \tau -\tau
(\varepsilon ),t\right) =0.
\end{equation*}%
Similarly $\overline{\partial }_{\left( K_{\varepsilon },J_{\varepsilon
}\right) }u_{\textrm{app}}^{\varepsilon }\left( \tau ,t\right) =0$ for $\tau
<-l/\varepsilon .$

Combining the above three pieces, we have

\begin{prop}
\label{prop:error-estimate}
\begin{equation}
\left\Vert \overline{\partial }_{\left( J,\varepsilon f\right)
}u_{\textrm{app}}^{\varepsilon }\right\Vert _{_{L_{\beta _{\delta ,\varepsilon
}}^{p}\left( \mathbb{R}\times S^{1}\right) }}\leq E\left( l\right)
\varepsilon ^{\frac{1}{p}},  \label{eq:error}
\end{equation}%
where the constant $E\left( l\right) \approx \left( C+l^{-\frac{p-1}{\delta }%
}\right) l^{\frac{p-1}{p}}$.
\end{prop}

\begin{rem}
If we use $u_{\textrm{app}}^{\varepsilon }=\exp _{\chi \left( \varepsilon \tau
\right) }\bigskip \left[ \left( 1-\kappa _{\varepsilon }^{0}\left( \tau
\right) \right) \exp _{\chi (\varepsilon \tau )}^{-1}\left( u_{\pm
}^{\varepsilon }\left( \tau ,t\right) \right) \right] $ for $\tau \in
\lbrack l/\varepsilon ,l/\varepsilon +1]$, then $\left\vert \eta \right\vert
=$ $\left\vert \overline{\partial }_{(J_{0},\varepsilon
f)}u_{\textrm{app}}^{\varepsilon }\right\vert $ is controlled pointwise by $%
C\varepsilon $ as above case $\left( 2\right) $, plus the $C^{1}$ difference
between $\exp _{\chi \left( \varepsilon \tau \right) }\bigskip \left[ \left(
1-\kappa _{\varepsilon }^{0}\left( \tau \right) \right) \exp _{\chi
(\varepsilon \tau )}^{-1}\left( u_{\pm }^{\varepsilon }\left( \tau ,t\right)
\right) \right] $ and $\kappa _{\varepsilon }^{0}\left( \tau \right) \chi
(\varepsilon \tau )+\left( 1-\kappa _{\varepsilon }^{0}\left( \tau \right)
\right) u_{-}^{\varepsilon }\left( \tau ,t\right) $, which is also of order $%
C\varepsilon $. Therefore we get the same pointwise estimate
\begin{equation*}
\left\vert \eta \right\vert =\left\vert \overline{\partial }%
_{(J_{0},\varepsilon f)}u_{\textrm{app}}^{\varepsilon }\right\vert \leq C\varepsilon .
\end{equation*}%
Continuing the remaining steps in case $\left( 2\right) $ we get the same $%
L_{\beta _{\delta ,\varepsilon }}^{p}$ estimate
\begin{equation*}
\left\Vert \eta \right\Vert _{L_{\beta _{\delta ,\varepsilon }}^{p}\left(
\pm \left[ l/\varepsilon -1,l/\varepsilon \right] \times S^{1}\right) }\leq
E\left( l\right) \varepsilon ^{\frac{1}{p}}.
\end{equation*}%
The case $\left( 1\right) $ and $\left( 3\right) $ are the same as above.
Therefore we still get
\begin{equation*}
\left\Vert \overline{\partial }_{\left( J_{0},\varepsilon f\right)
}u_{\textrm{app}}^{\varepsilon }\right\Vert _{_{L_{\beta _{\delta ,\varepsilon
}}^{p}\left( \mathbb{R}\times S^{1}\right) }}\leq E\left( l\right)
\varepsilon ^{\frac{1}{p}}.
\end{equation*}
\end{rem}

\begin{rem}
\label{exp-weight-dbar-error}The constant $E\left( l\right) $ is bounded if $%
l$ is in a bounded interval, because we see from the above estimates%
\begin{equation*}
E\left( l\right) \approx \left( C+l^{-\frac{p-1}{\delta }}\right) l^{\frac{%
p-1}{p}}.
\end{equation*}%
So if we assume that $l_{0}\leq l\leq L$, then we get uniform $\overline{%
\partial }_{\left( K_{\varepsilon },J_{\varepsilon }\right) }$ error
estimate. When $l\rightarrow $ $\infty $, $E\left( l\right) \rightarrow
\infty $, so the above estimate in $\left( \ref{d-bar-error}\right) $ is too
coarse. But notice that when $l\rightarrow \infty $, $\nabla f\left( \chi
\left( \pm l\right) \right) $ has exponential decay $e^{-cl}$, so using this
in $\left( \ref{d-bar-error}\right) $ one can get even better $\overline{%
\partial }_{\left( K_{\varepsilon },J_{\varepsilon }\right) }$ error
estimate with $E\left( l\right) \approx \left( Ce^{-cl}+l^{-\frac{p-1}{%
\delta }}\right) l^{\frac{p-1}{p}}\rightarrow 0$ as $l\rightarrow \infty $.
Thus the error estimate is uniform for all $l\geq l_{0}$.
\end{rem}

\section{The combined right inverse}

\label{sec:comb-rightinverse}

To keep notation simple, in the following we denote by $D^\varepsilon_{%
\Phi}:=D_{u_{\textrm{app}}^{\varepsilon }}\overline{\partial }_{\left( K_{\varepsilon
},J_{\varepsilon }\right) }$ the linearization of the section
\begin{equation*}
{\overline{\partial }}_{(K_{\varepsilon },J_{\varepsilon })}:{\mathcal{B}}%
_{\text{\rm res}}^{\varepsilon }(z_{-},z_{+};B;l/\varepsilon)\rightarrow {\mathcal{L}}_{%
\text{\textrm{\text{\rm res}}}}^{\varepsilon }(z_{-},z_{+};B;l/\varepsilon).
\end{equation*}%
We recall from \eqref{eq:tildeDe}  the definition of the linearization
\begin{equation*}
D^\varepsilon_{\Phi;\text{\textrm{para}}}:T_{(u,l/\varepsilon )}{\mathcal{B}}%
_{\text{\rm res}}^{\varepsilon }(z_{-},z_{+};B)\rightarrow {\mathcal{L}}%
_{\text{\rm res};(u,l/\varepsilon )}^{\varepsilon }(z_{-},z_{+};B)
\end{equation*}%
of the parameterized map ${\overline{\partial }}_{(K_{\varepsilon },J_{\varepsilon })}$ for which
$\ell$ is allowed to vary.

In this section, we first construct an approximate right inverse, denoted by
$Q_{\text{\textrm{para}}}^{\text{\textrm{app}};\varepsilon }$, by gluing the right inverses $Q_{\pm }$ of $D{%
\overline{\partial }}_{(K_{\pm },J_{\pm })}(u_{\pm })$ and another operator
that takes care of the part of the gradient segment in the middle.
We recall
\begin{equation*}
T_{(u,l/\varepsilon )}{\mathcal{B}}_{\text{\rm res}}^{\varepsilon }(z_{-},z_{+};B)=T_{u}{%
\mathcal{B}}_{\text{\rm res}}^{\varepsilon }(z_{-},z_{+};B;l/\varepsilon )\oplus
T_{l/\varepsilon }{\mathbb{R}}=W_{\rho _{\varepsilon }}^{(1,p)}(u^{\ast
}TM)\oplus {\mathbb{R}}
\end{equation*}%
and
\begin{equation*}
{\mathcal{L}}_{\text{\rm res};(u,l/\varepsilon )}^{\varepsilon }(z_{-},z_{+};B)={\mathcal{%
L}}_{\text{\rm res};u}^{\varepsilon }(z_{-},z_{+};B;l/\varepsilon )=L_{\rho _{\varepsilon
}}^{p}(u^{\ast }TM)
\end{equation*}%
from the definition \eqref{eq:CBe}, \, \eqref{eq:CLe}. Therefore the image
of $Q_{\text{\textrm{para}}}^{\text{\textrm{app}};\varepsilon }$ at $%
(u,l/\varepsilon )$ is decomposed into
\begin{equation*}
Q_{\text{\textrm{para}}}^{\text{\textrm{app}};\varepsilon }(\eta )=(\xi
_{\varepsilon },\mu ),\,\xi _{\varepsilon }\in W_{\rho _{\varepsilon
}}^{(1,p)}(u^{\ast }TM),\,\mu \in {\mathbb{R}}.
\end{equation*}%
We will define $(\xi _{\varepsilon },\mu)$ by describing the image of the
operator $Q_{\text{\textrm{para}}}^{\text{\textrm{app}};\varepsilon }(\eta )$%
.

We introduce several cut-off functions:

\begin{enumerate}
\item $\kappa _{0}^{\varepsilon }=\kappa _{0}^{\varepsilon }\left( \tau
\right) $ is the characteristic function of the interval $\left[
-l/\varepsilon ,l/\varepsilon \right] \subset \mathbb{R}$,

\item $\varphi _{0}^{K}$ is the cut-off function defined by
\begin{equation*}
\varphi_{0}^{K}\left( \tau \right) =
\begin{cases}
1 \quad & \mbox{ for }\, \left\vert \tau \right\vert \leq K \\
0 \quad & \mbox{ for }\, \left\vert \tau \right\vert >K+1,%
\end{cases}%
\end{equation*}

\item $\varphi_{+}^{K}$ is the cut-off function
\begin{equation*}
\varphi_{+}^{K}\left( \tau \right) =
\begin{cases}
1\quad & \mbox{ for }\, \tau \leq K \\
0 \quad & \mbox{ for }\, \tau >K+1,%
\end{cases}%
\end{equation*}
and $\varphi_{-}^{K}$ is the function defined by $\varphi_-^K\left( \tau
\right) =1-\varphi _{+}^{K}\left( \tau \right) $.
\end{enumerate}

Now let $\eta \in {\mathcal{L}}_{\text{\rm res}}^{\varepsilon }(z_{-},z_{+})$. We split
$\eta $ into 3 pieces of ${\mathbb{R}}$ according to the division of the
expression of the approximate solution \eqref{eq:uappe}:
\begin{equation}
\eta |_{(-\infty ,l/\varepsilon ]},\,\eta _{\lbrack l/\varepsilon ,\infty
)},\,\eta _{\lbrack -l/\varepsilon ,l/\varepsilon ]}.  \label{eq:divided-eta}
\end{equation}%
Multiplying the characteristic functions
\begin{equation*}
\kappa _{-}^{\varepsilon },\,\kappa _{+}^{\varepsilon },\,\kappa
_{0}^{\varepsilon }
\end{equation*}%
of the corresponding intervals, we regard each of them defined on the whole
cylinder ${\mathbb{R}}\times S^{1}$.

\begin{rem}
The smooth cut-off functions we used before in $\left( \ref{eq:betaR}\right)
,\left( \ref{eq:beta0}\right) $ are $\kappa _{\varepsilon }^{-},\,\kappa
_{\varepsilon }^{+},\kappa _{\varepsilon }^{0}$, \ whose notations are
similar to the characteristic functions here but lower and upper indices are
switched.
\end{rem}

Consider the translations of the first and the third pieces
\begin{equation*}
\eta_\pm^\varepsilon(\tau,t) := (\kappa_{\pm}^{\varepsilon }\eta) \circ
I_{\pm \tau(\varepsilon)}(\tau,t) = (\kappa_{\pm}^{\varepsilon }\eta)(\tau -
\pm \tau(\varepsilon),t)
\end{equation*}
whose supports become
\begin{equation*}
\left(-\infty,\frac{p-1}{\delta}S(\varepsilon)\right) \times S^1, \, \left(-%
\frac{p-1}{\delta}S(\varepsilon),\infty\right) \times S^1
\end{equation*}
respectively. Then we define
\begin{equation}  \label{eq:xiepm}
\xi^\varepsilon_\pm = Q_\pm(\eta_\pm^\varepsilon)\circ
I_{\pm(-\tau(\varepsilon))}.
\end{equation}

Now we consider the middle piece $\kappa _{0}^{\varepsilon }\eta $ which is
supported on $[-l/\varepsilon ,l/\varepsilon ]\times S^{1}$. To describe
this piece precisely, we need some careful examination how the Banach
manifold \eqref{eq:CBe} and the tangent vectors thereof at the approximate
solution $u_{\varepsilon }$ near $(u_{-},\chi ,u_{+},l)$ are made of, and
how the operator ${\overline{\partial }}_{(K_{\varepsilon },J_{\varepsilon
})}$ acts on $T_{u_{\varepsilon }}{\mathcal{B}}_{\text{\rm res}}^{\varepsilon
}(z_{-},z_{+};B)$.

%The \textquotedblleft thimble-flow-thimble" transversality guarantees the matching
%condition of the \textquotedblleft 0 mode" $\xi_{0}$ with the outside variations $%
%\xi _{\pm }$ of $u_{\pm }$ at the joint points.
%
%For the higher mode variations,
%similar construction and computation as in [OZ1]
%yields the desired approximate right inverse on the pre-glued solution from
%\textquotedblleft thimble-flow-thimble" configuration,
%because we have used exponential weight on the curves $u_{\pm }$ and power
%weight on $u_{\chi }$.
Recall $\chi$ satisfies $\dot \chi + \func{grad} f(\chi) = 0$. Since we
assume that the pair $(f,g)$ is Morse-Smale, the linearization
\begin{equation*}
D_\chi = \nabla_\tau + \nabla \func{grad}f(\chi)
\end{equation*}
is invertible as mentioned before. We denote by $Q_\chi$ its right inverse.

We denote the renormalized $\chi :[-l,l]\rightarrow M$ by
\begin{equation*}
\chi _{\varepsilon }:[-l/\varepsilon ,l/\varepsilon ]\rightarrow M,\quad
\chi _{\varepsilon }(\tau ):=\chi (\varepsilon \tau )
\end{equation*}%
which satisfies $\dot{\chi}_{\varepsilon }+\varepsilon \func{grad}f(\chi )=0$.

We apply the right inverse $Q_{\chi _{\varepsilon }}^{\text{\textrm{para}}}
$ (allowing $l$ to vary) of $D_{\chi _{\varepsilon }}^{\text{\textrm{para}}}$
to $\kappa _{0}^{\varepsilon }\eta $ and write it as
\begin{equation*}
Q_{\chi _{\varepsilon }}^{\text{\textrm{para}}}(\kappa _{0}^{\varepsilon
}\eta )=(\xi _{\chi _{\varepsilon }},\mu ).
\end{equation*}%
We denote by $\pi_l:[-l/\varepsilon ,l/\varepsilon ] \times S^1 \to [-l/\varepsilon ,l/\varepsilon ]$ the natural projection.
Then it follows from \eqref{eq:tildeDe} that the operator $D_{\Phi ;\text{\textrm{para}}}^{\varepsilon }$
induces an operator
\begin{equation*}
D_{\Phi ;\text{\textrm{para}}}^{\varepsilon,\chi }:W^{1,p}((\chi _{\varepsilon
} \circ \pi_l)^* TM)\oplus {\mathbb{R}}\rightarrow L^{p}((\chi
_{\varepsilon } \circ \pi_l)^* TM)
\end{equation*}%
which has the form
\begin{equation*}
D_{\Phi ;\text{\textrm{para}}}^{\varepsilon,\chi }(\xi ,\mu )=D_{\chi
_{\varepsilon }}(\xi )+\frac{\mu }{l}\nabla \varepsilon f(\chi _{\varepsilon
})
\end{equation*}%
with $D_{\chi _{\varepsilon }}=\frac{\partial }{\partial \tau }+J_{0}\frac{%
\partial }{\partial t}+\nabla \,$grad$\left( \varepsilon f\right) $. First
we define
\begin{equation*}
\xi _{\chi _{\varepsilon }}=Q_{\chi _{\varepsilon }}(\eta _{\chi
_{\varepsilon }})
\end{equation*}%
where $Q_{\chi _{\varepsilon }}$ is the right inverse of $D_{\chi
_{\varepsilon }}$ constructed in section \ref{sec:Qinmiddle}. Then we
determine $\mu $ by solving the $0$\emph{-mode matching condition}
\begin{equation}
\left( \xi _{\chi _{\varepsilon }}\right) _{0}\left( \pm l/\varepsilon
\right) +\frac{\varepsilon \mu }{l}\nabla f\left( p_{\pm }\right) =\xi _{\pm
}\left( o_{\pm }\right):  \label{eq:defmu}
\end{equation}%
the constant $\mu $ is always solvable by thimble-flow-thimble transversality in Section %
\ref{sec:tft-transversality}. Now we are ready to write down the formula for
$(\xi _{\varepsilon },\mu )=Q_{\text{\textrm{para}}}^{\text{\textrm{app}}%
;\varepsilon }(\eta )$. Here we define
\begin{eqnarray}
\xi _{\varepsilon } &=&%
\begin{cases}
\xi _{+}^{\varepsilon },\hfill  & l/\varepsilon +T\left( \varepsilon \right)
\leq \tau , \\
\func{Pal}_{\chi ,\varepsilon }\left[ \left( \xi _{\chi _{\varepsilon
}}\right) _{0}+\phi _{0}^{l/\varepsilon +T\left( \varepsilon \right) }\left(
\tau \right) \left( \xi _{\chi _{\varepsilon }}-\left( \xi _{\chi
_{\varepsilon }}\right) _{0}\right) \right]  &  \\
\qquad +\func{Pal}_{+,\varepsilon }\left[ \phi _{+}^{\left( l/\varepsilon
-T\left( \varepsilon \right) \right) }\left( \tau \right) \left( \xi
_{+}^{\varepsilon }-v_{+}\right) \right] ,\hfill  & l/\varepsilon \leq \tau
\leq l/\varepsilon +T\left( \varepsilon \right) , \\
\xi _{\chi _{\varepsilon }},\hfill  & -l/\varepsilon \leq \tau \leq
l/\varepsilon , \\
\func{Pal}_{\chi ,\varepsilon }\left[ \left( \xi _{\chi _{\varepsilon
}}\right) _{0}+\phi _{0}^{l/\varepsilon +T\left( \varepsilon \right) }\left(
\tau \right) \left( \xi _{\chi _{\varepsilon }}-\left( \xi _{\chi
_{\varepsilon }}\right) _{0}\right) \right]  &  \\
\qquad +\func{Pal}_{-,\varepsilon }\left[ \phi _{-}^{\left( l/\varepsilon
-T\left( \varepsilon \right) \right) }\left( \tau \right) \left( \xi
_{-}^{\varepsilon }-v_{-}\right) \right] ,\hfill  & \ -l/\varepsilon
-T\left( \varepsilon \right) \leq \tau \leq -l/\varepsilon , \\
\xi _{-}^{\varepsilon },\hfill  & \tau \leq -l/\varepsilon -T\left(
\varepsilon \right) ,%
\end{cases}
\notag  \label{eq:xi-e1} \\
&{}&
\end{eqnarray}%
where
\begin{eqnarray*}
\xi _{\chi _{\varepsilon }} &=&Q_{\chi _{\varepsilon }}\circ \left( \kappa
_{0}^{\varepsilon }\eta \right) , \\
~\left( \xi _{\chi _{\varepsilon }}\right) _{0} &=&\int_{S^{1}}\xi _{\chi
_{\varepsilon }}dt, \\
\xi _{\pm }^{\varepsilon } &=&Q_{\pm }(\eta _{\pm }^{\varepsilon })\circ
I_{\pm (-\tau (\varepsilon ))}
\end{eqnarray*}%
and $\mu $ is determined by \eqref{eq:defmu}.

Here we recall from \eqref{Vpm} that $v_{\pm }$ is defined as the \emph{%
cut-off vector field} of the constant vector field extending the vector $\xi
_{\pm }\left( o_{\pm }\right) $,%
\begin{equation}
v_{\pm }:=\varphi _{\pm }\left( \tau ,t\right) \func{Pal}_{\pm }\left( \tau
,t\right) \left( \xi _{\pm }\left( o_{\pm }\right) \right) ,
\label{v-extension}
\end{equation}%
$\varphi _{\pm }\left( \tau ,t\right) $ is a cut-off function, $\varphi
_{\pm }\left( \tau ,t\right) =1$ near $o_{\pm }$ and $\varphi _{\pm }\left(
\tau ,t\right) =0$ outside the cylindrical neighborhood of $o_{\pm }$, $%
\func{Pal}_{\pm }\left( \tau ,t\right) $ is the parallel transport along the
shortest geodesics from $u_{\pm }\left( o_{\pm }\right) $ to $u_{\pm }\left(
\tau ,t\right) $, and the $\func{Pal}_{\chi ,\varepsilon }$ and $\func{Pal}%
_{\pm ,\varepsilon }$ are parallel transports from $\chi $ and $u_{\pm }$ to
the corresponding points on $u_{\textrm{app}}^{\varepsilon }$ along the shortest
geodesics respectively.

We apply the interpolation to the regions
\begin{equation*}
\pm \lbrack l/\varepsilon -T\left( \varepsilon \right) ,l/\varepsilon
-T\left( \varepsilon \right) +1]\text{ and }\pm \left[ l/\varepsilon
+T\left( \varepsilon \right) ,l/\varepsilon +T\left( \varepsilon \right) +1%
\right] ,
\end{equation*}%
which avoids the peak $\tau =\pm l/\varepsilon $ of the weighting function $%
\beta _{\delta ,\varepsilon }$. Here we choose $T\left( \varepsilon \right)
>0$ so that as $\varepsilon \rightarrow 0$, it behaves as
\begin{equation*}
T\left( \varepsilon \right) \rightarrow \infty,\, \, \varepsilon T\left(
\varepsilon \right) \rightarrow 0,\, T\left( \varepsilon \right) <\frac{p-1}{%
\delta }S\left( \varepsilon \right)
\end{equation*}%
For example, we can take and fix
\begin{equation}  \label{eq:Te}
T\left( \varepsilon \right) =\frac{1}{3}\frac{p-1}{\delta }S(\varepsilon )
\end{equation}
henceforth.

We note that in this construction we have $\xi _{\chi _{\varepsilon }}$
defined for all $\tau \in \mathbb{R}$, solving $D_{\varepsilon }\xi _{\chi
_{\varepsilon }}=\kappa _{0}^{\varepsilon }\eta $ which is equivalent to the
first order linear PDE
\begin{equation*}
\frac{\partial \xi }{\partial \tau }+J_{0}\frac{\partial \xi }{\partial t}%
+\varepsilon \nabla _{\xi }\,\text{grad}\!f=\kappa _{0}^{\varepsilon }\eta
\end{equation*}%
on $\mathbb{R\times }S^{1}$, not just on $\left[ -l/\varepsilon
,l/\varepsilon \right] $. Since $\kappa _{0}^{\varepsilon }\eta \in L^{p}$, $%
\xi _{\chi _{\varepsilon }}$ lies in $W^{1,p}$ and in particular is
continuous on $\mathbb{R\times }S^{1}$. Therefore, considering the
evaluation of $\left( \xi _{\chi _{\varepsilon }}\right) _{0}$ at $\left(
\pm \left( l/\varepsilon +T\left( \varepsilon \right) \right) \right) $
makes sense. Thinking $\xi _{\chi _{\varepsilon }}$this way is to avoid its
cut-off too close to $\tau =\pm l/\varepsilon $, the peak of the weighting
function. Later we will also show the $W^{1,p}$ norm of $\xi _{\chi
_{\varepsilon }}$on $\left[ -l/\varepsilon ,l/\varepsilon \right] \mathbb{%
\times }S^{1}$ controls its $W^{1,p}$ norm on the whole $\mathbb{R\times }%
S^{1}$.

%The construction of $\left( Q^{\text{\rm app};\e}_{\text{\rm para}}\right) _{actual}$ $\eta
%=\left( \xi _{\varepsilon },\mu \right) $ can be summarized by the following
%diagram%
%\begin{equation*}
%\eta \text{ }%
%\begin{tabular}{lllll}
%$\overset{\kappa _{-}^{\varepsilon }}{\nearrow }\eta _{-}^{\varepsilon }$ & $%
%\rightarrow \eta _{-}$ & $\underrightarrow{Q_{-}}$ & $\xi _{-}$ & $%
%\rightarrow \widetilde{\xi _{-}^{\varepsilon }}$ $\ \ \ \ \ \ \ \ \overset{%
%\phi _{-}^{-l/\varepsilon +T\left( \varepsilon \right) }}{\searrow }$ \\
%$\underrightarrow{\kappa ^{\varepsilon }}$ $\eta _{\chi _{\varepsilon }}$ & $%
%\nearrow \left( \eta _{\chi _{\varepsilon }}\right) _{0}$ & $%
%\underrightarrow{\left( Q_{\chi _{\varepsilon }}\right) _{0}}$ & $\left(
%\left( \xi _{\chi _{\varepsilon }}\right) _{0},\mu \right) $ & $\rightarrow
%\left( \xi _{\chi _{\varepsilon }}\right) _{0}$ intropolation \\
%splitting & $\searrow \widetilde{\eta _{\chi _{\varepsilon }}}$ & $%
%\underrightarrow{\widetilde{Q_{\chi _{\varepsilon }}}}$ & $\widetilde{\xi
%_{\chi _{\varepsilon }}}$ & $\rightarrow $ $\ \widetilde{\xi _{\chi
%_{\varepsilon }}}$ $\ \ \ \ \ \ \overrightarrow{\phi _{0}^{l/\varepsilon
%+T\left( \varepsilon \right) }}\ \ \ \ \ \ $ \\
%$\overset{\kappa _{+}^{\varepsilon }}{\searrow }\eta _{+}^{\varepsilon }$ & $%
%\rightarrow \eta _{+}$ & $\underrightarrow{Q_{+}}$ & $\xi _{+}$ & $%
%\rightarrow \widetilde{\xi _{+}^{\varepsilon }}$ $\ \ \ \ \ \ \ \ \ \overset{%
%\phi _{+}^{l/\varepsilon -T\left( \varepsilon \right) }}{\nearrow }$%
%\end{tabular}%
%\left( \xi _{\varepsilon },\mu \right)
%\end{equation*}

\begin{rem}
In the construction of $\left( Q_{\text{\textrm{para}}}^{\varepsilon
}\right) _{actual}$, we can use the connection from the constant Euclidean
metric $g_{p_{\pm }}$ on the Weinstein neighborhood around $p_{\pm }$ to
identify different tangent spaces such that the vector sum $\pm $ makes
sense without parallel transport since the interpolation happens in a $%
C\varepsilon $ radius ball around $p_{\pm }$. We set this approximate right
inverse to be $Q_{\text{\textrm{para}}}^{\text{\textrm{app}};\varepsilon }$,
which is simpler than \eqref{eq:xi-e1} in its exposition and estimates:%
\begin{eqnarray*}
\xi _{\varepsilon } =
\begin{cases}
\xi _{+}^{\varepsilon }, \hfill & l/\varepsilon +T\left( \varepsilon \right)
\leq \tau , \\
\left( \xi _{\chi _{\varepsilon }}\right) _{0}+\phi _{0}^{l/\varepsilon
+T\left( \varepsilon \right) }\left( \tau \right) \left( \xi _{\chi
_{\varepsilon }}-\left( \xi _{\chi _{\varepsilon }}\right) _{0}\right) &  \\
\qquad +\phi _{+}^{\left( l/\varepsilon -T\left( \varepsilon \right) \right)
}\left( \tau \right) \left( \xi _{+}^{\varepsilon }-v_{+}\right), \hfill &
l/\varepsilon \leq \tau \leq l/\varepsilon +T\left( \varepsilon \right) , \\
\xi _{\chi _{\varepsilon }}, \hfill & -l/\varepsilon \leq \tau \leq
l/\varepsilon , \\
\left( \xi _{\chi _{\varepsilon }}\right) _{0}+\phi _{0}^{l/\varepsilon
+T\left( \varepsilon \right) }\left( \tau \right) \left( \xi _{\chi
_{\varepsilon }}-\left( \xi _{\chi _{\varepsilon }}\right) _{0}\right) &  \\
\qquad + \phi _{-}^{\left( l/\varepsilon -T\left( \varepsilon \right)
\right) }\left( \tau \right) \left( \xi _{-}^{\varepsilon }-v_{-}\right) ,
\hfill & \ -l/\varepsilon -T\left( \varepsilon \right) \leq \tau \leq
-l/\varepsilon , \\
\xi _{-}^{\varepsilon }, \hfill & \tau \leq -l/\varepsilon -T\left(
\varepsilon \right) ,%
\end{cases}%
\end{eqnarray*}

The interpolation in $\left( Q_{\text{\textrm{para}}}^{\text{\textrm{app}}%
;\varepsilon }\right) _{actual}$ uses the parallel transport from the
nonconstant metric $g$ inside the $C\varepsilon $ radius ball, but the
difference is a smooth tensor of order $C\varepsilon $, first in $C^{1}$
pointwise then in operator norm, namely%
\begin{equation*}
\left( Q_{\text{\textrm{para}}}^{\text{\textrm{app}};\varepsilon }\right)
_{actual}=Q_{\text{\textrm{para}}}^{\text{\textrm{app}};\varepsilon
}+H_{\varepsilon }
\end{equation*}%
with the operator norm $\left\Vert H_{\varepsilon }\right\Vert \leq
C\varepsilon $. This will not effect the approximate inverse. Therefore it
suffices to estimate the above (simpler) $Q_{\text{\textrm{para}}}^{\text{%
\textrm{app}};\varepsilon }$ in the next section.
\end{rem}

\section{Estimate of the combined right inverse}

\begin{prop}
\label{prop:uniform-inverse-bound} $Q_{\text{\textrm{para}}}^{\text{\textrm{%
app}};\varepsilon }$ has a uniform bound independent on $\varepsilon $. More
precisely, there exists a uniform constant $C$ such that for all $\eta $,
\begin{equation*}
\left\Vert Q_{\text{\textrm{para}}}^{\text{\textrm{app}};\varepsilon }\eta
\right\Vert _{\varepsilon }\leq C\left( \left\Vert \xi _{\chi _{\varepsilon
}}\right\Vert _{W_{\rho _{\varepsilon }}^{1,p}\left( \left[ -l/\varepsilon
,l/\varepsilon \right] \times S^{1}\right) }+\left\Vert \xi _{\pm
}\right\Vert _{W_{\alpha }^{1,p}\left( \Sigma _{\pm }\right) }\right) \leq
C\left\Vert \eta \right\Vert _{\varepsilon }.
\end{equation*}
\end{prop}

\begin{proof}
The proof is by splitting the Banach norm on $u_{\textrm{app}}^{\varepsilon }$ into
those associated to $u_{\chi _{\varepsilon }}$ and $u_{\pm }$. We also need
to estimates on

\begin{enumerate}
\item uniform convergence of $\xi _{\pm }\left( \tau \right) $ and $\xi
_{\chi _{\varepsilon }}\left( \tau \right) $ into $v_{\pm }$, because our
Banach norms involve first taking out the Morse-Bott variation,

\item and the matching constant $\mu$
\end{enumerate}

in terms of the norm $\Vert \eta \Vert _{\varepsilon }$. In the
interpolation region, $Q_{\text{\textrm{para}}}^{\text{\textrm{app}}%
;\varepsilon }\eta =\left( \xi _{\varepsilon },\mu \right) \,\ $ is given by
\begin{equation*}
\xi _{\varepsilon }=\left( \xi _{\chi _{\varepsilon }}\right) _{0}+\phi
_{0}^{l/\varepsilon +T\left( \varepsilon \right) }\left( \tau \right) \left(
\xi _{\chi _{\varepsilon }}-\left( \xi _{\chi _{\varepsilon }}\right)
_{0}\right) +\phi _{+}^{\left( l/\varepsilon -T\left( \varepsilon \right)
\right) }\left( \tau \right) \left( \xi _{+}^{\varepsilon }-v_{+}\right) ,
\end{equation*}%
when $\tau \geq 0$ and $\mu $ by solving the matching condition
\begin{equation}
\left( \xi _{\chi _{\varepsilon }}\right) _{0}\left( l/\varepsilon \right)
=v_{+}-\varepsilon \frac{\mu }{l}\nabla f\left( p_{+}\right) ,
\label{f-matching}
\end{equation}%
where $v_{+}=\xi _{+}\left( o_{+}\right) $, and $\mu $ is always solvable by
thimble-flow-thimble transversality in Section \ref{sec:tft-transversality}.

We also recall the definition of the Banach norm
\begin{eqnarray*}
\left\Vert \left( \xi ,\mu \right) \right\Vert _{\varepsilon } & = &
\left\Vert \xi \right\Vert _{\varepsilon }+\left\vert \mu \right\vert \\
& = & \left\Vert \widetilde{\xi }\right\Vert _{W_{\beta _{\delta
,\varepsilon }}^{1,p}\left( \mathbb{R\times }S^{1}\right) }+\left\Vert \xi
_{0}\left( \tau \right) \right\Vert _{W_{\varepsilon }^{1,p}\left( \left[
-l/\varepsilon ,l/\varepsilon \right] \right) }+\left\vert \xi _{0}\left(
\pm l/\varepsilon \right) \right\vert +\left\vert \mu \right\vert.
\end{eqnarray*}

With these preparation, we are now ready to carry out the estimate of the
norm $\Vert Q_{\text{\textrm{para}}}^{\text{\textrm{app}};\varepsilon }(\eta
)\Vert _{\varepsilon }$. For the $\mu$-component, since $\left( \left( \xi
_{\chi _{\varepsilon }}\right) _{0},\mu \right) =\left( Q_{\chi
_{\varepsilon }}^{\text{\textrm{para}}}\right) _{0}\circ \kappa
_{0}^{\varepsilon }\eta $ and $\left( Q_{\chi _{\varepsilon }}^{\text{%
\textrm{para}}}\right) _{0}$ has uniform bound (one can use $\left( \ref%
{e-eta-relation}\right) $ to see the $\left( Q_{\chi _{\varepsilon }}^{\text{%
\textrm{para}}}\right) _{0}$ has the same operator bound as $\left( Q_{\chi
}^{\text{\textrm{para}}}\right) _{0}$ for all $\varepsilon $), we have%
\begin{equation*}
\left\vert \mu \right\vert \leq C\left\Vert \left( \kappa _{0}^{\varepsilon
}\eta \right) _{0}\right\Vert _{L_{\varepsilon }^{p}\left( \left[
-l/\varepsilon ,l/\varepsilon \right] \right) }\leq C\left\Vert \eta
\right\Vert _{\varepsilon }.
\end{equation*}

For the $\left\Vert \xi _{\varepsilon }\right\Vert _{\varepsilon }$
component, we consider three regions separately:
\begin{enumerate}
\item [{(i)}]  on $[0,l/\varepsilon
]\times S^{1}$,
\item [{(ii)}] on $[l/\varepsilon ,\tau (\varepsilon )]\times
S^{1}=[l/\varepsilon ,l/\varepsilon +\frac{p-1}{\delta }S(\varepsilon
)]\times S^{1}$,
\item [{(iii)}] on $\R \setminus [-\tau(\varepsilon),\tau(\varepsilon)]$.
\end{enumerate}

{\bf (i) On $[0,l/\varepsilon ]\times S^{1}$:} Since $\phi _{0}^{l/\varepsilon
+T\left( \varepsilon \right) }\left( \tau \right) \equiv 1$, we have
\begin{eqnarray*}
\xi _{\varepsilon } &=&\left( \xi _{\chi _{\varepsilon }}\right) _{0}+\left(
\xi _{\chi _{\varepsilon }}-\left( \xi _{\chi _{\varepsilon }}\right)
_{0}\right) +\phi _{+}^{l/\varepsilon -T\left( \varepsilon \right) }\left(
\tau \right) \left( \xi _{+}^{\varepsilon }-v_{+}\right) \\
&=&\xi _{\chi _{\varepsilon }}+\phi _{+}^{\left( l/\varepsilon -T\left(
\varepsilon \right) \right) }\left( \tau \right) \left( \xi
_{+}^{\varepsilon }-v_{+}\right) .
\end{eqnarray*}%
Therefore for the $0$-mode $\left( \xi _{\varepsilon }\right) _{0}$ we have
\begin{eqnarray*}
\left\Vert \left( \xi _{\varepsilon }\right) _{0}\right\Vert
_{W_{\varepsilon }^{1,p}\left[ 0,l/\varepsilon \right] } &=&\left\Vert
\left( \xi _{\chi _{\varepsilon }}\right) _{0}+\phi _{+}^{l/\varepsilon
-T\left( \varepsilon \right) }\left( \tau \right) \left( \xi
_{+}^{\varepsilon }-v_{+}\right) _{0}\right\Vert _{W_{\varepsilon
}^{1,p}\left( \left[ 0,l/\varepsilon \right] \right) } \\
&\leq &\left\Vert \left( \xi _{\chi _{\varepsilon }}\right) _{0}\right\Vert
_{W_{\varepsilon }^{1,p}\left( \left[ 0,l/\varepsilon \right] \right)
}+\left\Vert \phi _{+}^{l/\varepsilon -T\left( \varepsilon \right) }\left(
\tau \right) \left( \xi _{+}^{\varepsilon }-v_{+}\right) _{0}\right\Vert
_{W_{\varepsilon }^{1,p}\left( \left[ 0,l/\varepsilon \right] \right) } \\
&\leq &\left\Vert \xi _{\chi _{\varepsilon }}\right\Vert _{W_{\rho
_{\varepsilon }}^{1,p}\left( \mathbb{R\times }S^{1}\right) }+C\left\Vert
\left( \xi _{+}^{\varepsilon }-v_{+}\right) _{0}\cdot \varepsilon ^{\frac{1-p%
}{p}}\right\Vert _{W^{1,p}\left( \left[ 0,l/\varepsilon \right] \times
S^{1}\right) } \\
&\leq &\left\Vert \xi _{\chi _{\varepsilon }}\right\Vert _{W_{\rho
_{\varepsilon }}^{1,p}\left( \mathbb{R\times }S^{1}\right) }+C\left\Vert \xi
_{+}\right\Vert _{W_{\alpha }^{1,p}\left( \left[ 0,l/\varepsilon \right]
\times S^{1}\right) },
\end{eqnarray*}%
where the third inequality is because $\left\Vert \cdot \right\Vert
_{W_{\varepsilon }^{1,p}}$ is a component of $\left\Vert \cdot \right\Vert
_{W_{\rho _{\varepsilon }}^{1,p}}$, and the last inequality holds because
for $\tau \in $ $\left[ 0,l/\varepsilon \right] \,$, the exponential weight
\begin{equation*}
e^{2\pi \delta (-\tau +\tau (\varepsilon ))}=e^{2\pi \delta \left( -\tau
+\tau (\varepsilon )\right) }\geq e^{2\pi \delta \left( \frac{p-1}{\delta }%
S\left( \varepsilon \right) \right) }=\left( 1+l/\varepsilon \right)
^{p-1}\geq C\varepsilon ^{1-p},
\end{equation*}%
then
\begin{eqnarray*}
&{}&\left\Vert \left( \xi _{+}^{\varepsilon }-v_{+}\right) _{0}\cdot
\varepsilon ^{\frac{1-p}{p}}\right\Vert _{W^{1,p}\left( \left[
0,l/\varepsilon \right] \times S^{1}\right) } \\
&\leq &C\left\Vert \left( \xi _{+}^{\varepsilon }-v_{+}\right) _{0}\cdot e^{%
\frac{2\pi \delta }{p}\left( -\tau +\tau (\varepsilon )\right) }\right\Vert
_{W^{1,p}\left( \left[ 0,l/\varepsilon \right] \times S^{1}\right) } \\
&=&C\left\Vert \left( \xi _{+}-v_{+}\right) _{0}\cdot e^{\frac{2\pi \delta }{%
p}\left\vert \tau \right\vert }\right\Vert _{W^{1,p}\left( \left[ -\tau
(\varepsilon ),-\frac{p-1}{\delta }S(\varepsilon )\right] \times
S^{1}\right) } \\
&\leq &C\left\Vert \xi _{+}\right\Vert _{W_{\alpha }^{1,p}\left( \Sigma
_{+}\right) }.
\end{eqnarray*}%
For the higher mode $\widetilde{\xi _{\varepsilon }}$, we have
\begin{eqnarray*}
\Vert \widetilde{\xi _{\varepsilon }}\Vert _{W_{\rho _{\varepsilon
}}^{1,p}([0,l/\varepsilon ]\times S^{1})} &=&\left\Vert \widetilde{\xi
_{\chi _{\varepsilon }}}+\phi _{+}^{l/\varepsilon -T\left( \varepsilon
\right) }\left( \tau \right) \left( \widetilde{\xi _{+}^{\varepsilon }-v_{+}}%
\right) \right\Vert _{W_{\rho _{\varepsilon }}^{1,p}([0,l/\varepsilon
]\times S^{1})} \\
&\leq &\left\Vert \widetilde{\xi _{\chi _{\varepsilon }}}\right\Vert
_{W_{\rho _{\varepsilon }}^{1,p}([0,l/\varepsilon ]\times
S^{1})}+C\left\Vert \widetilde{\left( \xi _{+}^{\varepsilon }-v_{+}\right) }%
\right\Vert _{W_{\rho _{\varepsilon }}^{1,p}([0,l/\varepsilon ]\times S^{1})}
\\
&\leq &C\left( \left\Vert \xi _{\chi _{\varepsilon }}\right\Vert _{W_{\rho
_{\varepsilon }}^{1,p}\left( \mathbb{R\times }S^{1}\right) }+\left\Vert \xi
_{+}\right\Vert _{W_{\alpha }^{1,p}\left( \Sigma _{+}\right) }\right) .
\end{eqnarray*}%
The last inequality holds because the exponential weight $e^{2\pi \delta
\left( -\tau +\tau (\varepsilon )\right) }$ of $W_{\alpha }^{1,p}$ is bigger
than the power weight $\varepsilon ^{1-p+\delta }\left( 1+\left\vert \tau
\right\vert \right) ^{\delta }$ of $W_{\rho _{\varepsilon }}^{1,p}$ on $%
[0,l/\varepsilon ]\times S^{1}$, by
\begin{equation*}
e^{2\pi \delta \left( -\tau +\tau (\varepsilon )\right) }\geq \left(
1+l/\varepsilon \right) ^{p-1}=C\varepsilon ^{1-p+\delta }\left(
1+l/\varepsilon \right) ^{\delta }\geq C\varepsilon ^{1-p+\delta }\left(
1+\left\vert \tau \right\vert \right) ^{\delta }.
\end{equation*}%
\ In the last inequality we also used that the projection operators
\begin{equation*}
P_{0}:\xi \rightarrow \xi _{0}\text{ and }\widetilde{P}:\xi \rightarrow
\widetilde{\xi }
\end{equation*}%
from $W_{\delta }^{1,p}$ to $W_{\delta }^{1,p}$ are uniformly bounded
operators by H\"{o}lder inequality on $S^{1}$
\begin{eqnarray*}
\int_{I}\left( \left\vert \xi _{0}\right\vert ^{p}+\left\vert \nabla _{\tau
}\xi _{0}\right\vert ^{p}\right) e^{2\pi \delta \left\vert \tau \right\vert
}d\tau &=&\int_{I}\left( \left\vert \int_{S^{1}}\xi dt\right\vert
^{p}+\left\vert \int_{S^{1}}\nabla _{\tau }\xi dt\right\vert ^{p}\right)
e^{2\pi \delta \left\vert \tau \right\vert }d\tau \\
&\leq &\int_{I}\left( \int_{S^{1}}\left\vert \xi \right\vert
^{p}dt+\int_{S^{1}}\left\vert \nabla \xi \right\vert ^{p}dt\right) e^{2\pi
\delta \left\vert \tau \right\vert }d\tau
\end{eqnarray*}%
for any interval $I$.
\newline
{\bf (ii) On $[l/\varepsilon ,\tau (\varepsilon )]\times S^{1}$:} On this region the
contribution to $\left\Vert \xi _{\varepsilon }\right\Vert _{\varepsilon }$
is
\begin{equation*}
\left\vert \left( \xi _{\varepsilon }\right) _{0}\left( l/\varepsilon
\right) \right\vert +\left\Vert \xi _{\varepsilon }\left( \tau ,t\right)
-\left( \xi _{\varepsilon }\right) _{0}\left( l/\varepsilon \right)
\right\Vert _{W_{\beta _{\delta ,\varepsilon }}^{1,p}\left( [l/\varepsilon
,\tau (\varepsilon )]\times S^{1}\right) .}
\end{equation*}%
From the matching condition \eqref{eq:matching}, we have
\begin{equation*}
\left( \xi _{\varepsilon }\right) _{0}\left( l/\varepsilon \right) =\left(
\xi _{\chi _{\varepsilon }}\right) _{0}\left( l/\varepsilon \right) +\left(
\xi _{+}\right) _{0}\left( l/\varepsilon \right) -v_{+}=\left( \xi
_{+}\right) _{0}\left( l/\varepsilon \right) -\varepsilon \frac{\mu }{l}%
\nabla f\left( p_{+}\right) .
\end{equation*}%
Therefore for the term $\left( \xi _{\varepsilon }\right) _{0}\left(
l/\varepsilon \right) $,
\begin{eqnarray*}
\left\vert \left( \xi _{\varepsilon }\right) _{0}\left( l/\varepsilon
\right) \right\vert  &\leq &\left\vert \left( \xi _{+}\right) _{0}\left(
l/\varepsilon \right) \right\vert +\left\vert \frac{\varepsilon }{l}\nabla
f\left( p_{+}\right) \right\vert \left\vert \mu \right\vert  \\
&\leq &C\left\Vert \xi _{+}\right\Vert _{W_{\alpha }^{1,p}\left( \Sigma
_{+}\right) }+C\varepsilon \left\vert \mu \right\vert  \\
&\leq &C\left\Vert \eta _{+}\right\Vert _{L_{\alpha }^{p}\left( \Sigma
_{+}\right) }+C\varepsilon \left\Vert \eta _{\varepsilon }\right\Vert
_{\varepsilon }\leq C\left\Vert \eta \right\Vert _{\varepsilon }.
\end{eqnarray*}%
For $\widetilde{\xi _{\varepsilon }}=\xi _{\varepsilon }\left( \tau
,t\right) -\left( \xi _{\varepsilon }\right) _{0}\left( l/\varepsilon
\right) $, since $\phi _{+}^{l/\varepsilon -T\left( \varepsilon \right)
}\left( \tau \right) =1$ on $[l/\varepsilon ,\tau (\varepsilon )]$, we have
\begin{eqnarray}
\xi _{\varepsilon }\left( \tau ,t\right) -\left( \xi _{\varepsilon }\right)
_{0}\left( l/\varepsilon \right)  &=&\left( \xi _{\chi _{\varepsilon
}}\right) _{0}\left( \tau \right) +\phi _{0}^{l/\varepsilon +T\left(
\varepsilon \right) }\left( \tau \right) \left( \xi _{\chi _{\varepsilon
}}-\left( \xi _{\chi _{\varepsilon }}\right) _{0}\right)   \notag \\
&{}&+\left( \xi _{+}^{\varepsilon }\left( \tau ,t\right) -v_{+}\right)
-\left( \xi _{+}^{\varepsilon }\right) _{0}\left( l/\varepsilon \right)
+\varepsilon \frac{\mu }{l}\nabla f\left( p_{+}\right)   \notag \\
&=&\left[ \xi _{+}^{\varepsilon }\left( \tau ,t\right) -\left( \xi
_{+}^{\varepsilon }\right) _{0}\left( l/\varepsilon \right) \right] +\phi
_{0}^{l/\varepsilon +T\left( \varepsilon \right) }\left( \tau \right) \left(
\xi _{\chi _{\varepsilon }}-\left( \xi _{\chi _{\varepsilon }}\right)
_{0}\right)   \notag \\
&&+\left[ \left( \xi _{\chi _{\varepsilon }}\right) _{0}\left( \tau \right)
-\left( \xi _{\chi _{\varepsilon }}\right) _{0}\left( l/\varepsilon \right) %
\right] ,  \label{difference-all}
\end{eqnarray}%
where in the last row we have used  %
\eqref{eq:matching} and \eqref{f-matching}. For the first term, same as the computation $\left(
10.55\right) \sim \left( 10.60\right) $ in [OZ1] Lemma 10.10, we have
\begin{eqnarray}
&{}&\int_{S^{1}}\int_{l/\varepsilon }^{\tau (\varepsilon )}\left\vert \xi
_{+}^{\varepsilon }\left( \tau ,t\right) -\left( \xi _{+}^{\varepsilon
}\right) _{0}\left( l/\varepsilon \right) \right\vert ^{p}  \notag \\
&{}&+\left\vert \nabla \left( \xi _{+}^{\varepsilon }\left( \tau ,t\right)
-\left( \xi _{+}^{\varepsilon }\right) _{0}\left( l/\varepsilon \right)
\right) \right\vert ^{p}e^{2\pi \delta \left( -\tau +\tau (\varepsilon
)\right) }d\tau dt\leq C\left\Vert \xi _{+}\right\Vert _{W_{\alpha
}^{1,p}\left( \Sigma _{+}\right) }^{p}.  \label{difference-1}
\end{eqnarray}%
For the second term, by comparing the exponential weight and power weight on
$\left[ l/\varepsilon ,l/\varepsilon +\frac{p-1}{\delta }S\left( \varepsilon
\right) \right] $
\begin{equation}
e^{2\pi \delta \left( -\tau +\tau (\varepsilon )\right) }=\left(
1+l/\varepsilon \right) ^{p-1}e^{-2\pi \left( \tau -l/\varepsilon \right)
}\leq C\varepsilon ^{1-p}\left( 1+\left\vert \tau \right\vert \right)
^{\delta },  \label{eq:edtaue}
\end{equation}%
we have
\begin{eqnarray}
&&\left\Vert \phi _{0}^{l/\varepsilon +T\left( \varepsilon \right) }\left(
\tau \right) \left( \xi _{\chi _{\varepsilon }}-\left( \xi _{\chi
_{\varepsilon }}\right) _{0}\right) e^{\frac{2\pi \delta }{p}\left( -\tau
+\tau (\varepsilon )\right) }\right\Vert _{W^{1,p}\left( \left[
l/\varepsilon ,\tau (\varepsilon )\right] \times S^{1}\right) }  \notag \\
&\leq &C\left\Vert \phi _{0}^{l/\varepsilon +T\left( \varepsilon \right)
}\left( \tau \right) \left( \xi _{\chi _{\varepsilon }}-\left( \xi _{\chi
_{\varepsilon }}\right) _{0}\right) \left( \varepsilon ^{1-p}\left(
1+\left\vert \tau \right\vert \right) ^{\delta }\right) ^{\frac{1}{p}%
}\right\Vert _{W^{1,p}\left( \left[ l/\varepsilon ,\tau (\varepsilon )\right]
\times S^{1}\right) }  \notag \\
&=&C\left\Vert \xi _{\chi _{\varepsilon }}-\left( \xi _{\chi _{\varepsilon
}}\right) _{0}\right\Vert _{W_{\rho _{\varepsilon }}^{1,p}\left( \mathbb{%
R\times }S^{1}\right) }  \notag \\
&\leq &C\left\Vert \xi _{\chi _{\varepsilon }}\right\Vert _{W_{\rho
_{\varepsilon }}^{1,p}\left( \mathbb{R\times }S^{1}\right) }.
\label{difference-2}
\end{eqnarray}%
For the third term, note that from $\left( \ref{Schauder-Convergence}\right)
$
\begin{equation*}
\left\vert \left( \xi _{\chi _{\varepsilon }}\right) _{0}\left( \tau \right)
-\left( \xi _{\chi _{\varepsilon }}\right) _{0}\left( l/\varepsilon \right)
\right\vert \leq \left\vert \varepsilon \tau \pm l\right\vert ^{\gamma
}\left\Vert \left( \xi _{\chi _{\varepsilon }}\right) _{0}\right\Vert
_{W_{\varepsilon }^{1,p}\left( [l/\varepsilon ,\tau (\varepsilon )]\right) }
\end{equation*}%
where $\gamma =1-\frac{1}{p}$, we have
\begin{eqnarray*}
&&\int_{S^{1}}\int_{l/\varepsilon }^{\tau (\varepsilon )}\left\vert \left(
\xi _{\chi _{\varepsilon }}\right) _{0}\left( \tau \right) -\left( \xi
_{\chi _{\varepsilon }}\right) _{0}\left( l/\varepsilon \right) \right\vert
^{p}e^{2\pi \delta \left( -\tau +\tau (\varepsilon )\right) }d\tau dt \\
&\leq &C\varepsilon ^{p\gamma }e^{2\pi \left( p-1\right) S(\varepsilon
)}\int_{S^{1}}\int_{0}^{\frac{p-1}{\delta }S(\varepsilon )}s^{p\gamma
}e^{-2\pi \delta s}dsdt\cdot \left\Vert \left( \xi _{\chi _{\varepsilon
}}\right) _{0}\right\Vert _{W_{\varepsilon }^{1,p}\left( [l/\varepsilon
,\tau (\varepsilon )]\right) }^{p} \\
&\leq &C\left( p,l\right) \left\Vert \left( \xi _{\chi _{\varepsilon
}}\right) _{0}\right\Vert _{W_{\varepsilon }^{1,p}\left( [l/\varepsilon
,\tau (\varepsilon )]\right) }^{p},
\end{eqnarray*}%
where $s=\tau -\frac{l}{\varepsilon }$ and in the last inequality we have
used that the integral $\int_{0}^{\infty }s^{p\gamma }e^{-2\pi \delta s}ds$
converges and $\varepsilon ^{p\gamma }e^{2\pi \left( p-1\right) S\left(
\varepsilon \right) }\approx l^{p-1}$. From linearized gradient operator $%
\frac{\partial }{\partial \tau }+A_{\varepsilon }(\tau )$ we also have
\begin{equation*}
\nabla _{\tau }\left( \left( \xi _{\chi _{\varepsilon }}\right) _{0}\left(
\tau \right) -v_{+}\right) =\nabla _{\tau }\left( \xi _{\chi _{\varepsilon
}}\right) _{0}\left( \tau \right) =\left( \eta _{\chi _{\varepsilon
}}\right) _{0}-A_{\varepsilon }(\tau )\left( \xi _{\chi _{\varepsilon
}}\right) _{0}.
\end{equation*}%
Since $\left( \xi _{\chi _{\varepsilon }}\right) _{0}\left( \tau \right)
-v_{+}$ is $t$-independent, and from \eqref{eq:edtaue} for $\tau \in \lbrack
l/\varepsilon ,\tau (\varepsilon )]$ we have

\begin{eqnarray*}
&&\left\Vert \nabla \left( \left( \xi _{\chi _{\varepsilon }}\right)
_{0}\left( \tau \right) -v_{+}\right) e^{-\frac{2\pi \delta }{p}\left( \tau
-\tau (\varepsilon )\right) }\right\Vert _{L^{p}\left( [l/\varepsilon ,\tau
(\varepsilon )]\right) } \\
&\leq &C\left\Vert \nabla _{\tau }\left( \left( \xi _{\chi _{\varepsilon
}}\right) _{0}\left( \tau \right) -v_{+}\right) \varepsilon ^{\frac{1-p}{p}%
}\right\Vert _{L^{p}\left( [l/\varepsilon ,\tau (\varepsilon )]\right) } \\
&=&C\left\Vert \nabla _{\tau }\left( \xi _{\chi _{\varepsilon }}\right)
_{0}\left( \tau \right) \right\Vert _{L_{\varepsilon }^{p}\left(
[l/\varepsilon ,\tau (\varepsilon )]\right) } \\
&\leq &C\left( \left\Vert \left( \eta _{\chi _{\varepsilon }}\right)
_{0}\right\Vert _{L_{\varepsilon }^{p}\left( [l/\varepsilon ,\tau
(\varepsilon )]\right) }+\left\Vert A_{\varepsilon }(\tau )\right\Vert
_{C^{0}}\left\Vert \left( \xi _{\chi _{\varepsilon }}\right) _{0}\right\Vert
_{L_{\varepsilon }^{p}\left( [l/\varepsilon ,\tau (\varepsilon )]\right)
}\right)  \\
&\leq &C\left( \left\Vert \left( \xi _{\chi _{\varepsilon }}\right)
_{0}\right\Vert _{W_{\varepsilon }^{1,p}\left( \mathbb{R}\right)
}+\left\Vert \left( \xi _{\chi _{\varepsilon }}\right) _{0}\right\Vert
_{W_{\varepsilon }^{1,p}\left( [l/\varepsilon ,\tau (\varepsilon )]\right)
}\right) .\text{ \ }
\end{eqnarray*}%
\ \ \ Note in the last inequality we must use $\left\Vert \left( \xi _{\chi
_{\varepsilon }}\right) _{0}\right\Vert _{W_{\varepsilon }^{1,p}\left(
\mathbb{R}\right) }$ instead of
\begin{equation*}
\left\Vert \left( \xi _{\chi _{\varepsilon }}\right) _{0}\right\Vert
_{W_{\varepsilon }^{1,p}\left( [l/\varepsilon ,\tau (\varepsilon )]\right) }
\end{equation*}%
since we have used the right inverse defined on whole $\mathbb{R}$.
Combining these we get%
\begin{equation}
\left\Vert \left( \xi _{\chi _{\varepsilon }}\right) _{0}\left( \tau \right)
-v_{+}\right\Vert _{W_{\varepsilon }^{1,p}\left( [l/\varepsilon ,\tau
(\varepsilon )]\right) }\leq C\left\Vert \left( \xi _{\chi _{\varepsilon
}}\right) _{0}\right\Vert _{W_{\varepsilon }^{1,p}\left( \mathbb{R}\right) }.
\label{difference-3}
\end{equation}%
In $\left( \ref{difference-2}\right) $ and $\left( \ref{difference-3}\right)
$, the interval $\mathbb{R}$ is bigger than $[-l/\varepsilon ,l/\varepsilon ]
$ where $\xi _{\chi _{\varepsilon }}$ was originally defined. However, we
will prove the following inequality
\begin{eqnarray}
\left\Vert \left( \xi _{\chi _{\varepsilon }}\right) _{0}\right\Vert
_{W_{_{\varepsilon }}^{1,p}\left( \mathbb{R}\right) } &\leq &C\left\Vert
\left( \xi _{\chi _{\varepsilon }}\right) _{0}\right\Vert _{W_{_{\rho
_{\varepsilon }}}^{1,p}\left( [-l/\varepsilon ,l/\varepsilon ]\right) },
\notag  \label{eq:xichie} \\
\left\Vert \widetilde{\xi _{\chi _{\varepsilon }}}\right\Vert _{W_{\rho
_{\varepsilon }}^{1,p}\left( \mathbb{R\times }S^{1}\right) }^{p} &\leq
&C\left\Vert \widetilde{\xi _{\chi _{\varepsilon }}}\right\Vert _{W_{\rho
_{\varepsilon }}^{1,p}\left( \left[ -l/\varepsilon ,l/\varepsilon \right]
\mathbb{\times }S^{1}\right) }^{p},
\end{eqnarray}%
later in Lemma \ref{lem:xie0} and \ref{lem:tildexie} where $C$ is
independent on $\varepsilon $. Therefore combining $\left( \ref%
{difference-all}\right) ,\left( \ref{difference-1}\right) ,\left( \ref%
{difference-2}\right) $ and $\left( \ref{difference-3}\right) ,$ we have
\begin{equation*}
\left\Vert \xi _{\varepsilon }-\left( \xi _{\varepsilon }\right) _{0}\left(
l/\varepsilon \right) \right\Vert _{W_{\beta _{\delta ,\varepsilon
}}^{1,p}\left( [l/\varepsilon ,\tau (\varepsilon )]\times S^{1}\right) }\leq
C\left( \left\Vert \xi _{+}\right\Vert _{W_{\alpha }^{1,p}\left( \Sigma
_{+}\right) }+\left\Vert \xi _{\chi _{\varepsilon }}\right\Vert _{W_{_{\rho
_{\varepsilon }}}^{1,p}\left( [-l/\varepsilon ,l/\varepsilon ]\times
S^{1}\right) }\right) .
\end{equation*}%
The estimate for $\tau \in \left[ -\tau (\varepsilon ),0\right] $ is similar.
\newline
{\bf (iii) For $\left\vert \tau \right\vert >\tau (\varepsilon )$:} On this region, $\xi _{\varepsilon
}=\xi _{\pm }^{\varepsilon }$ is a shift of $\xi _{\pm }$, hence%
\begin{equation*}
\left\Vert \widetilde{\xi _{\varepsilon }}\right\Vert _{W_{\beta _{\delta
,\varepsilon }}^{1,p}\left( \pm \left[ \tau (\varepsilon ),\infty \right]
\times S^{1}\right) }=\left\Vert \widetilde{\xi _{\pm }}\right\Vert
_{W_{\beta _{\delta ,\varepsilon }}^{1,p}\left( \pm \left[ 0,\infty \right]
\times S^{1}\right) }\leq \left\Vert \xi _{\pm }\right\Vert _{W_{\alpha
}^{1,p}\left( \Sigma _{\pm }\right) }.
\end{equation*}%
Combining these we have
\begin{eqnarray*}
\left\Vert \xi _{\varepsilon }\right\Vert _{\varepsilon } &=&\left\Vert
\widetilde{\xi _{\varepsilon }}\right\Vert _{W_{\beta _{\delta ,\varepsilon
}}^{1,p}\left( \mathbb{R}\times S^{1}\right) }+\left\Vert \left( \xi
_{\varepsilon }\right) _{0}\right\Vert _{W_{\varepsilon }^{1,p}\left( \left[
-l/\varepsilon ,l/\varepsilon \right] \right) }+\left\vert \left( \xi
_{\varepsilon }\right) _{0}\left( \pm l/\varepsilon \right) \right\vert \\
&\leq &C\left( \left\Vert \xi _{+}\right\Vert _{W_{\alpha }^{1,p}\left(
\Sigma _{+}\right) }+\left\Vert \xi _{-}\right\Vert _{W_{\alpha
}^{1,p}\left( \Sigma _{-}\right) }+\left\Vert \xi _{\chi _{\varepsilon
}}\right\Vert _{W_{_{\rho _{\varepsilon }}}^{1,p}\left( [-l/\varepsilon
,l/\varepsilon ]\times S^{1}\right) }\right) +C\varepsilon \left\vert \mu
\right\vert \\
&\leq &C\left( \left\Vert \eta _{+}\right\Vert _{L_{\alpha }^{p}\left(
\Sigma _{+}\right) }+\left\Vert \eta _{-}\right\Vert _{L_{\alpha }^{p}\left(
\Sigma _{-}\right) }+\left\Vert \eta _{\chi _{\varepsilon }}\right\Vert
_{L_{_{\rho _{\varepsilon }}}^{p}\left( [-l/\varepsilon ,l/\varepsilon
]\times S^{1}\right) }\right) \\
&=&C\left( \left\Vert \kappa _{+}^{\varepsilon }\eta \right\Vert _{L_{\alpha
}^{p}\left( \Sigma _{+}\right) }+\left\Vert \kappa _{-}^{\varepsilon }\eta
\right\Vert _{L_{\alpha }^{p}\left( \Sigma _{-}\right) }+\left\Vert \kappa
_{0}^{\varepsilon }\eta \right\Vert _{L_{_{\rho _{\varepsilon }}}^{p}\left(
[-l/\varepsilon ,l/\varepsilon ]\times S^{1}\right) }\right) \\
&\leq &C\left\Vert \eta \right\Vert _{\varepsilon },
\end{eqnarray*}

where we have use $\left( \ref{f-matching}\right) $ to control $\left( \xi
_{\varepsilon }\right) _{0}\left( \pm l/\varepsilon \right) $. Thus we have
obtained
\begin{equation*}
\left\Vert Q_{\text{\textrm{para}}}^{\text{\textrm{app}};\varepsilon }\eta
\right\Vert _{\varepsilon }=\left\Vert \left( \xi _{\varepsilon },\mu
\right) \right\Vert _{\varepsilon }\leq C\left\Vert \eta \right\Vert
_{\varepsilon }.
\end{equation*}%
The proposition is now proved.
\end{proof}

Finally it remains to prove \eqref{eq:xichie} which is in order.

\begin{lem}
\label{lem:xie0} We have%
\begin{equation*}
\left\Vert \left( \xi _{\chi _{\varepsilon }}\right) _{0}\right\Vert
_{W_{_{\varepsilon }}^{1,p}\left( \mathbb{R}\right) }\leq C\left\Vert \left(
\xi _{\chi _{\varepsilon }}\right) _{0}\right\Vert _{W_{_{\rho _{\varepsilon
}}}^{1,p}\left( [-l/\varepsilon ,l/\varepsilon ]\right) }.
\end{equation*}
\end{lem}

\begin{proof}
Since the right inverse $L_{_{\varepsilon }}^{p}\left( \mathbb{R}\right)
\rightarrow W_{_{\varepsilon }}^{1,p}\left( \mathbb{R}\right) ,\ \left( \eta
_{\chi _{\varepsilon }}\right) _{0}\rightarrow \ \left( \xi _{\chi
_{\varepsilon }}\right) _{0}\ \ $ is uniformly bounded, and $\left( \eta
_{\chi _{\varepsilon }}\right) _{0}$ is supported in $[-l/\varepsilon
,l/\varepsilon ]$, we have%
\begin{eqnarray*}
\left\Vert \left( \xi _{\chi _{\varepsilon }}\right) _{0}\right\Vert
_{W_{_{\varepsilon }}^{1,p}\left( \mathbb{R}\right) } &\leq &C\left\Vert
\left( \eta _{\chi _{\varepsilon }}\right) _{0}\right\Vert
_{L_{_{\varepsilon }}^{p}\left( \mathbb{R}\right) }=C\left\Vert \left( \eta
_{\chi _{\varepsilon }}\right) _{0}\right\Vert _{L_{_{\varepsilon
}}^{p}\left( [-l/\varepsilon ,l/\varepsilon ]\right) } \\
&=&C\left\Vert D_{\Phi }^{\varepsilon }\left( \xi _{\chi _{\varepsilon
}}\right) _{0}\right\Vert _{L_{_{\varepsilon }}^{p}\left( [-l/\varepsilon
,l/\varepsilon ]\right) }\leq C\left\Vert \nabla \left( \xi _{\chi
_{\varepsilon }}\right) _{0}\right\Vert _{L_{_{\varepsilon }}^{p}\left(
[-l/\varepsilon ,l/\varepsilon ]\right) } \\
&\leq &C\left\Vert \left( \xi _{\chi _{\varepsilon }}\right) _{0}\right\Vert
_{W_{_{\varepsilon }}^{1,p}\left( [-l/\varepsilon ,l/\varepsilon ]\right) }
\\
&\leq &C\left\Vert \left( \xi _{\chi _{\varepsilon }}\right) _{0}\right\Vert
_{W_{_{\rho _{\varepsilon }}}^{1,p}\left( [-l/\varepsilon ,l/\varepsilon
]\right) },
\end{eqnarray*}%
where the last inequality is because $W_{\varepsilon }^{1,p}$ is a component
of $W_{\rho _{\varepsilon }}^{1,p}$.
\end{proof}

Similarly we prove the following

\begin{lem}
\label{lem:tildexie} We have%
\begin{equation*}
\left\Vert \widetilde{\xi _{\chi _{\varepsilon }}}\right\Vert _{W_{\beta
_{\delta ,\varepsilon }}^{1,p}\left( \mathbb{R\times }S^{1}\right) }\leq
C\left\Vert \nabla \widetilde{\xi _{\chi _{\varepsilon }}}\right\Vert
_{L_{\beta _{\delta ,\varepsilon }}^{p}\left( \left[ -l/\varepsilon
,l/\varepsilon \right] \mathbb{\times }S^{1}\right) .}
\end{equation*}
\end{lem}

\begin{proof}
By previous proposition, $Q_{\textrm{app}}^{\varepsilon }$ is a uniformly bounded
operator, so for the higher mode $\widetilde{\xi _{\chi _{\varepsilon }}}$
of $\xi _{\chi _{\varepsilon }}$, we have
\begin{eqnarray*}
\left\Vert \widetilde{\xi _{\chi _{\varepsilon }}}\right\Vert _{W_{\beta
_{\delta ,\varepsilon }}^{1,p}\left( \mathbb{R\times }S^{1}\right) } &\leq
&C\left\Vert D^\varepsilon_{\Phi;\text{\textrm{para}}}\widetilde{\xi _{\chi
_{\varepsilon }}}\right\Vert _{L_{\beta _{\delta ,\varepsilon }}^{p}\left(
\mathbb{R\times }S^{1}\right) } \\
&=&C\left\Vert \kappa _{0}^{\varepsilon }\widetilde{\eta }_{\chi
_{\varepsilon }}\right\Vert _{L_{\beta _{\delta ,\varepsilon }}^{p}\left(
\mathbb{R\times }S^{1}\right) } \\
&=&C\left\Vert D^\varepsilon_{\Phi}\widetilde{\xi _{\chi _{\varepsilon }}}%
\right\Vert _{L_{\beta _{\delta ,\varepsilon }}^{p}\left( \left[
-l/\varepsilon ,l/\varepsilon \right] \mathbb{\times }S^{1}\right) } \\
&\leq &C\left\Vert \nabla \widetilde{\xi _{\chi _{\varepsilon }}}\right\Vert
_{L_{\beta _{\delta ,\varepsilon }}^{p}\left( \left[ -l/\varepsilon
,l/\varepsilon \right] \mathbb{\times }S^{1}\right) .}
\end{eqnarray*}
\end{proof}

This finishes the proof of Proposition \ref{prop:uniform-inverse-bound},
which establishes the construction of the right inverse with uniform bound
as $\varepsilon \rightarrow 0$ for $L\geq l\geq l_{0}$ for any given $%
L,l_{0}>0$.

We now justify that $Q^{\text{\textrm{app}};\varepsilon}_{\text{\textrm{para}%
}}$ is indeed an approximate right inverse.

\begin{prop}
\label{prop:approx-right-inverse}$\left\Vert \left( D^\varepsilon_{\Phi;%
\text{\textrm{para}}}\circ Q^{\text{\textrm{app}};\varepsilon}_{\text{%
\textrm{para}}} -I\right) \eta \right\Vert _{\varepsilon }<\frac{1}{2}%
\left\Vert \eta \right\Vert _{\varepsilon }.$
\end{prop}

\begin{proof}
By the definition of $Q_{\text{\textrm{para}}}^{\text{\textrm{app}}%
;\varepsilon }$, $\left( D_{\Phi ;\text{\textrm{para}}}^{\varepsilon }\circ
Q_{\text{\textrm{para}}}^{\text{\textrm{app}};\varepsilon }-I\right) \eta =0$
except on the intervals $\pm \left[ l/\varepsilon -T\left( \varepsilon
\right) ,l/\varepsilon +T\left( \varepsilon \right) \right] $. Let's
consider $\tau \in \left[ l/\varepsilon -T\left( \varepsilon \right)
,l/\varepsilon +T\left( \varepsilon \right) \right] $. The other interval is
the same. Recall $\widetilde{\xi _{\chi _{\varepsilon }}}{}=\xi _{\chi
_{\varepsilon }}-\left( \xi _{\chi _{\varepsilon }}\right) _{0}$, $%
\widetilde{\xi _{+}}=\xi _{+}-v_{+}$ and
\begin{equation*}
\xi _{\varepsilon }=\left( \xi _{\chi _{\varepsilon }}\right) _{0}+\phi
_{0}^{l/\varepsilon +T\left( \varepsilon \right) }\left( \tau \right)
\widetilde{\xi _{\chi _{\varepsilon }}}+\phi _{+}^{\left( l/\varepsilon
-T\left( \varepsilon \right) \right) }\left( \tau \right) \widetilde{\xi
_{+}^{\varepsilon }}.
\end{equation*}%
We compute%
\begin{eqnarray*}
&{}&D_{\Phi ;\text{\textrm{para}}}^{\varepsilon }\circ Q_{\text{\textrm{para}%
}}^{\text{\textrm{app}};\varepsilon }\eta -\eta  \\
&=&D_{\Phi ;\text{\textrm{para}}}^{\varepsilon }\left( \left( \xi _{\chi
_{\varepsilon }}\right) _{0}{},\mu \right) +\phi _{0}^{l/\varepsilon
+T\left( \varepsilon \right) }D_{\Phi ;\text{\textrm{para}}}^{\varepsilon }%
\widetilde{\xi _{\chi _{\varepsilon }}}{}+\phi _{+}^{l/\varepsilon -T\left(
\varepsilon \right) }D_{\Phi ;\text{\textrm{para}}}^{\varepsilon }\widetilde{%
\xi _{+}^{\varepsilon }}-\left( \eta _{0}+\widetilde{\eta }\right)  \\
&{}&+\left( \phi _{0}^{l/\varepsilon +T\left( \varepsilon \right) }\right)
^{\prime }\widetilde{\xi _{\chi _{\varepsilon }}}{}+\left( \phi
_{+}^{l/\varepsilon -T\left( \varepsilon \right) }\right) ^{\prime }%
\widetilde{\xi _{+}} \\
&=&D_{\Phi ;\text{\textrm{para}}}^{\varepsilon }\left( \left( \xi _{\chi
_{\varepsilon }}\right) _{0},\mu \right) +\phi _{0}^{l/\varepsilon +T\left(
\varepsilon \right) }D_{\Phi ;\text{\textrm{para}}}^{\varepsilon }\widetilde{%
\xi _{\chi _{\varepsilon }}} \\
&{}&+\phi _{+}^{l/\varepsilon -T\left( \varepsilon \right) }\left[
D_{u_{+}^{\varepsilon }}\overline{\partial }_{\left( K_{+}^{\varepsilon
},J_{+}^{\varepsilon }\right) }+A_{\varepsilon }\right] \left( \xi
_{+}^{\varepsilon }-v_{+}\right)  \\
&{}&-\left( \eta _{0}+\widetilde{\eta }\right) +\left( \phi
_{0}^{l/\varepsilon +T\left( \varepsilon \right) }\right) ^{\prime }%
\widetilde{\xi _{\chi _{\varepsilon }}}+\left( \phi _{+}^{l/\varepsilon
-T\left( \varepsilon \right) }\right) ^{\prime }\widetilde{\xi
_{+}^{\varepsilon }} \\
&=&D_{\Phi ;\text{\textrm{para}}}^{\varepsilon }\left( \left( \xi _{\chi
_{\varepsilon }}\right) _{0},\mu \right) +\phi _{0}^{l/\varepsilon +T\left(
\varepsilon \right) }D_{\Phi ;\text{\textrm{para}}}^{\varepsilon }\widetilde{%
\xi _{\chi _{\varepsilon }}} \\
&{}&+\phi _{+}^{l/\varepsilon -T\left( \varepsilon \right)
}D_{u_{+}^{\varepsilon }}\overline{\partial }_{\left( K_{+}^{\varepsilon
},J_{+}^{\varepsilon }\right) }\xi _{+}^{\varepsilon }-\left( \eta _{0}+%
\widetilde{\eta }\right)  \\
&{}&+\phi _{+}^{l/\varepsilon -T\left( \varepsilon \right) }\left(
A_{\varepsilon }\widetilde{\xi _{+}^{\varepsilon }}-D_{u_{+}^{\varepsilon }}%
\overline{\partial }_{\left( K_{+}^{\varepsilon },J_{+}^{\varepsilon
}\right) }v_{+}\right)  \\
&{}&+\left( \phi _{0}^{l/\varepsilon +T\left( \varepsilon \right) }\right)
^{\prime }\widetilde{\xi _{\chi _{\varepsilon }}}+\left( \phi
_{+}^{l/\varepsilon -T\left( \varepsilon \right) }\right) ^{\prime }%
\widetilde{\xi _{+}^{\varepsilon }},
\end{eqnarray*}%
where in the second identity we have used the notation $\widetilde{\xi
_{+}^{\varepsilon }}=\xi _{+}^{\varepsilon }-v_{+}$. By our construction
of the approximate right inverse $Q_{\text{\textrm{para}%
}}^{\text{\textrm{app}};\varepsilon }$, we have
\begin{equation*}
D_{\Phi ;\text{\textrm{para}}}^{\varepsilon }\left( \left( \xi _{\chi
_{\varepsilon }}\right) _{0},\mu \right) =\eta _{0},~~D_{\Phi ;\text{\textrm{%
para}}}^{\varepsilon }\widetilde{\xi _{\chi _{\varepsilon }}}{}=\kappa
_{0}^{\varepsilon }\widetilde{\eta ,}~\text{and}~D_{u_{+}^{\varepsilon }}%
\overline{\partial }_{\left( K_{+}^{\varepsilon },J_{+}^{\varepsilon
}\right) }\xi _{+}^{\varepsilon }=\kappa _{+}^{\varepsilon }\widetilde{\eta }
\end{equation*}%
in $\left[ l/\varepsilon -T\left( \varepsilon \right) ,l/\varepsilon
+T\left( \varepsilon \right) \right] $. Then substitution of these into the
above makes the first and second rows cancel.

It now remains to estimate the last two rows above.
For the term $A_{\varepsilon }\left( \tau \right) \widetilde{\xi
_{+}^{\varepsilon }}$, using $\left\Vert A_{\varepsilon }\right\Vert
_{C^{1}}\leq C\varepsilon $, we obtain
\begin{eqnarray*}
\left\Vert \phi _{+}^{l/\varepsilon -T\left( \varepsilon \right) }\left(
\tau \right) A_{\varepsilon }\left( \tau \right) \widetilde{\xi
_{+}^{\varepsilon }}\right\Vert _{\varepsilon } &\leq &C\varepsilon
\left\Vert \widetilde{\xi _{+}^{\varepsilon }}\right\Vert _{\varepsilon
}\leq C\varepsilon \left\Vert \xi _{+}\right\Vert _{\varepsilon } \\
&\leq &C\varepsilon \left\Vert \eta _{+}\right\Vert _{\varepsilon }\leq
C\varepsilon \left\Vert \eta \right\Vert _{\varepsilon },
\end{eqnarray*}%
where the first inequality holds because $\xi _{+}^{\varepsilon }$ is a
shift of $\xi _{+}$, the second holds by the definition of the norm $%
\left\Vert \cdot \right\Vert _{\varepsilon }$ and the third comes from the
boundedness of the right inverse of $D_{u_{+}}\overline{\partial }_{\left(
K_{+},J_{+}\right) }$.

For the term $D_{u_{+}^{\varepsilon }}\overline{\partial }_{\left( K_{+}^{\varepsilon
},J_{+}^{\varepsilon }\right) }v_{+}$, by the fact that on
\begin{equation*}
\tau \in \left[ l/\varepsilon -T\left( \varepsilon \right) -\tau \left(
\varepsilon \right) ,l/\varepsilon +T\left( \varepsilon \right) -\tau \left(
\varepsilon \right) \right] \subset \lbrack -\infty ,-1),
\end{equation*}%
$J_{+}=J_{0},$ and $v_{+}$ is a vector field obtained by parallel transport
of $\xi _{+}\left( o_{+}\right) ,$
\begin{equation*}
\left\vert v_{+}\right\vert =\left\vert \xi _{+}\left( o_{+}\right)
\right\vert ,\left\vert \nabla v_{+}(o_{+})\right\vert =0
\end{equation*}%
and
\begin{eqnarray*}
\left\vert D_{u_{+}}\overline{\partial }_{\left( K_{+},J_{+}\right)
}v_{+}\right\vert &=&\left\vert \left( \nabla v_{+}\right) ^{0,1}+\frac{1}{2}%
J_{0}\nabla _{v_{+}}J_{0}\partial u_{+}\right\vert \\
&\leq &C\left\vert du_{+}\right\vert \left\vert v_{+}\right\vert
=C\left\vert du_{+}\right\vert \left\vert \xi _{+}\left( o_{+}\right)
\right\vert \text{.}
\end{eqnarray*}%
Using the uniform exponential decay $\left\vert du_{+}\right\vert \leq
Ce^{2\pi \tau }$ as $\tau \rightarrow -\infty $, and noticing $\varepsilon
^{1-p}\leq e^{-2\pi \delta \left( \tau -\tau (\varepsilon )\right) }$ for $%
\tau \in $ $\left[ l/\varepsilon -T\left( \varepsilon \right) ,l/\varepsilon %
\right] $, we have
\begin{equation*}
|du_{+}(\tau )|\leq Ce^{2\pi \left( -l/\varepsilon +T(\varepsilon )\right) }
\end{equation*}%
for $\tau \in \lbrack l/\varepsilon -T(\varepsilon ),l/\varepsilon ]$, hence%
\begin{eqnarray*}
\left\Vert D_{u_{+}}\overline{\partial }_{\left( K_{+},J_{+}\right)
}v_{+}\right\Vert _{\varepsilon } &\leq &C|du_{+}(\tau )|\left\vert \xi
_{+}\left( o_{+}\right) \right\vert \\
&\leq &Ce^{2\pi \left( -l/\varepsilon +T\left( \varepsilon \right) \right)
}\left\Vert \xi _{+}\right\Vert _{W_{\alpha }^{1,p}\left( \Sigma _{+}\right)
} \\
&\leq &Ce^{2\pi \left( -l/\varepsilon +T\left( \varepsilon \right) \right)
}\left\Vert \eta \right\Vert _{L_{\alpha }^{p}\left( \Sigma _{+}\right) }.
\end{eqnarray*}%
Here we have used that $\left\vert \xi _{+}\left( o_{+}\right) \right\vert $
is part of the norm $\left\Vert \xi _{+}\right\Vert _{W_{\alpha
}^{1,p}\left( \Sigma _{+}\right) }$.

We estimate the remaining two terms%
\begin{equation*}
\left( \phi _{0}^{l/\varepsilon +T\left( \varepsilon \right) }\right)
^{\prime }\widetilde{\xi _{\chi _{\varepsilon }}},\text{ \ }\left( \phi
_{+}^{l/\varepsilon -T\left( \varepsilon \right) }\right) ^{\prime }%
\widetilde{\xi _{+}^{\varepsilon }}.
\end{equation*}

For the term $\left( \phi _{+}^{l/\varepsilon -T\left( \varepsilon \right) }\right)
^{\prime }\left( \tau \right) \widetilde{\xi _{+}^{\varepsilon }}$, it is
supported in $\left[ l/\varepsilon -T\left( \varepsilon \right)
,l/\varepsilon -T\left( \varepsilon \right) +1\right] $, where both the
power order weight and $\varepsilon $-adiabatic weight are dominated by the
exponential weight as the following%
\begin{equation*}
\varepsilon ^{1-p+\delta }\left( 1+\left\vert \tau \right\vert \right)
^{\delta }\leq \varepsilon ^{1-p}\leq e^{-2\pi \delta T\left( \varepsilon
\right) }\cdot e^{2\pi \delta \left\vert \tau -\tau (\varepsilon
)\right\vert }.
\end{equation*}%
Therefore we do not need to separate the $0$-mode and the higher mode parts
of $\phi _{l/\varepsilon -T\left( \varepsilon \right) }^{\prime }\left( \xi
_{+}-v_{+}\right) $ and immediately see
\begin{eqnarray*}
&{}&\left\Vert \left( \phi _{+}^{l/\varepsilon -T\left( \varepsilon \right)
}\right) ^{\prime }\left( \tau \right) \widetilde{\xi _{+}^{\varepsilon }}%
\right\Vert _{\varepsilon } \\
&\leq &e^{\frac{-2\pi \delta T\left( \varepsilon \right) }{p}}\left\Vert
\left( \phi _{+}^{l/\varepsilon -T\left( \varepsilon \right) }\right)
^{\prime }\left( \tau \right) \widetilde{\xi _{+}^{\varepsilon }}\right\Vert
_{W_{\delta }^{1,p}\left[ l/\varepsilon -T\left( \varepsilon \right)
,l/\varepsilon -T\left( \varepsilon \right) +1\right] } \\
&\leq &e^{\frac{-2\pi \delta T\left( \varepsilon \right) }{p}}\left\Vert \xi
_{+}\right\Vert _{W_{\alpha }^{1,p}\left( \Sigma _{+}\right) }\leq e^{\frac{%
-2\pi \delta T\left( \varepsilon \right) }{p}}C\left\Vert \eta
_{+}\right\Vert _{L_{\alpha }^{p}\left( \Sigma _{+}\right) }\leq Ce^{\frac{%
-2\pi \delta T\left( \varepsilon \right) }{p}}\left\Vert \eta \right\Vert
_{\varepsilon }.
\end{eqnarray*}

For the term $\left( \phi _{+}^{l/\varepsilon -T\left( \varepsilon \right) }\right)
^{\prime }\widetilde{\xi _{+}^{\varepsilon }}$, it is supported in $\left[
l/\varepsilon +T\left( \varepsilon \right) ,l/\varepsilon +T\left(
\varepsilon \right) +1\right] $. From the Sobolev embedding, and notice the
power weight $\rho _{\varepsilon }\left( \tau \right) =\varepsilon
^{1-p+\delta }\left( 1+\left\vert \tau \right\vert \right) ^{\delta }$ in $%
\left( \ref{power-e-weight}\right) $, we obtain
\begin{eqnarray*}
&{}&\left\vert \widetilde{\xi _{+}^{\varepsilon }}\right\vert  \\
&\leq &C\varepsilon ^{\frac{p-1-\delta }{p}}\left( 1+l/\varepsilon +T\left(
\varepsilon \right) \right) ^{-\frac{\delta }{p}}\left\Vert \widetilde{\xi
_{+}^{\varepsilon }}\right\Vert _{W_{\rho _{\varepsilon }}^{1,p}\left( \left[
l/\varepsilon +T\left( \varepsilon \right) ,l/\varepsilon +T\left(
\varepsilon \right) +1\right] \times S^{1}\right) } \\
&\leq &C\varepsilon ^{\frac{p-1-\delta }{p}}\left( \varepsilon /l\right) ^{%
\frac{\delta }{p}}\left\Vert \widetilde{\xi _{+}^{\varepsilon }}\right\Vert
_{W_{\rho _{\varepsilon }}^{1,p}\left( \left[ l/\varepsilon +T\left(
\varepsilon \right) ,l/\varepsilon +T\left( \varepsilon \right) +1\right]
\times S^{1}\right) } \\
&\leq &C\varepsilon ^{\frac{p-1-\delta }{p}}\left( \varepsilon /l\right) ^{%
\frac{\delta }{p}}\cdot C\left\Vert \widetilde{\eta }_{\chi _{\varepsilon
}}\right\Vert _{L_{\rho _{\varepsilon }}^{p}\left( \mathbb{R}\times
S^{1}\right) } \\
&\leq &C\varepsilon ^{\frac{p-1-\delta }{p}}\left( \varepsilon /l\right) ^{%
\frac{\delta }{p}}\left\Vert \eta \right\Vert _{\varepsilon },
\end{eqnarray*}%
Since $\left[ l/\varepsilon +T\left( \varepsilon \right) ,l/\varepsilon
+T\left( \varepsilon \right) +1\right] $ is contained outside $\left[
-l/\varepsilon ,l/\varepsilon \right] $, we needn't to distinguish its $0$%
-mode and higher mode in computing the weighted Sobolev norm of $\left( \phi
_{+}^{l/\varepsilon -T\left( \varepsilon \right) }\right) ^{\prime }\left(
\tau \right) \widetilde{\xi _{+}^{\varepsilon }}$. Note
\begin{equation*}
\left. \left( \phi _{+}^{l/\varepsilon -T\left( \varepsilon \right) }\right)
^{\prime }\left( \tau \right) \widetilde{\xi _{+}^{\varepsilon }}\right\vert
_{\tau =l/\varepsilon }=\left( \phi _{+}^{l/\varepsilon -T\left( \varepsilon
\right) }\right) ^{\prime }\left( \tau \right) \left( \xi _{\chi
_{\varepsilon }}-\left( \xi _{\chi _{\varepsilon }}\right) _{0}\right) \Big |%
_{\tau =l/\varepsilon }=0,
\end{equation*}%
and the weight there is
\begin{equation*}
e^{-2\pi \delta \left( \tau -l/\varepsilon -\frac{p-1}{\delta }S\left(
\varepsilon \right) \right) }\leq e^{-2\pi \delta \left( T\left( \varepsilon
\right) -\frac{p-1}{\delta }S(\varepsilon )\right) }=\left( 1+l/\varepsilon
\right) ^{p-1}e^{-2\pi \delta T\left( \varepsilon \right) },
\end{equation*}%
therefore%
\begin{eqnarray*}
\left\Vert \left( \phi _{+}^{l/\varepsilon -T\left( \varepsilon \right)
}\right) ^{\prime }\left( \tau \right) \widetilde{\xi _{+}^{\varepsilon }}%
\right\Vert _{\varepsilon } &\leq &C\varepsilon ^{\frac{p-1-\delta }{p}%
}\left( \varepsilon /l\right) ^{\frac{\delta }{p}}\left\Vert \eta
\right\Vert _{\varepsilon }\cdot \left( 1+l/\varepsilon \right) ^{\frac{p-1}{%
p}}e^{\frac{-2\pi \delta T\left( \varepsilon \right) }{p}} \\
&=&Ce^{\frac{-2\pi \delta T\left( \varepsilon \right) }{p}}\left\Vert \eta
\right\Vert _{\varepsilon }.
\end{eqnarray*}%
Combining these estimates we obtain%
\begin{eqnarray*}
&{}&\left\Vert D_{\Phi ;\text{\textrm{para}}}^{\varepsilon }\overline{%
\partial }_{\left( K_{\varepsilon },J_{\varepsilon }\right) }\circ Q_{\text{%
\textrm{para}}}^{\text{\textrm{app}};\varepsilon }(\eta )-\eta \right\Vert
_{\varepsilon } \\
&\leq &C\varepsilon \left\Vert \eta \right\Vert _{\varepsilon }+Ce^{\frac{%
-2\pi \delta T\left( \varepsilon \right) }{p}}\left\Vert \eta \right\Vert
_{\varepsilon }+Ce^{\frac{-2\pi \delta T\left( \varepsilon \right) }{p}%
}\left\Vert \eta \right\Vert _{\varepsilon } \\
&{}&\quad +Ce^{2\pi \left( -l/\varepsilon +T\left( \varepsilon \right)
\right) }\left\Vert \eta \right\Vert _{L_{\alpha }^{p}\left( \Sigma
_{+}\right) }
\end{eqnarray*}%
Since $T\left( \varepsilon \right) =\frac{p-1}{3\delta }S(\varepsilon
)\rightarrow \infty $ and $\varepsilon S(\varepsilon )\rightarrow 0$ as $%
\varepsilon \rightarrow 0$, we have for small enough $\varepsilon $,
\begin{equation*}
\left\Vert D_{\Phi ;\text{\textrm{para}}}^{\varepsilon }\overline{\partial }%
_{\left( K_{\varepsilon },J_{\varepsilon }\right) }\circ Q_{\text{\textrm{%
para}}}^{\text{\textrm{app}};\varepsilon }(\eta )-\eta \right\Vert
_{\varepsilon }\leq \frac{1}{2}\left\Vert \eta \right\Vert _{\varepsilon }.
\end{equation*}%
The proposition follows.
\end{proof}

By the above proposition, $D^\varepsilon_{\Phi;\text{\textrm{para}}}%
\overline{\partial }_{\left( K_{\varepsilon },J_{\varepsilon }\right) }\circ
Q_{\text{\textrm{para}}}^{\text{\textrm{app}};\varepsilon }$ is invertible,
and
\begin{eqnarray*}
\left\Vert \left( D^\varepsilon_{\Phi;\text{\textrm{para}}}\overline{%
\partial }_{\left( K_{\varepsilon },J_{\varepsilon }\right) }\circ Q_{\text{%
\textrm{para}}}^{\text{\textrm{app}};\varepsilon }\right) ^{-1}\right\Vert
&\leq &\left\Vert \Sigma _{k=0}^{\infty }\left( D^\varepsilon_{\Phi;\text{%
\textrm{para}}}\overline{\partial }_{\left( K_{\varepsilon },J_{\varepsilon
}\right) }\circ Q_{\text{\textrm{para}}}^{\text{\textrm{app}};\varepsilon
}-id\right) ^{k}\right\Vert \\
&\leq &\Sigma _{k=0}^{\infty }\left( \frac{1}{2}\right) ^{k}=1,
\end{eqnarray*}%
so we can construct the true right inverse of \ $D^\varepsilon_{\Phi;\text{%
\textrm{para}}}\overline{\partial }_{\left( K_{\varepsilon },J_{\varepsilon
}\right) }$ as
\begin{equation*}
Q_{\text{\textrm{para}}}^{\varepsilon }:=Q_{\text{\textrm{para}}}^{\text{%
\textrm{app}};\varepsilon }\circ \left( D^\varepsilon_{\Phi;\text{\textrm{%
para}}}\overline{\partial }_{\left( K_{\varepsilon },J_{\varepsilon }\right)
}\circ Q_{\text{\textrm{para}}}^{\text{\textrm{app}};\varepsilon }\right)
^{-1}.
\end{equation*}%
Since $Q_{\text{\textrm{para}}}^{\text{\textrm{app}};\varepsilon }$ is
uniformly bounded, so is $Q_{\text{\textrm{para}}}^{\varepsilon }$.

\begin{rem}
If we examine the proof of the right inverse bound of $Q_{\text{\textrm{para}%
}}^{\varepsilon }$, , we get $\left\Vert Q_{\text{\textrm{para}}%
}^{\varepsilon }\right\Vert \leq Cl^{\frac{p-1}{p}}$. It will increase as $%
l\rightarrow \infty $. But the $\overline{\partial }_{\left( K_{\varepsilon
},J_{\varepsilon }\right) }$-error estimate will have faster order decay $%
\left( Ce^{-cl}+l^{-\frac{p-1}{\delta }}\right) l^{\frac{p-1}{p}}\varepsilon
^{\frac{1}{p}}$, and the quadratic estimate in next section has the uniform
constant $C$, so we an still apply the implicit function theorem for all $%
l\geq l_{0}$.
\end{rem}

\section{\protect\bigskip Uniform quadratic estimate and implicit function
theorem\label{sec:quadratic}}

Consider the Banach spaces
\begin{eqnarray*}
X &=&\left\{ \left( \xi ,\mu \right) \,\Big|\,\xi \in \Gamma \left( \left(
u_{\textrm{app}}^{\varepsilon }\right) ^{\ast }TM\right) ,\mu \in T_{l/\varepsilon }%
\mathbb{R},\left\Vert \left( \xi ,\mu \right) \right\Vert _{\varepsilon
}=\left\Vert \xi \right\Vert _{\varepsilon }+\left\vert \mu \right\vert
<\infty \right\}  \\
Y &=&\left\{ \eta \,\Big|\,\eta \in \Gamma \left( \left(
u_{\textrm{app}}^{\varepsilon }\right) ^{\ast }TM\right) \otimes \Lambda ^{0,1}\left(
\mathbb{R\times }S^{1}\right) ,\left\Vert \eta \right\Vert _{\varepsilon
}<\infty \right\}
\end{eqnarray*}%
with the Banach norms $\left\Vert \cdot \right\Vert _{\varepsilon }$ defined
for $\xi $ and $\eta $ in section \ref{sec:off-shell}. $\ $ For sections $%
\xi ,\xi ^{\prime }$ in $\left( u_{\textrm{app}}^{\varepsilon }\right) ^{\ast }TM$
and $\mu ,\mu ^{\prime }$ in $T_{l/\varepsilon }\mathbb{R}$, let \
\begin{equation*}
\Upsilon _{u_{\textrm{app}}^{\varepsilon }}:\Gamma \left( \left( u_{\textrm{app}}^{\varepsilon
}\right) ^{\ast }TM\right) \times T_{l/\varepsilon }\mathbb{R\rightarrow }%
\Gamma \left( \left( u_{\textrm{app}}^{\varepsilon }\right) ^{\ast }TM\right) \otimes
\Lambda ^{0,1}\left( \mathbb{R\times }S^{1}\right)
\end{equation*}%
that%
\begin{equation*}
\Upsilon _{u_{\textrm{app}}^{\varepsilon }}\left( \xi ,\mu \right) \left( \tau
,t\right) =\func{Pal}_{\xi }^{-1}\left[ \overline{\partial }_{\left(
K_{\varepsilon },J_{\varepsilon }\right) }\exp _{u_{\textrm{app}}^{\varepsilon
}}\left( \xi \right) \right] \left( P_{\mu }\tau ,t\right)
\end{equation*}%
where $\func{Pal}_{\xi }^{-1}$ is the parallel transport from $\exp
_{u_{\textrm{app}}^{\varepsilon }}\xi \left( \tau ,t\right) $ to $u_{\textrm{app}}^{%
\varepsilon }\left( \tau ,t\right) $ along the shortest geodesic, and the
reparameterization map
\begin{equation*}
P_{\mu }:\left( \mathbb{R};-l/\varepsilon ,l/\varepsilon \right) \rightarrow
\left( \mathbb{R};-\left( l/\varepsilon -\mu \right) ,l/\varepsilon -\mu
\right) .
\end{equation*}%
between marked real lines is
\begin{equation}
P_{\mu }\left( \tau \right) =%
\begin{cases}
\tau +\mu  & \mbox{for }\,\tau <-l/\varepsilon , \\
\frac{l}{l-\varepsilon \mu }\tau  & \mbox{for }\,\left\vert \tau \right\vert
\leq l/\varepsilon , \\
\tau -\mu  & \mbox{for }\,\tau >l/\varepsilon ,%
\end{cases}
\label{P-mu}
\end{equation}%
The above definition of $\Upsilon _{u_{\textrm{app}}^{\varepsilon }}$ is slightly
imprecise since $P_{\mu }$ is only piecewise differentiable and we need to
smooth it to a sufficiently close diffeomorphism, but that can be done. Then
\begin{eqnarray*}
&&d\Upsilon _{u_{\textrm{app}}^{\varepsilon }}\left( \xi ,\mu \right) \left( \xi
^{\prime },\mu ^{\prime }\right) \left( \tau ,t\right)  \\
&=&\left. \frac{d}{ds}\right\vert _{s=0}\Upsilon _{u_{\textrm{app}}^{\varepsilon
}}\left( \xi +s\xi ^{\prime },\mu +s\mu ^{\prime }\right) \left( \tau
,t\right)  \\
&=&\left. \frac{d}{ds}\right\vert _{s=0}\func{Pal}_{\xi +s\xi ^{\prime
}}^{-1}\left[ \overline{\partial }_{\left( K_{\varepsilon },J_{\varepsilon
}\right) }\exp _{u_{\textrm{app}}^{\varepsilon }}\left( \xi +s\xi ^{\prime }\right) %
\right] \left( P_{\mu +s\mu ^{\prime }}\left( \tau \right) ,t\right)  \\
&=&%
\begin{cases}
\left. \frac{d}{ds}\right\vert _{s=0}\func{Pal}_{\xi +s\xi ^{\prime }}^{-1}%
\left[ \overline{\partial }_{\left( K_{\varepsilon },J_{\varepsilon }\right)
}\exp _{u_{\textrm{app}}^{\varepsilon }}\left( \xi +s\xi ^{\prime }\right) \right]
\left( \frac{l-\varepsilon \left( \mu +s\mu ^{\prime }\right) }{l}\tau
,t\right) \, & \mbox{if }\,\left\vert \tau \right\vert \leq l/\varepsilon
\\
\left. \frac{d}{ds}\right\vert _{s=0}\func{Pal}_{\xi +s\xi ^{\prime }}^{-1}%
\left[ \overline{\partial }_{\left( K_{\varepsilon },J_{\varepsilon }\right)
}\exp _{u_{\textrm{app}}^{\varepsilon }}\left( \xi +s\xi ^{\prime }\right) \right]
\left( \tau \pm \left( \mu +s\mu ^{\prime }\right) ,t\right) \, & \mbox{if }%
\,\left\vert \tau \right\vert >l/\varepsilon
\end{cases}%
\end{eqnarray*}%
and
\begin{equation*}
d\Upsilon _{u_{\textrm{app}}^{\varepsilon }}\left( 0,0\right) \left( \xi ^{\prime
},\mu ^{\prime }\right) =D_{u_{\textrm{app}}^{\varepsilon }}\overline{\partial }%
_{\left( K_{\varepsilon },J_{\varepsilon }\right) }\left( \xi ^{\prime },\mu
^{\prime }\right) .
\end{equation*}

\begin{prop}
\label{prop:quadratic}\bigskip For each given pair $\left( K_{\varepsilon
},J_{\varepsilon }\right) $, there exists uniform constants $C$ (depending
only on $l_{0}$ and $p$ but independent of $\varepsilon$) and $h_{0}$ such
that for all $\left( \xi ,\mu \right) $ and $\left( \xi ^{\prime },\mu
^{\prime }\right) $ in $\Gamma \left( \left( u_{\textrm{app}}^{\varepsilon }\right)
^{\ast }TM\right) \times T_{l/\varepsilon }\mathbb{R}$ with $0\leq
\left\Vert \left( \xi ,\mu \right) \right\Vert _{\varepsilon }\leq h_{0}\,$,
\begin{equation*}
\left\Vert d\Upsilon_{u_{\textrm{app}}^{\varepsilon }}\left( \xi ,\mu \right) \left(
\xi ^{\prime },\mu ^{\prime }\right) -D_{u_{\textrm{app}}^{\varepsilon }}\overline{%
\partial }_{\left( K_{\varepsilon },J_{\varepsilon }\right) }\left( \xi
^{\prime },\mu ^{\prime }\right) \right\Vert _{\varepsilon }\leq C\left\Vert
\left( \xi ,\mu \right) \right\Vert _{\varepsilon }\left\Vert \left( \xi
^{\prime },\mu ^{\prime }\right) \right\Vert _{\varepsilon }.
\end{equation*}
\end{prop}

\begin{proof}
We first consider the case when $\mu =\mu ^{\prime }=0$. We do the estimate
by considering the regions
$$
|\tau| \leq l/\varepsilon, \quad l/\varepsilon <\left\vert \tau \right\vert \leq l/\varepsilon +\frac{p-1%
}{\delta }S\left( \varepsilon \right), \quad |\tau| > \tau(\varepsilon)
$$
separately.

For $\left\vert \tau \right\vert \leq l/\varepsilon $, and any $t\in S^{1}$,
we have%
\begin{eqnarray*}
\left\vert \xi \left( \tau ,t\right) \right\vert  &\leq &\left\vert \xi
\left( \tau ,t\right) -\xi _{0}\left( \tau \right) \right\vert +\left\vert
\xi _{0}\left( \tau \right) \right\vert  \\
&\leq &C\left\Vert \tilde{\xi}\right\Vert _{W_{\beta _{\delta ,\varepsilon
}}^{1,p}\left( \left[ -l/\varepsilon ,l/\varepsilon \right] \times
S^{1}\right) }+C\left\Vert \xi _{0}\right\Vert _{W_{\varepsilon
}^{1,p}\left( \left[ -l/\varepsilon ,l/\varepsilon \right] \times
S^{1}\right) } \\
&\leq &C\left\Vert \xi \right\Vert _{\varepsilon },
\end{eqnarray*}%
where the second row is by $W^{1,p}\hookrightarrow C^{0}$ Sobolev embedding,
and the facts that the weight $\beta _{\delta ,\varepsilon }\,$\ is nowhere
less than $1$ and the length $l\geq l_{0}>0$. \

For $l/\varepsilon <\left\vert \tau \right\vert \leq l/\varepsilon +\frac{p-1%
}{\delta }S\left( \varepsilon \right) $, again by Sobolev embedding we have%
\begin{eqnarray*}
\left\vert \xi \left( \tau ,t\right) \right\vert  &\leq &\left\vert \xi
\left( \tau ,t\right) -\xi _{0}\left( o_{\pm }\right) \right\vert
+\left\vert \xi _{0}\left( o_{\pm }\right) \right\vert  \\
&\leq &C\left\Vert \tilde{\xi}\right\Vert _{W_{\beta _{\delta ,\varepsilon
}}^{1,p}\left( \pm \left[ l/\varepsilon ,l/\varepsilon +\frac{p-1}{\delta }%
S\left( \varepsilon \right) \right] \times S^{1}\right) }+\left\vert \xi
_{0}\left( o_{\pm }\right) \right\vert  \\
&\leq &C\left\Vert \xi \right\Vert _{\varepsilon }.
\end{eqnarray*}%
For $\left\vert \tau \right\vert >l/\varepsilon +\frac{p-1}{\delta }S\left(
\varepsilon \right) $ the weight $\beta _{\delta ,\varepsilon }\left( \tau
\right) $ is $1$ so%
\begin{equation}
\left\vert \xi \left( \tau ,t\right) \right\vert \leq C\left\Vert \xi
\right\Vert _{\varepsilon }  \label{e-Sobolev}
\end{equation}%
is standard Sobolev embedding. In the above the constants $C$ only depends
on $l_{0}$ and $p$. Combining these we have the uniform Sobolev constant
\begin{equation}
C\left( l_{0},p\right) :=\sup_{\xi \neq 0}\frac{\left\vert \xi \right\vert
_{\infty }}{\left\Vert \xi \right\Vert _{\varepsilon }},
\label{Sobolev-e-constant}
\end{equation}%
for our Banach norm $\left\Vert \cdot \right\Vert _{\varepsilon }$ of all $%
\varepsilon $ for $\xi \in \Gamma \left( \left( u_{\textrm{app}}^{\varepsilon
}\right) ^{\ast }TM\right) $ on $\mathbb{R\times }S^{1}$. The point estimate
in the proof of Proposition 3.5.3 in \cite{MS} yields
\begin{equation*}
\left\vert d\Upsilon _{u_{\textrm{app}}^{\varepsilon }}\left( \xi \right) \xi
^{\prime }-D_{u_{\textrm{app}}^{\varepsilon }}\overline{\partial }_{\left(
J_{0},\varepsilon f\right) }\xi ^{\prime }\right\vert \leq A\left(
\left\vert du_{\textrm{app}}^{\varepsilon }\right\vert \left\vert \xi \right\vert
\left\vert \xi ^{\prime }\right\vert +\left\vert \xi \right\vert \left\vert
\nabla \xi ^{\prime }\right\vert +\left\vert \nabla \xi \right\vert
\left\vert \xi ^{\prime }\right\vert \right) .
\end{equation*}%
Taking the $p$-th power and integrating $\left\vert \nabla \xi \right\vert $
and $\left\vert \nabla \xi ^{\prime }\right\vert $ over $\mathbb{R\times }%
S^{1}$ with respect to the weight $\beta _{\delta ,\varepsilon }$, while
using the Sobolev inequality (from the above definition of $C\left(
l_{0},p\right) $)
\begin{equation*}
\left\vert \xi \right\vert _{\infty }\leq C\left( l_{0},p\right) \left\Vert
\xi \right\Vert _{\varepsilon }
\end{equation*}%
for the terms $\left\vert \xi \right\vert $ and $\left\vert \xi ^{\prime
}\right\vert \,$, we obtain the uniform quadratic estimate%
\begin{equation*}
\left\Vert d\Upsilon _{u_{\textrm{app}}^{\varepsilon }}\left( \xi \right) \xi
^{\prime }-D_{u_{\textrm{app}}^{\varepsilon }}\overline{\partial }_{\left(
J,\varepsilon f\right) }\xi ^{\prime }\right\Vert _{\varepsilon }\leq
C\left\Vert \xi \right\Vert _{\varepsilon }\left\Vert \xi ^{\prime
}\right\Vert _{\varepsilon }
\end{equation*}%
where $C$ is dependent on $l_{0},p$ but independent on $\varepsilon $.

Next we include the $\mu $ and $\mu ^{\prime }$ in the quadratic estimate by
considering the general cases $\mu \neq \mu^{\prime}$.
For simplicity of notation we let
\begin{equation*}
u_{\xi }=\exp _{u_{\textrm{app}}^{\varepsilon }}\xi
\end{equation*}%
and write the pair $\left( u_{\textrm{app}}^{\varepsilon },l/\varepsilon \right) $ to
emphasize the parameter $l$ in the construction of $u_{\textrm{app}}^{\varepsilon }$.

For $\left\vert \tau \right\vert \leq l/\varepsilon $,
\begin{eqnarray}
&&\left\Vert d\Upsilon _{\left( u_{\textrm{app}}^{\varepsilon },\frac{l}{\varepsilon }%
\right) }\left( \xi ,\mu \right) \left( \xi ^{\prime },\mu ^{\prime }\right)
-d\Upsilon _{\left( u_{\textrm{app}}^{\varepsilon },\frac{l}{\varepsilon }\right)
}\left( 0,0\right) \left( \xi ^{\prime },\mu ^{\prime }\right) \right\Vert
\notag \\
&=&\left\Vert \left( d\Upsilon _{\left( u_{\textrm{app}}^{\varepsilon },\frac{l}{%
\varepsilon }-\mu \right) }\left( \xi \right) \xi ^{\prime } - \frac{\mu
^{\prime }\varepsilon \tau }{l}\func{Pal}_{\xi }^{-1}\frac{\partial u_{\xi }%
}{\partial \tau }\right) -\left( d\Upsilon _{\left( u_{\textrm{app}}^{\varepsilon },%
\frac{l}{\varepsilon }\right) }\left( 0\right) \xi ^{\prime }+\frac{\mu
^{\prime }\varepsilon \tau }{l}\frac{\partial u}{\partial \tau }\right)
\right\Vert   \notag \\
&\leq &\left\Vert d\Upsilon _{\left( u_{\textrm{app}}^{\varepsilon },\frac{l}{%
\varepsilon }-\mu \right) }\left( \xi \right) \xi ^{\prime }-d\Upsilon
_{\left( u_{\textrm{app}}^{\varepsilon },\frac{l}{\varepsilon }-\mu \right) }\left(
0\right) \xi ^{\prime }\right\Vert   \notag \\
&&+\left\Vert - d\Upsilon _{\left( u_{\textrm{app}}^{\varepsilon },\frac{l}{\varepsilon
}-\mu \right) }\left( 0\right) \xi ^{\prime }-d\Upsilon _{\left(
u_{\textrm{app}}^{\varepsilon },\frac{l}{\varepsilon }\right) }\left( 0\right) \xi
^{\prime }\right\Vert +\left\Vert \frac{\mu ^{\prime }\varepsilon \tau }{l}%
\func{Pal}_{\xi }^{-1}\left( \frac{\partial u_{\xi }}{\partial \tau }\right)
-\frac{\mu ^{\prime }\varepsilon \tau }{l}\frac{\partial u}{\partial \tau }%
\right\Vert   \notag \\
&\leq &C\left\Vert \xi \right\Vert _{\varepsilon }\left\Vert \xi ^{\prime
}\right\Vert _{\varepsilon }+C\left\vert \mu \right\vert \left\Vert \xi
^{\prime }\right\Vert _{\varepsilon }+C\left\vert \mu ^{\prime }\right\vert
\left\Vert \xi \right\Vert _{\varepsilon }.  \label{quadratic-C}
\end{eqnarray}%
Here the last inequality holds termwise for $\left\vert \varepsilon \tau \right\vert \leq l,$ $
\left\vert \mu \right\vert \leq h_{0}$ and $l\geq l_{0}>0$: the first term is by the $\mu =\mu ^{\prime }=0
$ case, the second term holds because the change of the conformal structure
on $\left[ -l/\varepsilon ,l/\varepsilon \right] \times S^{1}$ by $\mu $
affects $\overline{\partial }_{\left( K_{\varepsilon },J_{\varepsilon
}\right) }$ in a linear way, and the third term is by the property of
exponential map.

For $\left\vert \tau \right\vert >l/\varepsilon $, the estimate to get $%
\left( \ref{quadratic-C}\right) $ is similar (actually simpler) since
\begin{equation}
d\Upsilon_{\left( u_{\textrm{app}}^{\varepsilon },\frac{l}{\varepsilon }\right)
}\left( \xi ,\mu \right) \left( \xi ^{\prime },\mu ^{\prime }\right)
=d\Upsilon_{\left( u_{\textrm{app}}^{\varepsilon },\frac{l}{\varepsilon }-\mu \right)
}\left( \xi \right) \xi ^{\prime }+\mu ^{\prime }\func{Pal}_{\xi }^{-1}\left(%
\frac{\partial u_{\xi }}{\partial \tau }\right).  \label{outter-variation}
\end{equation}

Clearly
\begin{eqnarray*}
C\left\Vert \xi \right\Vert _{\varepsilon }\left\Vert \xi ^{\prime
}\right\Vert _{\varepsilon }+C\left\vert \mu \right\vert \left\Vert \xi
^{\prime }\right\Vert _{\varepsilon }+C\left\vert \mu ^{\prime }\right\vert
\left\Vert \xi \right\Vert _{\varepsilon } &\leq &C\left( \left\Vert \xi
\right\Vert _{\varepsilon }+\left\vert \mu \right\vert \right) \left(
\left\Vert \xi ^{\prime }\right\Vert _{\varepsilon }+\left\vert \mu ^{\prime
}\right\vert \right) \\
&=&C\left\Vert \left( \xi ,\mu \right) \right\Vert _{\varepsilon }\left\Vert
\left( \xi ^{\prime },\mu ^{\prime }\right) \right\Vert _{\varepsilon }.
\end{eqnarray*}%
so the proposition follows.
\end{proof}

\begin{rem}
\label{Sobolev-blow-up}When $l\rightarrow 0$ the Sobolev constant $C\left(
l_{0},p\right) $ in $\left( \ref{Sobolev-e-constant}\right) $ blows up so we
can not get uniform constant $C$ in the above quadratic estimate. Different
argument (blowing up the target) is needed for gluing. This is the nodal
Floer case and was treated in \cite{oh-zhu}. When $l_{0}\leq l\rightarrow
\infty $ the constant $C$ remains uniform.
\end{rem}

\section{Wrap-up of the proof of the main theorem}

To perturb $u_{\textrm{app}}^{\varepsilon }$ to be a true solution of the Floer
equation $\overline{\partial }_{\left( K_{\varepsilon },J_{\varepsilon
}\right) }u=0$, we need the following abstract implicit function theorem in
\cite{MS}

\begin{prop}
\label{prop:abstractglue} Let $X,Y$ be Banach spaces and $U$ be an open set
in $X$. The map $F:X\rightarrow Y$ is continuous differentiable. For $%
x_{0}\in U,D:=dF(x_{0}):X\rightarrow Y$ is surjective and has a bounded
linear right inverse $Q:Y\rightarrow X$, with $\left\Vert Q\right\Vert \leq
C $. Suppose that there exists $h>0$ such that $x\in B_{h}(x_{0})\subset U$%
\begin{equation*}
x\in B_{h}(x_{0})\subset U\implies \left\Vert dF(x)-D\right\Vert \leq \frac{1%
}{2C}.
\end{equation*}%
Suppose
\begin{equation*}
\left\Vert F(x_{0})\right\Vert \leq \frac{h}{4C},
\end{equation*}%
then there exists a unique $x\in B_{h}(x_{0})$ such that%
\begin{equation*}
F(x)=0,~x-x_{0}\in \func{Image}Q,~\left\Vert x-x_{0}\right\Vert \leq
2C\left\Vert F(x_{0})\right\Vert .
\end{equation*}
\end{prop}

For the remaining section, we will wrap up the gluing construction by
identifying the corresponding Banach spaces $X,\, Y$ and the nonlinear map $%
F $, the point $x_0$ and the right inverse $Q$ for the purpose of applying
this proposition.

For a fixed sufficiently small $\varepsilon _{0}>0$, let $\varepsilon \in
(0,\varepsilon _{0}]$. For any $u_{\textrm{app}}^{\varepsilon }$, we choose Banach
spaces
\begin{equation*}
X=\left( \Gamma (\left( u_{\textrm{app}}^{\varepsilon }\right) ^{\ast }TM)\times T_{l}%
\mathbb{R},\left\Vert \cdot \right\Vert _{\varepsilon }\right) ,\,Y=\left(
\Gamma \left( u_{\textrm{app}}^{\varepsilon }\right) ^{\ast }TM)\otimes \Lambda
^{0,1}\left( \mathbb{R\times }S^{1}\right) ,\left\Vert \cdot \right\Vert
_{\varepsilon }\right)
\end{equation*}%
with the Banach norms $\left\Vert \cdot \right\Vert _{\varepsilon }$ for $%
\xi $ and $\eta $ defined in section \ref{sec:off-shell}.

For any given positive function $C(\varepsilon_0)$ continuous at $\varepsilon_0 > 0$
such that $C(\varepsilon _{0})\rightarrow 0$ as $\varepsilon _{0}\rightarrow 0$,
we choose
\begin{equation*}
U_{\varepsilon }=\{\xi \in X\mid \Vert \xi \Vert _{\varepsilon
}<C(\varepsilon _{0})\},\,\quad x_{0}=0\in U_{\varepsilon }.
\end{equation*}%
In our case, we may take the function
$$
C(\varepsilon_0): = E(l) \varepsilon_0^{\frac1p}
$$
which is nothing but the right hand side function of \eqref{eq:error}.

Let $F$ be the map $\Upsilon_{u_{\textrm{app}}^{\varepsilon }}$ defined in the
beginning of the section. By Proposition \ref{prop:uniform-inverse-bound}
and Proposition \ref{prop:quadratic} the $C,h$ in our case are uniform while
$\left\Vert F\left( x_{0}\right) \right\Vert \leq C\varepsilon ^{\frac{1}{p}%
} $,\ so we can apply Proposition \ref{prop:abstractglue} to find a
perturbation $\left( \xi ,\mu \right) $ such that
\begin{equation}
\overline{\partial }_{(K_{\varepsilon },J_{\varepsilon })}(\exp
_{u_{\textrm{app}}^{\varepsilon }}\xi )\left( P_{\mu }\left( \tau \right) ,t\right) =0
\label{eq:F}
\end{equation}%
i.e., $\left( \exp _{u_{\textrm{app}}^{\varepsilon }}\xi \right) \left( P_{\mu
}\left( \tau \right) ,t\right) $ is a genuine solution for \eqref{eq:KJe}.

As before, we denote
\begin{eqnarray*}
&{}&{\mathcal{M}}^{\text{\rm tft}}(K_{-},J_{-};f,K_{+},J_{+};z_{-},z_{+};B)
\\
&=&{\mathcal{M}}(K_{-},J_{-};z_{-};A_{-}){}_{ev_{+}}\times _{ev_{0}}{\mathcal{M}}(f;[0,l ]){}_{ev_{l }}\times _{ev_{-}}{\mathcal{M}}(K_{+},J_{+};z_{+};A_{+})
\end{eqnarray*}for $B$ satisfying $A_- \# B \# A_+ = 0$ in $\pi_2(M)$,
and
\begin{equation*}
{\mathcal{M}}_{\leq \varepsilon _{0}}^{\text{\textrm{para}}%
}(K,J;z_{-},z_{+};B)=\bigcup _{\varepsilon \in (0,\varepsilon _{0}]}{%
\mathcal{M}}^{\varepsilon }(K_{\varepsilon },J_{\varepsilon };z_{-},z_{+};B).
\end{equation*}%
We denote by the smooth map
\begin{eqnarray}
&{}&\text{\textrm{Glue}}:(0,\varepsilon _{0}]\times {\mathcal{M}}_{(0;1,1)}^{%
\text{\textrm{tft}}}(K_{-},J_{-};f,K_{+},J_{+};z_{-},z_{+};B)  \notag
\label{eq:Glue} \\
&{}&\hskip1in\rightarrow \bigcup_{0<\varepsilon \leq \varepsilon _{0}}{%
\mathcal{M}}^{\varepsilon }(K_{\varepsilon },J_{\varepsilon };z_{-},z_{+};B)
\end{eqnarray}%
the composition of $\text{preG}$ followed by the map
\begin{equation*}
u_{\textrm{app}}^{\varepsilon }\rightarrow \left( \exp _{u_{\textrm{app}}^{\varepsilon }}\xi
\right) \left( P_{\mu }\cdot ,\cdot \right)
\end{equation*}%
obtained by solving the equation \eqref{eq:F} applying Proposition \ref%
{prop:abstractglue}. We also denote by $\text{\textrm{Glue}}_{\varepsilon }$
the slice of $\text{\textrm{Glue}}$ for $\varepsilon $.

The main gluing theorem is the following

\begin{thm}
\label{thm:maingluing} Let $(K_{\varepsilon},J_{\varepsilon})$ be the family
of Floer data defined in \eqref{eq:KR}. Then

\begin{enumerate}
\item there exists a topology on
\begin{equation*}
{\mathcal{M}}_{\leq \varepsilon _{0}}^{\text{\textrm{para}}%
}(K,J;z_{-},z_{+};B)=\bigcup_{0<\varepsilon \leq \varepsilon _{0}}{\mathcal{M%
}}^{\varepsilon }(K_{\varepsilon },J_{\varepsilon };z_{-},z_{+};B)
\end{equation*}
with respect to which the gluing construction defines a proper embedding
\begin{eqnarray*}
\text{\textrm{Glue}} & : &(0,\varepsilon _{0}]\times {\mathcal{M}}%
_{(0;1,1)}^{\text{\textrm{tft}}}(K_{-},J_{-};f,K_{+},J_{+};z_{-},z_{+};B) \\
&{}& \qquad \rightarrow {\mathcal{M}}_{\leq \varepsilon _{0}}^{\text{\textrm{%
para}}}(K,J;z_{-},z_{+};B)
\end{eqnarray*}
for sufficiently small $\varepsilon _{0}$.

\item the above mentioned topology can be embedded into
\begin{equation*}
{\mathcal{M}}^{\text{\textrm{tft}}}(K_{-},J_{-};f,K_{+},J_{+};z_{-},z_{+};B)%
\bigcup {\mathcal{M}}_{\leq \varepsilon _{0}}^{\text{\textrm{para}}%
}(K,J;z_{-},z_{+};B)
\end{equation*}
as a set,

\item the embedding $\text{\textrm{Glue}}$ smoothly extends to the embedding
\begin{eqnarray*}
\overline{\text{\textrm{Glue}}} &:&[0,\varepsilon _{0})\times {\mathcal{M}}^{%
\text{\textrm{tft}}}(K_{-},J_{-};f,K_{+},J_{+};z_{-},z_{+};B) \\
&\rightarrow &\overline{{\mathcal{M}}_{\leq \varepsilon _{0}}^{\text{\textrm{%
para}}}}(K,J;z_{-},z_{+};B)
\end{eqnarray*}%
that satisfies $\overline{\text{\textrm{Glue}}}(0,u_{-},\chi
,u_{+})=(u_{-},\chi ,u_{+})$.
\end{enumerate}
\end{thm}

One essential ingredient to establish to complete the proof of this theorem
is the following surjectivity whose proof will be given in the next section.

\begin{prop}
\label{prop:surjective} There exists some $\varepsilon _{0}>0,$ $\zeta
_{0}>0 $ and a function $\delta :\left( 0,\varepsilon _{0}\right)
\rightarrow \mathbb{R}_{+}$ such that the gluing map
\begin{equation*}
\text{\textrm{Glue}}:(0,\varepsilon _{0}]\times {\mathcal{M}}^{\text{\textrm{%
tft}}}(K_{-},J_{-};f,K_{+},J_{+};z_{-},z_{+};B)
\end{equation*}%
is surjective onto
\begin{equation*}
{\mathcal{M}}_{\leq \varepsilon _{0}}^{\text{\textrm{para}}%
}(K,J;z_{-},z_{+};B)\cap \bigcup_{\substack{ 0<\varepsilon \leq \varepsilon
_{0}  \\ 0<\zeta \leq \zeta _{0}}}V_{\zeta ,\delta \left( \varepsilon
\right) }^{\varepsilon }.
\end{equation*}%
Here $V_{\zeta ,\delta }^{\varepsilon }$ is the open subset given in %
\eqref{eq:Ve}.
\end{prop}

Once this is established, the standard argument employed by Donaldson will
complete the proof of this main theorem.

\section{Surjectivity}

\label{sec:surjectivity} In this section, we give the proof of Proposition %
\ref{prop:surjective}.

To prove surjectivity, we need to show that there exists $\varepsilon _{0}>0$%
, $0<\zeta _{0}<1$ and a function $\delta :(0,\varepsilon _{0}]\rightarrow
\mathbb{R}_{+}$ such that for $0<\varepsilon \leq \varepsilon _{0}$ and $%
0<\zeta \leq \zeta _{0}$, any pair $(\varepsilon ,u)\in {\mathcal{M}}_{\leq
\varepsilon _{0}}^{\text{\textrm{para}}}(K,J;z_{-},z_{+};B)$ with
\begin{equation}
d_{\text{\textrm{adia}}}^{\varepsilon ,\zeta }(u,(u_{-},\chi ,u_{+}))<\delta
\left( \varepsilon \right)  \label{adia-distance-delta}
\end{equation}%
lies in the image of the gluing map
\begin{equation*}
\text{\textrm{Glue}}_{\varepsilon }:{\mathcal{M}}^{\text{\textrm{tft}}%
}(K_{-},J_{-};f,K_{+},J_{+};z_{-},z_{+};B)\rightarrow {\mathcal{M}}%
(K_{\varepsilon },J_{\varepsilon };z_{-},z_{+};B).
\end{equation*}%
%
%
%
%
%
%
%
%
%
%
%
%
%
%
%We can actually take $\delta \left( \varepsilon \right) =\varepsilon $.

The above condition $\left( \ref{adia-distance-delta}\right) $ implies that

\begin{enumerate}
\item $E_{J,\Theta _{\varepsilon }}(u)<\delta \left( \varepsilon \right) ,$

\item $d_{\text{\textrm{H}}}\left( u([-R(\varepsilon ),R(\varepsilon
)]\times S^{1}),\chi ([-l,l])\right) <\delta \left( \varepsilon \right) ,$

\item $d_{C^{\infty }\left( \pm \lbrack \frac{1}{2\pi }\ln \zeta _{0},\infty
)\times S^{1}\right) }\left( u(\cdot \pm \tau \left( \varepsilon \right)
,\cdot ),u_{\pm }\right) <\delta \left( \varepsilon \right) ,$

\item $\func{diam}\left( u\left( \pm \left[ R\left( \varepsilon \right)
,\tau \left( \varepsilon \right) +\frac{1}{2\pi }\ln \zeta _{0}\right]
\times S^{1}\right) \right) <\delta \left( \varepsilon \right) .$
\end{enumerate}

By definition of $V_{\zeta ,\delta }^{\varepsilon }$, any element $u\in
V_{\zeta ,\delta }^{\varepsilon }$ can be expressed as
\begin{equation}
u\left( \tau ,t\right) =\exp _{u_{\textrm{app}}^{\varepsilon }}(\xi )\left( \tau
,t\right)  \label{eq:uexpxi}
\end{equation}%
with $u_{\textrm{app}}^{\varepsilon }=\text{preG}(\varepsilon ,u_{-},\chi ,u_{+})$
for some
\begin{equation*}
(u_{-},\chi ,u_{+})\in {\mathcal{M}}^{\text{\textrm{tft}}%
}(K_{-},J_{-};f,K_{+},J_{+};z_{-},z_{+};B)
\end{equation*}%
and $\xi \in \Gamma \left( \left( u_{\textrm{app}}^{\varepsilon }\right) ^{\ast
}TM\right) $ for some $\xi $. More precisely we introduce the off-shell
version of the pre-gluing map
\begin{equation*}
\widetilde{\text{preG}}:((u_{-},\xi _{-}),(\chi ,\xi _{0}),(u_{+},\xi
_{+}),\mu )\mapsto \text{preG}\left( \exp _{u_{-}}(\xi _{-}),\exp _{\chi
}(\xi _{0}),\exp _{u_{+}}(\xi _{+})\right) \left( P_{\mu }\tau ,t\right)
\end{equation*}%
by the same formula for preG as \eqref{eq:uappe} with $u_{\pm }$ and $\chi $
replaced by $\exp _{u_{\pm }}(\xi _{\pm })$ and $\exp _{\chi }(\xi _{0})$
respectively, and $\mu \in T_{R\left( \varepsilon \right) }\mathbb{R}_{+}$
corresponds to the reparameterization of $u_{\textrm{app}}^{\varepsilon }$ for $%
\left( \tau ,t\right) \in \left[ -R\left( \varepsilon \right) ,R\left(
\varepsilon \right) \right] \times S^{1}$ by $\left( \tau ^{\prime
},t\right) \in \left[ -R\left( \varepsilon \right) +\mu ,R\left( \varepsilon
\right) -\mu \right] \times S^{1}$ as in $\left( \ref{reparameterize-tau}%
\right) $.

The following is an immediate translation of the stable map convergence
together with adiabatic convergence result in Theorem \ref{thm:centrallimit}.

\begin{lem}
There exists some $\varepsilon _{0}>0$ such that for any $0<\varepsilon \leq
\varepsilon _{0}$ any
\begin{equation*}
u\in \mathcal{M}(K_{\varepsilon },J_{\varepsilon };z_{-},z_{+};B)\cap
\bigcup_{0<\zeta \leq \zeta _{0}}V_{\zeta ,\delta (\varepsilon
)}^{\varepsilon },
\end{equation*}%
can be expressed as
\begin{equation*}
u(\tau ,t)=\widetilde{\text{preG}}((u_{-},\xi _{-}),(\chi ,\xi
_{0}),(u_{+},\xi _{+}),\mu )
\end{equation*}%
for some $(u_{-},\chi ,u_{+})$ such that
\begin{equation}
\Vert \xi _{\pm }\Vert _{L^{\infty }}\leq \delta ,\quad \Vert \xi _{0}\Vert
_{L^{\infty }}\leq \delta ,\text{ \ }\left\vert \mu \right\vert \leq \delta .
\label{eq:Linftydelta}
\end{equation}
\end{lem}

By the above lemma, for $u\in \mathcal{M}(K_{\varepsilon },J_{\varepsilon
};z_{-},z_{+};B)\cap \bigcup_{0<\zeta \leq \zeta _{0}}V_{\zeta ,\delta
(\varepsilon )}^{\varepsilon }$ with $0<\varepsilon \leq \varepsilon _{0}$,
we can represent $u$ by tangent vectors $\left( \xi _{-},\xi _{0},\xi
_{+},\mu \right) $ via%
\begin{equation*}
u(\tau ,t)=\widetilde{\text{preG}}((u_{-},\xi _{-}),(\chi ,\xi
_{0}),(u_{+},\xi _{+}))(P_{\mu }\tau ,t).
\end{equation*}%
For notation brevity, we let $\xi =\left( \xi _{-},\xi _{0},\xi _{+}\right) $%
.

From now on, $u$ is represented by $\left( \xi ,\mu \right) $. To prove
Proposition \ref{prop:surjective}, it is enough to prove the following via
the local uniqueness property of the gluing construction.

\begin{prop}[Norm-convergence]
\label{prop:norm-converge} \ Let $\Vert \left( \xi ,\mu \right) \Vert
_{\varepsilon }$ be the Banach norm as defined in \eqref{eq:xienorm}. Then
there exists some $0<\zeta _{0}<1$ and a function $\delta :(0,\varepsilon
_{0}]\rightarrow {\mathbb{R}}_{+}$ such that
\begin{equation*}
\Vert \left( \xi ,\mu \right) \Vert _{\varepsilon }\rightarrow 0
\end{equation*}%
uniformly over $u\in \mathcal{M}(K_{\varepsilon },J_{\varepsilon
};z_{-},z_{+};B)\cap \bigcup_{0<\zeta \leq \zeta _{0}}V_{\zeta ,\delta
(\varepsilon )}^{\varepsilon }$ as $\varepsilon \rightarrow 0$.
\end{prop}

The rest of this section is devoted to the proof of this norm convergence.
We can actually take the function $\delta \left( \varepsilon \right) $ to be
$\delta \left( \varepsilon \right) =\varepsilon $ in this section.

We prove this by contradiction. Suppose that the Proposition is not true,
then there exist $0<\varepsilon _{j}\leq \varepsilon _{0},$ $0<\zeta
_{j}\leq \zeta _{0}$ and solutions $u^{\varepsilon _{j}}\in \mathcal{M}%
(K_{\varepsilon _{j}},J_{\varepsilon _{j}};z_{-},z_{+};B^j)$ represented by $%
\left( \xi _{j},\mu _{j}\right) $ such that
\begin{equation*}
d_{\text{\textrm{adia}}}^{\varepsilon _{j},\zeta _{j}}\left( u^{\varepsilon
_{j}},\left( u_{-},\chi ,u_{+}\right) \right) <\delta \left( \varepsilon
_{j}\right) =\varepsilon _{j}\rightarrow 0,
\end{equation*}%
but $\Vert \left( \xi _{j},\mu _{j}\right) \Vert _{\varepsilon
_{j}}\nrightarrow 0$. By the assumption $u^{\varepsilon _{j}}$, $%
u^{\varepsilon _{j}}$ cannot develop bubbles, so we may assume $%
u^{\varepsilon _{j}}\in \mathcal{M}(K_{\varepsilon _{j}},J_{\varepsilon
_{j}};z_{-},z_{+};B)$ for a fixed homology class $B$, and\ has the uniform $%
C^{1}$ estimate
\begin{equation}
|du^{\varepsilon _{j}}(\tau ,t)|\leq C<\infty .  \label{eq:du<C}
\end{equation}%
Therefore the sequence $u^{\varepsilon _{j}}$ is pre-compact on any given
compact interval $[a,b]\times S^{1}$.

\subsection{Exponential map and adiabatic renormalization}

In general, for solution $u^{\varepsilon _{j}}$ on ${\mathbb{R}}\times S^{1}$%
, \ we define $\xi _{j}^{app}\left( \tau ,t\right) \in
T_{u_{\textrm{app}}^{\varepsilon _{j}}\left( \tau ,t\right) }M$ $\ $by
\begin{equation*}
u^{\varepsilon _{j}}\left( \tau ,t\right) =\exp _{u_{\textrm{app}}^{\varepsilon
_{j}}\left( \tau ,t\right) }\xi _{j}^{app}\left( \tau ,t\right) .
\end{equation*}%
The exponential map makes sense because of the adiabatic convergence $%
u^{\varepsilon _{j}}\rightarrow \left( u_{-},\chi ,u_{+}\right) $.

For the later purpose, it turns out to be important to use the exponential
map at the \emph{genuine solution}
\begin{equation*}
u_{\textrm{glue}}^{\varepsilon _{j}}:=\func{Glue}(u_{-},\chi ,u_{+};\varepsilon _{j}).
\end{equation*}%
We also express the same sequence $u^{\varepsilon _{j}}$ as
\begin{equation}
u^{\varepsilon _{j}}\left( \tau ,t\right) =\exp _{u_{\textrm{glue}}^{\varepsilon
_{j}}\left( \tau ,t\right) }\xi _{j}\left( \tau ,t\right)
\label{eq:glue-xij}
\end{equation}%
for $\xi _{j}\left( \tau ,t\right) \in T_{u_{\textrm{glue}}^{\varepsilon _{j}}\left(
\tau ,t\right) }M$. By the triangle inequality, the error estimates %
\eqref{eq:error} and the construction of perturbation in implicit function
theorem, we have
\begin{equation}
\left\Vert \xi _{\textrm{app}}^{\varepsilon _{j}}-\func{Pal}_{\textrm{app}}^{glue}\left(
\varepsilon _{j}\right) \left( \xi _{\textrm{glue}}^{\varepsilon _{j}}\right)
\right\Vert _{\varepsilon _{j}}\leq \Vert E(u_{\textrm{app}}^{\varepsilon
_{j}},u_{\textrm{glue}}^{\varepsilon _{j}})\Vert _{\varepsilon _{j}}\leq CE(l
)\varepsilon _{j}^{\frac{1}{p}}  \label{eq:uapp-uglue}
\end{equation}%
where
\begin{equation*}
E(u_{\textrm{app}}^{\varepsilon _{j}},u_{\textrm{glue}}^{\varepsilon _{j}})=\exp
_{u_{\textrm{app}}^{\varepsilon _{j}}}^{-1}(u_{\textrm{glue}}^{\varepsilon _{j}}),
\end{equation*}%
(see \eqref{eq:map-E}), and $\func{Pal}_{\textrm{app}}^{glue}\left( \varepsilon
_{j}\right) :T_{u_{\textrm{glue}}^{\varepsilon _{j}}\left( \tau ,t\right)
}M\rightarrow T_{u_{\textrm{app}}^{\varepsilon _{j}}\left( \tau ,t\right) }M$ is the
parallel transport from $u_{\textrm{glue}}^{\varepsilon _{j}}\left( \tau ,t\right) $
to $u_{\textrm{app}}^{\varepsilon _{j}}\left( \tau ,t\right) $ along the minimal
geodesic connecting them. Therefore
\begin{eqnarray*}
\Vert \xi _{\textrm{app}}^{\varepsilon _{j}}\Vert _{\varepsilon _{j}} &\leq
&\left\Vert \func{Pal}_{\textrm{app}}^{glue}\left( \varepsilon _{j}\right) \left( \xi
_{\textrm{glue}}^{\varepsilon _{j}}\right) \right\Vert _{\varepsilon _{j}}+\Vert
E(u_{\textrm{app}}^{\varepsilon _{j}},u_{\textrm{glue}}^{\varepsilon _{j}})\Vert _{\varepsilon
_{j}} \\
&\leq &C\Vert \xi _{\textrm{glue}}^{\varepsilon _{j}}\Vert _{\varepsilon _{j}}+CE(l
)\varepsilon _{j}^{\frac{1}{p}}.
\end{eqnarray*}%
Now for the proof of Proposition \ref{prop:norm-converge} it will be enough to prove
the following.

\begin{prop}\label{prop:norm-convergence2}
Let $\xi_j: = \xi _{\textrm{glue}}^{\varepsilon _{j}}$ be as in \eqref{eq:glue-xij}. Then
\begin{equation*}
\Vert \xi _{\textrm{glue}}^{\varepsilon _{j}}\Vert _{\varepsilon _{j}}\rightarrow 0.
\end{equation*}
\end{prop}

The remaining subsection and the next will be occupied by the proof of this proposition.

We consider the exponential map
\begin{equation*}
\exp :{\mathcal{U}}\subset TM\rightarrow M;\quad \exp (x,\xi ):=\exp
_{x}(\xi )
\end{equation*}%
and denote
\begin{equation*}
D_{1}\exp (x,\xi ):T_{x}M\rightarrow T_{x}M
\end{equation*}%
the (covariant) partial derivative with respect to $x$ and
\begin{equation*}
d_{2}\exp (x,\xi ):T_{x}M\rightarrow T_{x}M
\end{equation*}%
the usual derivative $d_{2}\exp (x,\xi ):=T_{\xi }\exp
_{x}:T_{x}M\rightarrow T_{x}M$. We recall the basic property of the
exponential map
\begin{equation*}
D_{1}\exp (x,0)=d_{2}\exp (x,0)=id.
\end{equation*}%
Denote $\chi _{j}:=\func{cm}(u_{j})$. Then $\chi _{j}\rightarrow \chi $ on $%
[-l ,l ]$ by Theorem \ref{thm:centrallimit}. It is also useful to introduce
the map
\begin{equation}
E:V_{\Delta }\rightarrow {\mathcal{U}};\quad E(x,y)=\exp _{x}^{-1}(y)
\label{eq:map-E}
\end{equation}%
where $V_{\Delta }$ is the neighborhood of the diagonal of $M\times M$. Then
the following is the standard estimates on this map
\begin{eqnarray}
|d_{2}\exp (x,v)-\Pi _{x}^{\exp (x,v)}| &\leq &C|v|  \notag
\label{exp-derivatives} \\
|D_{1}\exp (x,v)-\Pi _{x}^{\exp (x,v)}| &\leq &C|v|
\end{eqnarray}%
for $v\in T_{x}M$ where $C$ is independent of $v$, as long as $|v|$ is
sufficiently small, say smaller than the injectivity radius of the metric on
$M$. (see \cite{katcher}.)

In the following calculations, to simplify the notation, we \emph{suppress
the subindex }$j$\emph{\ from various notations}, i.e. the $\varepsilon ,\xi
$ mean $\varepsilon _{j},\xi _{j}$ respectively. We compute
\begin{equation*}
\frac{\partial u^{\varepsilon }}{\partial \tau }=D_{1}\exp
(u_{\textrm{glue}}^{\varepsilon },\xi )\frac{\partial u_{\textrm{glue}}^{\varepsilon }}{%
\partial \tau }+d_{2}\exp (u_{\textrm{glue}}^{\varepsilon },\xi )\frac{\partial \xi }{%
\partial \tau }.
\end{equation*}%
Similarly we compute
\begin{equation*}
\frac{\partial u^{\varepsilon }}{\partial t}=D_{1}\exp
(u_{\textrm{glue}}^{\varepsilon },\xi )\frac{\partial u_{\textrm{glue}}^{\varepsilon }}{%
\partial t}+d_{2}\exp (u_{\textrm{glue}}^{\varepsilon },\xi )\frac{\partial \xi }{%
\partial t}.
\end{equation*}%
We introduce the following invertible linear operators $P(x,v):T_{x}M%
\rightarrow T_{x}M$ defined by
\begin{equation*}
P(x,v):=(d_{2}\exp (x,v))^{-1}\circ D_{2}\exp (x,v).
\end{equation*}%
Then we have the following inequality

\begin{lem}
\label{lem:P(x,v)}
\begin{equation*}
|P(x,v)-id |\leq C|v|
\end{equation*}%
for a universal constant $C>0$ depending only on the injectivity radius.
\end{lem}

\begin{proof}
This is an immediate consequence of \eqref{exp-derivatives}.
\end{proof}

Now we consider the equation
\begin{equation*}
0=\frac{\partial u^{\varepsilon }}{\partial \tau }+J^{\varepsilon }\left(
\frac{\partial u^{\varepsilon }}{\partial t}-X_{H^{\varepsilon
}}(u^{\varepsilon })\right) .
\end{equation*}%
We re-write
\begin{eqnarray}
0 &=&\left( D_{1}\exp (u_{\textrm{glue}}^{\varepsilon },\xi )\frac{\partial
u_{\textrm{glue}}^{\varepsilon }}{\partial t}+d_{2}\exp (u_{\textrm{glue}}^{\varepsilon },\xi )%
\frac{D\xi }{\partial t}\right)   \notag \\
&{}&\quad +J^{\varepsilon }\left( D_{1}\exp (u_{\textrm{glue}}^{\varepsilon },\xi )%
\frac{\partial u_{\textrm{glue}}^{\varepsilon }}{\partial t}+d_{2}\exp
(u_{\textrm{glue}}^{\varepsilon },\xi )\frac{D\xi }{\partial t}-X_{H^{\varepsilon
}}(u_{\textrm{glue}}^{\varepsilon })\right)   \notag \\
&=&d_{2}\exp (u_{\textrm{glue}}^{\varepsilon },\xi )\frac{D\xi }{\partial t}%
+J^{\varepsilon }\left( d_{2}\exp (u_{\textrm{glue}}^{\varepsilon },\xi )\frac{D\xi }{%
\partial t}-DX_{H^{\varepsilon }}(u_{\textrm{glue}}^{\varepsilon })(\xi ))\right)
\notag \\
&{}&+\left( D_{1}\exp (u_{\textrm{glue}}^{\varepsilon },\xi )\frac{\partial
u_{\textrm{glue}}^{\varepsilon }}{\partial t}+J^{\varepsilon }\left( D_{1}\exp
(u_{\textrm{glue}}^{\varepsilon },\xi )\frac{\partial u_{\textrm{glue}}^{\varepsilon }}{%
\partial t}-X_{H^{\varepsilon }}(u_{\textrm{glue}}^{\varepsilon })\right) \right)
\notag \\
&{}&+J^{\varepsilon }N(u_{\textrm{glue}}^{\varepsilon },\xi )  \label{eq:CR-in-psi}
\end{eqnarray}%
where $N(u_{\textrm{glue}}^{\varepsilon },\xi )$ is the higher order term
\begin{eqnarray}
N(u_{\textrm{glue}}^{\varepsilon },\xi ) &=&X_{H^{\varepsilon }}(u^{\varepsilon
})-X_{H^{\varepsilon }}(u_{\textrm{glue}}^{\varepsilon })-DX_{H^{\varepsilon
}}(u_{\textrm{glue}}^{\varepsilon })(\xi )  \notag  \label{eq:Nuepsilon} \\
&=&X_{H^{\varepsilon }}(\exp (u_{\textrm{glue}}^{\varepsilon },\xi
))-X_{H^{\varepsilon }}(u_{\textrm{glue}}^{\varepsilon })-DX_{H^{\varepsilon
}}(u_{\textrm{glue}}^{\varepsilon })(\xi )
\end{eqnarray}%
obtained from the (pointwise) Taylor expansion of $X_{H^{\varepsilon }}(\exp
(x,v))$.

Now we denote the pull-back
\begin{equation*}
\widehat{J}^{\varepsilon }:=(d_{2}\exp (u_{\textrm{glue}}^{\varepsilon },\xi
))^{-1}J^{\varepsilon }d_{2}\exp (u_{\textrm{glue}}^{\varepsilon },\xi )
\end{equation*}%
and
\begin{equation*}
\widetilde{J}^{\varepsilon }:=(P(u_{\textrm{glue}}^{\varepsilon },\xi
))^{-1}J^{\varepsilon }P(u_{\textrm{glue}}^{\varepsilon },\xi ).
\end{equation*}%
Then we obtain the pointwise inequalities
\begin{equation}
|\widehat{J}^{\varepsilon }-J^{\varepsilon }|\leq C|\xi |,\quad |\widetilde{J%
}^{\varepsilon }-J^{\varepsilon }|\leq C|\xi |.  \label{eq:inequality-Je}
\end{equation}%
We also have
\begin{eqnarray}
|(D_{1}\exp (u_{\textrm{glue}}^{\varepsilon },\xi ))^{-1}(X_{K^{\varepsilon
}}(u_{\textrm{glue}}^{\varepsilon }))-(X_{K^{\varepsilon }}(u_{\textrm{glue}}^{\varepsilon
}))| &\leq &C_{\varepsilon }|\xi |,  \notag \\
|(d_{2}\exp (u_{\textrm{glue}}^{\varepsilon },\xi ))^{-1}(DX_{K^{\varepsilon
}}(u_{\textrm{glue}}^{\varepsilon }))-(DX_{K^{\varepsilon }}(u_{\textrm{glue}}^{\varepsilon
}))| &\leq &C_{\varepsilon }|\xi |.  \label{eq:inequality-XHe}
\end{eqnarray}%
With this notation, we can simplify and write \eqref{eq:CR-in-psi} into
\begin{eqnarray}
&{}&\frac{D\xi }{\partial \tau }+\widehat{J}^{\varepsilon }\left( \frac{D\xi
}{\partial t}-(d_{2}\exp (u_{\textrm{glue}}^{\varepsilon }))^{-1}(DX_{K^{\varepsilon
}}(u_{\textrm{glue}}^{\varepsilon })(\xi ))\right)  \notag \\
&=&-P(u_{\textrm{glue}}^{\varepsilon },\xi )\left( \frac{\partial
u_{\textrm{glue}}^{\varepsilon }}{\partial \tau }+\widetilde{J}^{\varepsilon }\left(
\frac{\partial u_{\textrm{glue}}^{\varepsilon }}{\partial t}-(D_{1}\exp
(u_{\textrm{glue}}^{\varepsilon },\xi ))^{-1}(X_{K^{\varepsilon
}}(u_{\textrm{glue}}^{\varepsilon }))\right) \right)  \notag \\
&{}&-d_{2}\exp (u_{\textrm{glue}}^{\varepsilon },\xi )^{-1}(J^{\varepsilon
}N(u_{\textrm{glue}}^{\varepsilon },\xi ))  \label{eq:xi-on-chi}
\end{eqnarray}%
Combining the pointwise inequalities \eqref{eq:inequality-Je}, %
\eqref{eq:inequality-XHe} and the error estimate for $u_{\textrm{glue}}^{\varepsilon
} $, we obtain the differential inequality
\begin{equation*}
\left\vert \frac{D\xi }{\partial t}+\widehat{J}^{\varepsilon }\left( \frac{%
D\xi }{\partial t}-(DX_{K^{\varepsilon }}(u_{\textrm{glue}}^{\varepsilon })(\xi
))\right) \right\vert \leq C(\varepsilon )+C(|\xi |)|\xi |
\end{equation*}%
where $C(\varepsilon )$ is the error term for $u_{\textrm{glue}}^{\varepsilon }$.

For simplicity of exposition, we rewrite this equation (\emph{now with
subindex }$j$) into
\begin{eqnarray}
&{}&\frac{D\xi _{j}}{\partial \tau }+\widehat{J}^{\varepsilon _{j}}\left(
\tau ,t\right) \frac{\partial \xi _{j}}{\partial t}+A_{j}(u_{\textrm{glue}}^{%
\varepsilon _{j}}\left( \tau ,t\right) )\cdot \xi _{j}(\tau ,t)  \notag \\
&=&B_{j}\left( u_{\textrm{glue}}^{\varepsilon _{j}}\left( \tau ,t\right) ,\xi
_{j}(\tau ,t)\right) +E_{\varepsilon _{j}}\left( \tau ,t\right)
\label{eq:simplifiedqxi}
\end{eqnarray}%
where $A_{j},\,B_{j}$ and $E_{\varepsilon _{j}}$ are defined by
\begin{eqnarray*}
A_{j}(x)\xi _{j} &=&-\widehat{J}^{\varepsilon _{j}}DX_{H^{\varepsilon
_{j}}}(x)(\xi _{j}), \\
\quad B_{j}(x,\xi _{j}) &=&-d_{2}\exp (u_{\textrm{glue}}^{\varepsilon },\xi
)^{-1}(J^{\varepsilon _{j}}N^{\varepsilon _{j}}(x,\xi _{j})) \\
E_{\varepsilon _{j}}\left( {\tau },{t}\right)  &=&-P(u_{\textrm{glue}}^{\varepsilon
_{j}},\xi _{j})\left( \frac{\partial u_{\textrm{glue}}^{\varepsilon _{j}}}{\partial
\tau }+\widetilde{J}^{\varepsilon _{j}}\left( \frac{\partial
u_{\textrm{glue}}^{\varepsilon _{j}}}{\partial t}-(D_{1}\exp (u_{\textrm{glue}}^{\varepsilon
_{j}},\xi ))^{-1}(X_{K^{\varepsilon _{j}}}(u_{\textrm{glue}}^{\varepsilon
_{j}}))\right) \right) .
\end{eqnarray*}%
We have
\begin{equation}
|B_{j}(x,\xi _{j})|\leq C\,|\xi _{j}|^{2}  \label{eq:Bjquadratic}
\end{equation}%
for some uniform constant $C$, provided $\Vert \xi _{j}\Vert _{C^{0}}\leq
\text{injectivity radius of the metric $g$}$. (Note that this latter
requirement is automatic since $\Vert \xi _{j}\Vert _{\varepsilon
_{j}}\rightarrow 0$.) The pointwise inequality \eqref{eq:Bjquadratic}
follows from the inequality \eqref{exp-derivatives} and the pointwise
inequality
\begin{equation*}
|X_{H^{\varepsilon }}(\exp _{x}(v))-X_{H^{\varepsilon
}}(x)-DX_{H^{\varepsilon }}(x)(v)|\leq C|v|^{2}
\end{equation*}%
on $M$, where $C$ depends only on $H$, the injectivity radius of the metric as long
as $|v|\leq \text{injectivity radius of the metric $g$}$. The last inequality just follows
from the Taylor expansion with remainder.

In the transition region $\Omega _{\pm }\left( \varepsilon _{j}\right) :=\pm %
\left[ R(\varepsilon _{j}),\tau \left( \varepsilon _{j}\right) \right]
\times S^{1}$, we do not renormalize but consider \eqref{eq:simplifiedqxi}
itself. From the adiabatic convergence of $u_{j}\rightarrow \left(
u_{-},\chi ,u_{+}\right) $ we have
\begin{equation}
\left\vert \xi _{j}\left( \pm R\left( \varepsilon _{j}\right) ,t\right)
\right\vert ,\,|\xi _{j}(\pm \tau \left( \varepsilon _{j}\right) ,t)|\leq
\delta _{j}  \label{eq:xibdyatRejSej}
\end{equation}%
where
\begin{equation*}
\delta _{j}=\varepsilon _{j}\rightarrow 0\text{ \ \ as }j\rightarrow \infty .
\end{equation*}

We next consider the region $\Theta \left( \varepsilon _{j}\right)
=[-R(\varepsilon _{j}),R(\varepsilon _{j})]\times S^{1}$ for $\left( \tau
,t\right) $. By renomalizing the domain%
\begin{equation*}
(\overline{\tau },\overline{t})=(\varepsilon _{j}\tau ,\varepsilon _{j}t),\,%
\text{\ }\overline{u_{j}}\left( \overline{\tau },\overline{t}\right)
=u_{j}\left( \frac{\overline{\tau }}{\varepsilon _{j}},\frac{\overline{t}}{%
\varepsilon _{j}}\right) =u_{j}\left( \tau ,t\right) ,
\end{equation*}%
and applying the Hausdorff convergence of $u_{j}|_{[-R(\varepsilon
_{j}),R(\varepsilon _{j})]\times S^{1}}$ to $\chi |_{[-l,l]}$, we can write $%
\overline{u}_{j}$
\begin{equation}
\overline{u}_{j}(\overline{\tau },\overline{t})=\exp _{\chi (\overline{\tau }%
)}\overline{\xi }_{j}(\overline{\tau },\overline{t})  \label{eq:baruexpxi}
\end{equation}%
for $(\overline{\tau },\overline{t})\in \lbrack -l,l]\times {\mathbb{R}}%
/2\pi \varepsilon _{j}{\mathbb{Z}}$, a cylinder with radius $\varepsilon _{j}
$. Here we have
\begin{equation*}
\overline{\xi }_{j}(\overline{\tau },\overline{t})\in T_{\chi (\overline{%
\tau })}M.
\end{equation*}%
If we restrict $\left( \tau ,t\right) $ on $[-R(\varepsilon
_{j}),R(\varepsilon _{j})]\times S^{1}$, then $\xi _{j}\left( \tau ,t\right)
=$ $\overline{\xi }_{j}(\overline{\tau },\overline{t})$.

\bigskip Furthermore it easily follows from $\left( \ref{exp-derivatives}%
\right) $ and smoothness of the exponential map that%
\begin{eqnarray*}
&{}&|d_{2}\exp (\chi (\overline{\tau }),\overline{\xi }_{j}(\overline{\tau },%
\overline{t})))^{-1}(D_{1}\exp (\chi (\overline{\tau }),\overline{\xi }_{j}(%
\overline{\tau },\overline{t}))(\dot{\chi}(\overline{\tau }))-\dot{\chi}(%
\overline{\tau })| \\
&\leq &C|\overline{\xi }(\tau ,t)||\dot{\chi}(\overline{\tau })|
\end{eqnarray*}%
and
\begin{equation*}
|\exp _{\chi (\overline{\tau })}^{\ast }(\func{grad}f)(\overline{\xi }_{j}(%
\overline{\tau },\overline{t}))-\func{grad}f(\chi (\overline{\tau }))|\leq C|%
\overline{\xi }_{j}(\overline{\tau },\overline{t})|
\end{equation*}%
and
\begin{equation*}
|(\exp _{\chi (\overline{\tau })})^{\ast }J^{\varepsilon _{j}}(\chi (%
\overline{\tau }))-J^{\varepsilon _{j}}|\leq C|\overline{\xi }_{j}(\overline{%
\tau },\overline{t})|
\end{equation*}%
where the constant $C$ depends only on $M$. Therefore%
\begin{equation*}
\lim_{j\rightarrow \infty }\widehat{J}^{\varepsilon
_{j}}(u_{\textrm{glue}}^{\varepsilon _{j}})=J_{0}=\lim_{j\rightarrow \infty }%
\widetilde{J}^{\varepsilon _{j}}(u_{\textrm{glue}}^{\varepsilon _{j}}).
\end{equation*}

\subsection{Three-interval method: $L^2$-exponential estimates}

With those preparations made in the previous subsection, we will now prove
Proposition \ref{prop:norm-convergence2} by applying the three-interval method
to the equation \eqref{eq:simplifiedqxi}.

We first study the equation of $\xi _{j}$ on the transition regions
\begin{equation*}
\Omega _{\pm }\left( \varepsilon _{j}\right) :=\pm \left[ R(\varepsilon
_{j}),\tau \left( \varepsilon _{j}\right) \right] \times S^{1}.
\end{equation*}

Let $\delta _{0}>0$ be a number much smaller than the injectivity radius of $%
M$, and%
\begin{equation*}
h\left( \zeta \right) :=\frac{1}{2\pi }\left\vert \ln \zeta \right\vert >0.
\end{equation*}%
By the adiabatic convergence condition (3), (4), there exists small $0<\zeta
_{0}<1$ and integer $j_{0}$, such that for all $j\geq $ $j_{0}$ and $\delta
_{j}=\delta \left( \varepsilon _{j}\right) $,
\begin{eqnarray}
\text{dist }\left( u_{j}\left( \pm R\left( \varepsilon _{j}\right) \right)
,p_{\pm }\right) &<&\delta _{j},  \notag \\
d_{C^{\infty }\left( \pm \left[ \tau \left( \varepsilon _{j}\right) -h\left(
\zeta _{0}\right) ,\tau \left( \varepsilon _{j}\right) -h\left( \zeta
_{0}\right) +1\right] \times S^{1}\right) }\left( u_{j},u_{\pm
}^{\varepsilon _{j}}\right) &<&\delta _{j},\text{ }  \notag \\
\text{diam}\left( u_{j}\left( \pm \left[ R(\varepsilon _{j}),\tau \left(
\varepsilon _{j}\right) -h\left( \zeta _{0}\right) +1\right] \times
S^{1}\right) \right) &<&\delta _{j}.  \label{3-distances-small}
\end{eqnarray}%
We denote the region
\begin{equation*}
\Omega _{\pm \zeta _{0}}\left( \varepsilon _{j}\right) :=\pm \left[
R(\varepsilon _{j}),\tau \left( \varepsilon _{j}\right) -h\left( \zeta
_{0}\right) +1\right] \times S^{1}\subset \Omega _{\pm }\left( \varepsilon
_{j}\right) .
\end{equation*}%
Without loss of generality we assume $\delta _{j}\leq \delta _{0}$. Then for
$\varepsilon _{j}$ small, from $\left( \ref{3-distances-small}\right) $ we
have%
\begin{equation}
u_{j}\left( \Omega _{\pm \zeta _{0}}\left( \varepsilon _{j}\right) \right)
\subset B_{p_{\pm }}\left( 2\delta _{0}\right) .  \label{transition-small}
\end{equation}%
If we identify the neighborhood $B_{p_{\pm }}\left( 2\delta _{0}\right) $
into $T_{p_{\pm }}M$ by exponential map and deform the metric and almost
complex structure to standard $\left( g_{p_{\pm }},J_{p_{\pm }}\right) $,
then the $\xi _{j}$ can be simplified to%
\begin{equation*}
\xi _{j}=u_{j}-u_{\textrm{glue}}^{\varepsilon _{j}},
\end{equation*}%
where the \textquotedblleft $-$\textquotedblright\ is with respect to the
linear space structure on $T_{p_{\pm }}M$. Such simplification will not
affect the validity of the proof, as explained in section 8 and remark 11.1,
for it will only affect the $C^{1}$ pointwise estimate of a term of order $%
C\delta _{0}$.

We decompose the equation \eqref{eq:simplifiedqxi} into those of $0$-mode
and higher modes%
\begin{eqnarray}
\frac{\partial }{\partial \tau }\left( \xi _{j}\right) _{0} &=&\left(
B\right) _{0}\left( u_{\textrm{glue}}^{\varepsilon _{j}}\left( \tau ,t\right) ,\xi
_{j}\right) +\left( E_{\varepsilon _{j}}\right) _{0},  \notag \\
\overline{\partial }_{J}\widetilde{\xi _{j}} &=&\widetilde{B}\left(
u_{\textrm{glue}}^{\varepsilon _{j}}\left( \tau ,t\right) ,\xi _{j}\right) +%
\widetilde{E_{\varepsilon _{j}}}.  \label{eq:decouple-Floer}
\end{eqnarray}%
%
%
%
%
%
%
%
%
%
%
%
%
%
%
% and $%
%\widetilde{E_{\varepsilon }}\left( \tau ,t\right) $ supported in $%
%l/\varepsilon _{j}\leq \left\vert \tau \right\vert \leq l/\varepsilon _{j}+1$%
%, and%
%\begin{eqnarray*}
%\left\vert \left( E_{\varepsilon }\right) _{0}\right\vert &\leq &C\left\vert
%u_{\pm }^{\varepsilon _{j}}\left( \tau ,t\right) -\chi \left( \pm l\right)
%\right\vert \leq C\varepsilon ^{\frac{p-1}{\delta }}, \\
%\left\vert \widetilde{E_{\varepsilon }}\right\vert &\leq &\left\vert \chi
%\left( \varepsilon _{j}\tau \right) -\chi \left( \pm l\right) \right\vert
%\leq C\varepsilon .
%\end{eqnarray*}

\begin{rem}
In Theorem 1.2 of \cite{mundet-tian}, the higher mode exponential decay
estimate
\begin{equation}
\left\vert \widetilde{\xi _{j}}\right\vert \leq Ce^{-\sigma \left(
l/\varepsilon _{j}-\left\vert \tau \right\vert \right) }\text{ for }%
\left\vert \tau \right\vert \leq l/\varepsilon _{j}  \label{high-mode-decay}
\end{equation}%
has been obtained (Their notation for higher mode is $\phi _{0}\left(
t,\theta \right) $ instead of our $\widetilde{\xi _{j}}\left( \tau ,t\right)
$). Their observation was that $\left( \ref{high-mode-decay}\right) $ can be
reduced to a local $L^{2}$ estimate
\begin{equation}
\left\Vert \widetilde{\xi _{j}}\right\Vert _{L^{2}\left( Z_{II}\right) }\leq
\frac{1}{2}\left( \left\Vert \widetilde{\xi _{j}}\right\Vert _{L^{2}\left(
Z_{I}\right) }+\left\Vert \widetilde{\xi _{j}}\right\Vert _{L^{2}\left(
Z_{III}\right) }\right)  \label{local-elliptic-estimate}
\end{equation}%
on 3 sequential cylinders $Z_{I},Z_{II},Z_{III}\subset \left[ -l/\varepsilon
_{j},l/\varepsilon _{j}\right] \times $ $S^{1}$ of unit length, namely the
cylinders
\begin{equation*}
\left[ i-1,i\right] \times S^{1},\left[ i,i+1\right] \times S^{1},\left[
i+1,i+2\right] \times S^{1}
\end{equation*}%
for some integer $i$.
\end{rem}

To get the best $\sigma $ in the exponential decay (we need $\sigma $ to be
very close to $2\pi $), we recall in \cite{mundet-tian} they defined the
constant%
\begin{equation*}
\gamma \left( c\right) =\frac{1}{e^{c}+e^{-c}}.
\end{equation*}%
The importance of $\gamma \left( c\right) $ is due to the identity
\begin{equation}
\int_{0}^{1}e^{c\tau }d\tau =\gamma \left( c\right) \left[
\int_{-1}^{0}e^{c\tau }d\tau +\int_{1}^{2}e^{c\tau }d\tau \right]
\label{identity}
\end{equation}%
which will appear in the $L^{2}$-energy of $\xi _{j}$ on $3$ sequential unit
length cylinders later. Notice that when $c>0$, $\gamma \left( c\right) $ is
a strictly decreasing function of $c$.

We recall the following elementary but useful lemma.

\begin{lem}[\protect\cite{mundet-tian} Lemma 9.4]
For nonnegative numbers $x_{k}$ ($k=0,1,\cdots N$), if $\ $for $1\leq k\leq
N-1$,
\begin{equation*}
x_{k}\leq \gamma \left( x_{k-1}+x_{k+1}\right)
\end{equation*}%
for some fixed constant $\gamma \in (0,1/2)$, then for $1\leq k\leq N-1$,
\begin{equation*}
x_{k}\leq x_{0}\xi ^{-k}+x_{N}\xi ^{-\left( N-k\right) },
\end{equation*}
where $\xi =\frac{1+\sqrt{1-4\gamma ^{2}}}{2\gamma }$.
\end{lem}

\begin{rem}
If $\gamma =\gamma \left( c\right) =\left( e^{c}+e^{-c}\right) ^{-1}$, then
we can check $\xi =e^{c}$ and the above inequality becomes the exponential
decay estimate%
\begin{equation}
x_{k}\leq x_{0}e^{-ck}+x_{N}e^{-c\left( N-k\right) }\text{.}
\label{L2-exp-decay}
\end{equation}%
for $1\leq k\leq N-1$.
\end{rem}

Based on the above lemma, we first prove the following.

\begin{prop}\label{prop:three-interval}
For any $0<\upsilon <1$, there exists $N_{0}=N_{0}\left( \upsilon \right) $
depending on $\upsilon $, such that for all $j>N_{0}$, on any $3$ sequential
cylinders $Z_{I},Z_{II}$ and $Z_{III}$ in $\left[ -\tau \left( \varepsilon
\right) ,\tau (\varepsilon )\right] \times S^{1}$ we have%
\begin{equation}
\left\Vert d\xi _{j}\right\Vert _{L^{2}\left( Z_{II}\right) }\leq \gamma
\left( 4\pi \upsilon \right) \cdot \left( \left\Vert d\xi _{j}\right\Vert
_{L^{2}\left( Z_{I}\right) }+\left\Vert d\xi _{j}\right\Vert _{L^{2}\left(
Z_{III}\right) }\right)  \label{local-elliptic-estimate-g}
\end{equation}%
and on any $3$ sequential cylinders in $\left[ -R\left( \varepsilon \right)
,R(\varepsilon )\right] \times S^{1}$ we have
\begin{equation}
\left\Vert \widetilde{\xi _{j}}\right\Vert _{L^{2}\left( Z_{II}\right) }\leq
\gamma \left( 4\pi \upsilon \right) \cdot \left( \left\Vert \widetilde{\xi
_{j}}\right\Vert _{L^{2}\left( Z_{I}\right) }+\left\Vert \widetilde{\xi _{j}}%
\right\Vert _{L^{2}\left( Z_{III}\right) }\right) .
\label{local-elliptic-estimate-h}
\end{equation}
\end{prop}
\begin{proof}
We prove $\left( \ref{local-elliptic-estimate-g}\right) $ by contradiction.
Suppose that for some sequence $\xi _{j}\rightarrow 0$ $\left( \ref%
{local-elliptic-estimate-g}\right) $ is violated on $3$ sequential cylinders
$Z_{I}^{j},Z_{II}^{j}$ and $Z_{III}^{j}$ in $\left[ -\tau \left( \varepsilon
\right) ,\tau (\varepsilon )\right] \times S^{1}$:
\begin{equation*}
\left\Vert d\xi _{j}\right\Vert _{L^{2}\left( Z_{II}^{j}\right) }>\gamma
\left( 4\pi \upsilon \right) \left( \left\Vert d\xi _{j}\right\Vert
_{L^{2}\left( Z_{I}^{j}\right) }+\left\Vert d\xi _{j}\right\Vert
_{L^{2}\left( Z_{III}^{j}\right) }\right) .
\end{equation*}%
On $Z_{I}^{j}\cup Z_{II}^{j}\cup Z_{III}^{j}$ consider the rescaled sequence%
\begin{equation*}
\widehat{\xi _{j}}=\xi _{j}/\left\Vert \xi _{j}\right\Vert _{L^{\infty
}\left( Z_{I}^{j}\cup Z_{II}^{j}\cup Z_{III}^{j}\right) }\text{,}
\end{equation*}%
where the denominator $\left\Vert \xi _{j}\right\Vert _{L^{\infty }\left(
Z_{I}^{j}\cup Z_{II}^{j}\cup Z_{III}^{j}\right) }$ is never $0$, otherwise $%
\xi _{j}\equiv 0$ on $Z_{I}^{j}\cup Z_{II}^{j}\cup Z_{III}^{j}$,
contradicting our assumption. Using $\left( \ref{eq:simplifiedqxi}\right) $
we have
\begin{eqnarray*}
\left\vert \widehat{\xi _{j}}\right\vert _{L^{\infty }\left( Z_{I}^{j}\cup
Z_{II}^{j}\cup Z_{III}^{j}\right) } &=&1, \\
\left\Vert d\widehat{\xi _{j}}\right\Vert _{L^{2}\left( Z_{II}^{j}\right) }
&>&\gamma \left( 4\pi \upsilon \right) \cdot \left( \left\Vert d\widehat{\xi
_{j}}\right\Vert _{L^{2}\left( Z_{I}^{j}\right) }+\left\Vert d\widehat{\xi
_{j}}\right\Vert _{L^{2}\left( Z_{III}^{j}\right) }\right) , \\
\frac{D\widehat{\xi _{j}}}{\partial \tau }+\widehat{J}^{\varepsilon
_{j}}\left( \tau ,t\right) \frac{\partial \widehat{\xi _{j}}}{\partial t}%
+A_{j}(u_{\textrm{glue}}^{\varepsilon _{j}}\left( \tau ,t\right) )\cdot \widehat{\xi
_{j}} &=&\left( B_{j}\left( \xi _{j}\right) +E_{\varepsilon _{j}}\right)
/\left\Vert \xi _{j}\right\Vert _{L^{\infty }\left( Z_{I}^{j}\cup
Z_{II}^{j}\cup Z_{III}^{j}\right) }.
\end{eqnarray*}
Note $E_{\varepsilon _{j}}$ is before in  $\left( \ref{eq:simplifiedqxi}\right) $.

We need the following lemma

\begin{lem}
For $\left( \tau ,t\right) \in \lbrack -\tau (\varepsilon _{j}),\tau
(\varepsilon _{j})]\times S^{1}$, we have
\begin{equation}
|E_{\varepsilon _{j}}\left( {\tau },{t}\right) |\leq C|\xi _{j}(\tau
,t)|\left( \left\vert \frac{\partial u_{\textrm{glue}}^{\varepsilon _{j}}}{\partial t}%
(\tau ,t)\right\vert +|X_{K_{\varepsilon _{j}}}|\right) .  \label{E-error}
\end{equation}%
Especially for $Z_{I}^{j}\cup Z_{II}^{j}\cup Z_{III}^{j}\subset \left[ -\tau
\left( \varepsilon _{j}\right) ,\tau \left( \varepsilon _{j}\right) \right]
\times S^{1}$, we have
\begin{equation*}
\lim_{j\rightarrow \infty }E_{\varepsilon _{j}}\left( \tau ,t\right)
/\left\Vert \xi _{j}\right\Vert _{L^{\infty }\left( Z_{I}^{j}\cup
Z_{II}^{j}\cup Z_{III}^{j}\right) }=0.
\end{equation*}
\end{lem}

\begin{proof}
We recall that $u_{\textrm{glue}}^{\varepsilon _{j}}$ is a genuine solution for
\begin{equation}
{\frac{\partial u}{\partial \tau }}+J^{\varepsilon _{j}}\left( {\frac{%
\partial u}{\partial t}}-X_{K^{\varepsilon _{j}}}(u)\right) =0.
\label{eq:perturbedCRKe}
\end{equation}%
Therefore using Lemma \ref{lem:P(x,v)}, we obtain
\begin{equation*}
|E_{\varepsilon _{j}}\left( {\tau },{t}\right) |\leq C|\widetilde{J}%
^{\varepsilon _{j}}-J^{\varepsilon _{j}}|\left\vert \frac{\partial
u_{\textrm{glue}}^{\varepsilon _{j}}}{\partial t}\right\vert +C_{1}|(D_{1}\exp
_{u_{\textrm{glue}}^{\varepsilon _{j}}})^{-1}(\xi _{\varepsilon
_{j}})-id||X_{K_{\varepsilon _{j}}}(u)|.
\end{equation*}%
for some uniform constants $C$ and $C_{1}$. Now we
obtain $\left( \ref{E-error}\right) $ from \eqref{eq:inequality-Je}.

Next we prove the second statement.
By the definition \eqref{eq:Kepsilon} of $K_{\varepsilon _{j}}$,
\begin{equation*}
K_{\varepsilon _{j}}\left( \tau ,t,x\right) =\kappa _{0}^{\varepsilon
_{j}}\left( \tau \right) \cdot \varepsilon _{j}f\left( x\right) \text{ for }%
\left\vert \tau \right\vert \leq \tau \left( \varepsilon _{j}\right) ,
\end{equation*}%
so $|X_{K_{\varepsilon _{j}}}|\rightarrow 0$ as $\varepsilon _{j}\rightarrow
0$. We can also rewrite the equation $\left( \ref{eq:simplifiedqxi}\right) $
as
$$
\frac{D\xi _{j}}{\partial \tau }+\widehat{J}^{\varepsilon
_{j}}\left( \tau ,t\right) \frac{\partial \xi _{j}}{\partial t} =
-A_{j}(u_{\textrm{glue}}^{\varepsilon _{j}}\left( \tau ,t\right) )\cdot \xi
_{j} +  B_{j}(\xi _{j}) +E_{\varepsilon _{j}}.
$$
We also have inequality $\left\vert \xi _{j}\right\vert
_{C^{0}}\leq C\varepsilon _{j}^{1/p}$ by the error estimate$\left( \ref%
{eq:error}\right) $ and Sobolev embedding $\left( \ref{e-Sobolev}\right) $.
Therefore $\xi _{j}:=E\left( u_{\textrm{app}}^{\varepsilon
_{j}},u_{\textrm{glue}}^{\varepsilon _{j}}\right) $ satisfies
\begin{equation}\label{eq:delbar=G}
\frac{D\xi _{j}}{\partial \tau }+\widehat{J}^{\varepsilon
_{j}}\left( \tau ,t\right) \frac{\partial \xi _{j}}{\partial t} = G(\tau,t)
\end{equation}
with a smooth function $G: = -A_{j}(u_{\textrm{glue}}^{\varepsilon _{j}}\left( \tau ,t\right) )\cdot \xi
_{j} +  B_{j}(\xi _{j}) +E_{\varepsilon _{j}}$ which satisfies
$$
|G(\tau,t)| \leq C (\varepsilon_j^{1/p} + \varepsilon_j^{2/p} + \varepsilon_j).
$$
Here the last term follows from the formula for $E_{\varepsilon_j}(\tau,t)$ given
after $\left( \ref{eq:simplifiedqxi}\right) $:  the formula gives rise to
$$
|E_{\varepsilon_j}(\tau,t)| \leq  |P(u_{\textrm{app}}^{\varepsilon },\xi_j)| \left( \text{\textrm ``error term''} +
\left|id - (D_1\exp(u_{\textrm{app}}^{\varepsilon },\xi_j))^{-1}\right|\cdot |X_{K^{\varepsilon_j}}(u_{\textrm{app}}^{\varepsilon }|\right)
$$
which is bounded by $C(\varepsilon^{1/p} + \varepsilon_j)$.

Then by the interior Schauder estimate of $\xi _{j}$ applied to the equation \eqref{eq:delbar=G}
on $\left[ \tau -%
\frac{1}{2},\tau +\frac{1}{2}\right] \times S^{1}\subset \left[ \tau -1,\tau
+1\right] \times S^{1}$ we derive $\left\vert {\frac{\partial }{\partial t}}%
\xi _{j}\left( \tau ,t\right) \right\vert \leq C\varepsilon _{j}^{1/p}$.
Thus for $\left\vert \tau \right\vert \leq \tau \left( \varepsilon
_{j}\right) $, by the triangle inequality we have
\begin{eqnarray*}
\left\vert \frac{\partial u_{\textrm{glue}}^{\varepsilon _{j}}}{\partial t}\left(
\tau ,t\right) \right\vert  &\leq &C\left( \left\vert \frac{\partial
u_{\textrm{app}}^{\varepsilon _{j}}}{\partial t}\left( \tau ,t\right) \right\vert
+\left\vert {\frac{\partial }{\partial t}}\xi _{j}\left( \tau ,t\right)
\right\vert \right)  \\
&\leq &C\left( \varepsilon _{j}+\varepsilon _{j}^{1/p}\right) \rightarrow 0%
\text{,}
\end{eqnarray*}%
where in the last inequality we have used the exponential decay of $%
\left\vert du_{\pm }^{\varepsilon _{j}}\right\vert $ for $\left\vert \tau
\right\vert \leq \tau \left( \varepsilon _{j}\right) $. Therefore
\begin{equation*}
E_{\varepsilon _{j}}\left( \tau ,t\right) /\left\Vert \xi _{j}\right\Vert
_{L^{\infty }\left( Z_{I}^{j}\cup Z_{II}^{j}\cup Z_{III}^{j}\right) }\leq
C\left( \left\vert \frac{\partial u_{\textrm{glue}}^{\varepsilon _{j}}}{\partial t}%
(\tau ,t)\right\vert +|X_{K_{\varepsilon _{j}}}|\right) \rightarrow 0
\end{equation*}%
as $j\rightarrow \infty $. The lemma follows.
\end{proof}

\begin{rem}
Here is the place where we need to use the exponential map around the
genuine solution $u_{\textrm{glue}}^{\varepsilon }$ instead of $u_{\textrm{app}}^{\varepsilon }
$. If we had used the latter, we would not have the estimate given in this
lemma. This is because $u_{\textrm{app}}^{\varepsilon }$ is only an approximate
solution of \eqref{eq:perturbedCRKe} whose error term is
\begin{equation*}
\overline{\partial }_{J,K^{\varepsilon _{j}}}(u_{\textrm{app}}^{\varepsilon _{j}})=%
\frac{\partial u_{\textrm{app}}^{\varepsilon _{j}}}{\partial \tau }+J^{\varepsilon
_{j}}\left( \frac{\partial u_{\textrm{app}}^{\varepsilon _{j}}}{\partial t}%
-X_{K^{\varepsilon _{j}}}(u_{\textrm{app}}^{\varepsilon _{j}})\right),
\end{equation*}%
which is not vanishing in general and will enter equations \eqref{eq:CR-in-psi}$\sim$\eqref{eq:simplifiedqxi},
when we express $u^{\varepsilon _{j}}=\exp
_{u_{\textrm{app}}^{\varepsilon _{j}}}\xi _{\textrm{app}}^{\varepsilon _{j}}$. Then it seems
much harder to get
\begin{equation*}
E_{\varepsilon _{j}}\left( \tau ,t\right) /\left\Vert \xi
_{\textrm{app}}^{\varepsilon _{j}}\right\Vert _{L^{\infty }\left( Z_{I}^{j}\cup
Z_{II}^{j}\cup Z_{III}^{j}\right) }\rightarrow 0,
\end{equation*}%
a key to apply the three-interval method in the following to derive the
desired exponential decay of $\xi _{\textrm{app}}^{\varepsilon _{j}}$.
This is the reason why we first replaced $u_{\textrm{app}}^{\varepsilon _{j}}$ by the
genuine solution $u_{\textrm{glue}}^{\varepsilon _{j}}$ in the exponential map $%
\left( \ref{eq:glue-xij}\right) $ in the beginning of our derivation,
where $u_{\textrm{glue}}^{\varepsilon }$  is better in $\left( \ref{eq:CR-in-psi}\right) $,
because the left side of the equation is $0$, not the above complicated error term.
\end{rem}

Using the lemma and \eqref{eq:Bjquadratic}, we obtain
\begin{equation*}
\frac{|B\left( u_{\textrm{glue}}^{\varepsilon _{j}}\left( \tau ,t\right) ,\xi
_{j}(\tau ,t)\right) +E_{\varepsilon _{j}}\left( \tau ,t\right) |}{\Vert \xi
_{j}\Vert _{L^{\infty }(Z_{I}^{j}\cup Z_{II}^{j}\cup Z_{III}^{j})}}\leq
C\left( |\xi _{j}|+\left\vert \frac{\partial u_{\textrm{glue}}^{\varepsilon _{j}}}{%
\partial t}\right\vert +|X_{K^{\varepsilon _{j}}}|\right) \rightarrow 0
\end{equation*}%
uniformly over $Z_{I}^{j}\cup Z_{II}^{j}\cup Z_{III}^{j}\subset \lbrack
-\left( \tau \left( \varepsilon _{j}\right) -h\left( \zeta _{j}\right)
\right) ,\tau (\varepsilon _{j})-h\left( \zeta _{j}\right) ]\times S^{1}$.

After possibly shifting $Z_{I}^{j}\cup Z_{II}^{j}\cup Z_{III}^{j}$ and
taking subsequence of $\widehat{\xi _{j}}$, we can assume $\widehat{\xi _{j}}
$ $C^{1}$-converges to $\widehat{\xi _{\infty }}$ on a \emph{fixed} $%
Z_{I}\cup Z_{II}\cup Z_{III}$ (this is guaranteed by $C^{0}$ convergence from
our adiabatic convergence definition, and elliptic estimate on a length $5$
cylinder containing $Z_{I}^{j}\cup Z_{II}^{j}\cup Z_{III}^{j}$), which
satisfies%
\begin{eqnarray*}
\left\vert \widehat{\xi _{\infty }}\right\vert _{L^{\infty }\left( Z_{I}\cup
Z_{II}\cup Z_{III}\right) } &=&1, \\
\left\Vert d\widehat{\xi _{\infty }}\right\Vert _{L^{2}\left( Z_{II}\right)
} &\geq &\gamma \left( 4\pi \upsilon \right) \cdot \left( \left\Vert d%
\widehat{\xi _{\infty }}\right\Vert _{L^{2}\left( Z_{I}\right) }+\left\Vert d%
\widehat{\xi _{\infty }}\right\Vert _{L^{2}\left( Z_{III}\right) }\right) ,
\\
\left( \frac{\partial }{\partial \tau }+J_{0}\frac{\partial }{\partial t}%
\right) \widehat{\xi _{\infty }} &=&0.
\end{eqnarray*}%
Then $\widehat{\xi _{\infty }}$ is a nonzero holomorphic function by the
first and third identity. We write $\widehat{\xi _{\infty }}\left( \tau
,t\right) $ in Fourier series
\begin{equation*}
\widehat{\xi _{\infty }}\left( \tau ,t\right) =\Sigma _{k=-\infty }^{\infty
}a_{k}e^{2\pi k\tau }e^{2\pi kit}\text{ with }\left\Vert \widehat{\xi
_{\infty }}\right\Vert _{L^{2}\left( Z_{I}\cup Z_{II}\cup Z_{III}\right)
}\leq 3,
\end{equation*}%
where the $a_{k}$'s are constant vectors in $\mathbb{C}^{n}$. We can
explicitly compute
\begin{equation}
\left\Vert d\widehat{\xi _{\infty }}\right\Vert _{L^{2}\left( \left[ a,b%
\right] \times S^{1}\right) }^{2}=\Sigma _{k=-\infty }^{\infty }4\pi
^{2}k^{2}\left\vert a_{k}\right\vert ^{2}\cdot \int_{a}^{b}e^{4\pi k\tau
}d\tau .  \label{energy}
\end{equation}%
Multiplying $\left( \ref{identity}\right) $ by $e^{4\pi kZ}$ and letting $%
c=4\pi k$ there, we have
\begin{equation}
\int_{Z}^{Z+1}e^{4\pi k\tau }d\tau =\gamma \left( 4\pi k\right) \left[
\int_{Z-1}^{Z}e^{4\pi k\tau }d\tau +\int_{Z+1}^{Z+2}e^{4\pi k\tau }d\tau %
\right] .  \label{identity-shift}
\end{equation}%
By $\left( \ref{energy}\right) ,\left( \ref{identity-shift}\right) $, and $%
\gamma \left( 4\pi k\right) \leq \gamma \left( 4\pi \right) $ for any
integer $k\neq 0$, we see%
\begin{equation*}
\left\Vert d\widehat{\xi _{\infty }}\right\Vert _{L^{2}\left( Z_{II}\right)
}\leq \gamma \left( 4\pi \right) \cdot \left( \left\Vert d\widehat{\xi
_{\infty }}\right\Vert _{L^{2}\left( Z_{I}\right) }+\left\Vert d\widehat{\xi
_{\infty }}\right\Vert _{L^{2}\left( Z_{III}\right) }\right) .
\end{equation*}%
This contradicts with
\begin{equation*}
\left\Vert d\widehat{\xi _{\infty }}\right\Vert _{L^{2}\left( Z_{II}\right)
}\geq \gamma \left( 4\pi \upsilon \right) \cdot \left( \left\Vert d\widehat{%
\xi _{\infty }}\right\Vert _{L^{2}\left( Z_{I}\right) }+\left\Vert d\widehat{%
\xi _{\infty }}\right\Vert _{L^{2}\left( Z_{III}\right) }\right)
\end{equation*}%
since $\gamma \left( 4\pi \right) <\gamma \left( 4\pi \upsilon \right) $.
The proof of $\left( \ref{local-elliptic-estimate-h}\right) $ is similar.
The crucial point is that $\widetilde{\xi _{j}}$ contains no $0$-mode, so
the rescaled sequence \ $\widehat{\widetilde{\xi _{j}}}:=$ $\widetilde{\xi
_{j}}/\left\Vert \widetilde{\xi _{j}}\right\Vert _{L^{\infty }\left(
Z_{I}^{j}\cup Z_{II}^{j}\cup Z_{III}^{j}\right) }$ and its limit \ $\widehat{%
\widetilde{\xi _{\infty }}}$ \ are in the higher mode space. Since\ $%
\widehat{\widetilde{\xi _{\infty }}}$ $\ $is holomorphic, writing it in
Fourier series $\widehat{\widetilde{\xi _{\infty }}}\left( \tau ,t\right)
=\Sigma _{k\neq 0}b_{k}e^{2\pi k\tau }e^{2\pi kit}$, we have%
\begin{equation*}
\left\Vert \widehat{\widetilde{\xi _{\infty }}}\right\Vert _{L^{2}\left( %
\left[ a,b\right] \times S^{1}\right) }^{2}=\Sigma _{k\neq 0}\left\vert
b_{k}\right\vert ^{2}\cdot \int_{a}^{b}e^{4\pi \tau }d\tau .
\end{equation*}%
This is similar to $\left( \ref{energy}\right) $, and the remaining steps
are the same. We omit the details. This finishes the proof of Proposition \ref{prop:three-interval}.
\end{proof}

Combining the above Lemma and $\left( \ref{L2-exp-decay}\right) $ we have

\begin{cor}
For any $0<\upsilon <1$, there exists $N_{0}=N_{0}\left( \upsilon \right) $
depending on $\upsilon$, such that for all $j>N_{0}$ and $\left[ \tau,\tau+1%
\right] \subset \left[ -\tau \left( \varepsilon _{j}\right) ,\tau \left(
\varepsilon _{j}\right) \right] $, we have
\begin{eqnarray*}
&{}& \int_{\left[ \tau,\tau+1\right] \times S^{1}}\left\vert d\xi
_{j}\right\vert ^{2} \, dt\, d\tau \\
&{}& \quad \leq e^{-4\pi \upsilon \left( \tau \left( \varepsilon _{j}\right)
-\left\vert \tau \right\vert \right) }\left[ \int_{\left[ -\tau \left(
\varepsilon _{j}\right) -1,-\tau \left( \varepsilon _{j}\right) \right]
\times S^{1}}\left\vert d\xi _{j}\right\vert ^{2}+\int_{\left[ \tau \left(
\varepsilon _{j}\right) ,\tau \left( \varepsilon _{j}\right) +1\right]
\times S^{1}}\left\vert d\xi _{j}\right\vert ^{2}\right] ,
\end{eqnarray*}
and for $\left[ \tau,\tau+1\right] \subset \left[ -R \left( \varepsilon
_{j}\right) ,R \left( \varepsilon _{j}\right) \right] $, we have
\begin{eqnarray*}
&{}& \int_{\left[ \tau,\tau+1\right] \times S^{1}}\left\vert \widetilde\xi
_{j}\right\vert ^{2} \, dt\, d\tau \\
&{}& \quad \leq e^{-4\pi \upsilon \left( R \left( \varepsilon _{j}\right)
-\left\vert \tau \right\vert \right) }\left[ \int_{\left[ -R \left(
\varepsilon _{j}\right) -1,-R \left( \varepsilon _{j}\right) \right] \times
S^{1}}\left\vert \widetilde\xi _{j}\right\vert ^{2}+\int_{\left[ R \left(
\varepsilon _{j}\right) ,R \left( \varepsilon _{j}\right) +1\right] \times
S^{1}}\left\vert \widetilde\xi _{j}\right\vert ^{2}\right] .
\end{eqnarray*}
\end{cor}

From these results and standard elliptic estimate on the cylinder $\left[
\tau -\frac{1}{2},\tau +\frac{1}{2}\right] \times S^{1}$, we obtain the
following pointwise exponential decay estimate of $\xi _{j}$.

\begin{cor}
For any $0<\upsilon <1$, there exists $N_{0}=N_{0}\left( \upsilon \right) $
depending on $\upsilon$, such that for all $j>N_{0}$ and $\tau \in \left[
-\tau \left( \varepsilon _{j}\right) ,\tau \left( \varepsilon _{j}\right) %
\right] $, we have%
\begin{equation*}
\left\vert \nabla \xi _{j}\right\vert \leq Ce^{-2\pi \upsilon \left( \tau
\left( \varepsilon _{j}\right) -\left\vert \tau \right\vert \right) }\left(
\left\Vert d\xi _{j}\right\Vert _{L^{2}\left( \left[ -\tau \left(
\varepsilon _{j}\right) -1,-\tau \left( \varepsilon _{j}\right) \right]
\times S^{1}\right) }+\left\Vert d\xi _{j}\right\Vert _{L^{2}\left( \left[
\tau \left( \varepsilon _{j}\right) ,\tau \left( \varepsilon _{j}\right) +1%
\right] \times S^{1}\right) }\right) ,
\end{equation*}
and for $\tau \in \left[ -R \left( \varepsilon _{j}\right) ,R \left(
\varepsilon _{j}\right) \right] $, we have
\begin{equation*}
\left\vert \widetilde{\xi _{j}}\right\vert \leq Ce^{-2\pi \upsilon \left(
R\left( \varepsilon _{j}\right) -\left\vert \tau \right\vert \right) }\left(
\left\Vert \widetilde{\xi _{j}}\right\Vert _{L^{2}\left( \left[ -R\left(
\varepsilon _{j}\right) -1,-R\left( \varepsilon _{j}\right) \right] \times
S^{1}\right) }+\left\Vert \widetilde{\xi _{j}}\right\Vert _{L^{2}\left( %
\left[ R\left( \varepsilon _{j}\right) ,R\left( \varepsilon _{j}\right) +1%
\right] \times S^{1}\right) }\right) .
\end{equation*}%
The constant $\upsilon $ can be made arbitrarily close to $1$.
\end{cor}

%
%On the other hand, we have
%
%\begin{lem} On $\pm [R(\varepsilon),R(\varepsilon) +1]$,
%$$
%|\xi_j(\tau,t)| \leq C(\delta_j + CE(l) \varepsilon_j).
%$$
%\end{lem}
%\begin{proof}
%We recall the error estimate \eqref{eq:error}
%$$
%\|E_{\varepsilon _{j}}\left(\tau ,t\right)\|_{\varepsilon_j} =
%\|\delbar_{(J,\varepsilon f)} u^\varepsilon_{\textrm{glue}}\|_{L^p_{\beta_{\delta_j,\varepsilon_j}}(\R \times S^1)}.
%\leq E(l) \varepsilon_j^{\frac{1}{p}}
%$$
%In particular we have
%$$
%u_j = \exp_{u^{\varepsilon_J}_{\textrm{app}}}(\xi_j(\tau,t)) \cong u^{\varepsilon_j}_{\textrm{glue}} + \xi_j(\tau,t)
%$$
%with
%$$
%\|\xi_j\|_\e \leq C E(l) \varepsilon_j^{\frac{1}{p}}
%$$
%by the nature of the gluing construction. In particular
%\be\label{eq:onZ}
%\|\widetilde \xi_j\|_{W^{1,p}_{\beta_{\delta_j,\varepsilon_j}}(Z)}
%+ \|\xi_{j,0} \|_{W^{1,p}_{\beta_{\delta_j,\varepsilon_j}}(Z)} \leq E(l) \varepsilon_j^{\frac{1}{p}}
%\ee
%on $Z = \pm [R(\varepsilon) -1, R(\varepsilon) + 2]$.
%
%Here we recall the weighting function $\beta_{\delta_j,\varepsilon_j}$ on the region
%$[R(\epsilon_j), \tau(\epsilon_j)]$ is given by
%$$
%\beta_{\delta,\varepsilon}(\tau) = \varepsilon^{1-p+\delta}(1+|\tau|)^{\delta} \geq \varepsilon^{1-p}.
%$$
%Therefore we have
%$$
%\|\widetilde \xi_j\|_{W^{1,p}(Z)} \leq C E(l)\varepsilon_j^{1/p} (\varepsilon_j^{\frac{1-p}{p}})^{-1}
%=  C E(l)\varepsilon_j.
%$$
%In particular, we have $\|\widetilde \xi_j\|_{L^\infty(Z)} \leq C E(l) \varepsilon_j$
%by Sobolev embedding.
%This finishes the proof.
%\end{proof}

We fix a $\upsilon $ in the above lemma, such that $\frac{p-1}{\delta }%
\upsilon \geq 1$ . This is always possible since $\frac{p-1}{\delta }>1$.

\subsection{Exponential estimates for the gradient}

In this subsection, we will upgrade the $L^2$ and pointwise exponential estimates proved in the previous section to
that of the norm $\| \cdot \|_\varepsilon$.
We will assume $\delta _{j}\leq \varepsilon _{j}$ from now on. We estimate $\|\xi _{j}\|_\varepsilon$
by partitioning the cylinder $\R \times S^1$ into 4 regions.

\textbf{1. }We first study the Banach norm $\left\Vert \xi _{j}\right\Vert
_{\varepsilon _{j}}$ on the region $\Omega _{\pm \zeta _{0}}\left(
\varepsilon _{j}\right) $. From the exponential decay estimate we have for $%
\left\vert \tau \right\vert \leq \tau \left( \varepsilon _{j}\right)
-h\left( \zeta _{0}\right) $ that
\begin{eqnarray}
\left\vert \nabla \xi _{j}\left( \tau ,t\right) \right\vert &\leq &Ce^{2\pi
\upsilon \left[ \tau -\left( R\left( \varepsilon _{j}\right) +\frac{p-1}{%
\delta }S\left( \varepsilon _{j}\right) -h\left( \zeta _{0}\right) \right) %
\right] }\left\Vert \xi _{j}\right\Vert _{L^{2}\left( \left[ \tau \left(
\varepsilon _{0}\right) -h\left( \zeta _{0}\right) ,\tau \left( \varepsilon
_{0}\right) \right] \times S^{1}\right) }  \notag \\
&\leq &Ce^{2\pi \upsilon \left( \tau -R\left( \varepsilon _{j}\right)
\right) }\left( \frac{\varepsilon _{j}}{l}\right) ^{\frac{p-1}{\delta }%
\upsilon }\left( \frac{1}{\zeta _{0}}\right) ^{\upsilon }\delta _{j}.
\label{high-mode-exp-decay}
\end{eqnarray}%
Integrating $\nabla \xi _{j}\left( \tau ,t\right) $ from $R\left(
\varepsilon _{j}\right) $ to $\tau $, we have%
\begin{equation}
\left\vert \xi _{j}\left( \tau ,t\right) -\left( \xi _{j}\left(
R(\varepsilon _{j}),t\right) \right) _{0}\right\vert \leq Ce^{2\pi \upsilon
\left( \tau -R\left( \varepsilon _{j}\right) \right) }\left( \frac{%
\varepsilon _{j}}{l}\right) ^{\frac{p-1}{\delta }\upsilon }\left( \frac{1}{%
\zeta _{0}}\right) ^{\upsilon }\delta _{j}.  \label{difference-exp-decay}
\end{equation}%
Therefore by \eqref{high-mode-exp-decay} and \eqref
{difference-exp-decay} the estimate of $\left\Vert \xi _{j}\right\Vert
_{\varepsilon _{j}}$ restricted on the region%
\begin{equation*}
\left[ R\left( \varepsilon _{j}\right) ,R\left( \varepsilon _{j}\right) +%
\frac{p-1}{\delta }S\left( \varepsilon _{j}\right) -h\left( \zeta
_{0}\right) \right] \times S^{1}
\end{equation*}%
is
\begin{eqnarray*}
&&\int_{R(\varepsilon _{j})}^{R(\varepsilon _{j})+\frac{p-1}{\delta }S\left(
\varepsilon _{j}\right) -h\left( \zeta _{0}\right) }\int_{0}^{1}\left(
\left\vert \xi _{j}\left( \tau ,t\right) -\left( \xi _{j}\left( R\left(
\varepsilon _{j}\right) ,t\right) \right) _{0}\right\vert ^{p}+\left\vert
\nabla \xi _{j}\left( \tau ,t\right) \right\vert ^{p}\right) \\
&&\cdot e^{2\pi \delta \left( -\tau +R(\varepsilon _{j})+\frac{p-1}{\delta }%
S\left( \varepsilon _{j}\right) \right) }dtd\tau \\
&\leq &C\delta _{j}^{p}\left( \frac{1}{\zeta _{0}}\right) ^{\upsilon
p}\int_{0}^{\frac{p-1}{\delta }S\left( \varepsilon _{j}\right) }\left( \frac{%
\varepsilon _{j}}{l}\right) ^{\frac{p-1}{\delta }\upsilon p}e^{\upsilon
ps}\cdot e^{-\delta s}\left( \frac{l}{\varepsilon _{j}}\right) ^{p-1}d\tau \\
&{}& \qquad \text{ \ \ (where }s=2\pi \left( \tau -R\left( \varepsilon
_{j}\right) \right) \text{)} \\
&\leq &C\delta _{j}^{p}\left( \frac{1}{\zeta _{0}}\right) ^{\upsilon p}\cdot
\left( \frac{\varepsilon _{j}}{l}\right) ^{\frac{p-1}{\delta }\upsilon
p}\left( \frac{l}{\varepsilon _{j}}\right) ^{p-1}\cdot \frac{1}{\upsilon
p-\delta }e^{\left( \upsilon p-\delta \right) \cdot 2\pi \frac{p-1}{\delta }%
S\left( \varepsilon _{j}\right) } \\
&\leq &\frac{1}{\upsilon p-\delta }C\delta _{j}^{p}\left( \frac{1}{\zeta _{0}%
}\right) ^{\upsilon p}\cdot \left( \frac{\varepsilon _{j}}{l}\right) ^{\frac{%
p-1}{\delta }\upsilon p-\left( p-1\right) -\frac{\left( \upsilon p-\delta
\right) (p-1)}{\delta }} \\
&=&\frac{1}{\upsilon p-\delta }C\delta _{j}^{p}\left( \frac{1}{\zeta _{0}}%
\right) ^{\upsilon p}\cdot 1\rightarrow 0.
\end{eqnarray*}

\textbf{2. }We consider the regions
\begin{eqnarray*}
\Phi _{\zeta _{0}}\left( \varepsilon _{j}\right) &:&=[-\left( \tau \left(
\varepsilon _{j}\right) -h\left( \zeta _{0}\right) \right) ,\tau \left(
\varepsilon _{j}\right) -h\left( \zeta _{0}\right) ]\times S^{1}. \\
\Phi _{\zeta _{0}}^{+}\left( \varepsilon _{j}\right) &:&=[-\left( \tau
\left( \varepsilon _{j}\right) -h\left( \zeta _{0}\right) +1\right) ,\tau
\left( \varepsilon _{j}\right) -h\left( \zeta _{0}\right) +1]\times S^{1}.
\end{eqnarray*}%
From the interior Schauder estimate we have%
\begin{eqnarray}
&&\left\Vert \nabla \xi _{j}\right\Vert _{C^{0}\left( \Phi _{\zeta
_{0}}\left( \varepsilon _{j}\right) \right) }  \notag \\
&\leq &\left\Vert \xi _{j}\right\Vert _{C^{1,\alpha }\left( \Phi _{\zeta
_{0}}\left( \varepsilon _{j}\right) \right) }  \notag \\
&\leq &C\left( \left\Vert J^{\varepsilon _{j}}(\chi (\varepsilon _{j}\tau ))%
\frac{\partial \xi _{j}}{\partial t}+\varepsilon _{j}A_{j}\left( \chi
(\varepsilon _{j}\tau )\right) \cdot \xi _{j}\right\Vert _{C^{\alpha }\left(
\Phi _{\zeta _{0}}^{+}\left( \varepsilon _{j}\right) \right) }+\left\Vert
\xi _{j}\right\Vert _{C^{0}\left( \Phi _{\zeta _{0}}^{+}\left( \varepsilon
_{j}\right) \right) }\right)  \notag \\
&\leq &C\left( \left\Vert B_{j}\left( \chi (\varepsilon _{j}\tau ),\xi
_{j}\right) \right\Vert _{C^{0}\left( \Phi _{\zeta _{0}}^{+}\left(
\varepsilon _{j}\right) \right) }+\left\Vert E_{\varepsilon _{j}}\left( \tau
,t\right) \right\Vert _{C^{0}\left( \Phi _{\zeta _{0}}^{+}\left( \varepsilon
_{j}\right) \right) }+\left\Vert \xi _{j}\right\Vert _{C^{0}\left( \Phi
_{\zeta _{0}}^{+}\left( \varepsilon _{j}\right) \right) }\right)  \notag \\
&\leq &C\left( \left\Vert \nabla \xi _{j}\right\Vert _{C^{0}\left( \Phi
_{\zeta _{0}}^{+}\left( \varepsilon _{j}\right) \right) }\left\Vert \xi
_{j}\right\Vert _{C^{0}\left( \Phi _{\zeta _{0}}^{+}\left( \varepsilon
_{j}\right) \right) }+\left\Vert \xi _{j}\right\Vert _{C^{0}\left( \Phi
_{\zeta _{0}}^{+}\left( \varepsilon _{j}\right) \right) }^{2}\right.  \notag
\\
&&\left. +\left\Vert \xi _{j}\right\Vert _{C^{0}\left( \Phi _{\zeta
_{0}}^{+}\left( \varepsilon _{j}\right) \right) }\right) ,  \label{grad-xi}
\end{eqnarray}%
where the last inequality is because $B$ is quadratic, and $\left\Vert
E_{\varepsilon _{j}}\left( \tau ,t\right) \right\Vert _{C^{0}\left( \Phi
_{\zeta _{0}}^{+}\left( \varepsilon _{j}\right) \right) }\leq C\left\Vert
\xi _{j}\right\Vert _{C^{0}\left( \Phi _{\zeta _{0}}^{+}\left( \varepsilon
_{j}\right) \right) }$ from $\left( \ref{E-error}\right) $. Using the $%
C^{\infty }$ uniform convergence of $\xi _{j}$ outside $\Phi _{\zeta
_{0}}\left( \varepsilon _{j}\right) $, we have
\begin{equation*}
\left\Vert \nabla \xi _{j}\right\Vert _{C^{0}\left( \Phi _{\zeta
_{0}}^{+}\left( \varepsilon _{j}\right) \right) }\leq \left\Vert \nabla \xi
_{j}\right\Vert _{C^{0}\left( \Phi _{\zeta _{0}}\left( \varepsilon
_{j}\right) \right) }+\delta _{j}\text{.}
\end{equation*}%
Plugging in $\left( \ref{grad-xi}\right) $ and noting $\left\Vert \xi
_{j}\right\Vert _{C^{0}\left( \Phi _{\zeta _{0}}^{+}\left( \varepsilon
_{j}\right) \right) }\leq \delta _{j}$, we have
\begin{equation*}
\ \left( 1-C\delta _{j}\right) \left\Vert \nabla \xi _{j}\right\Vert
_{C^{0}\left( \Phi _{\zeta _{0}}\left( \varepsilon _{j}\right) \right) }\leq
C\left( \delta _{j}^{2}+\delta _{j}^{2}+\delta _{j}\right) ,
\end{equation*}%
so for $\delta _{j}<\min \left\{ \frac{1}{2C},1\right\} $ we have%
\begin{equation*}
\left\Vert \nabla \xi _{j}\right\Vert _{C^{0}\left( \Phi _{\zeta _{0}}\left(
\varepsilon _{j}\right) \right) }\leq 2C\left( 3\delta _{j}\right) =6C\delta
_{j}.
\end{equation*}

\textbf{3. }Then we study the equation of $\xi _{j}$ on the region%
\begin{equation*}
\Theta \left( \varepsilon _{j}\right) :=[-R(\varepsilon _{j}),R(\varepsilon
_{j})]\times S^{1}.
\end{equation*}%
For the higher mode $\widetilde{\xi _{j}}$, by $\left( \ref%
{high-mode-exp-decay}\right) $ we have
\begin{eqnarray}
\left\vert \nabla \widetilde{\xi _{j}}\left( \tau ,t\right) \right\vert
&=&\left\vert \nabla \xi _{j}\left( \tau ,t\right) -\int_{S^{1}}\nabla \xi
_{j}\left( \tau ,s\right) ds\right\vert   \notag \\
&=&\left\vert \int_{S^{1}}\left( \nabla \xi _{j}\left( \tau ,t\right)
-\nabla \xi _{j}\left( \tau ,s\right) \right) ds\right\vert \leq 2\sup_{t\in
S^{1}}\left\vert \nabla \xi _{j}\left( \tau ,t\right) \right\vert   \notag \\
&\leq &C\delta _{j}\left( \frac{1}{\zeta _{0}}\right) ^{\upsilon }\left(
\frac{\varepsilon _{j}}{l}\right) ^{\frac{p-1}{\delta }\upsilon }e^{2\pi
\upsilon \left( \tau -R\left( \varepsilon _{j}\right) \right) }\leq C\delta
_{j}\left( \frac{1}{\zeta _{0}}\right) ^{\upsilon }\varepsilon _{j}.
\label{eq:Schauderxibar}
\end{eqnarray}%
where in the last inequality we have used that $\frac{p-1}{\delta }\upsilon
>1$ and $\left\vert \tau \right\vert \leq R\left( \varepsilon _{j}\right) $.

We notice that on $\Theta \left( \varepsilon \right) =\left[ -R\left(
\varepsilon \right) ,R\left( \varepsilon \right) \right] \times S^{1}$, the
weighting function $\varepsilon ^{1-p}$ dominates the power weight $%
\left\Vert \cdot \right\Vert _{W_{\rho _{\varepsilon }}^{1,p}}$, up to
constant factor $\left( 2l\right) ^{\delta }$, because%
\begin{equation*}
\rho _{\varepsilon }\left( \tau \right) =\varepsilon ^{1-p+\delta }(1+|\tau
|)^{\delta }\leq \varepsilon ^{1-p+\delta }(2l/\varepsilon )^{\delta
}=\left( 2l\right) ^{\delta }\varepsilon ^{1-p}.
\end{equation*}%
So for higher mode $\widetilde{\xi _{j}}$ we obtain
\begin{eqnarray*}
\left\Vert \widetilde{\xi _{j}}\right\Vert _{W_{\beta _{\delta ,\varepsilon
_{j}}}^{1,p}\left( \Theta \left( \varepsilon _{j}\right) \right) }^{p} &\leq
&\int_{-R(\varepsilon _{j})}^{R(\varepsilon _{j})}\int_{0}^{1}\left(
\left\vert \widetilde{\xi _{j}}\right\vert ^{p}+|\nabla \widetilde{\xi _{j}}%
|^{p}\right) \left( 2l\right) ^{\delta }\varepsilon _{j}^{1-p}\,dt\,d\tau \\
&\leq &\left( 2l\right) ^{\delta }\int_{-R(\varepsilon _{j})}^{R\left(
\varepsilon _{j}\right) }2\left( C\delta _{j}\left( \frac{1}{\zeta _{0}}%
\right) ^{\upsilon }\varepsilon _{j}\right) ^{p}\varepsilon _{j}^{1-p}\,d\tau
\\
&=&2C^{p}\left( 2l\right) ^{\delta }\left( \frac{1}{\zeta _{0}}\right)
^{\upsilon p}\delta _{j}^{p}\rightarrow 0.
\end{eqnarray*}%
For the $0$-mode $\left( \xi _{j}\right) _{0}$, noticing that for $%
\left\vert \tau \right\vert \leq l/\varepsilon $ the error term is $0$, from
the equation of $\left( \xi \right) _{0}$ we have%
\begin{eqnarray}
\left\vert \nabla \left( \xi _{j}\right) _{0}\left( \tau \right) \right\vert
&=&\left\vert \frac{\partial }{\partial \tau }\left( \xi _{j}\right)
_{0}\right\vert =\left\vert \varepsilon _{j}A\left( \chi \left( \varepsilon
_{j}\tau \right) \right) \left( \xi _{j}\right) _{0}+\left( B\left( \xi
_{j}\right) \right) _{0}\right\vert  \notag \\
&\leq &C\left( \varepsilon _{j}\delta _{j}+\delta _{j}^{2}\right) \text{ (}%
\because B\left( \xi _{j}\right) \text{ is quadratic and }\left\vert \nabla
\xi _{j}\right\vert \leq 6C\delta _{j}\text{)}  \notag \\
&\leq &2C\varepsilon _{j}\delta _{j}\text{ ($\because $ }\delta _{j}\leq
\varepsilon _{j}\text{).}  \label{xi-0-control}
\end{eqnarray}%
Therefore
\begin{eqnarray*}
\left\Vert \left( \xi _{j}\right) _{0}\right\Vert _{W_{\varepsilon
_{j}}^{1,p}\left( \Theta \left( \varepsilon _{j}\right) \right) }^{p}
&=&\int_{-R(\varepsilon _{j})}^{R(\varepsilon _{j})}\left( \varepsilon
_{j}\left\vert \left( \xi _{j}\right) _{0}\right\vert ^{p}+\varepsilon
_{j}^{1-p}\,|\nabla \left( \xi _{j}\right) _{0}|^{p}\right) d\tau \\
&\leq &\int_{-R(\varepsilon _{j})}^{R(\varepsilon _{j})}\left( \varepsilon
_{j}\delta _{j}^{p}+\varepsilon _{j}^{1-p}\,\left( 2C\varepsilon _{j}\delta
_{j}\right) ^{p}\right) d\tau \\
&\leq &C\left( l\delta _{j}^{p}+R\left( \varepsilon _{j}\right) \varepsilon
_{j}\delta _{j}^{p}\right) =Cl\delta _{j}^{p}\rightarrow 0.
\end{eqnarray*}%
where the second inequality is by $\left( \ref{xi-0-control}\right) $.

By Sobolev embedding $\left\vert \left( \xi _{j}\left( \pm l/\varepsilon
_{j}\right) \right) _{0}\right\vert \leq C\left\Vert \left( \xi _{j}\right)
_{0}\right\Vert _{W_{\varepsilon _{j}}^{1,p}\left( \Theta \left( \varepsilon
_{j}\right) \right) }{}\leq C\delta _{j}.$

Combining these we have
\begin{eqnarray*}
\left. \left\Vert \xi _{j}\right\Vert _{\varepsilon _{j}}\right\vert
_{\Theta \left( \varepsilon _{j}\right) \cup \Omega _{\pm \zeta _{0}}\left(
\varepsilon _{j}\right) } &=&\left\Vert \widetilde{\xi _{j}}\right\Vert
_{W_{\beta _{\delta ,\varepsilon _{j}}}^{1,p}\left( \Theta \left(
\varepsilon _{j}\right) \cup \Omega _{\pm \zeta _{0}}\left( \varepsilon
_{j}\right) \right) }+\left\Vert \left( \xi _{j}\right) _{0}\right\Vert
_{W_{\varepsilon _{j}}^{1,p}\left( \Theta \left( \varepsilon _{j}\right)
\cup \Omega _{\pm \zeta _{0}}\left( \varepsilon _{j}\right) \right) }{} \\
&&+\left\vert \left( \xi _{j}\left( \pm l/\varepsilon _{j}\right) \right)
_{0}\right\vert \\
&\leq &C\delta _{j}\rightarrow 0,
\end{eqnarray*}%
where the constant $C$ is uniform for all $0<\varepsilon _{j}\leq
\varepsilon _{0}$ and $l\geq l_{0}$.

\textbf{4. }Outside the region $\Theta \left( \varepsilon _{j}\right) \cup
\Omega _{\pm \zeta _{0}}\left( \varepsilon _{j}\right) $, namely for $%
\left\vert \tau \right\vert >\tau \left( \varepsilon _{j}\right) -h\left(
\zeta _{0}\right) $, by the $C^{\infty }$ uniform convergence of $u_{j}$ to $%
u_{\pm }^{\varepsilon _{j}}$ it is easy to see
\begin{equation*}
\left. \left\Vert \xi _{j}\right\Vert _{\varepsilon _{j}}\right\vert
_{\Sigma _{\varepsilon _{j}}\backslash \left( \Theta \left( \varepsilon
_{j}\right) \cup \Omega _{\pm \zeta _{0}}\left( \varepsilon _{j}\right)
\right) }=\left. \left\Vert \xi _{j}\right\Vert _{\varepsilon
_{j}}\right\vert _{\Sigma _{\pm }\backslash U\pm \left( \zeta _{0}\right)
}\rightarrow 0.
\end{equation*}%
This finishes the proof
\begin{equation*}
\Vert \xi _{j}\Vert _{\varepsilon _{j}}\rightarrow 0
\end{equation*}
which is Proposition \ref{prop:norm-convergence2}. Hence is also finished
the proof of surjectivity, Proposition \ref{prop:surjective}.

\section{Variants of adiabatic gluing}

\label{sec:variants}

In this section, we specialize to the case
\begin{equation*}
H = 0
\end{equation*}
in our thimble-flow-thimble configuration by dividing the cases of closed
strings and of open strings. We discuss various cases to which similar
adiabatic gluing construction can be applied. Since the necessary analysis
will be small modifications of the current constructions, we will be brief
in our discussion.

\subsection{Pearl complex in the Hamiltonian case}

\label{subsec:ham-pearl}

Let $f:M\rightarrow {\mathbb{R}}$ be a Morse function. The the Floer
equation for the Hamiltonian $\varepsilon f$ is%
\begin{equation}  \label{Floer_e}
\frac{\partial u}{\partial \tau }+J(u)\left(\frac{\partial u}{\partial t}%
-\varepsilon X_{f}(u)\right)=0
\end{equation}%
for $u:{\mathbb{R\times }}S^{1}\rightarrow M$ with asymptote $u\left( \pm
\infty ,t\right) =z_{\pm }\left( t\right) $, where $z_{\pm }\left( t\right) $
are Hamiltonian $1$-periodic orbits of $\varepsilon f$. In the papers \cite%
{oh:dmj}, \cite{oh:adiabatic}, the first named author studied the adiabatic
degeneration of the moduli space of solutions satisfying the above equation
as $\varepsilon \rightarrow 0$. \cite{mundet-tian} studied similar adiabatic
degeneration for twisted holomorphic sections in Hamiltonian $S^{1}$%
-manifolds. The limiting moduli space as $\varepsilon =0$ consists of
sphere-flow-sphere configurations which Biran and Cornea call
\textquotedblleft pearl complexs\textquotedblright and is defined as the
following

\begin{defn}
The configuration
\begin{equation*}
u:=\left( p,\chi _{-\infty },u_{1,}\chi _{1},u_{2},\chi _{2},\cdots
,u_{k},\chi _{\infty },q\right)
\end{equation*}
is called a \emph{pearl configuration} if $u_{i}:S^{2}\cong {\mathbb{R}}%
\times S^{1}\rightarrow M$ are $J$-holomorphic spheres with marked points $%
u_{i}\left( o_{\pm }\right) $ where $o_{\pm }=\left\{ \pm \infty \right\}
\times S^{1}$, and each $\chi _{i}:\left[ -l_{i},l_{i}\right] \rightarrow M$
is a gradient segment of the Morse function $f$ connecting $u_{i}\left(
o_{+}\right) $ to $u_{i+1}\left( o_{-}\right) $, $\chi _{-\infty }$
connecting the critical point $p$ to $u_{1}\left( o_{-}\right) $ and $\chi
_{\infty }$ connecting $u_{k}\left( o_{+}\right) $ to the critical point $q$.
\end{defn}

We define the moduli space
\begin{eqnarray*}
\mathcal{M}_{2}\left( M,J;A_{i}\right) &=&\Big\{\left( u_{i},o_{\pm }\right)
|u_{i}:S^{2} \setminus \{o_\pm\} \cong {\mathbb{R}}\times S^{1}\rightarrow M,o_{\pm }\in S^{2}, \\
&{}&\quad \overline{\partial }_{J}u_{i}=0,\left[ u_{i}\right] =A_{i}\in
H_{2}\left( M,\mathbb{Z}\right) \Big\}\Big/{\mathbb{R}}\times S^{1}
\end{eqnarray*}%
where the last ${\mathbb{R}}\times S^{1}$ is the automorphism group ${%
\mathbb{R}}$-translation and $S^{1}$ rotation, and the evaluation maps%
\begin{equation*}
ev_{\pm }^{i}:\mathcal{M}_{2}\left( M,J;A_{i}\right) \rightarrow M,\text{ \ }%
u_{i}\rightarrow u_{i}\left( o_{\pm }\right) .
\end{equation*}%
Consider the map
\begin{eqnarray}
{}id\times \left( \Pi _{i=0}^{k-1}\phi _{f}^{2l_{i}}\circ ev_{+}^{i}\times
ev_{-}^{i+1}\right) \times id &:&  \label{eq:pearl} \\
W^{u}\left( p\right) \times \Pi _{i=1}^{k}\mathcal{M}_{2}\left(
M,J;A_{i}\right) \times W^{s}\left( p\right) &\rightarrow &\Pi
_{i=1}^{k+1}\left( M\times M\right) ,  \notag \\
\left( x,u_{1},\cdots u_{k},y\right) &\rightarrow &\left( x,\Pi
_{i=1}^{k-1}\left( \phi _{f}^{2l_{i}}u_{i}\left( o_{+}\right) ,u_{i+1}\left(
o_{-}\right) \right) ,y\right) ,  \notag
\end{eqnarray}%
where $\phi _{f}^{2l}$ is the time-2$l$ flow of the Morse function $f$, $%
W^{u}\left( p\right) $ and $W^{s}\left( q\right) $are unstable and stable
manifolds of $p$ and $q$ respectively.

\begin{defn}
The moduli space of pearl configuration with flow length vector $\vec{l}%
:=\left( l_{1},l_{2},\cdots l_{k}\right) $ connecting $p$ to $q$ with the $J$%
-holomorphic spheres $u_{i}\left( i=1,2,\cdots k\right) $ in homology class $%
\vec{A}=\left( A_{1},A_{2},\cdots A_{k}\right) $ is defined to be
\begin{equation*}
\mathcal{M}_{\text{\textrm{pearl}}}^{\vec{l}}\left( p,q;f;\vec{A}\right)
=\left( id\times \left( \Pi _{i=0}^{k-1}\phi _{f}^{2l_{i}}\circ
ev_{+}^{i}\times ev_{-}^{i+1}\right) \times id\right) ^{-1}\left( \Pi
_{i=1}^{k+1}\triangle \right) ,
\end{equation*}%
where $\triangle \subset M\times M$ is the diagonal.
\end{defn}

We can give the obvious $W^{1,p}$ Banach manifold $\mathcal{B}_{\text{%
\textrm{pearl}}}^{\vec{l}}\left( p,q\right) $ to host $\mathcal{M}_{\text{%
\textrm{pearl}}}^{\vec{l}}\left( p,q;f;\vec{A}\right) $, and a natural
section $e$ of the Banach bundle $\mathcal{L}_{\text{\textrm{pearl}}}^{\vec{l%
}}\left( p,q\right) \rightarrow $ $\mathcal{B}_{\text{\textrm{pearl}}}^{\vec{%
l}}\left( p,q\right) ,$
\begin{equation*}
e:\mathcal{B}_{\text{\textrm{pearl}}}^{\vec{l}}\left( p,q\right) \rightarrow
\mathcal{L}_{\text{\textrm{pearl}}}^{\vec{l}}\left( p,q\right) ,
\end{equation*}%
such that
\begin{equation*}
\mathcal{M}_{\text{\textrm{pearl}}}^{\vec{l}}\left( p,q;f;\vec{A}\right)
=e^{-1}\left( 0\right) .
\end{equation*}%
We let the linearization of $\ e$ at $u\in \mathcal{M}_{\text{\textrm{pearl}}%
}^{\vec{l}}\left( p,q;f;\vec{A}\right) $ to be $E\left( u\right) $ (See $%
\left( \ref{e}\right) $ and $\left( \ref{Eu}\right) $ for the definitions of
$e$ and $E\left( u\right) $).

We assume the pearl configuration $u:=\left( p,\chi _{-\infty },u_{1,}\chi
_{1},u_{2},\chi _{2},\cdots ,u_{k},\chi _{\infty },q\right) $ satisfies the
\textquotedblleft sphere-flow-sphere\textquotedblright\ transversality
defined as the following, which was also defined in \cite{biran-cor-1} for
the case with Lagrangian boundary condition.

\begin{defn}
\label{sfs-trans}The pearl configuration $u=\left( p,\chi _{-\infty
},u_{1,}\chi _{1},u_{2},\chi _{2},\cdots ,u_{k},\chi _{\infty },q\right) $
satisfies the \textquotedblleft sphere-flow-sphere\textquotedblright\
transversality if the map
\begin{equation*}
id\times \left( \Pi _{i=0}^{k-1}\phi _{f}^{2l_{i}}\circ ev_{+}^{i}\times
ev_{-}^{i+1}\right) \times id
\end{equation*}
in $\left( \ref{eq:pearl}\right) $ is transversal to the diagonal $\Pi
_{i=1}^{k+1}\triangle \subset \Pi _{i=1}^{k+1}\left( M\times M\right) $.
\end{defn}

The \textquotedblleft sphere-flow-sphere\textquotedblright\ transversality
is known to be achievable for generic $J$ and $f$ (see \cite{biran-cor-1}
for example). Small modifications of Proposition \ref{prop:dfdindex},
Corollary \ref{achieve-tft-trans} and Proposition \ref{family-tft-trans}
leads to the corresponding

\begin{prop}
For any $u\in \mathcal{M}_{\text{\textrm{pearl}}}^{\vec{l}}\left( p,q;f;\vec{%
A}\right) $, the operator $E_u$ is a Fredholm operator and
\begin{equation}
\func{Index}E_u=\mu _{f}\left( p\right) -\mu _{f}(q)+\Sigma
_{i=1}^{k}2c_{1}(A_{i}).
\end{equation}
where $\mu _{f}$ is the Morse index for critical points of $f$.
\end{prop}

\begin{cor}
Suppose that each $u_{i}\in {\mathcal{M}}_{2}(M,J;A_{i})$ is Fredholm
regular, and each gradient segment $\chi _{\pm \infty }$ belongs to a full
gradient trajectory that is Fredholm regular, then $u\in \mathcal{M}_{\text{%
\textrm{pearl}}}^{\vec{l}}\left( p,q;f;\vec{A}\right) $ is Fredholm regular
(in the sense that $E_u$ is surjective) if and only if the configuration $u$
satisfies the \textquotedblleft sphere-flow-sphere\textquotedblright\
transversality in definition \ref{sfs-trans}.
\end{cor}

\begin{prop}
\label{family-sfs-trans}Suppose that each $u_{i}\in {\mathcal{M}}%
_{2}(M,J;A_{i})$ is Fredholm regular. Then there exists a dense subset of $%
f\in C^{\infty }\left( M\right) $ such that any element $u$ in
\begin{equation*}
\mathcal{M}_{\text{\textrm{pearl}}}^{\text{\textrm{para}}}\left( p,q;f;\vec{A%
}\right) =\bigcup_{l_{i}>0}\mathcal{M}_{\text{\textrm{pearl}}}^{\vec{l}%
}\left( p,q;f;\vec{A}\right)
\end{equation*}%
is Fredholm regular, in the sense that $E_u$ is surjective. Therefore $%
\mathcal{M}_{\text{\textrm{pearl}}}^{\text{\textrm{para}}}\left( p,q;f;\vec{A%
}\right) $ is a smooth manifold with dimension equal to the index of $E_u$:
\begin{equation*}
\func{dim}\mathcal{M}_{\text{\textrm{pearl}}}^{\text{\textrm{para}}}\left(
p,q;f;\vec{A}\right) =\func{Index}E_u=\mu _{f}\left( p\right) -\mu
_{f}(q)+\Sigma _{i=1}^{k}2c_{1}(A_{i})+1.
\end{equation*}
\end{prop}

\ The gluing problem from \textquotedblleft pearl
configuration\textquotedblright\ to nearby Floer trajectories was mentioned
and left as a future work in \cite{oh:newton}, \cite{mundet-tian}. Now we
can study the gluing from the pearl configuration $u=\left( p,\chi _{-\infty
},u_{1,}\chi _{1},u_{2},\chi _{2},\cdots ,u_{k},\chi _{\infty },q\right) $
to nearby Floer trajectories satisfying $\left( \ref{Floer_e}\right) $,
using the techniques from this paper. We assume $J$ and $f$ are generic such
that $E_u$ in Proposition \ref{family-sfs-trans} is surjective. The outcome
is

\begin{thm}\label{thm:Hamiltonian-pearl}
For any pearl configuration $u=\left( p,\chi _{-\infty },u_{1,}\chi
_{1},u_{2},\chi _{2},\cdots ,u_{k},\chi _{\infty },q\right)$ in $\mathcal{M}%
_{\text{\textrm{pearl}}}^{\text{\textrm{para}}}\left( p,q;f;\vec{A}\right) $
whose $E_u$ is surjective, then for sufficiently small $\varepsilon >0$,
nearby $u$ we have the solutions $u_{\vec{\theta},\vec{d}}^{\varepsilon
}\left( \tau ,t\right) $ of Floer equation $\left( \ref{Floer_e}\right) $
parameterized by the gluing parameters $\varepsilon >0$, $\vec{\theta}%
=\left( \theta _{1},\theta _{2},\cdots ,\theta _{k}\right) \in \left(
S^{1}\right) ^{k}\,$and $\vec{d}=\left( d_{1},d_{2},\cdots ,d_{k}\right) \in
\left( {\mathbb{R}}\right) ^{k}$, with the asymptotes $u_{\vec{\theta},\vec{d%
}}^{\varepsilon }\left( -\infty ,t\right) =p$, and $u_{\vec{\theta},\vec{d}%
}^{\varepsilon }\left( +\infty ,t\right) =q$.
\end{thm}

\begin{proof} Since the main methodology and analytic details are essentially the same as
for the case of $(u_-,\chi,u_+)$ except increasing complexity of notations and writing,
we will be brief just by indicating what modifications should be made to apply
the methodology to the current case.

Locally the gluing is the same, in the sense that it is
the gluing of gradient segments (with noncritical joint points) with $J$%
-holomorphic curves. The construction of approximate solution is the same,
except that for any $J$-holomorphic sphere $u_{i}$ we have the $S^{1}\times
\mathbb{R}$ family of $J$-holomorphic spheres $u_{i}\left( \tau
+d_{i},\theta +\theta _{i}\right) $ in pregluing. The $\overline{\partial }$%
-error estimate is the same, with the pair $\left( K_{\varepsilon
},J_{\varepsilon }\right) $ in the piece $\left( 3\right) $ in section \ref%
{section:d-bar-error} replaced by $\left( \varepsilon f,J_{0}\right) $ hence
$\overline{\partial }_{\left( K_{\varepsilon },J_{\varepsilon }\right) }$
replaced by $\overline{\partial }_{\left( \varepsilon f,J_{0}\right) }$. The
weight $\beta _{\delta ,\varepsilon }\left( \tau \right) $ on $\chi _{\pm
\infty }$ is not of a polynomial weight but of an exponential weight when $%
\chi _{\pm \infty }\left( \tau \right) $ approaches the critical points $p,q$%
, but by remark \ref{exp-weight-dbar-error} this doe not destroy $\overline{%
\partial }_{J,\varepsilon f}$-error estimate. The construction and estimates
of the right inverse are the same. The quadratic estimate remains the same,
where the domain Riemann surface $\Sigma _{\varepsilon }\simeq {\mathbb{R}}%
\times S^{1}$ is equipped with standard measure on disjoint cylinders $\left[
-l_{i}/\varepsilon -\frac{p-1}{\delta }S\left( \varepsilon \right)
,l_{i}/\varepsilon +\frac{p-1}{\delta }S\left( \varepsilon \right) \right]
\times S^{1}$ and $(\pm \infty ,0]\times S^{1}$, glued with standard middle
annulus of $S^{2}$ (with cylindrical measure).
\end{proof}

Notice that we do not have the surjectivity part in this theorem, because a
sequence of solutions $u^{\varepsilon }$ of Floer equation $\left( \ref%
{Floer_e}\right) $ may develop multiple covering $J$-holomorphic spheres or
multiple covering of gradient segments in the limit, which lacks Fredholm
regularity; The bubbling sphere may also occur at any interior point on the
gradient segment $\chi $, or worse, the joint points. For that situation we
have not fully developed the gluing analysis.

\subsection{Pearl complex in the Lagrangian case}

Let $\left( M,\omega \right) $ be a compact symplectic manifold with
compatible almost complex structure $J$ and $L$ be a compact Lagrangian
submanifold. For a Morse function $f:L\rightarrow {\mathbb{R}}$, we extend $%
f $ \ to a Morse function $f:M\rightarrow {\mathbb{R}}$ such that in a
Weinstein neighborhood of $L\,$\ which is symplectomorphic to $T^{\ast }M$, $%
f$ is constant along each fiber. Let $\phi _{\varepsilon f}^{t}$ be the time-%
$t$ Hamiltonian flow of $f$. Then $L_{\varepsilon f}:=\phi _{\varepsilon
f}^{1}L$ is a Lagrangian submanifold Hamiltonian isotopic to $L$. For
generic $f$, $L$ and $L_{\varepsilon f}$ transversally intersect. Similar to
the Hamiltonian case, we study the $J$-holomorphic stripe $u:{\mathbb{R}}%
\times \left[ 0,1\right] \rightarrow M$ satisfying%
\begin{eqnarray}
\frac{\partial u}{\partial \tau }+J(u)\frac{\partial u}{\partial t} &=&0,%
\text{ }u\left( {\mathbb{R}},0\right) \in L\text{ and }u\left( {\mathbb{R}}%
,1\right) \in L_{\varepsilon f}  \notag \\
u\left( -\infty ,\cdot \right) &=&p, \, u\left( +\infty ,\cdot \right) =q
\label{Lagr-Floer-eq-e}
\end{eqnarray}%
where $p,q$ are intersections of $L$ and $L_{\varepsilon f}$. The limiting
moduli space as $\varepsilon \rightarrow 0$ consists of \textquotedblleft
pearl complexs\textquotedblright\ (or disk-flow-disk configurations) defined
as the following

\begin{defn}
The configuration
\begin{equation*}
u:=\left( p,\chi _{-\infty },u_{1,}\chi _{1},u_{2},\chi _{2},\cdots
,u_{k},\chi _{\infty },q\right)
\end{equation*}%
is called a pearl configuration if $u_{i}:D^{2}\backslash \left\{ \pm
1\right\} \cong {\mathbb{R}}\times [0,1]\rightarrow M$ are $J$-holomorphic
discs with lower and upper boundaries ending on $L$ and $L_{\varepsilon f}$
respectively, with marked points $u_{i}\left( o_{\pm }\right) $ where $%
o_{\pm }=\{\pm 1\}\in D^2=\left\{ \pm \infty \right\} \times \left[ 0,1%
\right] \in \mathbb{R}\times [0,1]$, and each $\chi _{i}:\left[ -l_{i},l_{i}%
\right] \rightarrow L$ is a gradient segment of the Morse function $f$
connecting $u_{i}\left( o_{+}\right) $ to $u_{i+1}\left( o_{-}\right) $, $%
\chi _{-\infty }$ connecting the critical point $p$ to $u_{1}\left(
o_{-}\right) $ and $\chi _{\infty }$ connecting $u_{k}\left( o_{+}\right) $
to the critical point $q$.
\end{defn}

Suppose that each $u_{i}$ is Fredholm regular. The disk-flow-disk
transversality has been defined in \cite{biran-cor-1} and proven to be
achievable for generic $\left( J,f\right) $. Parallel to the previous
section, we can define the pearl complex moduli space $\mathcal{M}_{\text{%
\textrm{pearl}}}^{\text{\textrm{para}}}\left( p,q;L,f;\vec{A}\right) $,
establish the index formula for $u$ in such moduli space, and show $E\left(
u\right) $ is surjective for generic $\left( J,f\right) $. These are more or
less standard so we omit details (See \cite{biran-cor-1} for precise
definitions of pearl complex moduli space and index formula). The following
is the gluing theorem.

\begin{thm}
For any pearl configuration $u=\left( p,\chi _{-\infty },u_{1,}\chi
_{1},u_{2},\chi _{2},\cdots ,u_{k},\chi _{\infty },q\right)$ in $\mathcal{M}%
_{\text{\textrm{pearl}}}^{\text{\textrm{para}}}\left( p,q;L,f;\vec{A}\right)
$ whose $E_u$ is surjective, then for sufficiently small $\varepsilon >0$,
nearby $u$ we have $J$-holomorphic stripes $u_{\vec{d}}^{\varepsilon }\left(
\tau ,t\right) $ of equation $\left( \ref{Lagr-Floer-eq-e}\right) $
parameterized by the gluing parameters $\varepsilon >0$ and $\vec{d}=\left(
d_{1},d_{2},\cdots ,d_{k}\right) \in  {\mathbb{R}} ^{k}$, with
the asymptotes $u_{\vec{d}}^{\varepsilon }\left( -\infty ,t\right) =p$, and $%
u_{\vec{d}}^{\varepsilon }\left( +\infty ,t\right) =q$.
\end{thm}

\begin{proof} Since the main methodology and analytic details are essentially the same as
in the case of Hamiltonian Floer trajectory case dealt in the main text of the present paper,
we will be brief just by indicating what modifications should be made to apply
the methodology to the current Lagrangian case.

 The equation $\left( \ref{Lagr-Floer-eq-e}\right) $ can be changed
into the following form which is very similar to the Hamiltonian case: For
any $u:{\mathbb{R}}\times \left[ 0,1\right] \rightarrow M$, let
\begin{equation*}
u^{\varepsilon }\left( \tau ,t\right) :=\phi _{\varepsilon f}^{-t}\circ
u\left( \tau ,t\right) .
\end{equation*}
Then $u$ satisfies $\left( \ref{Lagr-Floer-eq-e}\right) $\ if and only if $%
u^{\varepsilon }$ satisfies
\begin{eqnarray}
\frac{\partial u^{\varepsilon }}{\partial \tau }+J_{t}^{\varepsilon f}\left(
u^{\varepsilon }\right) \frac{\partial u^{\varepsilon }}{\partial t}%
+\varepsilon \nabla f\left( u^{\varepsilon }\right) &=&0,\text{ }%
u^{\varepsilon }\left( {\mathbb{R}},0\right) \text{ and }u^{\varepsilon
}\left( {\mathbb{R}},1\right) \subset L,  \notag \\
u^{\varepsilon }\left( -\infty ,\cdot \right) &=&p\quad \text{ and }\quad
u^{\varepsilon }\left( +\infty ,\cdot \right) =q.  \label{LH-Floer-eq-e}
\end{eqnarray}%
where $\left\{ J_{t}^{\varepsilon f}\right\} _{0\leq t\leq 1}:=\left\{
\left( \phi _{\varepsilon f}^{t}\right) _{\ast }J\left( \phi _{\varepsilon
f}^{-t}\right) _{\ast }\right\} _{0\leq t\leq 1}$is a $1$-parameter family
of compatible almost complex structures.

From the pearl configuration $u$ we can construct the approximate solution
for equation $\left( \ref{LH-Floer-eq-e}\right) $ from $J$-holomorphic discs
$u_{i}$ and gradient segments $\chi _{i}\left( i=1,2,\cdots ,k\right) $ and $%
\chi _{\pm \infty }$ as in the Hamiltonian case, but there is one
difference: The almost complex structure $J_{t}^{\varepsilon f}$ is $t$%
-dependent. \

The $t$-dependent $J_{t}^{\varepsilon f}$ seems to destroy the decomposition
of $0$-mode and higher modes of variation vector fields on $\chi $ as in the
Hamiltonian case, but actually here we have even better situation: the
linearized operator of $\overline{\partial }_{J_{t}^{\varepsilon
f},\varepsilon f}$ is canonically related to linearized operator of $\frac{d%
}{d\tau }-\func{grad}\left( \varepsilon f\right) $, and the right inverse
bound for $D_{\chi _{\varepsilon }}\overline{\partial }_{J_{t}^{\varepsilon
f},\varepsilon f}$ can be controlled by the right inverse bound of $\frac{D}{%
d\tau }-\nabla \func{grad}\left( \varepsilon f\right) $ up to a uniform
constant (as in the Proposition 6.1 of \cite{FO} and Proposition 4.6 in \cite%
{oh:imrn}). Therefore we do not need to decompose the $0$-mode and higher
modes of variation vector fields on $\chi $ in the Banach manifold setting,
and the weighting function $\beta _{\varepsilon ,\delta }\left( \tau \right)
$ on the portion over $\chi _{i}\left( \varepsilon \tau \right) $ $\left(
-l_{i}/\varepsilon \leq \tau \leq l_{i}/\varepsilon \right) $ is just the $%
\varepsilon $-adiabatic weight, namely
\begin{equation*}
\left\Vert \xi \right\Vert _{W_{\varepsilon }^{1,p}([-l_{i}/\varepsilon
,l_{i}/\varepsilon ]\times \left[ 0,1\right] )}^{p}=\int_{0}^{1}%
\int_{-l_{i}/\varepsilon }^{l_{i}/\varepsilon }\left( \varepsilon
^{2}\left\vert \xi _{0}\right\vert ^{p}+\varepsilon ^{2-p}\left\vert \nabla
\xi _{0}\right\vert ^{p}\right) d\tau dt
\end{equation*}%
as in \cite{FO}. The weighting function $\beta _{\varepsilon ,\delta }\left(
\tau \right) $ on the portion over $u_{i}^{\varepsilon }\left( \tau \right) $
remains the same, namely $\beta _{\varepsilon ,\delta }\left( \tau \right)
\approx e^{2\pi \delta \left\vert \tau \right\vert }$ for large $\left\vert
\tau \right\vert $ of the ends $u_{i}^{\varepsilon }\left( \tau \right) $.
The construction of combined right inverse and uniform quadratic estimates
are similar to the Hamiltonian case.
\end{proof}

The limit of moduli spaces of $J$-holomorphic $\left( k+1\right) $-gons
ending on $\left( L_{\varepsilon f_{0}},\cdots ,L_{\varepsilon f_{k}}\right)
$\ \ as $\varepsilon \rightarrow 0$ consists of \textquotedblleft cluster
complexes\textquotedblright , which are $J$-holomorphic discs $u_{i}$ ending
on $L$ connected with \emph{gradient flow trees} $\chi _{j}$ on $L$. In \cite%
{FO}, the gluing from gradient flow trees $\chi _{j}$ to \textquotedblleft
thin\textquotedblright\ $J$-holomorphic polygons was known. The gluing of
\textquotedblleft thin\textquotedblright\ $J$-holomorphic polygon with
\textquotedblleft thick\textquotedblright\ $J$-holomorphic discs $u_{i}$
locally is the same as the case of pearl complex as above.

\bigskip Again, we do not have the surjectivity part in the gluing theorem.
The reasons are similar to Hamiltonian case. The multi-covering is the
essential difficulty. In addition we also need the joint point to be immersed to
prove this surjectivity. However these are not issues for the pearl complex moduli spaces of virtual
dimension $0,1$ for monotone Lagrangian submanifolds, since it has been proved
in  \cite[Proposition 3.13]{biran-cor-1} that the $J$-holomorphic discs in
the pearl complex are Fredholm regular, simple (or non-multicovering) and
with disjoint images,  provided $\left( J,f\right) $ is generic. Then,
for simple $J$-holomorphic discs, relative version (with Lagrangian boundary
condition) of Theorem 2.6 in \cite{oh-zhu2} implies that the condition of
existence of non-immersed point on some $J$-holomorphic disk of the pearl
complex cut down the dimension of pearl complex moduli space by $%
\mathop{\kern0pt{\rm
dim}}\nolimits M-2$, provided $\left( J,f\right) $ is generic. So with
non-immersed condition the moduli space will have negative dimension if $%
\mathop{\kern0pt{\rm dim}}\nolimits M\geq 4$. When $%
\mathop{\kern0pt{\rm
dim}}\nolimits M=2$, \ $M$ is a Riemann surface, and the $J$-holomorphic
discs are just discs on surfaces $M$ with embedded boundary curve $L$. If
the $J$-holomorphic disk is simple, by Riemann mapping theorem it is
immersed. Hence we have

\begin{cor}
Let $(M,\omega )$ be monotone and $L\subset (M,\omega )$ a monotone
Lagrangian submanifold. Then for generic $\left( J,f\right) ,$ all $J$%
-holomorphic discs in pearl complex moduli spaces of virtual dimension $0,1$%
, are simple, Fredholm regular and immersed.
\end{cor}

Now with immersion condition, similar analysis as the Hamiltonian case can
establish surjectivity so we have

\begin{thm}
\label{thm:A-infty-isomorphism} Let $(M,\omega )$ be a monotone symplectic
manifold and $L\subset (M,\omega ) $ be a monotone Lagrangian submanifold.
Fix a Darboux neighborhood $U$ of $L$ and identify $U $ with a neighborhood
of the zero section in $T^{\ast }L$. Consider $k+1$ Hamiltonian deformations
of $L$ by autonomous Hamiltonian functions $F_{0},\ldots ,F_{k}:M\rightarrow
{\mathbb{R}}$ such that
\begin{equation*}
F_{i}=\chi f_{i}\circ \pi
\end{equation*}%
where $f_0,f_1,\ldots f_k: L\rightarrow {\mathbb{R}}$ are generic Morse
functions, $\chi $ is a cut-off function such that $\chi=1$ on $U$ and
supported nearby U. Now consider
\begin{equation*}
L_{i,\varepsilon }=\func{Graph}(\varepsilon df_{i})\subset U\subset M,\quad
i=0,\ldots ,k.
\end{equation*}%
Assume transversality of $L_{i}$'s of the type given in \cite{FO}. Consider
the intersections $p_{i}\in L_{i}\cap L_{i+1}$ such that
\begin{equation*}
\mathop{\kern0pt{\rm dim}}\nolimits{\mathcal{M}}(L_{0},\ldots
,L_{k};p_{0},\ldots ,p_{k})=0,\,1.
\end{equation*}%
Then when $\varepsilon $ is sufficiently small, the moduli space ${\mathcal{M%
}}(L_{0,\varepsilon },\ldots ,L_{k,\varepsilon };p_{0},\ldots ,p_{k})$ is
diffeomorphic to the moduli space of pearl complex defined in \cite%
{biran-cor-1}.
\end{thm}

We separately state the following special case of Theorem \ref%
{thm:A-infty-isomorphism}.

\begin{thm}
\label{thm:isomorphism} Let $(M,\omega)$ and $L$ be as in Theorem \ref%
{thm:A-infty-isomorphism}. Then there exists some $\varepsilon_0 > 0$ such
that the map
\begin{equation*}
\Theta_\varepsilon:(CF_*(L,L), \partial) \to (C(f,\rho,J), d)
\end{equation*}
induced by the diffeomorphism between the two relevant moduli spaces
defining the boundary maps for the two complexes induces a chain
isomorphism. In particular it the induces an isomorphism in homology
\begin{equation*}
QH_*(L) \cong HF_*(L,L)
\end{equation*}
for all $0 < \varepsilon \leq \varepsilon_0$.
\end{thm}

\section{Discussion and an example\label{discussion}}

\subsection{The immersion condition: an example of adiabatic limit in $%
\mathbb{CP}^{n}$}

Consider the standard $\mathbb{CP}^{n}$ equipped with Fubini-Study metric $\
g_{FS}$. Given a Morse function $f:\mathbb{CP}^{n}\rightarrow \mathbb{R}$,
we consider the Floer equation with the Hamiltonian $\varepsilon f$:%
\begin{equation*}
\frac{\partial u}{\partial \tau }+J_{0}\frac{\partial u}{\partial t}%
=-\varepsilon \,\nabla f\,\left( u\right) .
\end{equation*}%
Let $U_{0}\simeq \mathbb{C}^{n}$ be an affine chart, where
\begin{eqnarray*}
U_{0} &=&\left\{ \left[ z_{0}:z_{1}:\cdots :z_{n}\right] |z_{i}\in \mathbb{C}%
,\text{ }z_{0}\neq 0\in \mathbb{C}\right\} \\
&=&\left\{ \left[ 1:z_{1}:\cdots :z_{n}\right] |z_{i}\in \mathbb{C}\right\} ,%
\text{ \ }\left( i=1,2,\cdots ,n\right) .
\end{eqnarray*}%
Let
\begin{equation*}
\chi :\left[ -l,l\right] \mathbb{\rightarrow }U_{0}\simeq \left( \mathbb{C}%
^{n},g_{FS}\right)
\end{equation*}%
be a segment of a Morse trajectory of $f$ \ with end points $\chi \left( \pm
l\right) =p_{\pm }\in \mathbb{C}^{n}$, and assume the full Morse trajectory
extending $\chi $ to two critical endpoints is contained in the unit ball $%
B_{1}=\left\{ \left\vert z\right\vert \leq 1\right\} \subset $ $U_{0}$. Let $%
u_{\pm }:\left( \mathbb{R\times }S^{1},o_{\pm }\right) \rightarrow
U_{0}\simeq \left( \mathbb{C}^{n},g_{FS}\right) $ be two holomorphic curves,
\begin{equation*}
u_{\pm }\left( \tau ,t\right) =u\left( z\right) =A_{\pm }z^{\pm 1}+p_{\pm },
\end{equation*}%
where $A_{\pm }$ are two complex linearly independent vectors in $\mathbb{C}%
^{n}$, and
\begin{equation*}
z=e^{2\pi \left( \tau +it\right) },\text{ \ \ }o_{\pm }=\left\{ -\pm \infty
\right\} \times S^{1}.
\end{equation*}%
We thus have the \textquotedblleft disk-flow-disk\textquotedblright\ element
$\left( u_{-},\chi ,u_{+}\right) $ in $\mathbb{CP}^{n}$ with immersed joint
points
\begin{equation*}
p_{\pm }=\chi \left( \pm l\right) =u_{\pm }\left( o_{\pm }\right) .
\end{equation*}

\bigskip We can construct a family of maps $u_{\varepsilon }:\mathbb{R\times
}S^{1}\rightarrow \mathbb{C}^{n+1}$ approximately satisfies the above Floer
equation when $\varepsilon $ is small. Let
\begin{equation*}
u_{\varepsilon }\left( \tau ,t\right) =\varepsilon ^{\nu }e^{-2\pi \frac{l}{%
\varepsilon }}\left( A_{+}z+A_{-}z^{-1}\right) +\chi \left( \varepsilon \tau
\right)
\end{equation*}%
with a fixed $\nu >0$. Then%
\begin{equation*}
\frac{\partial u_{\varepsilon }}{\partial \tau }+J_{0}\frac{\partial
u_{\varepsilon }}{\partial t}=\varepsilon \nabla f\left( \chi \left(
\varepsilon \tau \right) \right) \approx \varepsilon \nabla f\left(
u_{\varepsilon }\left( \tau ,t\right) \right) ,
\end{equation*}%
where $\chi \left( \varepsilon \tau \right) \approx u_{\varepsilon }\left(
\tau ,t\right) $ for $\left\vert \tau \right\vert \leq l/\varepsilon $ is
justified by the following Proposition, and for $\left\vert \tau \right\vert
>l/\varepsilon $, the following Proposition also proves $u_{\varepsilon
}\approx \varepsilon ^{\nu }e^{-2\pi \frac{l}{\varepsilon }}\left(
A_{+}z+A_{-}z^{-1}\right) $, hence
\begin{equation*}
\frac{\partial u_{\varepsilon }}{\partial \tau }+J_{0}\frac{\partial
u_{\varepsilon }}{\partial t}\approx \left( \frac{\partial }{\partial \tau }%
+J_{0}\frac{\partial }{\partial t}\right) \left[ \varepsilon ^{\nu }e^{-2\pi
\frac{l}{\varepsilon }}\left( A_{+}z+A_{-}z^{-1}\right) \right] =0\approx
\varepsilon \nabla f\left( u_{\varepsilon }\left( \tau ,t\right) \right) .
\end{equation*}

Although $u_{\varepsilon }$ is only an approximate solution, it is very
explicit and illustrates the mechanism of adiabatic convergence. It also
indicates the convergence rates and $L^{2}$-energy distribution of true
solutions on different regions of $\mathbb{R\times }S^{1}$.

\begin{prop}
The adiabatic limit of $u_{\varepsilon }\left( \tau ,t\right) $ as $%
\varepsilon \rightarrow 0$ is the \textquotedblleft
disk-flow-disk\textquotedblright\ configuration $\left( u_{-},\chi
,u_{+}\right) $.
\end{prop}

\begin{proof}
We recall some useful inequalities of the Fubini-Study metric $g_{FS}$ on
the affine chart $\mathbb{C}^{n}$. Let $g_{st}$ be the standard Euclidean
metric on $\mathbb{C}^{n}$ and the Euclidean norm be $\left\vert \cdot
\right\vert $. Note that $g_{FS}\left( z\right) \leq \frac{1}{1+\left\vert
z\right\vert ^{2}}g_{st}\left( z\right) \leq g_{st}\left( z\right) $ for any
$z\in \mathbb{C}^{n}$, so for any vectors $p,q\in \mathbb{C}^{n}$, we have%
\begin{eqnarray}
\text{dist}_{g_{FS}}\left( p,q\right)  &\leq &\left\vert p-q\right\vert ,
\notag \\
\text{dist}_{g_{FS}}\left( p,q\right)  &\leq &2\frac{\left\vert
p-q\right\vert }{\left\vert p\right\vert }\text{ }\ \text{if }\left\vert
p-q\right\vert <\frac{\left\vert p\right\vert }{2},  \label{g_FS}
\end{eqnarray}%
where the second inequality is because both $p$ and $q$ are outside the
Euclidean ball of radius $\left\vert p\right\vert /2$, and the Fubini-Study
metric satisfies $g_{FS}\leq \frac{1}{\left\vert z\right\vert ^{2}}g_{st}$.

Let $R\left( \varepsilon \right) =l/\varepsilon $. We first check that $%
d_{H}\left( u_{\varepsilon }\left( \tau ,t\right) \left( \left[ -R\left(
\varepsilon \right) ,R\left( \varepsilon \right) \right] \times S^{1}\right)
,\chi \left( \left[ -l,l\right] \right) \right) \rightarrow 0$.

(i) For $\left\vert \tau \right\vert \leq R\left( \varepsilon \right)
=l/\varepsilon $: we have
\begin{eqnarray*}
\text{dist}_{g_{FS}}\left( u_{\varepsilon }\left( \tau ,t\right) ,\chi
\left( \varepsilon \tau \right) \right) &\leq &\left\vert u_{\varepsilon
}\left( \tau ,t\right) -\chi \left( \varepsilon \tau \right) \right\vert \\
&=&\left\vert \varepsilon ^{\nu }e^{-2\pi \frac{l}{\varepsilon }}\left(
A_{+}z+A_{-}z^{-1}\right) \right\vert \\
&\leq &\varepsilon ^{\nu }e^{-2\pi \frac{l}{\varepsilon }}\cdot \left(
\left\vert A_{+}\right\vert e^{2\pi \frac{l}{\varepsilon }}+\left\vert
A_{+}\right\vert e^{2\pi \frac{l}{\varepsilon }}\right) \\
&=&\varepsilon ^{\nu }\left( \left\vert A_{-}\right\vert +\left\vert
A_{+}\right\vert \right) \rightarrow 0
\end{eqnarray*}%
uniformly as $\varepsilon \rightarrow 0$.

Let
\begin{equation}
b\left( \varepsilon \right) =-\frac{1}{2\pi }\ln \varepsilon ,\text{ \ }%
S\left( \varepsilon \right) =\nu b\left( \varepsilon \right) ,  \label{be}
\end{equation}%
and the shift of $u_{\pm }$ be
\begin{equation*}
u_{\pm }^{\varepsilon }\left( \tau ,t\right) :=u_{\pm }\left( \tau -\pm
l/\varepsilon -\nu b\left( \varepsilon \right) ,t\right) =\varepsilon ^{\nu
}e^{-2\pi \frac{l}{\varepsilon }}A_{\pm }z^{\pm 1}+p_{\pm }.
\end{equation*}

(ii) For $\left\vert \tau \right\vert \geq R\left( \varepsilon \right)
=l/\varepsilon $: we consider two cases, one with $l/\varepsilon \leq
\left\vert \tau \right\vert \leq l/\varepsilon +2\nu b\left( \varepsilon
\right) $ and the other with $\left\vert \tau \right\vert \geq l/\varepsilon
+2\nu b\left( \varepsilon \right) $, separately.

If $l/\varepsilon \leq \tau \leq l/\varepsilon +2\nu b\left( \varepsilon
\right) $, then
\begin{eqnarray*}
\text{dist}_{g_{FS}}\left( u_{\varepsilon }\left( \tau ,t\right)
,u_{+}^{\varepsilon }\left( \tau ,t\right) \right) &\leq &\left\vert
u_{\varepsilon }\left( \tau ,t\right) -u_{+}^{\varepsilon }\left( \tau
,t\right) \right\vert \\
&\leq &\left\vert \chi _{\varepsilon }\left( \tau \right) -p_{+}\right\vert
+\left\vert \varepsilon ^{\nu }e^{-2\pi \frac{l}{\varepsilon }%
}A_{-}z^{-1}\right\vert \\
&\leq &\left\vert \varepsilon \nabla f\right\vert _{C^{0}}\cdot \left\vert
\tau -l/\varepsilon \right\vert +\varepsilon ^{a}\left\vert A_{-}\right\vert
\\
&\leq &2\left\vert \nabla f\right\vert _{C^{0}}\nu \cdot \varepsilon b\left(
\varepsilon \right) +\varepsilon ^{a}\left\vert A_{-}\right\vert \rightarrow
0
\end{eqnarray*}%
uniformly as $\varepsilon \rightarrow 0$.

If $\tau \geq l/\varepsilon +2\nu b\left( \varepsilon \right) $, we have
\begin{eqnarray*}
\left\vert u_{+}^{\varepsilon }\left( \tau ,t\right) \right\vert
&=&\left\vert \varepsilon ^{\nu }e^{-2\pi \frac{l}{\varepsilon }%
}A_{+}z\right\vert =\left\vert A_{+}\right\vert e^{2\pi \left( \tau
-l/\varepsilon -\nu b\left( \varepsilon \right) \right) } \\
&\geq &\left\vert A_{+}\right\vert e^{2\pi b\left( \varepsilon \right) \nu
}=\left\vert A_{+}\right\vert \varepsilon ^{-\nu }\rightarrow \infty , \\
\left\vert u_{\varepsilon }\left( \tau ,t\right) -u_{+}^{\varepsilon }\left(
\tau ,t\right) \right\vert  &=&\left\vert \varepsilon ^{\nu }e^{-2\pi \frac{l%
}{\varepsilon }}A_{-}z^{-1}+\chi \left( \varepsilon \tau \right) \right\vert
\leq \left\vert \varepsilon ^{\nu }e^{-2\pi \frac{l}{\varepsilon }%
}A_{-}z^{-1}\right\vert +\left\vert \chi \left( \varepsilon \tau \right)
\right\vert  \\
&\leq &\varepsilon ^{a}\left\vert A_{-}\right\vert e^{-2\pi \left(
l/\varepsilon +\tau \right) }+1\rightarrow 1<\frac{\left\vert
u_{+}^{\varepsilon }\left( \tau ,t\right) \right\vert }{2},
\end{eqnarray*}%
when $\varepsilon $ is small, hence by $\left( \ref{g_FS}\right) $
\begin{eqnarray*}
\text{dist}_{g_{FS}}\left( u_{\varepsilon }\left( \tau ,t\right)
,u_{+}^{\varepsilon }\left( \tau ,t\right) \right)  &\leq &2\frac{\left\vert
u_{\varepsilon }\left( \tau ,t\right) -u_{+}^{\varepsilon }\left( \tau
,t\right) \right\vert }{\left\vert u_{\varepsilon }\left( \tau ,t\right)
\right\vert } \\
&\leq &2\frac{\left\vert \varepsilon ^{\nu }e^{-2\pi \frac{l}{\varepsilon }%
}A_{-}z^{-1}+\chi \left( \varepsilon \tau \right) \right\vert +\left\vert
p_{+}\right\vert }{\left\vert \varepsilon ^{\nu }e^{-2\pi \frac{l}{%
\varepsilon }}A_{+}z\right\vert -\left\vert \varepsilon ^{\nu }e^{-2\pi
\frac{l}{\varepsilon }}A_{-}z^{-1}+\chi \left( \varepsilon \tau \right)
\right\vert } \\
&\leq &2\frac{\left( \varepsilon ^{\nu }\left\vert A_{-}\right\vert e^{-2\pi
\left( l/\varepsilon +\tau \right) }+1\right) +1}{\left\vert
A_{+}\right\vert \varepsilon ^{-\nu }-\left( \varepsilon ^{a}\left\vert
A_{-}\right\vert e^{-2\pi \left( l/\varepsilon +\tau \right) }+1\right) } \\
&\leq &2\frac{3}{\left\vert A_{+}\right\vert \varepsilon ^{-\nu }-2}%
\rightarrow 0
\end{eqnarray*}%
uniformly as $\varepsilon \rightarrow 0$. Combining these we have for $\tau
\geq R\left( \varepsilon \right) $,%
\begin{equation*}
\text{dist}_{g_{FS}}\left( u_{\varepsilon }\left( \tau ,t\right)
,u_{+}^{\varepsilon }\left( \tau ,t\right) \right) \rightarrow 0,\text{ }
\end{equation*}%
or equivalently%
\begin{equation*}
\text{dist}_{g_{FS}}\left( u_{\varepsilon }\left( \tau +R\left( \varepsilon
\right) +\nu b\left( \varepsilon \right) ,t\right) ,u_{+}\left( \tau
,t\right) \right) \rightarrow 0
\end{equation*}%
uniformly as $\varepsilon \rightarrow 0$ on $[-\nu b\left( \varepsilon
\right) ,+\infty )\times S^{1}$, especially for\ any $[K,+\infty )\times
S^{1}$ for any fixed $K\in \mathbb{R}$.

The case when $\tau \leq -l/\varepsilon $ is similar; we have
\begin{equation*}
\text{dist}_{g_{FS}}\left( u_{\varepsilon }\left( \tau -R\left( \varepsilon
\right) -\nu b\left( \varepsilon \right) ,t\right) ,u_{-}\left( \tau
,t\right) \right) \rightarrow 0
\end{equation*}
uniformly as $\varepsilon \rightarrow 0$ on $(-\infty ,\nu b\left(
\varepsilon \right) ]\times S^{1}$, especially for\ any $(-\infty ,K]\times
S^{1}$ for any fixed $K\in \mathbb{R}$.

Next we compute the energy $E\left( u_{\varepsilon }\right) $ on $\Theta
_{\varepsilon }:=\left[ -R\left( \varepsilon \right) ,R\left( \varepsilon
\right) \right] \times S^{1}$. We have
\begin{eqnarray*}
\frac{\partial }{\partial \tau }u_{\varepsilon }\left( \tau ,t\right)
&=&2\pi \varepsilon ^{\nu }e^{-2\pi \frac{l}{\varepsilon }}\left(
A_{+}e^{2\pi \left( \tau +it\right) }-A_{-}e^{-2\pi \left( \tau +it\right)
}\right) +\varepsilon \nabla f\left( \chi \left( \varepsilon \tau \right)
\right)  \\
\frac{\partial }{\partial t}u_{\varepsilon }\left( \tau ,t\right)  &=&2\pi
i\varepsilon ^{\nu }e^{-2\pi \frac{l}{\varepsilon }}\left( A_{+}e^{2\pi
\left( \tau +it\right) }-A_{-}e^{-2\pi \left( \tau +it\right) }\right) \text{
\ \ \ \ \ \ \ \ \ } \\
\left\vert du_{\varepsilon }\left( \tau ,t\right) \right\vert _{g_{st}}
&\leq &C\varepsilon ^{\nu }e^{-2\pi \frac{l}{\varepsilon }}\cdot e^{2\pi
\left\vert \tau \right\vert }+C\varepsilon .
\end{eqnarray*}%
For $\left\vert \tau \right\vert \leq R\left( \varepsilon \right) $,
noticing $\left( \ref{g_FS}\right) $, from the above third inequality we have%
\begin{eqnarray*}
\int_{\left[ -R\left( \varepsilon \right) ,R\left( \varepsilon \right) %
\right] \times S^{1}}\left\vert du_{\varepsilon }\right\vert
_{g_{FS}}^{2}d\tau dt &\leq &\int_{\left[ -R\left( \varepsilon \right)
,R\left( \varepsilon \right) \right] \times S^{1}}\left\vert du_{\varepsilon
}\right\vert _{g_{st}}^{2}d\tau dt \\
&\leq &2C^{2}\int_{\left[ -R\left( \varepsilon \right) ,R\left( \varepsilon
\right) \right] \times S^{1}}\left( \varepsilon ^{2\nu }e^{-4\pi \left(
\frac{l}{\varepsilon }-\left\vert \tau \right\vert \right) }+\varepsilon
^{2}\right) d\tau dt \\
&=&2C^{2}\left[ \varepsilon ^{2\nu }\frac{1-e^{-4\pi \frac{l}{\varepsilon }}%
}{2\pi }+2l\varepsilon \right]  \\
&\leq &\widetilde{C}\left( \varepsilon ^{2\nu }+\varepsilon \right)
\rightarrow 0.
\end{eqnarray*}
\end{proof}

\bigskip If the joint points of $u_{\pm }$ are not immersed, in the next
example we will see extra family of approximate solutions of the Floer
equation beyond our pre-gluing construction. Let
\begin{equation*}
u_{\varepsilon }\left( \tau ,t\right) =\varepsilon ^{\nu }\left[ e^{-2\pi k%
\frac{l}{\varepsilon }}A_{+}z^{k}+e^{-2\pi m\frac{l}{\varepsilon }%
}A_{-}z^{-m}\right] +\beta \left( \varepsilon \right) P\left( z\right) +\chi
\left( \varepsilon \tau \right) ,
\end{equation*}%
where $k,m>0\ $are integers,\ and at least one of them $>1$, $P\left(
z\right) $ is any Laurent polynomial of intermediate degree between $z^{k%
\text{ }}$ and $z^{-m}$ with $\mathbb{C}^{n}$ vector-valued coefficients,
and $\beta \left( \varepsilon \right) $ is a fast vanishing real constant
when $\varepsilon \rightarrow 0$.\ Then for fixed $l>0$, similarly we can
show $u_{\varepsilon }$ has the adiabatic limit $\left( u_{-},\chi
,u_{+}\right) $, where $u_{\pm }\left( z\right) $ are the holomorphic
spheres in $\mathbb{CP}^{n}$ that in the affine chart $U_{0}\simeq \mathbb{C}%
^{n}$,
\begin{equation*}
u_{+}\left( z\right) =A_{+}z^{k}+p_{+},\text{ \ \ \ \ \ \ \ }u_{-}\left(
z\right) =A_{-}z^{-m}+p_{-}
\end{equation*}%
for $z=e^{2\pi \left( \tau +it\right) }$. But other than the approximate
solutions
\begin{equation*}
u_{\varepsilon }\left( \tau ,t\right) =\varepsilon ^{\nu }\left[ e^{-2\pi k%
\frac{l}{\varepsilon }}A_{+}z^{k}+e^{-2\pi m\frac{l}{\varepsilon }%
}A_{-}z^{-m}\right] +\chi \left( \varepsilon \tau \right)
\end{equation*}%
of the Floer equation, now we have extra family of approximate solutions
from the term $\beta \left( \varepsilon \right) P\left( z\right) $. This
gives evidence that we can not prove surjectivity if there is non-immersed
joint point.

The adiabatic degeneration with Lagrangian boundary condition can be
constructed similarly, with the Lagrangian $L=$ $\mathbb{RP}^{n}$ and $\chi
\left( \tau \right) $ inside $\mathbb{RP}^{n}$, and $u_{\varepsilon }:%
\mathbb{R\times }\left[ 0,\frac{1}{2}\right] \rightarrow U_{0}\simeq \mathbb{%
C}^{n}$ (where $\left[ 0,\frac{1}{2}\right] $ should be thought as half of $%
S^{1}=\left[ 0,1\right] /\symbol{126}$)
\begin{equation*}
u_{\varepsilon }\left( \tau ,t\right) =\varepsilon ^{\nu }\left[ e^{-2\pi k%
\frac{l}{\varepsilon }}A_{+}z^{k}+e^{-2\pi m\frac{l}{\varepsilon }%
}A_{-}z^{-m}\right] +\beta \left( \varepsilon \right) P\left( z\right) +\chi
\left( \varepsilon \tau \right)
\end{equation*}%
where $A_{\pm }\in \mathbb{R}^{n}\subset \mathbb{C}^{n}$, and $P\left(
z\right) $ is any Laurent polynomial of intermediate degree between $z^{k%
\text{ }}$ and $z^{-m}$ with $\mathbb{R}^{n}$ vector-valued coefficients.
Then $u_{\varepsilon }$ is an approximate solution of Floer equation $\frac{%
\partial u}{\partial \tau }+J_{0}\frac{\partial u}{\partial t}+\varepsilon
\nabla f\left( u\right) =0$ with Lagrangian boundary condition $u\left(
\mathbb{R\times }\left\{ 0,\frac{1}{2}\right\} \right) \subset L$. As $%
\varepsilon \rightarrow 0\,$, $u_{\varepsilon }$ degenerates to the
\textquotedblleft disk-flow-disk\textquotedblright\ configuration $\left(
u_{-},\chi ,u_{+}\right) $.

\end{document}